\documentclass[a4paper]{amsart}

\RequirePackage{amsmath} 
\RequirePackage{amssymb}
\usepackage{amscd,latexsym,amsthm,amsfonts,amssymb,amsmath,amsxtra}
\usepackage[colorlinks=true,urlcolor=blue,citecolor=blue]{hyperref}
\usepackage{color}
\usepackage[all]{xy}
\usepackage{bm}
\usepackage{mathtools}
\usepackage{ mathrsfs }
\usepackage{xcolor}
\usepackage{enumitem}
\usepackage{thmtools, thm-restate}

\let\Re\undefined
\let\Im\undefined

\DeclareMathOperator{\Re}{Re}
\DeclareMathOperator{\Im}{Im}
\DeclareMathOperator{\Tr}{Tr}

\DeclareMathOperator{\Spec}{Spec}
\newcommand{\floor}[1]{{\left\lfloor#1\right\rfloor}}

	\newcommand{\Res}{\operatorname{Res}}
	\newcommand{\Eis}{\operatorname{Eis}}
	\newcommand{\Ad}{\operatorname{Ad}}
	\newcommand{\fin}{\operatorname{fin}}
	\newcommand{\diag}{\operatorname{diag}}
	\newcommand{\Vol}{\operatorname{Vol}}
	\newcommand{\Ind}{\operatorname{Ind}}
	\newcommand{\RNum}[1]{\uppercase\expandafter{\romannumeral #1\relax}}

\begin{document}
\theoremstyle{plain}
\newtheorem{thm}{Theorem}[section]
	
\newtheorem{cor}[thm]{Corollary}

\newtheorem{thmx}{Theorem}
\renewcommand{\thethmx}{\Alph{thmx}} 

\newtheorem{lemma}[thm]{Lemma}  
\newtheorem{prop}[thm]{Proposition}
	
\theoremstyle{definition}
\newtheorem{defn}[thm]{Definition}
\newtheorem{que}[thm]{Question}
\theoremstyle{remark}
	
\newtheorem{remark}[thm]{Remark}	
\numberwithin{equation}{section}
	
\title[]{Symmetric Spectral Reciprocity for $\mathrm{GL}(2)$ and Uniform Subconvexity}%
\author{Liyang Yang}
\address{Department of Mathematics, Texas A\&M University, College Station, TX 77843, USA}
\email{liyangy@tamu.edu}

\begin{abstract}
We establish an explicit symmetric spectral reciprocity formula for twisted fourth moments of $\mathrm{GL}_{2}$ $L$-functions over an arbitrary number field. A Rankin--Selberg kernel gives a new analytic regularization of the formal Petersson identity for products of Eisenstein series, based on the rank decomposition $4=2+2$. The formula preserves the nonnegativity of the amplified fourth-moment spectral side at general levels.

We further identify a new \emph{square-level concentration phenomenon}: after amplification, the dominant off-diagonal contribution comes from representations whose conductor is the square of a product of two distinct primes. Using square-level local test functions in the relative trace formula, we prove new hybrid subconvexity bounds for twisted $L$-functions. Combining these bounds with symmetric spectral reciprocity and amplification, we obtain the uniform estimate
\begin{align*}
L(1/2,\pi)\ll_{F,\varepsilon} C(\pi)^{\frac{1}{4}-\frac{1}{120}+\varepsilon}
\end{align*}
for every unitary cuspidal automorphic representation $\pi$ of $\mathrm{GL}_{2}/F$ with arbitrary central character. We also obtain stronger bounds for fixed central characters, a prime-level fifth-moment estimate, and applications to Rankin--Selberg and ring class $L$-functions.
\end{abstract}

\maketitle
\tableofcontents

\section{Introduction}

The subconvexity problem for automorphic $L$-functions on $\mathrm{GL}_{2}$
asks for a power saving over the convexity bound
\begin{align*}
L(1/2,\pi) \ll_{F,\varepsilon} C(\pi)^{\frac{1}{4}+\varepsilon}
\end{align*}
for unitary cuspidal representations $\pi$ of $\mathrm{GL}_{2}/F$, where $C(\pi)$ is the analytic conductor.
Such bounds have applications to equidistribution and class group arithmetic.

A powerful strategy for subconvexity is to study higher moments through spectral identities. In the present setting, the starting point is the formal Petersson identity
\begin{equation}
\label{c1.1}
\langle E_{1}E_{2}, E_{3}E_{4} \rangle
=
\langle E_{1}\overline{E_{3}}, \overline{E_{2}}E_{4} \rangle
\end{equation}
for Eisenstein series, which reflects the rank decomposition $4=2+2$. Michel and Venkatesh~\cite{MV10} used a special case in their proof of subconvexity. Since the pairings are divergent, deriving an explicit quantitative spectral reciprocity formula requires regularization; to our knowledge, no such formula has previously been extracted from this identity.

Let $F$ be a number field with ring of adeles $\mathbb{A}_{F}$.
Fix complex parameters $s_{j}\in\mathbb{C}$ for $1\le j\le 4$,
and an automorphic datum $\mathfrak{X}=(\boldsymbol{\chi},\boldsymbol{\omega},\boldsymbol{\Phi})$ consisting of unitary Hecke characters
$\chi_{j},\omega_{j}$ of $F^{\times}\backslash\mathbb{A}_{F}^{\times}$
with $\omega_{1}\omega_{2}=\omega_{3}\omega_{4}$, and Schwartz--Bruhat functions $\Phi_{j}$.
These data determine Eisenstein series
$E_{j}=E_{j}(\cdot,s_{j})$ in the isobaric representations
$\chi_{j}|\cdot|^{s_{j}}\boxplus\chi_{j}^{-1}\omega_{j}|\cdot|^{-s_{j}}$.

Using a Rankin--Selberg kernel, we regularize \eqref{c1.1} and obtain an explicit \emph{symmetric spectral reciprocity} formula for twisted fourth moments of $\mathrm{GL}_{2}$ $L$-functions. For paired data, the spectral side is nonnegative.

We identify a \emph{square-level concentration phenomenon} in the amplified dual spectrum. We exploit it by constructing local test functions for these square conductors and proving refined orbital-integral estimates in the relative trace formula of \cite{Yan23c}, including bounds for nonsplit character sums. The resulting hybrid subconvexity bounds for twisted $L$-functions, combined with symmetric reciprocity, yield the uniform estimate
\begin{align*}
L(1/2,\pi)\ll_{F,\varepsilon} C(\pi)^{\frac{1}{4}-\frac{1}{120}+\varepsilon}
\end{align*}
for all unitary cuspidal representations $\pi$ of $\mathrm{GL}_{2}/F$.

As a further application of the reciprocity formula, we obtain a Blomer--Khan-type fifth-moment estimate at prime level over an arbitrary number field.

\subsection{A Symmetric Spectral Reciprocity}
The following theorem relates the weighted fourth moments arising from the two spectral expansions of \eqref{c1.1}.
\begin{thmx}\label{thmA}
Let $\mathfrak{X}$ be an automorphic datum (see \textsection\ref{sec2.2}). Define the region  
\begin{multline*}
\mathcal{R}^{\heartsuit}:=\big\{\mathbf{s}\in \mathbb{C}^{4}:\ |\Re(s_{1})|+|\Re(s_{2})|<1/2,\ |\Re(s_{3})|+|\Re(s_{4})|<1/2,\\
|\Re(s_{1})|+|\Re(s_{3})|<1/2,\ |\Re(s_{2})|+|\Re(s_{4})|<1/2\big\}.
\end{multline*}
We have the following identity of meromorphic functions on $\mathcal{R}^{\heartsuit}$:
\begin{multline}\label{1.3}
\mathcal{J}_{\mathrm{\Spec}}^{\heartsuit}(\mathbf{s},\mathfrak{X})=\ \mathcal{I}_{\Spec}^{\heartsuit}(\mathbf{s},\mathfrak{X})+\widetilde{\Psi}(W_{1},\overline{W_{3}},\mathcal{C}_{2}\overline{h_{4}})-\widetilde{\Psi}(W_{1},W_{2}^{*},\overline{\mathcal{C}_{3}}\,\overline{h_{4}})\\
+\widetilde{\Psi}(W_{\mathcal{C}_{1},\overline{\mathcal{C}_{3}}},W_{2}^{*},\overline{h_{4}})-\widetilde{\Psi}(W_{\mathcal{C}_{1},\mathcal{C}_{2}},\overline{W_{3}},\overline{h_{4}})\\
+\widetilde{\Psi}(W_{1},W_{\mathcal{C}_{2},\overline{h_{4}^{\Diamond}}}^{*},\overline{h_{3}})-\widetilde{\Psi}(W_{1},W_{\overline{\mathcal{C}_{3}},\overline{h_{4}^{\Diamond}}}^{*},h_{2}),
\end{multline}
where $\mathcal{C}_{j}$ is the constant term of $E_{j}(\cdot,s_{j})$. See \textsection\ref{sec2.1.1}--\textsection\ref{sec2.1.auxiliary} for the sections and Whittaker functions, \textsection\ref{sec2.1.rs-periods} for the Rankin--Selberg periods $\widetilde{\Psi}$, and Definitions~\ref{defn3.2} and~\ref{defn3.7} for the spectral sides.
\end{thmx}

At the central parameters, the original moment is
\begin{multline*}
\mathcal{J}_{\mathrm{\Spec}}^{\heartsuit}(\mathbf{0},\mathfrak{X})=\int_{\pi\in \mathcal{A}([G],\omega_{1}\omega_{2})} L(1/2,\pi\times\chi_{3}^{-1}\chi_{4}^{-1})L(1/2,\pi\times\chi_{3}\omega_{3}^{-1}\chi_{4}^{-1})\\
L(1/2,\widetilde{\pi}\times\chi_{1}\chi_{2})L(1/2,\widetilde{\pi}\times\chi_{1}^{-1}\omega_{1}\chi_{2})\alpha_{\mathfrak{X}}(\pi)d\mu_{\pi},
\end{multline*}
and the dual moment is
\begin{multline*}
\mathcal{I}_{\mathrm{Spec}}^{\heartsuit}(\mathbf{0}, \mathfrak{X})=\int_{\sigma\in \mathcal{A}([G],\omega_{1}\omega_{3}^{-1})} L(1/2,\sigma\times\chi_{2}\chi_{4}^{-1})
L(1/2,\sigma\times\chi_{2}^{-1}\omega_{2}\chi_{4}^{-1})\\
L(1/2,\widetilde{\sigma}\times\chi_{1}\chi_{3}^{-1})L(1/2,\widetilde{\sigma}\times\chi_{1}^{-1}\omega_{1}\chi_{3}^{-1})\widehat{\alpha_{\mathfrak{X}}}(\sigma)d\mu_{\sigma}.
\end{multline*}
Here $\mathcal{A}([G],\omega_{1}\omega_{2})$ denotes the automorphic spectrum of $G=\mathrm{GL}_{2}/F$ with central character $\omega_{1}\omega_{2}$. The spectral integrals denote the cuspidal sums and continuous integrals in Definitions~\ref{defn3.2} and~\ref{defn3.7}; the weights $\alpha_{\mathfrak{X}}$ and $\widehat{\alpha}_{\mathfrak{X}}$ are obtained by factoring out the displayed products of $L$-functions from the corresponding periods.

All six correction terms in \eqref{1.3} are explicit integrals in the initial data.

\begin{remark}
The formula \eqref{1.3} is a symmetric analogue of the reciprocity formulas of Blomer and Khan \cite{BK19,BK19a}, arising from the rank decomposition $4=2+2$. It holds over arbitrary number fields, with general levels and central characters.
\end{remark}

For the subconvexity application of Theorem \ref{thmA}, fix a nonzero integral ideal $\mathfrak{q}$ and a unitary Hecke character $\omega$ with conductor ideal $\mathfrak{q}'\supseteq\mathfrak{q}$.
Write $\mathcal{F}_{0}(\mathfrak{q},\omega)$ for the unitary cuspidal representations of $\mathrm{GL}_{2}/F$ with central character $\omega$ and level dividing $\mathfrak{q}$.
At infinity, choose $\mathbf{C}_{v}>1$ and put $\mathbf{C}_{\infty}=\prod_{v\mid\infty}\mathbf{C}_{v}$.

Take $\chi_{1}=\chi_{2}=\chi_{3}=\chi_{4}= \omega_{1}=\omega_{3}=\mathbf{1}$ and $\omega_{2}=\omega_{4}=\omega$, and choose
\begin{align*}
\alpha_{\mathfrak{X}}(\pi)\approx \lambda_{\pi}(\mathfrak{n})\mathbf{1}_{\pi\in \mathcal{F}_{0}(\mathfrak{q},\omega),\ \& \ C_{v}(\pi)\leq \mathbf{C}_{v},\ v\mid\infty}.
\end{align*}
Here $C_{v}(\pi)$ is the analytic conductor at $v$. We show that the dual Archimedean weight decays exponentially and obtain the following schematic relation for the main contributions:
\begin{multline}\label{f1.1}
\sum_{\substack{\pi\in \mathcal{F}_{0}(\mathfrak{q},\omega)\\
C_{v}(\pi)\leq \mathbf{C}_{v},\ v\mid\infty}}\lambda_{\pi}(\mathfrak{n})|L(1/2,\pi)|^{4}\rightsquigarrow D_{F}^{\frac{3}{2}}\mathbf{C}_{\infty}N_{F}(\mathfrak{q})N_{F}(\mathfrak{n})^{-\frac{1}{2}}\\
+\mathbf{C}_{\infty}^{\frac{1}{2}}N_{F}(\mathfrak{q})^{\frac{1}{2}}N_{F}(\mathfrak{n})^{-\frac{1}{2}}\big[\mathbf{C}_{\infty}C_{\infty}(\omega)^{-1}\big]^{\vartheta}\sum_{\sigma}\sum_{\mathfrak{a}\mid\mathfrak{q}\mathfrak{q}'^{-1}}|\lambda_{\sigma}(\mathfrak{a})||L(1/2,\sigma)|^{3}|L(1/2,\sigma\times\overline{\omega})|.
\end{multline}
Here $D_{F}$ is the absolute discriminant of $F$, $\vartheta$ is an admissible Ramanujan exponent, and the dual sum is over $\sigma\in\mathcal{F}_{0}(\mathfrak{n},\mathbf{1})$ unramified at all Archimedean places, with $C_{\infty}(\sigma)\leq(\mathbf{C}_{\infty} N_{F}(\mathfrak{q}\mathfrak{n}))^{\varepsilon}$. The omitted correction terms are absorbed in the subsequent moment estimates; see Remark~\ref{rmk10.degenerate}. In \eqref{f1.1}, the dual weight contains $N_{F}(\mathfrak{q})^{1/2}$, but no additional power of $D_{F}$.

For $F=\mathbb{Q}$, our formula recovers the prime-level arithmetic weights of \cite{BK19,BK19a}.

\begin{remark}
In \cite{Yan26b}, simple supercuspidal projections in symmetric reciprocity yield level-aspect Weyl bounds for $\mathrm{GL}_{2}$ over arbitrary number fields, uniformly in the central character and including non-self-dual families of conductor exponent $3$. In \cite{Yan26c}, the Rankin--Selberg kernel regularization yields a non-split Motohashi formula for class group moments and uniform positive-proportion nonvanishing in CM relative class group families. Both applications are independent of the uniform subconvexity bound proved here.
\end{remark}

\subsection{Square-Level Concentration and Refined Hybrid Subconvexity}
We bound the dual moment in \eqref{f1.1} using hybrid subconvexity for $L(1/2,\sigma\times\overline{\omega})$.
The hybrid estimate of \cite[Corollary 1.9]{Yan23c} yields a Burgess-type bound in the $\omega$-aspect.

The dominant off-diagonal contribution after amplification comes from pairs of distinct amplifier primes $\mathfrak{p}_{1}$ and $\mathfrak{p}_{2}$ and is supported on representations $\sigma$ of conductor $\mathfrak{p}_{1}^{2}\mathfrak{p}_{2}^{2}$. The other dual conductors give smaller contributions, bounded by the general hybrid estimate of \cite{Yan23c}. In the relative trace formula (RTF), we replace the Gross--Prasad test vector used there, at one or both ramified primes, by special test functions whose local integrals reduce to finite-field character sums.

Viewing the split and nonsplit character sums as trace functions associated with Kummer sheaves, we apply Katz's theorem and quasi-orthogonality to obtain the following schematic bounds (Theorems \ref{thm10.7} and \ref{thm10.8} in \textsection\ref{sec10.6}).

\begin{thm}[See Theorem \ref{thm10.7} for a precise formulation]
Let $\varepsilon>0$ and assume 
$L\gg_{F,\varepsilon} 
C_{\infty}(\sigma)^{\frac{1}{2}+\varepsilon}
C_{\fin}(\sigma)^{1+\varepsilon}
C_{\fin}(\overline{\omega})^{\varepsilon}$. Then for the relevant $\sigma$ and $\omega$ we have
\begin{multline*}
|L(1/2,\sigma\times\overline{\omega})|^{2}
\ll_{F,\varepsilon}  
L^{-1+\varepsilon}\,
C_{\fin}(\sigma)^{\frac{1}{2}+\varepsilon}
C_{\fin}(\overline{\omega})^{1+\varepsilon}
C_{\infty}(\sigma)^{1+\varepsilon}
C_{\infty}(\sigma\times\overline{\omega})^{\frac{1}{2}+\varepsilon}
\\
+\, 
L^{1+\varepsilon}\,
C_{\fin}(\sigma)^{\frac{1}{4}+\varepsilon}
C_{\fin}(\overline{\omega})^{\frac{1}{2}+\varepsilon}
C_{\infty}(\sigma)^{\frac{1}{2}+\varepsilon}
C_{\infty}(\sigma\times\overline{\omega})^{\frac{1}{4}+\varepsilon}.
\end{multline*}
\end{thm}

\subsection{Applications to the Subconvexity Problem}

Amplifying \eqref{f1.1}, we combine \cite[Corollary~1.9]{Yan23c} with its square-level refinements (Theorems \ref{thm10.7} and \ref{thm10.8}) over complementary conductor ranges to obtain the following explicit subconvexity bound.

\begin{thmx}\label{thmB}
Let $F$ be a number field and $\pi$ a unitary cuspidal automorphic representation of $\mathrm{GL}_{2}/F$ with central character $\omega$. Write $C(\pi),C(\omega)$ for the analytic conductors. For $0<\varepsilon<10^{-2}$ and an admissible Ramanujan exponent $0\leq\vartheta\leq7/64$, we have
\begin{multline}\label{1.1}
L(1/2,\pi)\ll_{F,\varepsilon} \min\Big\{C(\pi)^{\frac{13}{60}+\frac{\vartheta}{15}+\varepsilon}C(\omega)^{\frac{1}{40}-\frac{\vartheta}{15}}\mathbf{1}_{C(\pi)\leq C(\omega)^{\frac{9-14\vartheta}{7(1-2\vartheta)}}}\\
+C(\pi)^{\frac{5}{22}+\frac{\vartheta}{22}+\varepsilon}C(\omega)^{\frac{1}{88}-\frac{\vartheta}{22}}\mathbf{1}_{C(\omega)^{\frac{9-14\vartheta}{7(1-2\vartheta)}}<C(\pi)\leq C(\omega)^{\frac{35-52\vartheta}{26(1-2\vartheta)}}}\\
+C(\pi)^{\frac{1}{4}-\frac{1-2\vartheta}{31}+\varepsilon}C(\omega)^{\frac{3}{124}-\frac{2\vartheta}{31}}\mathbf{1}_{C(\omega)^{\frac{35-52\vartheta}{26(1-2\vartheta)}}<C(\pi)\leq C(\omega)^{\frac{2(8-11\vartheta)}{11(1-2\vartheta)}}}\\
+C(\pi)^{\frac{19}{84}+\frac{\vartheta}{21}+\varepsilon}C(\omega)^{\frac{1}{84}-\frac{\vartheta}{21}}\mathbf{1}_{C(\pi)>C(\omega)^{\frac{2(8-11\vartheta)}{11(1-2\vartheta)}}},\\
C(\pi)^{\frac{5}{24}+\frac{\vartheta}{12}+\varepsilon}C(\omega)^{\frac{1}{24}-\frac{\vartheta}{12}+\varepsilon},\ 
C(\pi)^{\frac{7}{32}+\frac{\vartheta}{16}+\varepsilon}C(\omega)^{\frac{3}{128}-\frac{\vartheta}{16}}\Big\}.
\end{multline}	
\end{thmx}

Theorem \ref{thmB} sharpens the subconvex bounds for automorphic representations over number fields established by Wu \cite{Wu22}. In particular, replacing $\pi$ with $\pi \otimes |\cdot|^{it}$ in \eqref{1.1} yields subconvexity for $L(1/2+it,\pi)$ uniformly in all parameters. 

\subsubsection{Uniform Subconvexity for $\mathrm{GL}_{2}$} 
Since $C(\omega)\leq C(\pi)$, \eqref{1.1} gives the following uniform bound.
\begin{cor}[Uniform Subconvexity]\label{cor1.1}
Let $F$ be a number field, and let $\pi$ be a unitary cuspidal automorphic representation of $\mathrm{GL}_{2}/F$. For $j\in \mathbb{Z}_{\geq 0}$, let $L^{(j)}(s,\pi)$ denote the $j$-th derivative of $L(s,\pi)$. We have
\begin{equation}\label{1.5}
L^{(j)}(1/2,\pi)\ll_{F,j,\varepsilon} C(\pi)^{\frac{1}{4}-\frac{1}{120}+\varepsilon}.	
\end{equation}
\end{cor}
For $j=0$, the implied constant is $\ll_{\varepsilon}D_{F}^{A_{\varepsilon}}$, where $A_{\varepsilon}$ depends only on $\varepsilon$, with no restriction on $[F:\mathbb{Q}]$; here $C(\pi)$ includes $D_{F}^{2}$. See \textsection\ref{sec10.field}.

\begin{remark}
The saving $1/120$ also appears in \cite[Theorem 3]{DFI94} for holomorphic cusp forms of \textit{trivial} nebentypus under hypothesis (15) of \emph{loc.\ cit.}; here the same saving follows by a different method, for arbitrary central characters and without that hypothesis.
\end{remark}

For comparison, Blomer and Khan \cite[Theorem 1]{BK19a} considered Hecke--Maass cusp forms $f$ of prime level $q$, primitive central character, and Laplace eigenvalue $1/4+t_{f}^{2}$. They proved
\begin{equation}\label{eq1.6}
L(1/2+i\tau,f)\ll_{\tau,\varepsilon} (qT)^{\varepsilon}\big[q^{\frac{1}{4}-\frac{1}{128}}T^{\frac{1}{2}-\frac{1-2\vartheta}{20}}+q^{\frac{1}{8}}T^{\frac{1}{2}}\big],
\end{equation}
where $\tau\in \mathbb{R}$ and $T=1+|t_{f}|$, so that the analytic conductor is $qT^{2}$. 

The estimate \eqref{eq1.6} is not uniform in the twisting parameter $\tau$. In the setting of \cite{BK19a}, our bound \eqref{1.5} instead yields
\begin{equation}\label{1.7}
L(1/2+i\tau,f)=L(1/2,f\otimes|\cdot|^{i\tau})\ll_{\varepsilon} (q(1+|\tau+t_{f}|)(1+|\tau-t_{f}|))^{\frac{1}{4}-\frac{1}{120}+\varepsilon}.
\end{equation}

In the extreme conductor-dropping case $\tau=t_{f}$, the analytic conductor is $qT$ rather than $qT^{2}$, and \eqref{1.7} becomes
\begin{equation}\label{1.10}
L(1/2+it_{f},f)\ll_{\varepsilon} (q(1+|t_{f}|))^{\frac{1}{4}-\frac{1}{120}+\varepsilon}.
\end{equation} 

At these special values, subconvex bounds follow implicitly from \cite{MV10}, while \cite{Wu22} gives the explicit exponent $1/4-5/2048$, compared with $1/4-1/120$ in \eqref{1.10}.

\subsubsection{Subconvexity of Ring Class $L$-functions}
Applying Corollary \ref{cor1.1} to the automorphic induction of $\chi$ gives the following bound. If the induction is not cuspidal, it is an isobaric sum of two Hecke characters, and the corresponding $\mathrm{GL}_{1}$ subconvex bounds suffice.
\begin{cor}\label{cor1.3}
Let $F$ be a number field of degree $d=[F:\mathbb{Q}]$. Let $E/F$ be a quadratic extension. Let $\mathcal{O}$ be an $\mathcal{O}_{F}$-order in $E$, and let $D$ be the absolute norm of the relative discriminant of $\mathcal{O}$ over $F$. Let $\chi$ be a character of the Picard group $\mathrm{Pic}(\mathcal{O})$. For $t\in \mathbb{R}$, we have 
\begin{align*}
L(1/2+it,\chi)\ll_{F,\varepsilon} ((1+|t|)^{2d}D)^{\frac{1}{4}-\frac{1}{120}+\varepsilon}.	
\end{align*}
\end{cor}

Corollary \ref{cor1.3} improves the exponent of \cite[Corollary 1]{BHM07} from $1/4-1/1889$ to $1/4-1/120$. Such subconvexity bounds were a key ingredient in the cubic analogue of Duke's equidistribution theorem, established by Einsiedler, Lindenstrauss, Michel and Venkatesh \cite[Hypotheses A.1 \& A.2]{ELMV11}. For further applications of Corollary \ref{cor1.3} to Picard groups, see \textsection\ref{sec1.3}. 

\subsubsection{Fixed Central Characters}
For fixed central characters, Theorem \ref{thmB} gives the following bound.
\begin{cor}\label{cor1.6}
 Let $d=[F:\mathbb{Q}]$ and $t\in \mathbb{R}$. Let $\pi$ be a unitary cuspidal automorphic representation of $\mathrm{GL}_{2}/F$ with a \textit{fixed} central character $\omega$. Then 
\begin{equation}\label{1.9}
L(1/2+it,\pi)\ll_{F,\omega,\varepsilon} C(\pi)^{\frac{1}{4}-\frac{1-2\vartheta}{24}+\varepsilon}(1+|t|)^{\frac{d}{2}-\frac{(1-2\vartheta)d}{24}+\varepsilon}.	
\end{equation}
\end{cor}

Indeed, $C(\pi\otimes|\cdot|^{it})\ll_{F} C(\pi)(1+|t|)^{2d}$ and $C(\omega|\cdot|^{2it})\ll_{F,\omega} (1+|t|)^{d}$. Applying the second bound in the minimum of Theorem \ref{thmB} gives \eqref{1.9}.

The exponent $1/4-(1-2\vartheta)/24$ in \eqref{1.9} was obtained by Blomer and Khan \cite[Theorem 4]{BK19} for square-free levels $q$ over $\mathbb{Q}$ with $(q,6)=1$ and trivial central character. For $t=0$, \eqref{1.9} extends their bound to arbitrary number fields, general levels and fixed central characters, uniformly in the Archimedean parameters.

\subsubsection{Quadratic Central Characters}
When the central character is quadratic, the subconvexity bound in Theorem \ref{thmB} can be further improved using the weighted moment estimates \eqref{10.47} for the cuspidal and Eisenstein spectra.
 
\begin{thm}\label{cor1.7}
Let $F$ be a number field, and let $\pi$ be a unitary cuspidal automorphic representation of $\mathrm{GL}_{2}/F$ with central character $\omega$. Suppose $\omega^{2}=\mathbf{1}$. Then
\begin{equation}\label{1.8}
 L(1/2,\pi)\ll_{F,\varepsilon}
C(\pi)^{\frac{1}{4}-\frac{9(1 - 2\vartheta)}{4(59 + 2\vartheta)}+\varepsilon
}
C(\omega)^{
\frac{6-18\vartheta}{4(59+2\vartheta)}+\varepsilon
}.
\end{equation}	
In particular, we have
\begin{align*}
L(1/2,\pi)\ll_{F,\varepsilon}
C(\pi)^{\frac{1}{4}-\frac{1}{80}+\varepsilon
}.
\end{align*}
\end{thm}

Let $E/F$ be a quadratic extension with associated quadratic Hecke character $\eta_{E/F}$. For a class group character $\chi$ of $E$ trivial on $\mathbb{A}_{F}^{\times}$, the automorphic induction has central character $\eta_{E/F}$. If the induction is cuspidal, Theorem~\ref{cor1.7} gives an unconditional discriminant saving $1/80$. This is stronger than the limiting saving $1/88$ from the toric second moment and DFI-type amplifier of \cite[\textsection 10]{Yan26c}, even with the expected supply of small split primes. Here amplification takes place in a larger $\mathrm{GL}_{2}$ fourth-moment family and requires no split-prime hypothesis.

\subsubsection{Rankin--Selberg $L$-functions}
Applying Theorem \ref{thm10.7} at both amplifier primes in the proof of \cite[Theorem A]{HY26} improves the saving from $1/64$ to $1/60$.

\begin{thmx}\label{thmC} 
Let $0<\varepsilon<10^{-2}$, and let $\pi$ and $\pi'$ be unitary cuspidal 
automorphic representations of $\mathrm{GL}_{2}/F$. For $j\in \mathbb{Z}_{\geq 0}$, we have 
\begin{align*}
L^{(j)}(1/2,\pi\times\pi')\ll_{F,\pi',j,\varepsilon}C(\pi)^{\frac{1}{2}-\frac{1}{60}+\varepsilon}.
\end{align*}
The implied constant depends polynomially on $C(\pi')$.
\end{thmx}

\subsection{Applications to Picard Group Arithmetic}\label{sec1.3}
Following \cite[Theorems 2.7 and 2.8]{DFI02}, we consider the following generalization of classical questions on ideal class groups of imaginary quadratic fields.
\begin{que}
Let $E/F$ be a quadratic extension and $\mathcal{O}$ be an $\mathcal{O}_{F}$-order in $E$. Let $D$ be the absolute norm of the relative discriminant of $\mathcal{O}$ over $F$. Let $\chi$ be a non-trivial character of the Picard group $\mathrm{Pic}(\mathcal{O})$ and $H$ be a cyclic subgroup of $\mathrm{Pic}(\mathcal{O})$. For an invertible integral $\mathcal{O}$-ideal $\mathfrak{a}$, let $N(\mathfrak{a})$ be its absolute norm. Define   
\begin{align*}
N(\mathcal{O};\chi):=\min_{\mathfrak{a}}\big\{N(\mathfrak{a}):\ \chi([\mathfrak{a}])\neq 1\big\},\qquad N(\mathcal{O},H):=\min_{\mathfrak{a}:\ H=\langle [\mathfrak{a}]\rangle}\big\{N(\mathfrak{a}) \big\},
\end{align*}
where both minima are over invertible integral $\mathcal{O}$-ideals.
How large can $N(\mathcal{O};\chi)$ and $N(\mathcal{O},H)$ be in terms of $D$ and $[\mathrm{Pic}(\mathcal{O}): H]$?
\end{que}

By Minkowski's theorem, each ideal class of $\mathrm{Pic}(\mathcal{O})$ is represented by an integral ideal of norm $\ll_{F}\sqrt{D}$. Consequently, 
\begin{align*}
N(\mathcal{O};\chi)\ll_{F}\sqrt{D}\  \ \text{and}\ \  N(\mathcal{O},H)\ll_{F}\sqrt{D}.
\end{align*}

In the special case $F=\mathbb{Q}$ and $\mathcal{O}=\mathcal{O}_{E}$ is the maximal order, Duke--Friedlander--Iwaniec \cite[Theorems 2.7 and 2.8]{DFI02} proved
\begin{equation}\label{1.14}
N(\mathcal{O};\chi)\ll D^{\frac{1}{2}-\frac{1}{11521}}\  \ \text{and}\ \  N(\mathcal{O},H)\ll [\mathrm{Pic}(\mathcal{O}): H]^{2}D^{\frac{1}{2}-\frac{1}{11521}},
\end{equation}
with absolute but ineffective constants.

Using Corollary \ref{cor1.3}, we extend and sharpen \eqref{1.14} for arbitrary $\mathcal{O}_{F}$-orders in quadratic extensions of $F$:
\begin{thmx}\label{thmeC}
Let $0<\varepsilon<10^{-2}$. Then
\begin{align*}
N(\mathcal{O};\chi)\ll_{F,\varepsilon} D^{\frac{1}{2}-\frac{1}{60}+\varepsilon}\  \ \text{and}\ \  N(\mathcal{O},H)\ll_{F,\varepsilon} [\mathrm{Pic}(\mathcal{O}): H]^{2}D^{\frac{1}{2}-\frac{1}{60}+\varepsilon},
\end{align*}
with ineffective constants.
\end{thmx}

\subsection{Fourth-Moment Estimates}  
Estimating the terms in Theorem \ref{thmA} gives the following moment bounds, used in the proof of Theorem \ref{thmB}.
\begin{thmx}\label{thmE}
Let $\mathfrak{q}\subseteq \mathcal{O}_{F}$ be an integral ideal, and $\omega$ be a unitary character of $F^{\times}\backslash\mathbb{A}_{F}^{\times}$. For $v\mid\infty$, let $\mathbf{C}_{v}>1$ and $\mathbf{C}_{\infty}:=\prod_{v\mid\infty}\mathbf{C}_{v}$. Then 
\begin{equation}\label{10.1}
\sum_{\substack{\pi\in \mathcal{F}_{0}(\mathfrak{q},\omega)\\
C_{v}(\pi)\leq \mathbf{C}_{v},\ v\mid\infty}}|L(1/2,\pi)|^{4}\ll_{F,\varepsilon} \mathbf{C}_{\infty}^{1+\varepsilon}N_{F}(\mathfrak{q})^{1+\varepsilon}.
\end{equation}
\end{thmx}

Theorem \ref{thmE} sharpens \cite[Theorem 6.6]{Wu14} by improving the dependence on the spectral parameters and extends \cite[Theorem 6.3]{Wu22} from trivial to arbitrary central characters. Theorem \ref{thmF} gives the corresponding Eisenstein estimate over general number fields.

\begin{thmx}\label{thmF}
 Let $\mathfrak{q}\subseteq \mathcal{O}_{F}$ be an integral ideal, and $\omega$ be a unitary character of $F^{\times}\backslash\mathbb{A}_{F}^{\times}$. Let $\mathfrak{X}(\mathfrak{q},\omega)$ be the set of unitary characters $\mu$ of $F^{\times}\backslash\mathbb{A}_{F}^{(1)}$ such that the representation $\mu\boxplus \overline{\mu}\omega$ has level dividing $\mathfrak{q}$. For $v\mid\infty$, let $\mathbf{C}_{v}>1$ and $\mathbf{C}_{\infty}:=\prod_{v\mid\infty}\mathbf{C}_{v}$. Then 
\begin{align*}
\sum_{\substack{\mu\in \mathfrak{X}(\mathfrak{q},\omega)}}\int_{\substack{t\in\mathbb{R}\\ C_{v}(\mu|\cdot|^{it})C_{v}(\mu\overline{\omega}|\cdot|^{it})\leq\mathbf{C}_{v},\ v\mid\infty}}\frac{|L(1/2+it,\mu)L(1/2-it,\overline{\mu}\omega)|^{4}}{|L(1+2it,\mu^{2}\overline{\omega})|^{2}}dt\ll_{F,\varepsilon} \mathbf{C}_{\infty}^{1+\varepsilon}N_{F}(\mathfrak{q})^{1+\varepsilon}.
\end{align*}
\end{thmx}

\begin{cor}\label{cor1.eighth}
We have
\begin{align*}
\sum_{\substack{\mu\in \widehat{C}_{F}\\ C_{v}(\mu)\leq \mathbf{C}_{v},\ v\mid\infty\\ \text{$\mu$ of level dividing $\mathfrak{q}$}}}|L(1/2,\mu)|^{8}\ll_{F,\varepsilon} \mathbf{C}_{\infty}^{2+\varepsilon}N_{F}(\mathfrak{q})^{2+\varepsilon}.
\end{align*}
\end{cor}

When $F=\mathbb{Q}$, Theorems \ref{thmE} and \ref{thmF} follow from the large sieve 
\cite[Theorem 7.34]{IK04}. Over general number fields such estimates are used in the fifth 
moment of $\mathrm{GL}_{2}$ $L$-functions \cite{Yan25} and in 
subconvexity for $\mathrm{GL}_{2}\times\mathrm{GL}_{2}$ over number fields 
\cite{HY26}.

\subsection{A Fifth-Moment Application}
For a prime ideal $\mathfrak{p}$, let $\mathcal{F}_{\mathrm{new}}(\mathfrak{p},\mathbf{1})$ denote the set of cuspidal representations of $\mathrm{PGL}_{2}(\mathbb{A}_{F})$ of exact arithmetic conductor $\mathfrak{p}$.

\begin{cor}\label{cor1.fifth}
Let $F$ be a number field, let $\Omega_{\infty}$ be a compact set of unitary spherical Archimedean representations of $\mathrm{PGL}_{2}/F$, and let $\mathfrak{p}$ be a prime ideal of $\mathcal{O}_{F}$. Then
\begin{equation}\label{1.fifth}
\sum_{\substack{\pi\in\mathcal{F}_{\mathrm{new}}(\mathfrak{p},\mathbf{1})\\ \pi_{\infty}\in\Omega_{\infty}}}L(1/2,\pi)^{5}
\ll_{F,\Omega_{\infty},\varepsilon}N_{F}(\mathfrak{p})^{1+\vartheta+\varepsilon}.
\end{equation}
\end{cor}

When $F=\mathbb{Q}$, this is the fifth-moment estimate of Blomer and Khan \cite[Theorem~3]{BK19}. For totally real fields, the prime-level case is contained in the prime-power result of \cite[Theorem~E]{Yan25}, obtained from asymmetric reciprocity. Corollary \ref{cor1.fifth}, proved in \textsection\ref{sec10.fifth}, follows from symmetric reciprocity over arbitrary number fields, including those with complex places.

\subsection{Symmetry vs.\ Asymmetry}

We compare our formula with the asymmetric reciprocity of
Blomer and Khan \cite{BK19,BK19a}. Write
\begin{align*}
L(1/2,\pi)^{4} = L(1/2,\pi\times\Pi)
\end{align*}
as a Rankin--Selberg $L$-function on $\mathrm{GL}_{2}\times\mathrm{GL}_{4}$,
where $\Pi$ is a suitable representation of $\mathrm{GL}_{4}/F$.
The distinction lies in the decomposition of the $\mathrm{GL}_{4}$ datum
and in how the resulting fourth moment is transformed.

\begin{itemize}

\item
Blomer and Khan use the asymmetric decomposition $4=3+1$, taking
$\Pi=(\mathbf{1}\boxplus\mathbf{1}\boxplus\mathbf{1})\boxplus\mathbf{1}$,
with the minimal Eisenstein representation on $\mathrm{GL}_{3}$ in the first summand.
Their method uses the Kuznetsov and $\mathrm{GL}_{3}$ Voronoi formulas.

This reciprocity has been extended to automorphic representations over number fields in \cite[Theorem B]{Yan25}.
The asymmetric formula involves
\(
L(1/2,\pi)^{3}\overline{L(1/2,\pi)},
\)
together with the root-number factor from the functional equation.
At prime level with primitive central character, an explicit formula for this factor gives a nonnegative fourth moment \cite[\textsection 7]{BK19a}.

\item Our approach uses the symmetric decomposition $4=2+2$, taking
$\Pi=(\mathbf{1}\boxplus\mathbf{1})\boxplus(\mathbf{1}\boxplus\mathbf{1})$.
For paired data, the spectral weights are squared periods, so nonnegativity is preserved under Hecke amplification at general levels (Corollary \ref{cor5.2}).

The decomposition $4=2+2$ is classical.
Earlier applications in the level and spectral aspects use the binary
additive divisor problem; see
Duke--Friedlander--Iwaniec \cite{DFI94},
Kowalski--Michel--VanderKam \cite{KMV00}, and Jutila \cite{Jut01}.
The Kuznetsov--Motohashi formula \cite{Mot03} gives a spectral
formula for the same decomposition by dualizing both
$\mathrm{GL}_{2}$ components.
As observed in \cite{BK19}, the resulting formula is essentially
self-dual, requiring a suitable difference of the holomorphic and Maass
spectra; in particular, everywhere-positive test functions are generally
unavailable.

Our formula is obtained by switching Rankin--Selberg kernels.
With a suitable choice of $\mathfrak{X}$, Theorem \ref{thmA} recovers in principle the Kuznetsov--Motohashi formula in a regularized form.
For fixed central character, the Ramanujan conjecture yields square-root cancellation in the corresponding twisted fourth-moment estimate.
\end{itemize}

\subsection{Regularization via the Rankin--Selberg Kernel}
Products of Eisenstein series require regularization for spectral expansions and period integrals. Zagier \cite{Zag81}, Michel--Venkatesh \cite{MV10}, and Nelson \cite{Nel19} regularize the integral or spectral pairing of the non-$L^{2}$ function. Wu \cite{Wu22a} uses a distributional identity with Godement--Jacquet zeta integrals on the spectral side. Templier \cite{Tem11} uses an explicit automorphic subtraction adapted to the Heegner-point setting.

We begin in a region of absolute convergence and isolate a Rankin--Selberg kernel. Let $E_{i}(g,s_{i})$, $i=1,2$, be Eisenstein series, with $\Re(s_{1})\gg1$ and $\Re(s_{2})-|\Re(s_{1})|\gg1$. We consider
\begin{align*}
\Eis(\mathcal{G}_{1} h_{2})(g,s_{1},s_{2})
:=
\sum_{\delta\in B(F)\backslash G(F)}
\mathcal{G}_{1}(\delta g,s_{1})h_{2}(\delta g,s_{2}),
\end{align*}
which is characterized by the identity
\begin{align*}
\bigl\langle \Eis(\mathcal{G}_{1} h_{2})(\cdot,s_{1},s_{2}),\phi_{0}\bigr\rangle
=
\Psi(\overline{W_{0}},W_{1},h_{2})
\end{align*}
for every automorphic form $\phi_{0}$ of central character $\omega_{1}\omega_{2}$ for which the pairing converges absolutely. Here $W_{0}=W_{\phi_{0}}$, $h_{i}$ and $W_{i}$ are the inducing sections and their associated Whittaker functions, and $\mathcal{G}_{i}=E_{i}-\mathcal{C}_{i}$.

We also have
\begin{align*}
\Eis(\mathcal{G}_{1} h_{2})(g,s_{1},s_{2})
=
E_{1}(g,s_{1})E_{2}(g,s_{2})-\Eis(\mathcal{C}_{1} h_{2})(g,s_{1},s_{2}),
\end{align*}
where $\Eis(\mathcal{C}_{1} h_{2})$ is a sum of two explicit Eisenstein series.

In \textsection\ref{sec3.2} we establish a spectral expansion for $\Eis(\mathcal{G}_{1} h_{2})(g,s_{1},s_{2})$. We derive the reciprocity formula by comparing the inner products
\begin{align*}
\bigl\langle \Eis(\mathcal{G}_{1}h_{2})(\cdot,s_{1},s_{2}),\Eis(\mathcal{G}_{3}h_{4})(\cdot,\overline{s_{3}},\overline{s_{4}})\bigr\rangle
\end{align*}
and
\begin{align*}
\bigl\langle \Eis(\mathcal{G}_{1}\overline{h_{3}})(\cdot,s_{1},\overline{s_{3}}),\Eis(\overline{\mathcal{G}_{2}}h_{4})(\cdot,s_{2},\overline{s_{4}})\bigr\rangle,
\end{align*}
first in the domain of absolute convergence via their geometric expansions, and then by meromorphically continuing their spectral decompositions. See \textsection\ref{sec2.2} and \textsection\ref{sec3.zero-free}--\textsection\ref{sec3.6}.

The spectral coefficients of the kernel are Rankin--Selberg periods, as in the relative trace formula of \cite{Yan19a} and the asymmetric reciprocity of \cite{Yan25}. The correction terms come from constant terms of Eisenstein series and spectral residues.

This switching uses a feature specific to $\mathrm{GL}_{2}$: the mirabolic Eisenstein series entering the Rankin--Selberg integral have Whittaker coefficients that do not vanish identically. For $\mathrm{GL}_{n}$ with $n>2$, the corresponding series are nongeneric, so the same switching does not yield a dual Rankin--Selberg moment.

Taking a split torus period of the same Rankin--Selberg kernel expansion yields a Motohashi formula over number fields, using the meromorphic continuation developed here. The non-split analogue is established in \cite{Yan26c}, with applications to class group $L$-functions.

\subsection{The Relative Trace Formula and Square-Level Concentration}
Combining \eqref{f1.1} with amplification and the general hybrid bound of \cite[Corollary 1.9]{Yan23c} yields
\begin{align*}
L(1/2,\pi) \ll_{F,\varepsilon} C(\pi)^{\tfrac{1}{4}-\tfrac{1}{128}+\varepsilon},
\end{align*}
a bound weaker than \eqref{1.5}.

The special local test functions include unipotent translates and character idempotents arising from depth-zero supercuspidal representations. They improve the bounds for the orbital integrals, especially the singular terms. The contributions from the other conductors are bounded by the general hybrid estimate \eqref{10.21}; see the comparison following \eqref{f10.23}.

The two choices of local test functions give the four piecewise bounds in \eqref{1.1}. Using special test functions at both $\mathfrak{p}_{1}$ and $\mathfrak{p}_{2}$ gives the first two bounds, whose optimized ranges meet at $C(\pi)=C(\omega)^{\frac{9-14\vartheta}{7(1-2\vartheta)}}$. Using a special test function at one prime and the Gross--Prasad test function at the other gives the last two bounds, whose ranges meet at $C(\pi)=C(\omega)^{\frac{2(8-11\vartheta)}{11(1-2\vartheta)}}$. The two constructions overlap between these transition points, and within this overlap the second and third bounds coincide at $C(\pi)=C(\omega)^{\frac{35-52\vartheta}{26(1-2\vartheta)}}$. Taking the better of the two constructions gives the four ranges in \eqref{1.1}; the other two bounds come from the general hybrid estimate. In the balanced range $C(\pi)\asymp C(\omega)$, the first term gives the uniform saving $1/120$, while the bounds improve when $C(\omega)$ is smaller relative to $C(\pi)$.

\subsection{Outline of the Paper}
\subsubsection{Part 1: The Symmetric Spectral Reciprocity}
\begin{itemize}
\item In \textsection\ref{sec2}, we set up the basic automorphic data and establish a
coarse reciprocity formula in Proposition \ref{prop2.5} by switching Rankin--Selberg kernels in a region of absolute convergence.

\item In \textsection\ref{sec3}, meromorphic continuation to $\mathcal{R}^{\heartsuit}$ gives Theorem~\ref{thmA} and its equivalent indexed form \eqref{eq3.indexed-reciprocity}.
\end{itemize}

\subsubsection{Part 2: Quantitative Applications and Subconvexity}
In this part we specialize the automorphic data in Theorem \ref{thmA} to the subconvexity problem. 
\begin{itemize}
\item In \textsection\ref{sec4} we construct the automorphic data and apply the contour integral of \textsection\ref{sec4.regularization} to \eqref{eq3.indexed-reciprocity} to obtain Theorem \ref{thmD}.
\item In \textsection\ref{sec5}--\textsection\ref{sec9} we establish the lower and upper bounds for the terms in Theorem \ref{thmD}. 
\item In \textsection\ref{sec10} we assemble these bounds to derive Theorems \ref{thmE} and \ref{thmF}. Together with the hybrid bound of \cite[Corollary 1.9]{Yan23c} and amplification, this gives Theorem \ref{thmBweak}. The square-level refinements in Theorems \ref{thm10.7} and \ref{thm10.8} then yield Theorem \ref{thmB} and its applications. We also prove the fifth-moment estimate in Corollary \ref{cor1.fifth}.
\end{itemize}

\subsection{Notation}
\subsubsection{Number Fields and Measures}\label{1.1.1}
Let $F$ be a number field with ring of integers $\mathcal{O}_{F}.$ Let $N_{F}$ be the absolute norm and $\mathfrak{D}_{F}$ the different of $F.$ Let $\mathbb{A}_{F}$ be the adele ring of $F.$ For each place $v$ of $F$, let $F_{v}$ be the corresponding local field. For a non-Archimedean place $v,$ let $\mathcal{O}_{v}$ be the ring of integers of $F_{v}$ and $\mathfrak{p}_{v}$ its maximal ideal. Fix a uniformizer $\varpi_{v}\in\mathfrak{p}_{v}.$ Denote by $e_{v}(\cdot)$ the valuation normalized by $e_{v}(\varpi_{v})=1.$ Let $d_{v}=e_{v}(\mathfrak{D}_{F})$ be the different exponent and let $q_{v}$ be the cardinality of $\mathcal{O}_{v}/\mathfrak{p}_{v}.$ We use $v\mid\infty$ to indicate an Archimedean place and write $v<\infty$ if $v$ is non-Archimedean. Let $|\cdot|_{v}$ be the normalized absolute value on $F_{v}.$ For ideles, write $|\cdot|=\prod_{v}|\cdot|_{v}$.

Let $\psi_{\mathbb{Q}}$ be the additive character of $\mathbb{Q}\backslash\mathbb{A}_{\mathbb{Q}}$ whose real component is $t\mapsto\exp(2\pi it)$, and put $\psi_{F}=\psi_{\mathbb{Q}}\circ\Tr_{F}=\prod_{v}\psi_{v}$. For each place $v$, let $dt_{v}$ be self-dual relative to $\psi_{v}$ and set $d^{\times}t_{v}=\zeta_{F_{v}}(1)dt_{v}/|t_{v}|_{v}$, where $\zeta_{F_{v}}$ is the local zeta factor, with the standard Gamma factors at infinity. Set $dt=\prod_{v}dt_{v}$ and $d^{\times}t=\prod_{v}d^{\times}t_{v}$. Then $\Vol(F\backslash\mathbb{A}_{F};dt)=1$ and $\Vol(\mathcal{O}_{v}^{\times},d^{\times}t_{v})=\Vol(\mathcal{O}_{v},dt_{v})=q_{v}^{-d_{v}/2}$ for $v<\infty$. 

Let $\mathbb{A}_{F}^{(1)}$ be the group of norm-one ideles, put $C_{F}:=F^{\times}\backslash\mathbb{A}_{F}^{(1)}$, and write $\widehat{C}_{F}$ for its Pontryagin dual. With measure $dr/r$ in the norm variable, put $\rho_{F}:=\Vol(C_{F})=\prod_{v\mid\infty}\zeta_{F_{v}}(1)\underset{s=1}{\Res}\ \zeta_{F}(s)$, where $\zeta_{F}(s)=\prod_{v<\infty}\zeta_{F_{v}}(s)$ is the finite Dedekind zeta function.
	
\subsubsection{Reductive Groups}
For an algebraic group $H$ over $F$, put $[H]:=H(F)\backslash H(\mathbb{A}_{F}).$ For every unipotent subgroup $U$ of $H,$ we choose the Haar measure on $U(\mathbb{A}_{F})$ for which $[U]$ has volume $1,$ with $U(F)$ equipped with counting measure. We normalize the Haar measure on a maximal compact subgroup $K$ of $H(\mathbb{A}_{F})$ to have total mass $1.$ When $H$ is split, we equip its maximal split torus with the Tamagawa measure induced from that of $\mathbb{A}_{F}^{\times}.$

In this paper we set $G=\mathrm{GL}_{2}$. Let $B$ be the group of upper triangular matrices in $G$.  
Let $\overline{G}=Z\backslash G$, where $Z$ is the center of $G.$ Let $N$ be the unipotent radical of $B$. Let $w=\begin{pmatrix}
	& 1\\
1	& 
\end{pmatrix}$ be the Weyl element. Write $a(y)=\diag(y,1)$. The quotient measures are given by $dg=|y|^{-1}dn\,d^{\times}y\,dk$ in the Iwasawa coordinates $g=na(y)k$ on $\overline{G}(\mathbb{A}_{F})$.

Let $K=\prod_{v}K_{v}$ be a maximal compact subgroup of $G(\mathbb{A}_{F}),$ where $K_{v}=\mathrm{U}_{2}(\mathbb{C})$ if $v$ is complex, $K_{v}=\mathrm{O}_{2}(\mathbb{R})$ if $v$ is real, and $K_{v}=G(\mathcal{O}_{v})$ if $v<\infty.$ For $v<\infty$ and $m\in\mathbb{Z}_{\geq 0},$ define 
\begin{align*}
K_{0,v}[m]:=\Big\{\begin{pmatrix}
a&b\\
c&d
\end{pmatrix}\in G(\mathcal{O}_{v}):\ c\in \mathfrak{p}_{v}^{m}\Big\}.
\end{align*}
For an integral ideal $\mathfrak{q}\subseteq \mathcal{O}_{F}$, we define $K_{0}[\mathfrak{q}]:=\prod_{v<\infty}K_{0,v}[e_{v}(\mathfrak{q})]$. We write $\mathcal{A}_{0}([G],\omega)$ for the cuspidal automorphic spectrum with central character $\omega$.
\subsubsection{Whittaker Functions}\label{sec1.3.3}
Let $\theta$ be the generic character induced by $\psi_{F}$:
\begin{align*}
\theta\left(\begin{pmatrix}
1& u\\
&1
\end{pmatrix}\right)=\psi_{F}(u),\  \ u=(u_{v})_{v}\in \mathbb{A}_{F}.
\end{align*}
Then $\theta=\otimes_{v}\theta_{v}$, where $\theta_{v}\left(\begin{pmatrix}
1& u_{v}\\
&1
\end{pmatrix}\right):=\psi_{v}(u_{v})$, for every place $v$. 

Let $\pi$ be a generic automorphic representation of $G/F$. Let $\phi\in \pi$ be an automorphic form. Define the Whittaker function of $\phi$ (relative to $\theta$) by 
\begin{equation}\label{eq1.3}
W_{\phi}(g):=\int_{[N]}\phi(ug)\overline{\theta}(u)du.
\end{equation}
Let $W_{\phi}^{*}(g):=W_{\phi}(\diag(-1,1)
g)$ be the Whittaker function of $\phi$ relative to $\overline{\theta}$. 

For a unitary generic representation $\pi_{v}$, we use the Kirillov inner product
\begin{align*}
\langle W,W'\rangle_{v}:=\int_{F_{v}^{\times}}W(a(y))\overline{W'(a(y))}d^{\times}y,
\end{align*}
as in \cite[(3.8)]{MV10}. Eisenstein sections use the compact-model norm specified in \textsection\ref{sec3}.

\subsubsection{Other Conventions}
We write $\Lambda(s,\cdot)=L_{\infty}(s,\cdot)L(s,\cdot)$, without conductor powers.

Let $\Phi=\otimes_{v}\Phi_{v}\in \mathcal{S}(\mathbb{A}_{F}^{2})$. Let $\widehat{\Phi}$ be the Fourier transform of $\Phi$, namely,
\begin{equation}\label{1.6}
\widehat{\Phi}(t_{1},t_{2}):=\int_{\mathbb{A}_{F}}\int_{\mathbb{A}_{F}}\Phi(b_{1},b_{2})\overline{\psi}_{F}(b_{1}t_{1}+b_{2}t_{2})db_{1}db_{2}.
\end{equation}

For meromorphic functions $F_{1}(s)$ and $F_{2}(s)$, write $F_{1}(s)\propto F_{2}(s)$ when $F_{1}(s)/F_{2}(s)$ is holomorphic on the stated domain.

Throughout, $\varepsilon$ denotes an arbitrarily small positive number whose value may vary from one occurrence to another; fixed multiples of $\varepsilon$ are absorbed by rescaling it.

\textbf{Acknowledgements}
I thank Peter Sarnak for his encouragement, and Gergely Harcos, Peter Humphries, Rizwanur Khan, Xiannan Li, Wenzhi Luo, Riad Masri, Philippe Michel, Han Wu and Matthew Young for helpful comments and suggestions. 

\part{Symmetric Spectral Reciprocity}
\section{A Coarse Symmetric Spectral Reciprocity}\label{sec2}
We derive a coarse reciprocity identity from a Rankin--Selberg kernel.

\subsection{Eisenstein Series}\label{sec2.1}
Let $j\in\{1,2,3,4\}$, and let $\chi_{j},\omega_{j}\in \widehat{F^{\times}\backslash \mathbb{A}_{F}^{\times}}$ be unitary Hecke characters. Suppose $\omega_{1}\omega_{2}=\omega_{3}\omega_{4}$. Let $\Phi_{j}=\otimes_{v}\Phi_{j,v}$ be a Schwartz--Bruhat function on $\mathbb{A}_{F}^{2}$.
\subsubsection{Godement Sections}\label{sec2.1.1}
For $\Re(s)\gg 1$, we define 
\begin{equation}\label{2.1}
h_{j}(g,\Phi_{j},\chi_{j},\omega_{j},s):=\chi_{j}(\det g)|\det g|^{s+\frac{1}{2}}\int_{\mathbb{A}_{F}^{\times}}\Phi_{j}((0,t)g)\chi_{j}^{2}\omega_{j}^{-1}(t)|t|^{2s+1}d^{\times}t,
\end{equation} 
which is a Tate integral representing $\Lambda(2s+1,\chi_{j}^{2}\omega_{j}^{-1})$. For $a,d\in \mathbb{A}_{F}^{\times}$ and $b\in \mathbb{A}_{F}$, it satisfies
\begin{align*}
h_{j}\left(\begin{pmatrix}
a& b\\
& d
\end{pmatrix}
g,\Phi_{j},\chi_{j},\omega_{j},s\right)=\chi_{j}(a)\chi_{j}^{-1}\omega_{j}(d)\Big|\frac{a}{d}\Big|^{s+\frac{1}{2}}h_{j}(g,\Phi_{j},\chi_{j},\omega_{j},s).
\end{align*}
Hence $h_{j}(g,\Phi_{j},\chi_{j},\omega_{j},s)$ is a Godement section in the induced representation $\chi_{j}|\cdot|^{s}\boxplus \chi_{j}^{-1}\omega_{j}|\cdot|^{-s}$.

Define the Eisenstein series associated with $h_{j}(\cdot,\Phi_{j},\chi_{j},\omega_{j},s)$ as 
\begin{equation}\label{2.2}
E_{j}(g,\Phi_{j},\chi_{j},\omega_{j},s):=\sum_{\delta\in B(F)\backslash G(F)}h_{j}(\delta g,\Phi_{j},\chi_{j},\omega_{j},s). 
\end{equation}

By Poisson summation, $E_{j}(g,\Phi_{j},\chi_{j},\omega_{j},s)$ has a meromorphic continuation to all $s$ and satisfies the functional equation
\begin{align*}
E_{j}(g,\Phi_{j},\chi_{j},\omega_{j},s)=E_{j}(g^{\iota},\widehat{\Phi_{j}},\chi_{j}^{-1},\omega_{j}^{-1},-s),
\end{align*}
where $g^{\iota}$ is the transpose inverse of $g$, and $\widehat{\Phi_{j}}$ is the Fourier transform of $\Phi_{j}$ defined by \eqref{1.6}. 

\subsubsection{Whittaker Expansions}\label{sec2.1.whittaker}
Let $j\in\{1,2,3,4\}$ and $\Re(s_{j})\gg 1$. Write
\begin{align*}
h_{j}(g,s_{j})=h_{j}(g,\Phi_{j},\chi_{j},\omega_{j},s_{j}),\ \ \ \ E_{j}(g,s_{j})=E_{j}(g,\Phi_{j},\chi_{j},\omega_{j},s_{j}),
\end{align*}
the functions defined by \eqref{2.1} and \eqref{2.2}, respectively.

Let $W_{j}(\cdot,s_{j}):=W_{E_{j}(\cdot,s_{j})}(\cdot)$ be the Whittaker function defined by \eqref{eq1.3}. For $g\in G(\mathbb{A}_{F})$, it is given explicitly by
\begin{align*}
W_{j}(g,s_{j})=\chi_{j}(\det g)|\det g|^{s_{j}+\frac{1}{2}}\int_{\mathbb{A}_{F}^{\times}}\int_{\mathbb{A}_{F}}\Phi_{j}((t,bt)g)\overline{\psi_{F}}(b)db\chi_{j}^{2}\omega_{j}^{-1}(t)|t|^{2s_{j}+1}d^{\times}t,
\end{align*}
where the $b$-integration is performed first. The resulting $t$-integral converges absolutely for every $s_{j}\in\mathbb{C}$, so $W_{j}(g,s_{j})$ is entire in $s_{j}$. We have the Euler product $W_{j}(g,s_{j})=\prod_{v}W_{j,v}(g_{v},s_{j})$, where
\begin{multline}\label{c2.5}
W_{j,v}(g_{v},s_{j})=\chi_{j,v}(\det g_{v})|\det g_{v}|_{v}^{s_{j}+\frac{1}{2}}\int_{F_{v}^{\times}}\int_{F_{v}}\Phi_{j,v}((t_{v},b_{v}t_{v})g_{v})\\\overline{\psi}_{v}(b_{v})db_{v}\chi_{j,v}^{2}\omega_{j,v}^{-1}(t_{v})|t_{v}|_{v}^{2s_{j}+1}d^{\times}t_{v}.
\end{multline}

\subsubsection{Fourier Expansions}\label{sec2.1.fourier}
We have the Fourier expansion 
\begin{equation}\label{eq2.6}
E_{j}(g,s_{j})=\mathcal{C}_{j}(g,s_{j})+\mathcal{G}_{j}(g,s_{j}),
\end{equation}
where 
\begin{align*}
\mathcal{C}_{j}(g,s_{j})&=h_{j}(g,s_{j})+h_{j}^{\Diamond}(g,s_{j}),\ \ \ \ h_{j}^{\Diamond}(g,s_{j}):=\int_{N(\mathbb{A}_{F})}h_{j}(wug,s_{j})du,\\
\mathcal{G}_{j}(g,s_{j})&=\sum_{\alpha\in F^{\times}}W_{E_{j}(\cdot,s_{j})}\left(\begin{pmatrix}
\alpha\\
& 1
\end{pmatrix}g\right)=\sum_{\alpha\in F^{\times}}W_{j}\left(\begin{pmatrix}
\alpha\\
& 1
\end{pmatrix}g,s_{j}\right).
\end{align*}

\subsubsection{Auxiliary Whittaker Functions}\label{sec2.1.auxiliary}
Define
\begin{align*}
&W_{h_{i},h_{j}}(g,s_{i},s_{j}):=\int_{N(\mathbb{A}_{F})}h_{i}(wug,s_{i})h_{j}(wug,s_{j})\overline{\theta}(u)du,\\
&W_{h_{i}^{\Diamond},h_{j}}(g,s_{i},s_{j}):=\int_{N(\mathbb{A}_{F})}h_{i}^{\Diamond}(wug,s_{i})h_{j}(wug,s_{j})\overline{\theta}(u)du.
\end{align*}
We use the same notation for products of these sections, their intertwining sections, and their conjugates. Their starred versions follow the convention of \textsection\ref{sec1.3.3}; thus $W^{*}(g)=W(\diag(-1,1)g)$, or equivalently, $\overline{\theta}(u)$ is replaced by $\theta(u)$ in the corresponding global integral. Outside their convergence regions, these integrals are defined by meromorphic continuation.

The product $h_{i}(\cdot,s_{i})h_{j}(\cdot,s_{j})$ is a section in $\chi_{i}\chi_{j}|\cdot|^{1/2+s_{i}+s_{j}}\boxplus \overline{\chi_{i}\chi_{j}}\omega_{i}\omega_{j}|\cdot|^{-1/2-s_{i}-s_{j}}$ and $h_{i}^{\Diamond}(\cdot,s_{i})h_{j}(\cdot,s_{j})$ is a section in $\chi_{i}^{-1}\omega_{i}\chi_{j}|\cdot|^{1/2-s_{i}+s_{j}}\boxplus \chi_{i}\chi_{j}^{-1}\omega_{j}|\cdot|^{-1/2+s_{i}-s_{j}}$. 

\subsubsection{Rankin--Selberg Periods}\label{sec2.1.rs-periods}
Let $\pi_{0}$ and $\pi_{1}$ be generic automorphic representations of $G(\mathbb{A}_{F})$. Let $W_{0}=\otimes_{v}W_{0,v}$ and $W_{1}=\otimes_{v}W_{1,v}$ be Whittaker vectors relative to $\overline{\theta}$ and $\theta$, respectively. Let $\pi_{2}=\chi_{2}|\cdot|^{s_{2}}\boxplus \chi_{2}'|\cdot|^{-s_{2}}$ be an induced representation of $G(\mathbb{A}_{F})$, where $\chi_{2}$ and $\chi_{2}'$ are unitary, and let $h_{2}=\otimes_{v}h_{2,v}$ be a section in $\pi_{2}$. Suppose the product of the central characters of $\pi_{0}$, $\pi_{1}$, and $\pi_{2}$ is trivial.

For sufficiently large $\Re(s_{2})$, we define 
\begin{align*}
\Psi(W_{0},W_{1},h_{2}):=\int W_{0}(x)W_{1}(x)h_{2}(x)dx=\prod_{v}\Psi_{v}(W_{0,v},W_{1,v},h_{2,v}),
\end{align*}
where $x$ ranges through $N(\mathbb{A}_{F})\backslash\overline{G}(\mathbb{A}_{F})$ and 
\begin{equation}\label{c2.6}
\Psi_{v}(W_{0,v},W_{1,v},h_{2,v}):=\int_{N(F_{v})\backslash\overline{G}(F_{v})}W_{0,v}(x_{v})W_{1,v}(x_{v})h_{2,v}(x_{v})dx_{v}.
\end{equation}

Rankin--Selberg theory gives a meromorphic continuation $\widetilde{\Psi}(W_{0},W_{1},h_{2})$ to $s_{2}\in\mathbb{C}$ and, in the convergence region, expresses the integral as a holomorphic multiple of $\Lambda(1/2+s_{2},\pi_{0}\times\pi_{1}\times\chi_{2})$. We use the same notation when the Whittaker functions and the section vary meromorphically in auxiliary parameters.

\subsection{Rankin--Selberg Kernel}
Fix $i,j\in\{1,2,3,4\}$ with $\Re(s_{i}),\Re(s_{j})\gg 1$ and $\Re(s_{j})-|\Re(s_{i})|\geq 10$. Define
\begin{align*}
\Eis(\mathcal{G}_{i}h_{j})(g,s_{i},s_{j}):=\sum_{\delta\in B(F)\backslash G(F)}\mathcal{G}_{i}(\delta g, s_{i})h_{j}(\delta g, s_{j}).
\end{align*}

For any automorphic form $\phi_{0}$ with central character $\omega_{1}\omega_{2}$, the following Rankin--Selberg identity holds in a common region of absolute convergence:
\begin{equation}\label{2.10}
\bigl\langle \Eis(\mathcal{G}_{1} h_{2})(\cdot,s_{1},s_{2}),\,\phi_{0}\bigr\rangle
=
\Psi(\overline{W_{0}},W_{1},h_{2}),
\end{equation}
where $W_{0}$ denotes the Whittaker function attached to $\phi_{0}$. In particular, the left-hand side of \eqref{2.10} vanishes whenever $\phi_{0}$ belongs to the residual spectrum.

We call $\Eis(\mathcal{G}_{i} h_{j})(g,s_{i},s_{j})$ the \textit{Rankin--Selberg kernel}. Define also
\begin{align*}
\Eis(\mathcal{C}_{i} h_{j})(g,s_{i},s_{j})
:=
\sum_{\delta\in B(F)\backslash G(F)} \mathcal{C}_{i}(\delta g,s_{i})h_{j}(\delta g,s_{j}).
\end{align*}
Since $\mathcal{C}_{i}(g,s_{i})=h_{i}(g,s_{i})+h_{i}^{\Diamond}(g,s_{i})$, \textsection\ref{sec2.1.auxiliary} shows that $\Eis(\mathcal{C}_{i} h_{j})(g,s_{i},s_{j})$ is a sum of two Eisenstein series. Moreover, the Fourier expansion \eqref{eq2.6} yields
\begin{equation}\label{2.6}
\Eis(\mathcal{G}_{i} h_{j})(\cdot,s_{i},s_{j})
=
E_{i}(\cdot,s_{i})E_{j}(\cdot,s_{j})-\Eis(\mathcal{C}_{i} h_{j})(\cdot,s_{i},s_{j}),
\end{equation}
which is the geometric expansion of the Rankin--Selberg kernel $\Eis(\mathcal{G}_{i} h_{j})(\cdot,s_{i},s_{j})$.

\begin{lemma}
Suppose $\Re(s_{i}),\Re(s_{j})\gg 1$ and $\Re(s_{j})-|\Re(s_{i})|\geq 10$. Then $\Eis(\mathcal{C}_{i}h_{j})(\cdot,s_{i},s_{j})$ and $\Eis(\mathcal{G}_{i}h_{j})(\cdot,s_{i},s_{j})$ converge absolutely.
\end{lemma}
\begin{proof}
The two product sections $h_{i}(\cdot,s_{i})h_{j}(\cdot,s_{j})$ and $h_{i}^{\Diamond}(\cdot,s_{i})h_{j}(\cdot,s_{j})$ have inducing parameters $1/2+s_{i}+s_{j}$ and $1/2-s_{i}+s_{j}$, respectively; see \textsection\ref{sec2.1.auxiliary}. Their real parts are greater than $1/2$ in the stated range, so both associated Eisenstein series converge absolutely. Their sum is $\Eis(\mathcal{C}_{i}h_{j})(\cdot,s_{i},s_{j})$. Since $\mathcal{G}_{i}=E_{i}-\mathcal{C}_{i}$ and $E_{i}$ is automorphic, the same absolute bounds give convergence of $\Eis(\mathcal{G}_{i}h_{j})(\cdot,s_{i},s_{j})$.
\end{proof}

\begin{lemma}\label{lem2.2}
 Let $\Re(s_{i})\gg 1$ and $\Re(s_{j})-|\Re(s_{i})|\geq 10$. Then 
\begin{equation}\label{e2.12}
\Eis(\mathcal{G}_{i}h_{j})(\cdot,s_{i},s_{j})\in L^{1}([\overline{G}],\omega_{i}\omega_{j})\cap L^{2}([\overline{G}],\omega_{i}\omega_{j}). 
\end{equation}
\end{lemma}
\begin{proof}
By definition and the Bruhat decomposition, we have
\begin{equation}\label{equ2.5}
\Eis(\mathcal{G}_{i}h_{j})(g,s_{i},s_{j})=\mathcal{G}_{i}(g, s_{i})h_{j}(g, s_{j})+\sum_{\beta\in F}\mathcal{G}_{i}(wn(\beta)g, s_{i})h_{j}(wn(\beta)g, s_{j}),
\end{equation}
where $n(\beta)=\begin{pmatrix}
1& \beta\\
& 1
\end{pmatrix}$. Since $E_{i}=\mathcal{C}_{i}+\mathcal{G}_{i}$ is automorphic, we have
\begin{equation}\label{e2.5}
\mathcal{G}_{i}(wn(\beta)g,s_{i})=\mathcal{G}_{i}(g,s_{i})+\mathcal{C}_{i}(g,s_{i})-\mathcal{C}_{i}(wn(\beta)g,s_{i}).
\end{equation}

Substituting \eqref{e2.5} into \eqref{equ2.5}, we can express $\Eis(\mathcal{G}_{i}h_{j})(\cdot,s_{i},s_{j})$ as 
\begin{align*}
\mathcal{G}_{i}(g, s_{i})E_{j}(g, s_{j})+\mathcal{C}_{i}(g, s_{i})\sum_{\beta\in F}h_{j}(wn(\beta)g, s_{j})-\sum_{\beta\in F}\mathcal{C}_{i}(wn(\beta)g, s_{i})h_{j}(wn(\beta)g, s_{j}).
\end{align*}

Let $g=\begin{pmatrix}
y\\
& 1
\end{pmatrix}g'$ be in a Siegel set of $[\overline{G}]$, where $|y|\gg 1$ and $g'$ lies in a compact set of $G(\mathbb{A}_{F})$. By the standard bounds for Whittaker functions (see, for example, \cite[Proposition 3.2.3]{MV10}), we have $\mathcal{G}_{i}(g,s_{i})\ll_{F,m,h_{i}(\cdot,s_{i})} |y|^{-m}$. Hence
\begin{equation}\label{eq2.5}
\mathcal{G}_{i}(g, s_{i})E_{j}(g, s_{j})\ll_{F,\Phi_{i},\Phi_{j},s_{i},s_{j}} |y|^{-100}
\end{equation}
by taking $m$ sufficiently large.

Since $\Phi_{j}$ is a Schwartz--Bruhat function, Poisson summation gives
\begin{equation}\label{eq2.8}
\mathcal{C}_{i}(g, s_{i})\sum_{\beta\in F}h_{j}(wn(\beta)g, s_{j})\ll_{F,\Phi_{i},\Phi_{j},s_{i},s_{j}} |y|^{1+\Re(s_{i})-\Re(s_{j})},
\end{equation}
and 
\begin{equation}\label{eq2.9}
\sum_{\beta\in F}\mathcal{C}_{i}(wn(\beta)g, s_{i})h_{j}(wn(\beta)g, s_{j})\ll_{F,\Phi_{i},\Phi_{j},s_{i},s_{j}} |y|^{1+\Re(s_{i})-\Re(s_{j})}.
\end{equation}

It then follows from \eqref{eq2.5}, \eqref{eq2.8}, and \eqref{eq2.9} that 
\begin{align*}
\Eis(\mathcal{G}_{i}h_{j})(g,s_{i},s_{j})\ll_{F,\Phi_{i},\Phi_{j},s_{i},s_{j}} |y|^{1+\Re(s_{i})-\Re(s_{j})}.
\end{align*}
Since $\Re(s_{j})-\Re(s_{i})\geq 10$, $\Eis(\mathcal{G}_{i}h_{j})(\cdot,s_{i},s_{j})$ is both $L^{1}$ and $L^{2}$-integrable over a Siegel set. Therefore, \eqref{e2.12} holds.
\end{proof}

\subsection{A Coarse Symmetric Spectral Reciprocity}\label{sec2.2}
Let $\langle\cdot,\cdot \rangle$ be the Petersson inner product. Let $\mathcal{R}$ be the set of points $\mathbf{s}=(s_{1},s_{2},s_{3},s_{4})\in \mathbb{C}^{4}$ satisfying $\Re(s_{1})\geq 10$, $\Re(s_{2}-s_{1})\geq 10$, $\Re(s_{3}-s_{2})\geq 100$ and $\Re(s_{4}-s_{1}-s_{2}-s_{3})\geq 1000$.

With the data of \textsection\ref{sec2.1}, write $\boldsymbol{\chi}=(\chi_{j})_{j=1}^{4}$, $\boldsymbol{\omega}=(\omega_{j})_{j=1}^{4}$, and $\boldsymbol{\Phi}=(\Phi_{j})_{j=1}^{4}$. We call $\mathfrak{X}:=(\boldsymbol{\chi},\boldsymbol{\omega},\boldsymbol{\Phi})$ an automorphic datum.

\begin{defn}[The Spectral Side]
Let $\mathbf{s}\in \mathcal{R}$. Define  
\begin{equation}\label{eq2.12}
\mathcal{J}_{\mathrm{RS}}(\mathbf{s},\mathfrak{X}):=\langle \Eis(\mathcal{G}_{1}h_{2})(\cdot,s_{1},s_{2}),\Eis(\mathcal{G}_{3}h_{4})(\cdot,\overline{s_{3}},\overline{s_{4}})\rangle.
\end{equation}
\end{defn}

\begin{defn}[The Dual Side]
Let $\mathbf{s}\in \mathcal{R}$. Define 
\begin{equation}\label{e2.13}
\mathcal{I}_{\mathrm{Dual}}(\mathbf{s},\mathfrak{X}):=\langle \Eis(\mathcal{G}_{1}\overline{h_{3}})(\cdot,s_{1},\overline{s_{3}}),\Eis(\overline{\mathcal{G}_{2}}h_{4})(\cdot,s_{2},\overline{s_{4}})\rangle.
\end{equation}
\end{defn}

\begin{defn}\label{defn2.5}
 Let $\mathbf{s}\in \mathbb{C}^{4}$.  Define 
\begin{align*}
&\widetilde{\Psi}_{\mathrm{Geo}}^{(1)}(\mathbf{s},\mathfrak{X}):=\widetilde{\Psi}(W_{1}(\cdot,s_{1}),\overline{W_{3}(\cdot,\overline{s_{3}})},h_{2}(\cdot,s_{2})\overline{h_{4}(\cdot,\overline{s_{4}})}),\\
&\widetilde{\Psi}_{\mathrm{Geo}}^{(2)}(\mathbf{s},\mathfrak{X}):=\widetilde{\Psi}(W_{1}(\cdot,s_{1}),\overline{W_{3}(\cdot,\overline{s_{3}})},h_{2}^{\Diamond}(\cdot,s_{2})\overline{h_{4}(\cdot,\overline{s_{4}})}),\\
&\widetilde{\Psi}_{\mathrm{Geo}}^{(3)}(\mathbf{s},\mathfrak{X}):=-\widetilde{\Psi}(W_{1}^{*}(\cdot,s_{1}),W_{2}(\cdot,s_{2}),\overline{h_{3}(\cdot,\overline{s_{3}})}\overline{h_{4}(\cdot,\overline{s_{4}})}),\\
&\widetilde{\Psi}_{\mathrm{Geo}}^{(4)}(\mathbf{s},\mathfrak{X}):=-\widetilde{\Psi}(W_{1}^{*}(\cdot,s_{1}),W_{2}(\cdot,s_{2}),\overline{h_{3}^{\Diamond}(\cdot,\overline{s_{3}})}\overline{h_{4}(\cdot,\overline{s_{4}})}),\\
&\widetilde{\Psi}_{\mathrm{Geo}}^{(5)}(\mathbf{s},\mathfrak{X}):=\widetilde{\Psi}(W_{h_{1},\overline{h_{3}}}^{*}(\cdot,s_{1},\overline{s_{3}}),W_{2}(\cdot,s_{2}),\overline{h_{4}(\cdot,\overline{s_{4}})}),\\
&\widetilde{\Psi}_{\mathrm{Geo}}^{(6)}(\mathbf{s},\mathfrak{X}):=\widetilde{\Psi}(W_{h_{1}^{\Diamond},\overline{h_{3}}}^{*}(\cdot,s_{1},\overline{s_{3}}),W_{2}(\cdot,s_{2}),\overline{h_{4}(\cdot,\overline{s_{4}})}),\\
&\widetilde{\Psi}_{\mathrm{Geo}}^{(7)}(\mathbf{s},\mathfrak{X}):=-\widetilde{\Psi}(W_{h_{1},h_{2}}(\cdot,s_{1},s_{2}),\overline{W_{3}(\cdot,\overline{s_{3}})},\overline{h_{4}(\cdot,\overline{s_{4}})}),\\
&\widetilde{\Psi}_{\mathrm{Geo}}^{(8)}(\mathbf{s},\mathfrak{X}):=-\widetilde{\Psi}(W_{h_{1}^{\Diamond},h_{2}}(\cdot,s_{1},s_{2}),\overline{W_{3}(\cdot,\overline{s_{3}})},\overline{h_{4}(\cdot,\overline{s_{4}})}).
\end{align*} 
The periods $\widetilde{\Psi}$ are defined in \textsection\ref{sec2.1.rs-periods}.
\end{defn}

By Lemma \ref{lem2.2} and the Cauchy--Schwarz inequality, $\mathcal{J}_{\mathrm{RS}}(\mathbf{s},\mathfrak{X})$ and $\mathcal{I}_{\mathrm{Dual}}(\mathbf{s},\mathfrak{X})$ are well defined for $\mathbf{s}\in \mathcal{R}$. They satisfy the following identity:
\begin{prop}\label{prop2.5}
Let $\mathbf{s}\in \mathcal{R}$. Suppose $\omega_{1}\omega_{2}=\omega_{3}\omega_{4}$.
\begin{itemize}
\item The defining inner products for $\mathcal{J}_{\mathrm{RS}}(\mathbf{s},\mathfrak{X})$ and $\mathcal{I}_{\mathrm{Dual}}(\mathbf{s},\mathfrak{X})$ converge absolutely.
\item The function $\mathcal{J}_{\mathrm{RS}}(\mathbf{s},\mathfrak{X})-\mathcal{I}_{\mathrm{Dual}}(\mathbf{s},\mathfrak{X})$ admits a meromorphic continuation to $\mathbf{s}\in \mathbb{C}^{4}$, given explicitly by
\begin{equation}\label{eq2.10}
\mathcal{J}_{\mathrm{RS}}(\mathbf{s},\mathfrak{X})-\mathcal{I}_{\mathrm{Dual}}(\mathbf{s},\mathfrak{X})=\sum_{i=1}^{8}\widetilde{\Psi}_{\mathrm{Geo}}^{(i)}(\mathbf{s},\mathfrak{X}).
\end{equation}
\end{itemize}
\end{prop} 
\begin{proof}
By \eqref{2.6} and Lemma \ref{lem2.2}, $\langle E_{1}(\cdot,s_{1})E_{2}(\cdot,s_{2}),\Eis(\mathcal{G}_{3}h_{4})(\cdot,\overline{s_{3}},\overline{s_{4}})\rangle$ can be expressed as
\begin{equation}\label{eq2.13}
\mathcal{J}_{\mathrm{RS}}(\mathbf{s},\mathfrak{X})+\langle \Eis(\mathcal{C}_{1}h_{2})(\cdot,s_{1},s_{2}),\Eis(\mathcal{G}_{3}h_{4})(\cdot,\overline{s_{3}},\overline{s_{4}})\rangle.
\end{equation}
Both terms converge absolutely for $\mathbf{s}\in\mathcal{R}$. Likewise, we can decompose the inner product $\langle E_{1}(\cdot,s_{1})\overline{E_{3}(\cdot,\overline{s_{3}})},\Eis(\overline{\mathcal{G}_{2}}h_{4})(\cdot,s_{2},\overline{s_{4}})\rangle$ as 
\begin{equation}\label{eq2.14}
\mathcal{I}_{\mathrm{Dual}}(\mathbf{s},\mathfrak{X})+\langle \Eis(\mathcal{C}_{1}\overline{h_{3}})(\cdot,s_{1},\overline{s_{3}}),\Eis(\overline{\mathcal{G}_{2}}h_{4})(\cdot,s_{2},\overline{s_{4}})\rangle.
\end{equation}

On the other hand, by unfolding $\Eis(\mathcal{G}_{3}h_{4})(\cdot,\overline{s_{3}},\overline{s_{4}})$ and $\Eis(\overline{\mathcal{G}_{2}}h_{4})(\cdot,s_{2},\overline{s_{4}})$,  
\begin{align*}
\langle E_{1}(\cdot,s_{1})E_{2}(\cdot,s_{2}),\Eis(\mathcal{G}_{3}h_{4})(\cdot,\overline{s_{3}},\overline{s_{4}})\rangle-\langle E_{1}(\cdot,s_{1})\overline{E_{3}(\cdot,\overline{s_{3}})},\Eis(\overline{\mathcal{G}_{2}}h_{4})(\cdot,s_{2},\overline{s_{4}})\rangle
\end{align*}
is equal to 
\begin{equation}\label{2.15}
\int_{N(F)\backslash \overline{G}(\mathbb{A}_{F})}\Big[E_{2}(g,s_{2})\overline{W_{3}(g,\overline{s_{3}})}-W_{2}(g,s_{2})\overline{E_{3}(g,\overline{s_{3}})}\Big]E_{1}(g,s_{1})\overline{h_{4}(g,\overline{s_{4}})}dg.
\end{equation}

Since $\mathbf{s}\in \mathcal{R}$, the parameter $\Re(s_{4})$ is sufficiently large for the integral in \eqref{2.15} to converge absolutely, by the rapid decay of Whittaker functions (see \cite[Proposition 3.2.3]{MV10}) and the moderate growth of Eisenstein series. Similarly,
\begin{equation}\label{eq2.22}
\int_{N(F)\backslash \overline{G}(\mathbb{A}_{F})}\Big[\mathcal{G}_{2}(g,s_{2})\overline{W_{3}(g,\overline{s_{3}})}-W_{2}(g,s_{2})\overline{\mathcal{G}_{3}(g,\overline{s_{3}})}\Big]E_{1}(g,s_{1})\overline{h_{4}(g,\overline{s_{4}})}dg
\end{equation}
converges absolutely. Expanding $\mathcal{G}_{2}$ and $\mathcal{G}_{3}$ and changing variables term by term, noting that $E_{1}(g,s_{1})\overline{h_{4}(g,\overline{s_{4}})}=E_{1}(\diag(\delta,1)g,s_{1})\overline{h_{4}(\diag(\delta,1)g,\overline{s_{4}})}$ for all $\delta\in F^{\times}$, we deduce that \eqref{eq2.22} vanishes.

By the Fourier expansion \eqref{eq2.6} of Eisenstein series, we have 
\begin{multline*}
E_{2}(g,s_{2})\overline{W_{3}(g,\overline{s_{3}})}-W_{2}(g,s_{2})\overline{E_{3}(g,\overline{s_{3}})}=\mathcal{C}_{2}(g,s_{2})\overline{W_{3}(g,\overline{s_{3}})}\\ -W_{2}(g,s_{2})\overline{\mathcal{C}_{3}(g,\overline{s_{3}})}
+\mathcal{G}_{2}(g,s_{2})\overline{W_{3}(g,\overline{s_{3}})}-W_{2}(g,s_{2})\overline{\mathcal{G}_{3}(g,\overline{s_{3}})}.
\end{multline*}

Since \eqref{eq2.22} vanishes and $\mathcal{C}_{2}(\cdot,s_{2})$ and $\overline{\mathcal{C}_{3}(\cdot,\overline{s_{3}})}$ are left invariant by $N(\mathbb{A}_{F})$, the integral \eqref{2.15} becomes
\begin{multline}\label{c2.18}
\int_{N(\mathbb{A}_{F})\backslash \overline{G}(\mathbb{A}_{F})}\mathcal{C}_{2}(g,s_{2})\int_{[N]}\overline{W_{3}(ug,\overline{s_{3}})}E_{1}(ug,s_{1})du\overline{h_{4}(g,\overline{s_{4}})}dg\\
-\int_{N(\mathbb{A}_{F})\backslash \overline{G}(\mathbb{A}_{F})}\overline{\mathcal{C}_{3}(g,\overline{s_{3}})}\int_{[N]}W_{2}(ug,s_{2})E_{1}(ug,s_{1})du\overline{h_{4}(g,\overline{s_{4}})}dg.
\end{multline}

Evaluating the two inner integrals in \eqref{c2.18} by Fourier orthogonality and decomposing $\mathcal{C}_{2}(\cdot,s_{2})$ and $\overline{\mathcal{C}_{3}(\cdot,\overline{s_{3}})}$, we obtain
\begin{equation}\label{2.16}
\Psi_{\mathrm{Geo}}^{(1)}(\mathbf{s},\mathfrak{X})+\Psi_{\mathrm{Geo}}^{(2)}(\mathbf{s},\mathfrak{X})+\Psi_{\mathrm{Geo}}^{(3)}(\mathbf{s},\mathfrak{X})+\Psi_{\mathrm{Geo}}^{(4)}(\mathbf{s},\mathfrak{X}),
\end{equation}
where 
\begin{align*}
\Psi_{\mathrm{Geo}}^{(1)}(\mathbf{s},\mathfrak{X})&:=\int_{N(\mathbb{A}_{F})\backslash \overline{G}(\mathbb{A}_{F})}W_{1}(g,s_{1})\overline{W_{3}(g,\overline{s_{3}})}h_{2}(g,s_{2})\overline{h_{4}(g,\overline{s_{4}})}dg,\\
\Psi_{\mathrm{Geo}}^{(2)}(\mathbf{s},\mathfrak{X})&:=\int_{N(\mathbb{A}_{F})\backslash \overline{G}(\mathbb{A}_{F})}W_{1}(g,s_{1})\overline{W_{3}(g,\overline{s_{3}})}h_{2}^{\Diamond}(g,s_{2})\overline{h_{4}(g,\overline{s_{4}})}dg,\\
\Psi_{\mathrm{Geo}}^{(3)}(\mathbf{s},\mathfrak{X})&:=-\int_{N(\mathbb{A}_{F})\backslash \overline{G}(\mathbb{A}_{F})}W_{1}^{*}(g,s_{1})W_{2}(g,s_{2})\overline{h_{3}(g,\overline{s_{3}})h_{4}(g,\overline{s_{4}})}dg,\\
\Psi_{\mathrm{Geo}}^{(4)}(\mathbf{s},\mathfrak{X})&:=-\int_{N(\mathbb{A}_{F})\backslash \overline{G}(\mathbb{A}_{F})}W_{1}^{*}(g,s_{1})W_{2}(g,s_{2})\overline{h_{3}^{\Diamond}(g,\overline{s_{3}})h_{4}(g,\overline{s_{4}})}dg.
\end{align*}

Combining \eqref{eq2.13}, \eqref{eq2.14}, and \eqref{2.16}, we obtain 
\begin{multline*}
\mathcal{J}_{\mathrm{RS}}(\mathbf{s},\mathfrak{X})=\mathcal{I}_{\mathrm{Dual}}(\mathbf{s},\mathfrak{X})+\langle \Eis(\mathcal{C}_{1}\overline{h_{3}})(\cdot,s_{1},\overline{s_{3}}),\Eis(\overline{\mathcal{G}_{2}}h_{4})(\cdot,s_{2},\overline{s_{4}})\rangle\\
-\langle \Eis(\mathcal{C}_{1}h_{2})(\cdot,s_{1},s_{2}),\Eis(\mathcal{G}_{3}h_{4})(\cdot,\overline{s_{3}},\overline{s_{4}})\rangle+\sum_{i=1}^{4}\Psi_{\mathrm{Geo}}^{(i)}(\mathbf{s},\mathfrak{X}).
\end{multline*}

Therefore, \eqref{eq2.10} follows from the above decomposition and Lemma \ref{lem2.6}, since each $\Psi_{\mathrm{Geo}}^{(i)}(\mathbf{s},\mathfrak{X})$ admits the meromorphic continuation $\widetilde{\Psi}_{\mathrm{Geo}}^{(i)}(\mathbf{s},\mathfrak{X})$.
\end{proof}

\begin{lemma}\label{lem2.6}
Let $\mathbf{s}\in \mathcal{R}$. Suppose $\omega_{1}\omega_{2}=\omega_{3}\omega_{4}$. Then
\begin{align*}
&\langle \Eis(\mathcal{C}_{1}h_{2})(\cdot,s_{1},s_{2}),\Eis(\mathcal{G}_{3}h_{4})(\cdot,\overline{s_{3}},\overline{s_{4}})\rangle=-\Psi_{\mathrm{Geo}}^{(7)}(\mathbf{s},\mathfrak{X})-\Psi_{\mathrm{Geo}}^{(8)}(\mathbf{s},\mathfrak{X}),\\
&\langle \Eis(\mathcal{C}_{1}\overline{h_{3}})(\cdot,s_{1},\overline{s_{3}}),\Eis(\overline{\mathcal{G}_{2}}h_{4})(\cdot,s_{2},\overline{s_{4}})\rangle=\Psi_{\mathrm{Geo}}^{(5)}(\mathbf{s},\mathfrak{X})+\Psi_{\mathrm{Geo}}^{(6)}(\mathbf{s},\mathfrak{X}),
\end{align*}
which converge absolutely. Here $\Psi_{\mathrm{Geo}}^{(i)}(\mathbf{s},\mathfrak{X})$, $5\leq i\leq 8$, is defined by 
\begin{align*}
\Psi_{\mathrm{Geo}}^{(5)}(\mathbf{s},\mathfrak{X})&:=\int_{N(\mathbb{A}_{F})\backslash \overline{G}(\mathbb{A}_{F})}W_{h_{1},\overline{h_{3}}}^{*}(g,s_{1},\overline{s_{3}})W_{2}(g,s_{2})\overline{h_{4}(g,\overline{s_{4}})}dg,\\
\Psi_{\mathrm{Geo}}^{(6)}(\mathbf{s},\mathfrak{X})&:=\int_{N(\mathbb{A}_{F})\backslash \overline{G}(\mathbb{A}_{F})}W_{h_{1}^{\Diamond},\overline{h_{3}}}^{*}(g,s_{1},\overline{s_{3}})W_{2}(g,s_{2})\overline{h_{4}(g,\overline{s_{4}})}dg,\\
\Psi_{\mathrm{Geo}}^{(7)}(\mathbf{s},\mathfrak{X})&:=-\int_{N(\mathbb{A}_{F})\backslash \overline{G}(\mathbb{A}_{F})}W_{h_{1},h_{2}}(g,s_{1},s_{2})\overline{W_{3}(g,\overline{s_{3}})}\overline{h_{4}(g,\overline{s_{4}})}dg,\\
\Psi_{\mathrm{Geo}}^{(8)}(\mathbf{s},\mathfrak{X})&:=-\int_{N(\mathbb{A}_{F})\backslash \overline{G}(\mathbb{A}_{F})}W_{h_{1}^{\Diamond},h_{2}}(g,s_{1},s_{2})\overline{W_{3}(g,\overline{s_{3}})}\overline{h_{4}(g,\overline{s_{4}})}dg.
\end{align*} 
\end{lemma}
\begin{proof}
By definition, $\langle \Eis(\mathcal{C}_{1}h_{2})(\cdot,s_{1},s_{2}),\Eis(\mathcal{G}_{3}h_{4})(\cdot,\overline{s_{3}},\overline{s_{4}})\rangle$ is equal to  
\begin{align*}
\int_{[\overline{G}]} \Eis(\mathcal{C}_{1}h_{2})(g,s_{1},s_{2})\overline{\Eis(\mathcal{G}_{3}h_{4})(g,\overline{s_{3}},\overline{s_{4}})}dg,
\end{align*}
which, by unfolding the kernel $\Eis(\mathcal{G}_{3}h_{4})(g,\overline{s_{3}},\overline{s_{4}})$, becomes
\begin{equation}\label{2.12}
\int_{N(\mathbb{A}_{F})\backslash \overline{G}(\mathbb{A}_{F})}\int_{[N]}\Eis(\mathcal{C}_{1}h_{2})(ug,s_{1},s_{2})\overline{\theta}(u)du\overline{W_{3}(g,\overline{s_{3}})h_{4}(g,\overline{s_{4}})}dg.
\end{equation}

By the Bruhat decomposition,
\begin{align*}
\int_{[N]}\Eis(\mathcal{C}_{1}h_{2})(ug,s_{1},s_{2})\overline{\theta}(u)du=\int_{N(\mathbb{A}_{F})}\mathcal{C}_{1}(wug,s_{1})h_{2}(wug,s_{2})\overline{\theta}(u)du.
\end{align*}

By the definition of the auxiliary Whittaker functions in \textsection\ref{sec2.1.auxiliary}, we obtain
\begin{equation}\label{2.13}
\int_{[N]}\Eis(\mathcal{C}_{1}h_{2})(ug,s_{1},s_{2})\overline{\theta}(u)du=W_{h_{1},h_{2}}(g,s_{1},s_{2})+W_{h_{1}^{\Diamond},h_{2}}(g,s_{1},s_{2}).
\end{equation}

Thus \eqref{2.12} and \eqref{2.13} give
\begin{align*}
\langle \Eis(\mathcal{C}_{1}h_{2})(\cdot,s_{1},s_{2}),\Eis(\mathcal{G}_{3}h_{4})(\cdot,\overline{s_{3}},\overline{s_{4}})\rangle=-\Psi_{\mathrm{Geo}}^{(7)}(\mathbf{s},\mathfrak{X})-\Psi_{\mathrm{Geo}}^{(8)}(\mathbf{s},\mathfrak{X}).
\end{align*}
The second identity follows in the same way, using the $\overline{\theta}$-Whittaker coefficient. All unfoldings and changes of variables above are justified by absolute convergence in $\mathcal{R}$.
\end{proof}

\begin{remark}
In the dual side of \eqref{eq2.10}, we may replace $\Eis(\mathcal{G}_{1}\overline{h_{3}})$ by $\Eis(h_{1}\overline{\mathcal{G}_{3}})$, continued meromorphically using \eqref{2.6}. As in the proof of Proposition~\ref{prop2.5}, we have
\begin{multline*}
\mathcal{I}_{\mathrm{Dual}}(\mathbf{s},\mathfrak{X})=\big\langle \Eis(h_{1}\overline{\mathcal{G}_{3}})(\cdot,s_{1},\overline{s_{3}}),\Eis(\overline{\mathcal{G}_{2}}h_{4})(\cdot,s_{2},\overline{s_{4}})\big\rangle\\
+\big\langle \Eis(h_{1}\overline{h_{3}^{\Diamond}}-h_{1}^{\Diamond}\overline{h_{3}})(\cdot,s_{1},\overline{s_{3}}),\Eis(\overline{\mathcal{G}_{2}}h_{4})(\cdot,s_{2},\overline{s_{4}})\big\rangle.
\end{multline*}
Thus this replacement adds the last inner product to the right-hand side of \eqref{eq2.10}.
\end{remark}

\section{Meromorphic Continuation}\label{sec3}
Let $\mathcal{J}_{\mathrm{RS}}(\mathbf{s},\mathfrak{X})$ and $\mathcal{I}_{\mathrm{Dual}}(\mathbf{s},\mathfrak{X})$ be defined by \eqref{eq2.12} and \eqref{e2.13}, respectively. In this section, we derive their spectral expansions and meromorphically continue them to a neighborhood of $\mathbf{s}=\mathbf{0}$.
 
\subsection{The $\mathrm{GL}_{2}$-spectrum}
We fix unitary Hecke-character extensions of the characters in $\widehat{C}_{F}$ and use $(\mu,\overline{\mu}\omega)$ in the continuous spectrum. Replacing $\mu$ by $\mu|\cdot|^{it}$ replaces $\lambda$ by $\lambda+it$, so the integral over $i\mathbb{R}$ is independent of the extension. In the residue formulas, $\mu=\eta$ means equality on $\mathbb{A}_{F}^{(1)}$; the period itself uses the full Hecke character $\eta$. Indeed, if $\eta=\mu|\cdot|^{it}$, then
\begin{align*}
\underset{\lambda=b}{\Res}\ \widetilde{\Psi}(\lambda,\eta;\mathbf{s},\mathfrak{X})=\underset{\lambda=b+it}{\Res}\ \widetilde{\Psi}(\lambda,\mu;\mathbf{s},\mathfrak{X}),
\end{align*}
and the same convention applies to $\widetilde{\Psi}^{*}$.
For $\mu_{1}, \mu_{2}\in \widehat{C}_{F}$, let $H(\mu_{1},\mu_{2})$ be the space of functions $h\in L^{2}(K)$ satisfying
\begin{align*}
h\left(\begin{pmatrix}a&b\\&d\end{pmatrix}k\right)=\mu_{1}(a)\mu_{2}(d)h(k),\qquad k\in K,\quad \begin{pmatrix}a&b\\&d\end{pmatrix}\in B(\mathbb{A}_{F})\cap K.
\end{align*}

Equip this space with the inner product $\langle f,h\rangle_{H}=\rho_{F}\int_{K} f(k)\overline{h(k)}dk$. Local compact models use the probability measure on $K_{v}$.

Given $h\in H(\mu_{1},\mu_{2})$ and $\lambda\in \mathbb{C}$, we define 
\begin{equation}\label{eq2.2}
S(h)\left(\begin{pmatrix}
a& b\\
&d
\end{pmatrix}k,\lambda\right):=\mu_{1}(a)\mu_{2}(d)\Big|\frac{a}{d}\Big|^{\lambda+1/2}h(k),
\end{equation}
where $a,d\in \mathbb{A}_{F}^{\times}$, $b\in \mathbb{A}_{F}$, and $k\in K$. Moreover, if $h$ is $K$-finite, then $S(h)$ is a finite linear combination, with meromorphic scalar coefficients, of the Godement sections defined by \eqref{2.1}, divided by $\Lambda(2\lambda+1,\mu_{1}\mu_{2}^{-1})$ (see \cite[pages~6--7]{JZ87}). Right translation on these sections defines the induced representation $\pi_{\mu_{1},\mu_{2},\lambda}$. We realize it on $H(\mu_{1},\mu_{2})$ by
\begin{equation}\label{e2.2}
\pi_{\mu_{1},\mu_{2},\lambda}(g)h(k)=S(h)(kg,\lambda),\ \ h\in H(\mu_{1},\mu_{2}),\ k\in K,\ g\in G(\mathbb{A}_{F}). 
\end{equation} 

Define the associated Eisenstein series by
\begin{align*}
E(g,h,\lambda):=\sum_{\gamma\in B(F)\backslash G(F)}S(h)(\gamma g,\lambda),\ \ \Re(\lambda)>1/2.
\end{align*}

$E(g,h,\lambda)$ converges absolutely in $\Re(\lambda)>1/2$ and admits a meromorphic continuation to $\mathbb{C}$. In $\Re(\lambda)\geq0$, a pole is possible only when $\mu_{1}\mu_{2}^{-1}=|\cdot|^{it}$, and is then simple and located at $\lambda=(1-it)/2$.

Fix an orthonormal basis $\mathfrak{B}(\mu_{1},\mu_{2})$ of $K$-finite functions in the Hilbert space $H(\mu_{1},\mu_{2})$. For $\pi\in \mathcal{A}_{0}([G],\omega)$, we denote by $\mathfrak{B}(\pi)$ an orthonormal basis of $\pi$.
In the residual sum below, the extension of $\mu$ is chosen so that $\mu^{2}=\omega$ as full Hecke characters. By \cite{GJ79}, an $L^{2}$-function $\varphi$ on $[G]$ with central character $\omega$ has the spectral expansion
\begin{equation}\label{spec}
\varphi(g)=\varphi_{\Res}(g)+\varphi_{\mathrm{Cusp}}(g)+\varphi_{\mathrm{Eis}}(g), 
\end{equation}
where 
\begin{align*}
&\varphi_{\Res}(g)=\frac{1}{\Vol([\overline{G}])}\sum_{\substack{\mu\in \widehat{C}_{F},\ \mu^{2}=\omega}}\langle\varphi,\mu\circ\det\rangle\mu(\det g),
\\
&\varphi_{\mathrm{Cusp}}(g)=\sum_{\pi\in \mathcal{A}_{0}([G],\omega)}\sum_{\phi\in\mathfrak{B}(\pi)}\langle\varphi,\phi\rangle\phi(g),\\
&\varphi_{\mathrm{Eis}}(g)=\sum_{\substack{\mu\in \widehat{C}_{F}}} \frac{1}{4\pi i}\int_{i\mathbb{R}}\sum_{h\in \mathfrak{B}(\mu,\overline{\mu}\omega)}\langle\varphi,E(\cdot,h,\lambda)\rangle E(g,h,\lambda)d\lambda.
\end{align*}

\subsection{Spectral Expansion of the Rankin--Selberg Kernel}\label{sec3.2}

Combining \eqref{2.10}, Lemma \ref{lem2.2}, and the spectral expansion \eqref{spec}, we obtain the following result.

\begin{lemma}\label{lem3.1}
Let $\Re(s_{1})\geq 10$, $\Re(s_{2})-|\Re(s_{1})|\geq 10$ and $\omega=\omega_{1}\omega_{2}$. The following identity holds in $L^{2}([\overline{G}],\omega)$:
\begin{multline}\label{eq3.4}
\Eis(\mathcal{G}_{1}h_{2})(g,s_{1},s_{2})
=\sum_{\substack{\pi\in \mathcal{A}_{0}([G],\omega)\\ \phi\in\mathfrak{B}(\pi)}}
\Psi(\overline{W_{\phi}},W_{1}(\cdot,s_{1}),h_{2}(\cdot,s_{2}))\phi(g)\\
+\sum_{\mu\in \widehat{C}_{F}}
\frac{1}{4\pi i}\int_{i\mathbb{R}}
\sum_{h\in \mathfrak{B}(\mu,\overline{\mu}\omega)}
\Psi(\overline{W_{E(\cdot,h,-\overline{\lambda})}},W_{1}(\cdot,s_{1}),h_{2}(\cdot,s_{2}))E(g,h,\lambda)\,d\lambda.
\end{multline}
\end{lemma}

\begin{remark}
Combining the geometric expansion \eqref{2.6} and the spectral expansion \eqref{eq3.4} gives an identity for the Rankin--Selberg kernel, analogous to the pre-trace formula. The residual spectrum does not contribute.
\end{remark}
 
\subsection{Zero-free Regions of Hecke $L$-functions}\label{sec3.zero-free}
Let $\xi$ be a unitary Hecke character over $F$. There are effectively computable constants $c>0$ and $c_{\xi}>0$, with $c_{\xi}$ depending on $F$ and $\xi$, such that the Hecke $L$-function $\Lambda(s,\xi)$ has no zero in the region
\begin{align*}
\Big\{s=\sigma+it:\ \sigma> 1-\frac{c}{c_{\xi}+\log (t^{2}+1)}\Big\}.
\end{align*}
Then $\Lambda(1+2s,\xi)\neq 0$ in $s\in \mathcal{D}_{\xi}$, which is defined by 
\begin{equation}\label{eq3.12}
\mathcal{D}_{\xi}:=\Big\{s=\sigma+it:\ -\frac{c}{10[c_{\xi}+\log (t^{2}+1)]}< \sigma< \frac{c}{10[c_{\xi}+\log (t^{2}+1)]}\Big\}.
\end{equation}    
Upon replacing $c$ with a smaller constant, we may assume that if $s\in\mathcal{D}_{\xi}$, then $|\Re(s)|<10^{-1}$.

\subsection{Meromorphic Continuation of \texorpdfstring{$\mathcal{J}_{\mathrm{RS}}(\mathbf{s},\mathfrak{X})$}{}}
Let $\mathfrak{X}$ be the automorphic datum of \textsection\ref{sec2.2}, and put $\omega=\omega_{1}\omega_{2}$.

\subsubsection{Period Integrals}\label{sec3.3.1}
Let $\pi\in \mathcal{A}_{0}([G],\omega)$. We define 
\begin{align*}
\Psi(\pi;\mathbf{s},\mathfrak{X}):=\sum_{\phi\in\mathfrak{B}(\pi)}\Psi(\overline{W_{\phi}},W_{1}(\cdot,s_{1}),h_{2}(\cdot,s_{2}))\Psi(W_{\phi},\overline{W_{3}(\cdot,\overline{s_{3}})},\overline{h_{4}(\cdot,\overline{s_{4}})}).
\end{align*}

Let $\mu\in \widehat{C}_{F}$, and $\lambda\in i\mathbb{R}$. Define $\Psi(\lambda,\mu;\mathbf{s},\mathfrak{X})$ by 
\begin{equation}\label{c3.5}
\frac{1}{2}\sum_{h}\Psi(\overline{W_{E(\cdot,h,-\overline{\lambda})}},W_{1}(\cdot,s_{1}),h_{2}(\cdot,s_{2}))\Psi(W_{E(\cdot,h,\lambda)},\overline{W_{3}(\cdot,\overline{s_{3}})},\overline{h_{4}(\cdot,\overline{s_{4}})}),
\end{equation}
where $h$ ranges over $\mathfrak{B}(\mu,\overline{\mu}\omega)$. 

\begin{lemma}
 Let $\mathbf{s}\in \mathcal{R}$. Then 
\begin{equation}\label{eq3.5}
\mathcal{J}_{\mathrm{RS}}(\mathbf{s},\mathfrak{X})=\mathcal{J}_{\mathrm{Cusp}}(\mathbf{s},\mathfrak{X})+\mathcal{J}_{\mathrm{Eis}}(\mathbf{s},\mathfrak{X}),
\end{equation}
where 
\begin{align*}
\mathcal{J}_{\mathrm{Cusp}}(\mathbf{s},\mathfrak{X}):=\sum_{\pi\in \mathcal{A}_{0}([G],\omega)}\Psi(\pi;\mathbf{s},\mathfrak{X}),\quad \mathcal{J}_{\mathrm{Eis}}(\mathbf{s},\mathfrak{X}):=\sum_{\substack{\mu\in \widehat{C}_{F}}}\frac{1}{2\pi i}\int_{i\mathbb{R}}\Psi(\lambda,\mu;\mathbf{s},\mathfrak{X})d\lambda.
\end{align*}
\end{lemma}
\begin{proof}
By Lemma \ref{lem3.1}, we can write $\mathcal{J}_{\mathrm{RS}}(\mathbf{s},\mathfrak{X})$ as the sum of 
\begin{align*}
\sum_{\pi\in \mathcal{A}_{0}([G],\omega)}\sum_{\phi\in\mathfrak{B}(\pi)}\langle\Eis(\mathcal{G}_{1}h_{2})(\cdot,s_{1},s_{2}),\phi\rangle\langle \phi,\Eis(\mathcal{G}_{3}h_{4})(\cdot,\overline{s_{3}},\overline{s_{4}})\rangle
\end{align*}
and
\begin{align*}
\sum_{\substack{\mu}} \frac{1}{4\pi i}\int_{i\mathbb{R}}\sum_{h}\langle\Eis(\mathcal{G}_{1}h_{2})(\cdot,s_{1},s_{2}),E(\cdot,h,\lambda)\rangle \langle E(\cdot,h,\lambda),\Eis(\mathcal{G}_{3}h_{4})(\cdot,\overline{s_{3}},\overline{s_{4}})\rangle d\lambda,
\end{align*}
where $\mu\in \widehat{C}_{F}$ and $h\in \mathfrak{B}(\mu,\overline{\mu}\omega)$. Applying \eqref{2.10} and noting that $\lambda=-\overline{\lambda}$ for $\lambda\in i\mathbb{R}$ identifies these sums with $\mathcal{J}_{\mathrm{Cusp}}(\mathbf{s},\mathfrak{X})$ and $\mathcal{J}_{\mathrm{Eis}}(\mathbf{s},\mathfrak{X})$, respectively. This proves \eqref{eq3.5}.
\end{proof}

\subsubsection{Meromorphic Continuation}
\begin{lemma}\label{lem3.3}
 Let $\mathcal{D}_{\overline{\mu}^{2}\omega}$ be defined by \eqref{eq3.12} in \textsection\ref{sec3.zero-free}.
\begin{itemize}
\item $\Psi(\lambda,\mu;\mathbf{s},\mathfrak{X})$ converges absolutely in the region 
\begin{align*}
\begin{cases}
\Re(s_{2})-|\Re(s_{1})|-|\Re(\lambda)|>1/2,\\
\Re(s_{4})-|\Re(s_{3})|-|\Re(\lambda)|>1/2,\\
2\lambda\in \mathcal{D}_{\overline{\mu}^{2}\omega}.
\end{cases}	
\end{align*} 
\item $\Psi(\lambda,\mu;\mathbf{s},\mathfrak{X})$ admits a meromorphic continuation $\widetilde{\Psi}(\lambda,\mu;\mathbf{s},\mathfrak{X})$ to $(\lambda,\mathbf{s})\in \mathbb{C}^{5}$. For $|\Re(\lambda)|<1/4$, it satisfies
\begin{multline}\label{f3.7}
\widetilde{\Psi}(\lambda,\mu;\mathbf{s},\mathfrak{X})\propto\ \Lambda(1+2\lambda,\mu^{2}\overline{\omega})^{-1}\Lambda(1-2\lambda,\overline{\mu}^{2}\omega)^{-1}\\
 \Lambda(1/2+s_{2}+s_{1}-\lambda,\overline{\mu}\chi_{1}\chi_{2})\Lambda(1/2+s_{2}+s_{1}+\lambda,\mu\overline{\omega}\chi_{1}\chi_{2})\\
\Lambda(1/2+s_{2}-s_{1}-\lambda,\overline{\mu}\chi_{1}^{-1}\omega_{1}\chi_{2})\Lambda(1/2+s_{2}-s_{1}+\lambda,\mu\overline{\omega}\chi_{1}^{-1}\omega_{1}\chi_{2})\\
\Lambda(1/2+s_{4}+s_{3}+\lambda,\mu\chi_{3}^{-1}\chi_{4}^{-1})\Lambda(1/2+s_{4}+s_{3}-\lambda,\overline{\mu}\omega\chi_{3}^{-1}\chi_{4}^{-1})\\
\Lambda(1/2+s_{4}-s_{3}+\lambda,\mu\chi_{3}\omega_{3}^{-1}\chi_{4}^{-1})\Lambda(1/2+s_{4}-s_{3}-\lambda,\overline{\mu}\omega\chi_{3}\omega_{3}^{-1}\chi_{4}^{-1}).
\end{multline}
\end{itemize}
\end{lemma}
\begin{proof}
Write $\mathbf{e}_{1}=(1,0)$ and $\mathbf{e}_{2}=(0,1)$. Unfolding $W_{1}$ and $h_{2}$, we express the first Rankin--Selberg period in \eqref{c3.5} as
\begin{equation}\label{e3.9}
\int_{G(\mathbb{A}_{F})}\overline{W_{E(\cdot,h,-\overline{\lambda})}(g)}h_{1}(wg,s_{1})\Phi_{2}(\mathbf{e}_{2}g)\chi_{2}(\det g)|\det g|^{\frac{1}{2}+s_{2}}dg.
\end{equation}

Substituting the definition of $h_{1}(\cdot,s_{1})$ into \eqref{e3.9}, along with the change of variables $t\mapsto t^{-1}$ and $g\mapsto \diag(t,1)g$, we obtain 
\begin{multline}\label{e3.10}
\Psi(\overline{W_{E(\cdot,h,-\overline{\lambda})}},W_{1}(\cdot,s_{1}),h_{2}(\cdot,s_{2}))=\int_{G(\mathbb{A}_{F})}\chi_{1}\chi_{2}(\det g)|\det g|^{1+s_{1}+s_{2}}\\
\chi_{1}(-1)\int_{\mathbb{A}_{F}^{\times}}\overline{W_{E(\cdot,h,-\overline{\lambda})}\left(\begin{pmatrix}
t\\
& 1
\end{pmatrix}g\right)}\chi_{1}^{-1}\chi_{2}\omega_{1}(t)|t|^{s_{2}-s_{1}}d^{\times}t\Phi_{1}(\mathbf{e}_{1}g)\Phi_{2}(\mathbf{e}_{2}g)dg.
\end{multline}

By Bruhat decomposition and the definition of Whittaker functions,  
\begin{align*}
W_{E(\cdot,h,-\overline{\lambda})}\left(\begin{pmatrix}
t\\
& 1
\end{pmatrix}g\right)&=\int_{N(\mathbb{A}_{F})}S(h)\left(wu\begin{pmatrix}
t\\
& 1
\end{pmatrix}
g,-\overline{\lambda}\right)\overline{\theta}(u)du.
\end{align*}

Using \eqref{e2.2} and the orthonormal expansion in $H(\mu,\overline{\mu}\omega)$, we obtain
\begin{align*}
W_{E(\cdot,h,-\overline{\lambda})}\left(\begin{pmatrix}
t\\
& 1
\end{pmatrix}g\right)=\sum_{h'\in \mathfrak{B}(\mu,\overline{\mu}\omega)}\langle \pi_{\mu,\overline{\mu}\omega,-\overline{\lambda}}(g)h,h'\rangle W_{E(\cdot,h',-\overline{\lambda})}\left(\begin{pmatrix}
t\\
& 1
\end{pmatrix}\right).
\end{align*}

Plugging this into \eqref{e3.10}, we can express $\Psi(\overline{W_{E(\cdot,h,-\overline{\lambda})}},W_{1}(\cdot,s_{1}),h_{2}(\cdot,s_{2}))$ as 
\begin{multline}\label{e3.11}
\chi_{1}(-1)\sum_{h'\in \mathfrak{B}(\mu,\overline{\mu}\omega)}\int_{G(\mathbb{A}_{F})}\overline{\langle \pi_{\mu,\overline{\mu}\omega,-\overline{\lambda}}(g)h,h'\rangle}\Phi_{1}(\mathbf{e}_{1}g)\Phi_{2}(\mathbf{e}_{2}g)\chi_{1}\chi_{2}(\det g)\\
|\det g|^{1+s_{1}+s_{2}}dg\int_{\mathbb{A}_{F}^{\times}}\overline{W_{E(\cdot,h',-\overline{\lambda})}\left(\begin{pmatrix}
t\\
& 1
\end{pmatrix}\right)}\chi_{1}^{-1}\chi_{2}\omega_{1}(t)|t|^{s_{2}-s_{1}}d^{\times}t.
\end{multline}

The two integrals in \eqref{e3.11} represent the following $L$-functions:
\begin{itemize}
\item The integral over $g$ is a Godement-Jacquet integral, representing the $L$-function $\Lambda(1/2+s_{2}+s_{1}-\lambda,\overline{\mu}\chi_{1}\chi_{2})\Lambda(1/2+s_{2}+s_{1}+\lambda,\mu\overline{\omega}\chi_{1}\chi_{2})$.
\item Using Godement sections, the integral over $t$ is a holomorphic multiple of
\begin{align*}
\frac{\Lambda(1/2+s_{2}-s_{1}-\lambda,\overline{\mu}\chi_{1}^{-1}\omega_{1}\chi_{2})\Lambda(1/2+s_{2}-s_{1}+\lambda,\mu\overline{\omega}\chi_{1}^{-1}\omega_{1}\chi_{2})}{\Lambda(1-2\lambda,\overline{\mu}^{2}\omega)}.
\end{align*}
\end{itemize}

After extracting these $L$-factors, the remaining factor is holomorphic for $|\Re(\lambda)|<1/4$ and meromorphic globally. The second period is treated similarly. Repeated integration by parts on $K$ gives locally normal convergence of the basis sum away from the poles, proving the continuation and \eqref{f3.7}.
\end{proof}

Let $\widetilde{\Psi}(\pi;\mathbf{s},\mathfrak{X})$ be the meromorphic continuation of $\Psi(\pi;\mathbf{s},\mathfrak{X})$.
\begin{defn}\label{defn3.2}
 Let $\mathbf{s}\in\mathcal{R}_{\mathrm{RS}}:=\{\mathbf{s}:\ |\Re(s_{1})|+|\Re(s_{2})|<1/2,\ |\Re(s_{3})|+|\Re(s_{4})|<1/2\}$. We define 
\begin{equation}\label{reg1}
\mathcal{J}_{\mathrm{Spec}}^{\heartsuit}(\mathbf{s},\mathfrak{X})=\mathcal{J}_{\mathrm{Cusp}}^{\heartsuit}(\mathbf{s},\mathfrak{X})+\mathcal{J}_{\mathrm{Eis}}^{\heartsuit}(\mathbf{s},\mathfrak{X}),
\end{equation}
where 
\begin{align*}
\mathcal{J}_{\mathrm{Cusp}}^{\heartsuit}(\mathbf{s},\mathfrak{X}):=\sum_{\pi\in \mathcal{A}_{0}([G],\omega)}\widetilde{\Psi}(\pi;\mathbf{s},\mathfrak{X}),\quad \mathcal{J}_{\mathrm{Eis}}^{\heartsuit}(\mathbf{s},\mathfrak{X}):=\sum_{\substack{\mu\in \widehat{C}_{F}}}\frac{1}{2\pi i}\int_{i\mathbb{R}}\widetilde{\Psi}(\lambda,\mu;\mathbf{s},\mathfrak{X})d\lambda,
\end{align*}
\end{defn} 
Continuation from the initial convergence region uses deformed contours; returning them to $i\mathbb{R}$ produces the residues in Definition \ref{defn3.5} and Proposition \ref{prop3.4}.

Integration by parts, together with the Weyl law and the continuous Plancherel formula, gives locally normal convergence on the deformed contours away from the poles, justifying the continuation below.

\begin{defn}\label{defn3.5}
 Let $\mathbf{s}\in \mathbb{C}^{4}$ and $\mu\in\widehat{C}_{F}$. The following residues are first taken where the indicated pole is simple and distinct from the other poles, and are then continued meromorphically in $\mathbf{s}$:
\begin{align*}
&\widetilde{\Psi}_{\mathrm{RS},\mu}^{(1)}(\mathbf{s},\mathfrak{X}):=\mathbf{1}_{\mu=\chi_{1}\chi_{2}}\underset{\lambda=s_{2}+s_{1}-\frac{1}{2}}{\Res}\ \widetilde{\Psi}(\lambda,\chi_{1}\chi_{2};\mathbf{s},\mathfrak{X}),\\
&\widetilde{\Psi}_{\mathrm{RS},\mu}^{(2)}(\mathbf{s},\mathfrak{X}):=-\mathbf{1}_{\mu=\omega\overline{\chi_{1}\chi_{2}}}\underset{\lambda=\frac{1}{2}-s_{2}-s_{1}}{\Res}\ \widetilde{\Psi}(\lambda,\omega\overline{\chi_{1}\chi_{2}};\mathbf{s},\mathfrak{X}),\\
&\widetilde{\Psi}_{\mathrm{RS},\mu}^{(3)}(\mathbf{s},\mathfrak{X}):=\mathbf{1}_{\mu=\overline{\chi_{1}}\omega_{1}\chi_{2}}\underset{\lambda=s_{2}-s_{1}-\frac{1}{2}}{\Res}\ \widetilde{\Psi}(\lambda,\overline{\chi_{1}}\omega_{1}\chi_{2};\mathbf{s},\mathfrak{X}),\\
&\widetilde{\Psi}_{\mathrm{RS},\mu}^{(4)}(\mathbf{s},\mathfrak{X}):=-\mathbf{1}_{\mu=\omega\chi_{1}\overline{\omega_{1}\chi_{2}}}\underset{\lambda=\frac{1}{2}+s_{1}-s_{2}}{\Res}\ \widetilde{\Psi}(\lambda,\omega\chi_{1}\overline{\omega_{1}\chi_{2}};\mathbf{s},\mathfrak{X}),\\
&\widetilde{\Psi}_{\mathrm{RS},\mu}^{(5)}(\mathbf{s},\mathfrak{X}):=\mathbf{1}_{\mu=\omega\chi_{3}\omega_{3}^{-1}\chi_{4}^{-1}}\underset{\lambda=s_{4}-s_{3}-\frac{1}{2}}{\Res}\ \widetilde{\Psi}(\lambda,\omega\chi_{3}\omega_{3}^{-1}\chi_{4}^{-1};\mathbf{s},\mathfrak{X}),\\
&\widetilde{\Psi}_{\mathrm{RS},\mu}^{(6)}(\mathbf{s},\mathfrak{X}):=-\mathbf{1}_{\mu=\overline{\chi_{3}}\omega_{3}\chi_{4}}\underset{\lambda=\frac{1}{2}+s_{3}-s_{4}}{\Res}\ \widetilde{\Psi}(\lambda,\overline{\chi_{3}}\omega_{3}\chi_{4};\mathbf{s},\mathfrak{X}),\\
&\widetilde{\Psi}_{\mathrm{RS},\mu}^{(7)}(\mathbf{s},\mathfrak{X}):=\mathbf{1}_{\mu=\omega\chi_{3}^{-1}\chi_{4}^{-1}}\underset{\lambda=s_{4}+s_{3}-\frac{1}{2}}{\Res}\ \widetilde{\Psi}(\lambda,\omega\chi_{3}^{-1}\chi_{4}^{-1};\mathbf{s},\mathfrak{X}),\\
&\widetilde{\Psi}_{\mathrm{RS},\mu}^{(8)}(\mathbf{s},\mathfrak{X}):=-\mathbf{1}_{\mu=\chi_{3}\chi_{4}}\underset{\lambda=\frac{1}{2}-s_{3}-s_{4}}{\Res}\ \widetilde{\Psi}(\lambda,\chi_{3}\chi_{4};\mathbf{s},\mathfrak{X}).
\end{align*}	
The global correction terms are
\begin{align*}
\widetilde{\Psi}_{\mathrm{RS}}^{(i)}(\mathbf{s},\mathfrak{X}):=\sum_{\mu\in\widehat{C}_{F}}\widetilde{\Psi}_{\mathrm{RS},\mu}^{(i)}(\mathbf{s},\mathfrak{X}),\qquad 1\leq i\leq 8,
\end{align*}
where each sum has at most one nonzero term.
\end{defn}

Let $\mu\in \widehat{C}_{F}$. We define 
\begin{align*}
\mathcal{J}_{\mu}(\mathbf{s},\mathfrak{X}):=\frac{1}{2\pi i}\int_{i\mathbb{R}}\Psi(\lambda,\mu;\mathbf{s},\mathfrak{X})d\lambda,
\end{align*}
where $\Re(s_{2})-|\Re(s_{1})|>1/2,$ and 
$\Re(s_{4})-|\Re(s_{3})|>1/2$. For $\mathbf{s}\in\mathcal{R}_{\mathrm{RS}}$, define
\begin{align*}
\mathcal{J}_{\mu}^{\heartsuit}(\mathbf{s},\mathfrak{X}):=\frac{1}{2\pi i}\int_{i\mathbb{R}}\widetilde{\Psi}(\lambda,\mu;\mathbf{s},\mathfrak{X})d\lambda.
\end{align*}

\begin{prop}\label{prop3.4}
 Let $\mathbf{s}=(s_{1},s_{2},s_{3},s_{4})\in \mathbb{C}^{4}$ satisfy $\Re(s_{2})-|\Re(s_{1})|>1/2,$ and 
$\Re(s_{4})-|\Re(s_{3})|>1/2$. 
\begin{itemize}
\item $\mathcal{J}_{\mu}(\mathbf{s},\mathfrak{X})$ admits a meromorphic continuation $\widetilde{\mathcal{J}}_{\mu}(\mathbf{s},\mathfrak{X})$ to 
\begin{equation}\label{equ3.12}
\big\{\mathbf{s}\in \mathbb{C}^{4}:\ \Re(s_{2})-|\Re(s_{1})|>-1/2,\ \Re(s_{4})-|\Re(s_{3})|>-1/2\big\}.
\end{equation}
\item For $\mathbf{s}\in \mathcal{R}_{\mathrm{RS}}$, we have
\begin{align*}
\widetilde{\mathcal{J}}_{\mu}(\mathbf{s},\mathfrak{X})&=\mathcal{J}_{\mu}^{\heartsuit}(\mathbf{s},\mathfrak{X})-\sum_{i=1}^{8}\widetilde{\Psi}_{\mathrm{RS},\mu}^{(i)}(\mathbf{s},\mathfrak{X})
\end{align*}
as an identity of meromorphic functions. Here $\widetilde{\Psi}_{\mathrm{RS},\mu}^{(i)}(\mathbf{s},\mathfrak{X})$ is defined by Definition \ref{defn3.5}.
\end{itemize}

\end{prop}
\begin{proof}
Fix the extension of $\mu$. If $\eta$ agrees with $\mu$ on $\mathbb{A}_{F}^{(1)}$, write $\eta=\mu|\cdot|^{it_{\eta}}$. By \eqref{f3.7}, the poles that can cross $i\mathbb{R}$ in \eqref{equ3.12} are the eight poles indexed in Definition \ref{defn3.5}, with their displayed locations shifted by $it_{\eta}$. The numerator factors may also have poles where their arguments have real part zero, but these cannot cross $i\mathbb{R}$ in this region, since $\Re(s_{2}\pm s_{1}),\Re(s_{4}\pm s_{3})>-1/2$.

Start in the region of absolute convergence and continue along a path in \eqref{equ3.12}, initially keeping the poles distinct. The zero-free region for the denominator in \eqref{f3.7} allows us to deform $i\mathbb{R}$ near each pole as it crosses. Rapid decay in vertical strips gives locally normal convergence on these contours. The residue theorem gives the meromorphic continuation, including parameters where poles coincide.

For $\mathbf{s}\in\mathcal{R}_{\mathrm{RS}}$, the poles in the odd-numbered terms of Definition \ref{defn3.5} have moved from the right of $i\mathbb{R}$ to the left, and those in the even-numbered terms have moved from the left to the right. Returning the deformed contour to $i\mathbb{R}$ therefore subtracts the odd residues and adds the even residues. The change of variable $\lambda\mapsto\lambda+it_{\eta}$ identifies each shifted residue with the corresponding period evaluated at the full character $\eta$. This gives the asserted residue formula.
\end{proof}

\subsubsection{Meromorphic Continuation of \texorpdfstring{$\mathcal{J}_{\mathrm{RS}}(\mathbf{s},\mathfrak{X})$}{}}
Equations \eqref{eq3.5} and \eqref{reg1}, together with Proposition \ref{prop3.4}, give the following theorem.
\begin{thm}\label{thm3.5}
There is a meromorphic continuation $\widetilde{\mathcal{J}}_{\mathrm{RS}}(\mathbf{s},\mathfrak{X})$ of $\mathcal{J}_{\mathrm{RS}}(\mathbf{s},\mathfrak{X})$ to the region \eqref{equ3.12}. In particular, for $\mathbf{s}\in \mathcal{R}_{\mathrm{RS}}$, it is given explicitly by
\begin{align*}
\widetilde{\mathcal{J}}_{\mathrm{RS}}(\mathbf{s},\mathfrak{X})=\mathcal{J}_{\mathrm{Spec}}^{\heartsuit}(\mathbf{s},\mathfrak{X})-\sum_{i=1}^{8}\widetilde{\Psi}_{\mathrm{RS}}^{(i)}(\mathbf{s},\mathfrak{X}),
\end{align*}
where $\mathcal{J}_{\mathrm{Spec}}^{\heartsuit}(\mathbf{s},\mathfrak{X})$ is defined by \eqref{reg1} (see Definition \ref{defn3.2}), and $\widetilde{\Psi}_{\mathrm{RS}}^{(i)}(\mathbf{s},\mathfrak{X})$ is defined by Definition \ref{defn3.5}.
\end{thm}

\subsection{Meromorphic Continuation of \texorpdfstring{$\mathcal{I}_{\mathrm{Dual}}(\mathbf{s},\mathfrak{X})$}{}}\label{sec3.dual-continuation}
Let $\omega'=\omega_{1}\overline{\omega_{3}}$. For $\sigma\in \mathcal{A}_{0}([G],\omega')$, we define 
\begin{align*}
\Psi^{*}(\sigma;\mathbf{s},\mathfrak{X}):=\sum_{\phi\in\mathfrak{B}(\sigma)}\Psi(\overline{W_{\phi}},W_{1}(\cdot,s_{1}),\overline{h_{3}(\cdot,\overline{s_{3}})})\Psi(W_{\phi}^{*},W_{2}(\cdot,s_{2}),\overline{h_{4}(\cdot,\overline{s_{4}})}).
\end{align*}
Denote by $\widetilde{\Psi}^{*}(\sigma;\mathbf{s},\mathfrak{X})$ the meromorphic continuation of $\Psi^{*}(\sigma;\mathbf{s},\mathfrak{X})$.

Let $\mu\in \widehat{C}_{F}$, and $\lambda\in i\mathbb{R}$. Define $\Psi^{*}(\lambda,\mu;\mathbf{s},\mathfrak{X})$ by 
\begin{align*}
\frac{1}{2}\sum_{h}\Psi(\overline{W_{E(\cdot,h,-\overline{\lambda})}},W_{1}(\cdot,s_{1}),\overline{h_{3}(\cdot,\overline{s_{3}})})\Psi(W_{E(\cdot,h,\lambda)}^{*},W_{2}(\cdot,s_{2}),\overline{h_{4}(\cdot,\overline{s_{4}})}),
\end{align*}
where $h$ ranges over $\mathfrak{B}(\mu,\overline{\mu}\omega')$. 

Parallel to Lemma \ref{lem3.3}, we have the following meromorphic continuation. 
\begin{lemma}\label{lem4.1}
 Let $\mathcal{D}_{\overline{\mu}^{2}\omega'}$ be defined by \eqref{eq3.12} in \textsection\ref{sec3.zero-free}.
\begin{itemize}
\item $\Psi^{*}(\lambda,\mu;\mathbf{s},\mathfrak{X})$ converges absolutely in the region 
\begin{align*}
\begin{cases}
\Re(s_{3})-|\Re(s_{1})|-|\Re(\lambda)|>1/2,\\
\Re(s_{4})-|\Re(s_{2})|-|\Re(\lambda)|>1/2,\\
2\lambda\in \mathcal{D}_{\overline{\mu}^{2}\omega'}.
\end{cases}	
\end{align*}

\item $\Psi^{*}(\lambda,\mu;\mathbf{s},\mathfrak{X})$ admits a meromorphic continuation $\widetilde{\Psi}^{*}(\lambda,\mu;\mathbf{s},\mathfrak{X})$ to $(\lambda,\mathbf{s})\in \mathbb{C}^{5}$. For $|\Re(\lambda)|<1/4$, it satisfies
\begin{multline*}
\widetilde{\Psi}^{*}(\lambda,\mu;\mathbf{s},\mathfrak{X})\propto \ \Lambda(1+2\lambda,\mu^{2}\overline{\omega}')^{-1}\Lambda(1-2\lambda,\overline{\mu}^{2}\omega')^{-1}\\
\Lambda(1/2+s_{3}+s_{1}-\lambda,\overline{\mu}\chi_{1}\chi_{3}^{-1})\Lambda(1/2+s_{3}+s_{1}+\lambda,\mu\omega'^{-1}\chi_{1}\chi_{3}^{-1})\\
\Lambda(1/2+s_{3}-s_{1}-\lambda,\overline{\mu}\chi_{1}^{-1}\omega_{1}\chi_{3}^{-1})\Lambda(1/2+s_{3}-s_{1}+\lambda,\mu\omega'^{-1}\chi_{1}^{-1}\omega_{1}\chi_{3}^{-1})\\
\Lambda(1/2+s_{4}+s_{2}+\lambda,\mu\chi_{2}\chi_{4}^{-1})\Lambda(1/2+s_{4}+s_{2}-\lambda,\overline{\mu}\omega'\chi_{2}\chi_{4}^{-1})\\
\Lambda(1/2+s_{4}-s_{2}+\lambda,\mu\chi_{2}^{-1}\omega_{2}\chi_{4}^{-1})\Lambda(1/2+s_{4}-s_{2}-\lambda,\overline{\mu}\omega'\chi_{2}^{-1}\omega_{2}\chi_{4}^{-1}).
\end{multline*}
\end{itemize}
\end{lemma}

\begin{defn}\label{defn3.7}
Let $\mathbf{s}\in\mathcal{R}_{\mathrm{Dual}}:=\{\mathbf{s}:\ |\Re(s_{1})|+|\Re(s_{3})|<1/2,\ |\Re(s_{2})|+|\Re(s_{4})|<1/2\}$. We define 
\begin{equation}\label{reg2}
\mathcal{I}_{\mathrm{Spec}}^{\heartsuit}(\mathbf{s},\mathfrak{X})=\mathcal{I}_{\mathrm{Cusp}}^{\heartsuit}(\mathbf{s},\mathfrak{X})+\mathcal{I}_{\mathrm{Eis}}^{\heartsuit}(\mathbf{s},\mathfrak{X}),
\end{equation}
where
\begin{align*}
\mathcal{I}_{\mathrm{Cusp}}^{\heartsuit}(\mathbf{s},\mathfrak{X}):=\sum_{\sigma\in \mathcal{A}_{0}([G],\omega')}\widetilde{\Psi}^{*}(\sigma;\mathbf{s},\mathfrak{X}),\quad 
\mathcal{I}_{\mathrm{Eis}}^{\heartsuit}(\mathbf{s},\mathfrak{X}):=\sum_{\substack{\mu\in \widehat{C}_{F}}}\frac{1}{2\pi i}\int_{i\mathbb{R}}\widetilde{\Psi}^{*}(\lambda,\mu;\mathbf{s},\mathfrak{X})d\lambda.
\end{align*}
\end{defn} 

The convergence argument after Definition \ref{defn3.2} applies on $\mathcal{R}_{\mathrm{Dual}}$, again with the continuous integral taken on $i\mathbb{R}$.

\begin{defn}\label{defn3.9}
 Let $\mathbf{s}\in \mathbb{C}^{4}$ and $\mu\in\widehat{C}_{F}$. As in Definition~\ref{defn3.5}, we first take the following residues when the poles are simple and distinct, and then continue meromorphically in $\mathbf{s}$:
\begin{align*}
&\widetilde{\Psi}_{\mathrm{Dual},\mu}^{(1)}(\mathbf{s},\mathfrak{X}):=-\mathbf{1}_{\mu=\chi_{1}\chi_{3}^{-1}}\underset{\lambda=s_{3}+s_{1}-1/2}{\Res}\ \widetilde{\Psi}^{*}(\lambda,\chi_{1}\chi_{3}^{-1};\mathbf{s},\mathfrak{X}),\\
&\widetilde{\Psi}_{\mathrm{Dual},\mu}^{(2)}(\mathbf{s},\mathfrak{X}):=\mathbf{1}_{\mu=\omega'\overline{\chi_{1}}\chi_{3}}\underset{\lambda=1/2-s_{3}-s_{1}}{\Res}\ \widetilde{\Psi}^{*}(\lambda,\omega'\overline{\chi_{1}}\chi_{3};\mathbf{s},\mathfrak{X}),\\
&\widetilde{\Psi}_{\mathrm{Dual},\mu}^{(3)}(\mathbf{s},\mathfrak{X}):=-\mathbf{1}_{\mu=\overline{\chi_{1}}\omega_{1}\overline{\chi_{3}}}\underset{\lambda=s_{3}-s_{1}-1/2}{\Res}\ \widetilde{\Psi}^{*}(\lambda,\overline{\chi_{1}}\omega_{1}\overline{\chi_{3}};\mathbf{s},\mathfrak{X}),\\
&\widetilde{\Psi}_{\mathrm{Dual},\mu}^{(4)}(\mathbf{s},\mathfrak{X}):=\mathbf{1}_{\mu=\omega'\chi_{1}\overline{\omega_{1}}\chi_{3}}\underset{\lambda=1/2+s_{1}-s_{3}}{\Res}\ \widetilde{\Psi}^{*}(\lambda,\omega'\chi_{1}\overline{\omega_{1}}\chi_{3};\mathbf{s},\mathfrak{X}),\\
&\widetilde{\Psi}_{\mathrm{Dual},\mu}^{(5)}(\mathbf{s},\mathfrak{X}):=-\mathbf{1}_{\mu=\omega'\overline{\chi_{2}}\omega_{2}\overline{\chi_{4}}}\underset{\lambda=s_{4}-s_{2}-1/2}{\Res}\ \widetilde{\Psi}^{*}(\lambda,\omega'\overline{\chi_{2}}\omega_{2}\overline{\chi_{4}};\mathbf{s},\mathfrak{X}),\\
&\widetilde{\Psi}_{\mathrm{Dual},\mu}^{(6)}(\mathbf{s},\mathfrak{X}):=\mathbf{1}_{\mu=\chi_{2}\overline{\omega_{2}}\chi_{4}}\underset{\lambda=1/2+s_{2}-s_{4}}{\Res}\ \widetilde{\Psi}^{*}(\lambda,\chi_{2}\overline{\omega_{2}}\chi_{4};\mathbf{s},\mathfrak{X}),\\
&\widetilde{\Psi}_{\mathrm{Dual},\mu}^{(7)}(\mathbf{s},\mathfrak{X}):=-\mathbf{1}_{\mu=\omega'\chi_{2}\overline{\chi_{4}}}\underset{\lambda=s_{4}+s_{2}-1/2}{\Res}\ \widetilde{\Psi}^{*}(\lambda,\omega'\chi_{2}\overline{\chi_{4}};\mathbf{s},\mathfrak{X}),\\
&\widetilde{\Psi}_{\mathrm{Dual},\mu}^{(8)}(\mathbf{s},\mathfrak{X}):=\mathbf{1}_{\mu=\overline{\chi_{2}}\chi_{4}}\underset{\lambda=1/2-s_{2}-s_{4}}{\Res}\ \widetilde{\Psi}^{*}(\lambda,\overline{\chi_{2}}\chi_{4};\mathbf{s},\mathfrak{X}),
\end{align*}
The global correction terms are
\begin{align*}
\widetilde{\Psi}_{\mathrm{Dual}}^{(i)}(\mathbf{s},\mathfrak{X}):=\sum_{\mu\in\widehat{C}_{F}}\widetilde{\Psi}_{\mathrm{Dual},\mu}^{(i)}(\mathbf{s},\mathfrak{X}),\qquad 1\leq i\leq 8,
\end{align*}
where each sum has at most one nonzero term.
\end{defn}

The spectral expansion of $\mathcal{I}_{\mathrm{Dual}}$ in the initial region follows by applying Lemma \ref{lem3.1} to the two Rankin--Selberg kernels. Repeating the contour shifts in Proposition \ref{prop3.4}, now with the pairs $(s_{1},s_{3})$ and $(s_{2},s_{4})$, gives the following result.
\begin{thm}\label{thm4.1}
 There is a meromorphic continuation $\widetilde{\mathcal{I}}_{\mathrm{Dual}}(\mathbf{s},\mathfrak{X})$ of $\mathcal{I}_{\mathrm{Dual}}(\mathbf{s},\mathfrak{X})$ to the region
\begin{align*}
\begin{cases}
\Re(s_{3})-|\Re(s_{1})|>-1/2,\\
\Re(s_{4})-|\Re(s_{2})|>-1/2.
\end{cases}	
\end{align*}
In particular, for $\mathbf{s}\in \mathcal{R}_{\mathrm{Dual}}$, it is given explicitly by
\begin{align*}
\widetilde{\mathcal{I}}_{\mathrm{Dual}}(\mathbf{s},\mathfrak{X})=\mathcal{I}_{\mathrm{Spec}}^{\heartsuit}(\mathbf{s},\mathfrak{X})+\sum_{i=1}^{8}\widetilde{\Psi}_{\mathrm{Dual}}^{(i)}(\mathbf{s},\mathfrak{X}),
\end{align*}
where $\mathcal{I}_{\mathrm{Spec}}^{\heartsuit}(\mathbf{s},\mathfrak{X})$ is defined by \eqref{reg2} (see Definition \ref{defn3.7}), and $\widetilde{\Psi}_{\mathrm{Dual}}^{(i)}(\mathbf{s},\mathfrak{X})$ is defined in Definition \ref{defn3.9}.
\end{thm}

\subsection{Integral Formulas for the Residues}
Following the method of \cite[\textsection 7]{Yan23a}, we write the residual terms as geometric integrals in the initial data, using the auxiliary Whittaker integrals of \textsection\ref{sec2.1.auxiliary}.

\begin{prop}\label{prop3.residual-integrals}
We have
\begin{equation}\label{eq3.rs-residual-integral}
\widetilde{\Psi}_{\mathrm{RS}}^{(1)}=\widetilde{\Psi}_{\mathrm{RS}}^{(2)}=-2^{-1}\widetilde{\Psi}(W_{h_{1}^{\Diamond},h_{2}^{\Diamond}},\overline{W_{3}},\overline{h_{4}}),
\end{equation}
and
\begin{equation}\label{eq3.dual-residual-integral}
\widetilde{\Psi}_{\mathrm{Dual}}^{(1)}=\widetilde{\Psi}_{\mathrm{Dual}}^{(2)}=2^{-1}\widetilde{\Psi}(W_{h_{1}^{\Diamond},\overline{h_{3}^{\Diamond}}}^{*},W_{2},\overline{h_{4}}).
\end{equation}
Here each unbarred section or Whittaker function is evaluated at its parameter $s_{j}$, and a bar includes evaluation at $\overline{s_{j}}$ followed by conjugation.

Replacing $h_{1}^{\Diamond}$ by $h_{1}$ on the right gives $i=3,4$ in each formula. In \eqref{eq3.rs-residual-integral}, exchanging $1\leftrightarrow3$, $2\leftrightarrow4$ and conjugating gives $i=7,8$; making the same change after the replacement gives $i=5,6$. In \eqref{eq3.dual-residual-integral}, exchanging $1\leftrightarrow2$, $3\leftrightarrow4$, without conjugation, gives $i=7,8$, and gives $i=5,6$ after the replacement. All identities hold meromorphically in the parameters.
\end{prop}
\begin{proof}
We first prove \eqref{eq3.rs-residual-integral} when the poles are simple and distinct. Put $f=h_{1}^{\Diamond}h_{2}^{\Diamond}$, $\eta=\omega(\chi_{1}\chi_{2})^{-1}$ and $b=1/2-s_{1}-s_{2}$. Then $f$ is a section in $\eta|\cdot|^{b}\boxplus\overline{\eta}\omega|\cdot|^{-b}$. By \eqref{2.6}, the constant term of $\Eis(\mathcal{G}_{1}h_{2})$ consists of $h_{1}h_{2}^{\Diamond}$, $f$, and the negatives of the intertwining sections of $h_{1}h_{2}$ and $h_{1}^{\Diamond}h_{2}$, up to a rapidly decreasing function in the cusp. Their exponents in $|y|$ are $1+s_{1}-s_{2}$, $1-s_{1}-s_{2}$, $-s_{1}-s_{2}$ and $s_{1}-s_{2}$, respectively.

Initially take the parameters in the range of Lemma~\ref{lem3.1} and unfold \eqref{2.10} against $E(\cdot,h,-\overline{\lambda})$. The bounds in the proof of Lemma~\ref{lem2.2} make the kernel bounded, so the integral over $|y|\leq1$ is holomorphic for $\Re(\lambda)<-1/2$, a region containing $b$. After subtracting the four terms above, the integral over $|y|\geq1$ is entire in $\lambda$. Since the poles are distinct, only $f(a(y)k)=\eta(y)|y|^{b+1/2}f(k)$ contributes to the pole at $\lambda=b$. Since $\Vol(C_{F})=\rho_{F}$ and $\int_{1}^{\infty}r^{b-\lambda}dr/r=(\lambda-b)^{-1}$, we obtain
\begin{align*}
\underset{\lambda=b}{\Res}\,\widetilde{\Psi}(\overline{W_{E(\cdot,h,-\overline{\lambda})}},W_{1},h_{2})=\rho_{F}\int_{K} f(k)\overline{h(k)}dk=\langle f,h\rangle_{H}.
\end{align*}

Choose the orthonormal basis of the compact model according to $K$-types. On a small circle around $b$ containing no other pole, integration by parts as in the proof of Lemma~\ref{lem3.3} gives rapid decay in the $K$-types and justifies taking residues term by term. The expansion of $f$ converges in $C^{\infty}(K)$, so Parseval replaces the sum over $h$ in \eqref{c3.5} by the Jacquet integral $W_{h_{1}^{\Diamond},h_{2}^{\Diamond}}$. The factor $1/2$ in \eqref{c3.5} and the minus sign in Definition~\ref{defn3.5} give the asserted formula for $i=2$.

On $\lambda\in i\mathbb{R}$, the Eisenstein functional equation and unitarity of the intertwining operator give $\widetilde{\Psi}(\lambda,\mu)=\widetilde{\Psi}(-\lambda,\overline{\mu}\omega)$. This extends meromorphically. The change $\lambda\mapsto-\lambda$ reverses the residue sign, giving the same formula for $i=1$. Using the term $h_{1}h_{2}^{\Diamond}$ instead of $f$ gives $i=3,4$; exchanging the two periods and conjugating gives the other four RS terms.

For the Dual terms, apply this calculation to the sections $(h_{1},\overline{h}_{3},\overline{h}_{2},h_{4})$, with parameters $(s_{1},s_{3},s_{2},s_{4})$, conjugating the inducing characters of $h_{2}$ and $h_{3}$. In the second period, the change of variables $g\mapsto\diag(-1,1)g$ transfers the star from $W_{2}$ to the first Whittaker function. All inducing characters are trivial on $F^{\times}$, so this introduces no scalar. The signs in Definition~\ref{defn3.9} are opposite to those in Definition~\ref{defn3.5}, giving \eqref{eq3.dual-residual-integral} and the stated replacements.

The Jacquet and outer Rankin--Selberg integrals have a common convergence region after continuing the intertwining sections; for example, for \eqref{eq3.rs-residual-integral}, take $\Re(s_{1}+s_{2})<0$ and $\Re(s_{4})$ sufficiently large. When poles coincide, each residue is continued meromorphically in the parameters. The values at $\mathbf{0}$ in Part~2 use the contour integral of \textsection\ref{sec4.regularization}.
\end{proof}

\subsection{The Symmetric Spectral Reciprocity}\label{sec3.6}

An equivalent indexed form of Theorem~\ref{thmA} on $\mathcal{R}^{\heartsuit}$ is:
\begin{equation}\label{eq3.indexed-reciprocity}
\mathcal{J}_{\mathrm{\Spec}}^{\heartsuit}(\mathbf{s},\mathfrak{X})=\ \mathcal{I}_{\Spec}^{\heartsuit}(\mathbf{s},\mathfrak{X})+\sum_{i=1}^{8}\Big[\widetilde{\Psi}_{\mathrm{Geo}}^{(i)}(\mathbf{s},\mathfrak{X})+\widetilde{\Psi}_{\mathrm{RS}}^{(i)}(\mathbf{s},\mathfrak{X})+\widetilde{\Psi}_{\mathrm{Dual}}^{(i)}(\mathbf{s},\mathfrak{X})\Big],
\end{equation}
where the three families of correction terms are defined in Definitions~\ref{defn2.5}, \ref{defn3.5} and~\ref{defn3.9}, respectively. Expressions involving $\mathcal{C}_{j}=h_{j}+h_{j}^{\Diamond}$ in \eqref{1.3} are expanded linearly in each section, with the parameter convention of Proposition~\ref{prop3.residual-integrals}.

\begin{proof}[Proof of Theorem~\ref{thmA}]
Since $\mathcal{R}^{\heartsuit}=\mathcal{R}_{\mathrm{RS}}\cap \mathcal{R}_{\mathrm{Dual}}$, the identity in Proposition~\ref{prop2.5} extends meromorphically from its region of absolute convergence to $\mathcal{R}^{\heartsuit}$. Substituting Theorems~\ref{thm3.5} and~\ref{thm4.1} gives \eqref{eq3.indexed-reciprocity}. Combining Definition~\ref{defn2.5} with Proposition~\ref{prop3.residual-integrals} and adding each pair of equal residues yields the six terms in \eqref{1.3}. The change of variables $g\mapsto a(-1)g$ transfers stars between the Whittaker factors without a scalar, since the inducing Hecke characters are trivial on $F^{\times}$.
\end{proof}

\part{Subconvexity for \texorpdfstring{$\mathrm{GL}_{2}$}{GL2}}
In this part we specialize the automorphic datum $\mathfrak{X}$ in Theorem \ref{thmA} to prove the explicit subconvexity bound in Theorem \ref{thmB}.  

\section{Choice of the Automorphic Data}\label{sec4}
\subsection{Automorphic Data}\label{sec4.data}
Let $\mathfrak{n}, \mathfrak{q}\subseteq\mathcal{O}_{F}$ be nonzero integral ideals satisfying $\mathfrak{n}+\mathfrak{q}\mathfrak{D}_{F}=\mathcal{O}_{F}$. Write $\mathfrak{n}=\prod_{v<\infty}\mathfrak{p}_{v}^{l_{v}}$, where $l_{v}=e_{v}(\mathfrak{n})$. Let $\mathbf{d}_{\mathfrak{n}}=(\mathbf{d}_{v})_{v}\in G(\mathbb{A}_{F})$, where $\mathbf{d}_{v}=I_{2}$ if $v\mid\infty$ or $v\nmid\mathfrak{n}$, and $\mathbf{d}_{v}:=\begin{pmatrix}
1&\\
&\varpi_{v}^{l_{v}}
\end{pmatrix}$ if $v\mid\mathfrak{n}$. 

\subsubsection{Archimedean Conductor Parameter}\label{sec4.1.1}
For each $v\mid\infty$, let $\mathbf{C}_{v}\geq 1$, and set $\mathbf{C}_{\infty}:=\prod_{v\mid\infty}\mathbf{C}_{v}$.  

For a Hecke character $\chi$ or an automorphic representation $\pi$ of $\mathrm{GL}_{2}/F$, write $C_{v}(\chi)$ and $C_{v}(\pi)$ for the local conductors, and $C_{\fin}$ and $C_{\infty}$ for their finite and Archimedean products. We use the full analytic conductors $C(\chi)=D_{F}C_{\fin}(\chi)C_{\infty}(\chi)$ and $C(\pi)=D_{F}^{2}C_{\fin}(\pi)C_{\infty}(\pi)$.  

\subsubsection{Hecke Operators}
Let $\pi=\otimes_{v}\pi_{v}$ be a unitary automorphic representation of $G(\mathbb{A}_{F})$, with central character $\omega=\otimes_{v}\omega_{v}$. Let $\mathfrak{n}$ be as above, and write $\mathfrak{p}=\mathfrak{p}_{v}$ at a finite place $v$. The $\mathfrak{n}$-th Hecke operator is defined by 
\begin{equation}\label{hecke..}
T_{\mathfrak{n}}(\varphi):=\prod_{v\mid\mathfrak{n}}q_{v}^{-\frac{l_{v}}{2}}\sum_{\substack{i+j=l_{v}\\
i\geq j\geq 0}}\int_{K_{v}\diag(\varpi_{v}^{i},\varpi_{v}^{j})K_{v}} \pi_{v}(y_{v})\varphi dy_{v},\ \ \varphi\in \pi.
\end{equation}

If $\varphi$ is right invariant under $\prod_{v\mid\mathfrak{n}}K_{v}$, then it is an eigenvector of $T_{\mathfrak{n}}=\prod_{v\mid\mathfrak{n}}T_{\mathfrak{p}_{v}^{l_{v}}}$. We denote the eigenvalue of $T_{\mathfrak{p}_{v}^{l_{v}}}$ by $\lambda_{\pi}(\mathfrak{p}_{v}^{l_{v}})$. 

Define also the modified Hecke operator
\begin{align*}
T_{\mathfrak{p}^{l_{v}}}^{*}(\varphi):=q_{v}^{-\frac{l_{v}}{2}}\int_{K_{v}\diag(\varpi_{v}^{l_{v}},1)K_{v}} \pi_{v}(y_{v})\varphi dy_{v},\ \ \varphi\in \pi.
\end{align*} 

When $\varphi$ is right-$K_{v}$-invariant, we denote the eigenvalue of $T_{\mathfrak{p}^{l_{v}}}^{*}$ by $\lambda_{\pi}^{*}(\mathfrak{p}^{l_{v}})$. For $\pi$ unramified at every $v\mid\mathfrak{n}$, set $\lambda_{\pi}^{*}(\mathfrak{n}):=\prod_{v\mid\mathfrak{n}}\lambda_{\pi}^{*}(\mathfrak{p}_{v}^{l_{v}})$.

\subsubsection{Hecke relations}
Let $v<\infty$. We have 
\begin{equation}\label{eq2.3}
\overline{\lambda_{\pi}(\mathfrak{p}^{l_{v}})}=\omega_{v}^{-1}(\varpi_{v}^{l_{v}})\cdot \lambda_{\pi}(\mathfrak{p}^{l_{v}}),
\end{equation}
and the Hecke relation
\begin{equation}\label{hecke.}
\lambda_{\pi}(\mathfrak{p}^{l_{v}})^{2}=\sum_{j=0}^{l_{v}}\omega_{v}(\varpi_{v}^{j})\lambda_{\pi}(\mathfrak{p}^{2l_{v}-2j}).
\end{equation}

The Hecke relation \eqref{hecke.} gives the following identities for the modified eigenvalues.
\begin{lemma}\label{lemma4.1}
We have
\begin{equation}\label{c4.5}
\begin{cases}
\overline{\lambda_{\pi}^{*}(\mathfrak{p})}=\omega_{v}^{-1}(\varpi_{v})\lambda_{\pi}^{*}(\mathfrak{p}),\ \ \ \overline{\lambda_{\pi}^{*}(\mathfrak{p}^{2})}=\omega_{v}^{-2}(\varpi_{v})\lambda_{\pi}^{*}(\mathfrak{p}^{2}),\\
|\lambda_{\pi}^{*}(\mathfrak{p})|^{2}=\omega_{v}^{-1}(\varpi_{v})\lambda_{\pi}^{*}(\mathfrak{p}^{2})+q_{v}^{-1}+1,\\
|\lambda_{\pi}^{*}(\mathfrak{p}^{2})|^{2}=\omega_{v}^{-2}(\varpi_{v})\lambda_{\pi}^{*}(\mathfrak{p}^{4})+\omega_{v}^{-1}(\varpi_{v})(1-q_{v}^{-1})\lambda_{\pi}^{*}(\mathfrak{p}^{2})+q_{v}^{-1}+1.
\end{cases}
\end{equation}	
\end{lemma}
\begin{proof}
It follows from the definition \eqref{hecke..} that  
\begin{align*}
\lambda_{\pi}(\mathfrak{p}^{l_{v}})=\sum_{0\leq j\leq l_{v}/2}\omega_{v}(\varpi_{v}^{j})\lambda_{\pi}^{*}(\mathfrak{p}^{l_{v}-2j})q_{v}^{-j}.
\end{align*}
We obtain
\begin{equation}\label{f4.5}
\begin{cases}
\lambda_{\pi}(\mathfrak{p})=\lambda_{\pi}^{*}(\mathfrak{p}),\\
\lambda_{\pi}(\mathfrak{p}^{2})=\lambda_{\pi}^{*}(\mathfrak{p}^{2})+\omega_{v}(\varpi_{v})q_{v}^{-1},\\
\lambda_{\pi}(\mathfrak{p}^{3})=\lambda_{\pi}^{*}(\mathfrak{p}^{3})+\omega_{v}(\varpi_{v})\lambda_{\pi}(\mathfrak{p})q_{v}^{-1},\\
\lambda_{\pi}(\mathfrak{p}^{4})=\lambda_{\pi}^{*}(\mathfrak{p}^{4})+\omega_{v}(\varpi_{v})\lambda_{\pi}^{*}(\mathfrak{p}^{2})q_{v}^{-1}+\omega_{v}(\varpi_{v}^{2})q_{v}^{-2}.
\end{cases}
\end{equation}

The first line of \eqref{c4.5} follows from \eqref{eq2.3} and \eqref{f4.5}. 
Consequently, we obtain the second line of \eqref{c4.5}: 
\begin{align*}
|\lambda_{\pi}^{*}(\mathfrak{p})|^{2}=|\lambda_{\pi}(\mathfrak{p})|^{2}=\omega_{v}^{-1}(\varpi_{v})\lambda_{\pi}(\mathfrak{p}^{2})+1=\omega_{v}^{-1}(\varpi_{v})\lambda_{\pi}^{*}(\mathfrak{p}^{2})+q_{v}^{-1}+1.
\end{align*}

Moreover, using the second relation in \eqref{f4.5} together with \eqref{eq2.3}, we obtain
\begin{equation}\label{eq4.5}
|\lambda_{\pi}^{*}(\mathfrak{p}^{2})|^{2}
=\omega_{v}^{-1}(\varpi_{v}^{2})\lambda_{\pi}(\mathfrak{p}^{2})^{2}-2\omega_{v}^{-1}(\varpi_{v})q_{v}^{-1}\lambda_{\pi}(\mathfrak{p}^{2})+q_{v}^{-2}.
\end{equation}

Substituting the expansion of $\lambda_{\pi}(\mathfrak{p}^{2})^{2}$ (see \eqref{hecke.}) into \eqref{eq4.5}, along with the relations in \eqref{f4.5}, we thus obtain the last relation in \eqref{c4.5}.  
\end{proof}

\subsubsection{Ramanujan Bounds}\label{sec4.ramanujan}
Let $\vartheta$ be an admissible Ramanujan exponent: when a local component $\pi_{v}$ of a unitary cuspidal automorphic representation of $\mathrm{GL}_{2}/F$ is written as a normalized parabolic induction from tempered representations, all real inducing exponents have absolute value at most $\vartheta$. This includes ramified places and implies that $L_{v}(s,\pi_{v})$ is holomorphic in $\Re(s)>\vartheta$. By \cite{KS03} and \cite[Definition~1 and Theorem~1]{BB11}, one may take $\vartheta=7/64$.

\subsection{Construction of Eisenstein Series}\label{sec4.4}
Let $\mathfrak{n}$ and $\mathfrak{q}$ be as in \textsection\ref{sec4.data}. Let $\omega=\otimes_{v}\omega_{v}$ be a unitary Hecke character of $F^{\times}\backslash\mathbb{A}_{F}^{\times}$ with finite conductor ideal $\mathfrak{q}'\supseteq\mathfrak{q}$; thus its arithmetic conductor is $N_{F}(\mathfrak{q}')$. 
For $v\mid\mathfrak{q}$ and $k_{v}=\begin{pmatrix}a_{v}&b_{v}\\c_{v}&d_{v}\end{pmatrix}\in K_{0,v}[e_{v}(\mathfrak{q})]$, set $\eta_{\omega,v}(k_{v}):=\omega_{v}(d_{v})$; for $v\nmid\mathfrak{q}$, let $\eta_{\omega,v}$ be trivial on $K_{v}$. Since $e_{v}(\mathfrak{q}')\leq e_{v}(\mathfrak{q})$, the product $\eta_{\omega}:=\otimes_{v<\infty}\eta_{\omega,v}$ defines a character of $K_{0}[\mathfrak{q}]$. Throughout the paper, if a representation $V$ has central character $\omega$, the notation $V^{K_{0}[\mathfrak{q}]}$ denotes its $\eta_{\omega}$-isotypic subspace:
\begin{align*}
V^{K_{0}[\mathfrak{q}]}:=\{\varphi\in V:\ \pi(k)\varphi=\eta_{\omega}(k)\varphi\text{ for every }k\in K_{0}[\mathfrak{q}]\}.
\end{align*}
For $\omega=\mathbf{1}$, this is the space of $K_{0}[\mathfrak{q}]$-fixed vectors.
\begin{defn}[Construction of $\boldsymbol{\chi}$ and $\boldsymbol{\omega}$]
Take
\begin{equation}\label{4.5}
\chi_{1}=\chi_{2}=\chi_{3}=\chi_{4}=\mathbf{1},\ \ \omega_{1}=\omega_{3}=\mathbf{1},\ \ \omega_{2}=\omega_{4}=\omega.
\end{equation} 
\end{defn}

For each $v\mid\infty$, let $\alpha_{v}$ be a fixed nonnegative smooth function on $\mathbb{R}$, supported in $|t|\leq 10^{-1}$ and identically equal to $1$ on $|t|\leq 20^{-1}$. Wherever it occurs below, $\omega_{v}$ is extended by zero at the origin. Implied constants in the period estimates may depend on $F$, $\varepsilon$, and these functions $\alpha_{v}$.
\begin{itemize}
\item Let $v\mid\infty$, and $\mathbf{C}_{v}$ be defined as in \textsection \ref{sec4.1.1}. We define
\begin{equation}\label{f4.6}
\Phi_{v}^{*}(t_{1,v},t_{2,v}):=
\begin{cases}
e^{-\pi(t_{1,v}^{2}+t_{2,v}^{2})},\ \ &\text{if $F_{v}\simeq\mathbb{R}$}\\
e^{-2\pi(t_{1,v}\overline{t_{1,v}}+t_{2,v}\overline{t_{2,v}})},\ \ &\text{if $F_{v}\simeq\mathbb{C}$},
\end{cases}
\end{equation}
where $t_{1,v},t_{2,v}\in F_{v}$, and 
\begin{equation}\label{eq6.28}
\Phi_{v}(t_{1,v},t_{2,v}):=\mathbf{C}_{v}^{1/2}\omega_{v}(t_{2,v})\alpha_{v}(\mathbf{C}_{v}|t_{1,v}|_{v})\alpha_{v}(|t_{2,v}|_{v}-1),\ \ v\mid\infty.
\end{equation}

\item Let $v\mid\mathfrak{q}$, so that $m_{v}:=e_{v}(\mathfrak{q})\geq 1$ (see \textsection\ref{1.1.1}). We set
\begin{align*}
\Phi_{v}^{*}(t_{1,v},t_{2,v}):=\mathbf{1}_{\mathcal{O}_{v}}(t_{1,v})\mathbf{1}_{\mathcal{O}_{v}}(t_{2,v}),
\end{align*}
where $t_{1,v},t_{2,v}\in F_{v}$, and 
\begin{equation}\label{equa4.5}
\Phi_{v}(t_{1,v},t_{2,v}):=\Vol(K_{0,v}[m_{v}])^{-1}\mathbf{1}_{\mathfrak{p}_{v}^{m_{v}}}(t_{1,v})\mathbf{1}_{\mathcal{O}_{v}^{\times}}(t_{2,v})\omega_{v}(t_{2,v}).
\end{equation}
\item For the remaining finite places $v\nmid\mathfrak{q}$, we set
\begin{align*}
\Phi_{v}(t_{1,v},t_{2,v}):=\mathbf{1}_{\mathcal{O}_{v}}(t_{1,v})\mathbf{1}_{\mathcal{O}_{v}}(t_{2,v}).
\end{align*}
\end{itemize}
The conductor condition $e_{v}(\mathfrak{q}')\leq m_{v}$ implies that $\Phi_{v}$ transforms under $K_{0,v}[m_{v}]$ according to $\eta_{\omega,v}$. Thus its Godement section has the type specified above, whereas the section constructed from $\Phi_{v}^{*}$ is spherical at every finite place dividing $\mathfrak{q}$.

Let $\Phi(t_{1},t_{2}):=\otimes_{v}\Phi_{v}(t_{1,v},t_{2,v})\in \mathcal{S}(\mathbb{A}_{F}^{2})$, and $\Phi^{*}(t_{1},t_{2}):=\otimes_{v\mid\mathfrak{q}\infty}\Phi_{v}^{*}(t_{1,v},t_{2,v})\otimes \otimes_{v\nmid\mathfrak{q}\infty}\Phi_{v}(t_{1,v},t_{2,v})\in \mathcal{S}(\mathbb{A}_{F}^{2})$, where $t_{i}=(t_{i,v})_{v}\in \mathbb{A}_{F}$, $1\leq i\leq 2$. 

\begin{defn}[Construction of $\boldsymbol{\Phi}$]
Define
\begin{equation}\label{e4.10}
\Phi_{1}(t_{1},t_{2}):=\alpha_{\mathfrak{n}}|\det \mathbf{d}_{\mathfrak{n}}|^{1/2}\cdot\Phi^{*}((t_{1},t_{2})\mathbf{d}_{\mathfrak{n}}),\ \ \Phi_{3}(t_{1},t_{2}):=\Phi^{*}((t_{1},t_{2})),
\end{equation}
where $\alpha_{\mathfrak{n}}:=\prod_{v\mid\mathfrak{n}}\overline{\omega}_{v}(\varpi_{v}^{l_{v}})q_{v}^{l_{v}/2}(1+q_{v}^{-1})$, $t_{i}=(t_{i,v})_{v}\in \mathbb{A}_{F}$, $1\leq i\leq 2$, and 
\begin{equation}\label{e4.11}
\Phi_{2}(t_{1},t_{2}):=|\det \mathbf{d}_{\mathfrak{n}}|^{1/2}\cdot \Phi((t_{1},t_{2})\mathbf{d}_{\mathfrak{n}}),\ \ \Phi_{4}(t_{1},t_{2}):=\Phi((t_{1},t_{2})).
\end{equation}
\end{defn}

Consequently, for $s\in \mathbb{C}$, we have 
\begin{align*}
h_{1}(g,s)=\alpha_{\mathfrak{n}}|\det\mathbf{d}_{\mathfrak{n}}|^{-s}h_{3}(g\mathbf{d}_{\mathfrak{n}},s),\quad h_{2}(g,s)=|\det\mathbf{d}_{\mathfrak{n}}|^{-s}h_{4}(g\mathbf{d}_{\mathfrak{n}},s),
\end{align*}
as identities of meromorphic functions in $s$. The same identities hold with $h_{j}$ replaced by $W_{j}$. In particular, the determinant factors equal $1$ at $s=0$.

\subsection{Regularization via Contour Integrals}\label{sec4.regularization}
For $0<\varepsilon<10^{-3}$, define 
\begin{align*}
\mathbf{B}_{\varepsilon}^{4}:=\big\{(s_{1},s_{2},s_{3},s_{4})\in \mathbb{C}^{4}:\ |s_{1}|\leq \varepsilon,\ |s_{2}|\leq 2\varepsilon,\ |s_{3}|\leq 5\varepsilon,\ |s_{4}|\leq 10\varepsilon\big\}.
\end{align*}
Then $\mathbf{B}_{\varepsilon}^{4}\subset\mathcal{R}^{\heartsuit}$. Let $\mathbf{S}_{\varepsilon}^{4}:=\{\mathbf{s}\in\mathbb{C}^{4}:\ |s_{1}|=\varepsilon,\ |s_{2}|=2\varepsilon,\ |s_{3}|=5\varepsilon,\ |s_{4}|=10\varepsilon\}$, with each circle positively oriented.

Fix $\varepsilon_{0}>0$ sufficiently small. After canceling removable Tate factors, choose $0<\varepsilon\asymp_{F}\varepsilon_{0}$ so that $\mathbf{S}_{\varepsilon}^{4}$ stays a distance $\gg_{F}\varepsilon$ from the remaining poles of the correction terms and their Archimedean factors. We keep the same distance from the poles of $h_{2,v}^{\Diamond}$ and $h_{4,v}^{\Diamond}$ for $v\mid\infty$. Such a choice is possible because the radii $1,2,5,10$ keep every nonempty signed sum of the $s_{j}$ away from zero, while the remaining $O_{F}(1)$ pole locations can be avoided by adjusting $\varepsilon$. All terms use this same contour. Define
\begin{align*}
\widetilde{\Psi}_{*}^{(i)}(\mathbf{0}\mid\mathfrak{X}_{\mathfrak{n}}):=\frac{1}{16\pi^{4}}\iiiint_{\mathbf{S}_{\varepsilon}^{4}}\frac{\widetilde{\Psi}_{*}^{(i)}(\mathbf{s},\mathfrak{X}_{\mathfrak{n}})}{s_{1}s_{2}s_{3}s_{4}}ds_{1}ds_{2}ds_{3}ds_{4},\quad *\in \{\mathrm{Geo}, \mathrm{RS}, \mathrm{Dual}\}.
\end{align*} 

The spectral sides in \eqref{reg1} and \eqref{reg2} are holomorphic in this polydisc, so Cauchy's formula gives their values at $\mathbf{0}$. The correction terms at $\mathbf{0}$ mean the contour integrals above.

For these data, the spectral residues in \eqref{eq3.indexed-reciprocity} have Rankin--Selberg factorizations and can be bounded by the same local period estimates as the two moments. Applying the contour integral above to this identity gives the following theorem.

\begin{thmx}\label{thmD}
 Let $\mathfrak{X}_{\mathfrak{n}}=(\boldsymbol{\chi},\boldsymbol{\omega},\boldsymbol{\Phi})$ be the automorphic datum constructed in \textsection \ref{sec4.4} from $\mathfrak{q}$, $\mathfrak{n}$, $\omega$, the parameters $\mathbf{C}_{v}$, and the functions $\alpha_{v}$ for $v\mid\infty$. Then 
\begin{align*}
\mathcal{J}_{\mathrm{\Spec}}^{\heartsuit}(\mathbf{0},\mathfrak{X}_{\mathfrak{n}})=\ \mathcal{I}_{\Spec}^{\heartsuit}(\mathbf{0},\mathfrak{X}_{\mathfrak{n}})+\sum_{i=1}^{8}\Big[\widetilde{\Psi}_{\mathrm{Geo}}^{(i)}(\mathbf{0}\mid\mathfrak{X}_{\mathfrak{n}})+\widetilde{\Psi}_{\mathrm{RS}}^{(i)}(\mathbf{0}\mid\mathfrak{X}_{\mathfrak{n}})+\widetilde{\Psi}_{\mathrm{Dual}}^{(i)}(\mathbf{0}\mid\mathfrak{X}_{\mathfrak{n}})\Big].
\end{align*}
\end{thmx}

\section{The Original Spectral Side \texorpdfstring{$\mathcal{J}_{\mathrm{\Spec}}^{\heartsuit}(\mathbf{0},\mathfrak{X}_{\mathfrak{n}})$}{}}\label{sec5}

Recall from Definition \ref{defn3.2} that
\begin{align*}
\mathcal{J}_{\mathrm{Spec}}^{\heartsuit}(\mathbf{s},\mathfrak{X}_{\mathfrak{n}})=\mathcal{J}_{\mathrm{Cusp}}^{\heartsuit}(\mathbf{s},\mathfrak{X}_{\mathfrak{n}})+\mathcal{J}_{\mathrm{Eis}}^{\heartsuit}(\mathbf{s},\mathfrak{X}_{\mathfrak{n}}),\tag{\ref{reg1}}
\end{align*}
where 
\begin{align*}
&\mathcal{J}_{\mathrm{Cusp}}^{\heartsuit}(\mathbf{s},\mathfrak{X}_{\mathfrak{n}}):=\sum_{\pi\in \mathcal{A}_{0}([G],\omega)}\widetilde{\Psi}(\pi;\mathbf{s},\mathfrak{X}_{\mathfrak{n}}),\\
&\mathcal{J}_{\mathrm{Eis}}^{\heartsuit}(\mathbf{s},\mathfrak{X}_{\mathfrak{n}}):=\sum_{\substack{\mu\in \widehat{C}_{F}}}\frac{1}{2\pi i}\int_{i\mathbb{R}}\widetilde{\Psi}(\lambda,\mu;\mathbf{s},\mathfrak{X}_{\mathfrak{n}})d\lambda,
\end{align*}
where $\widetilde{\Psi}(\pi;\mathbf{s},\mathfrak{X}_{\mathfrak{n}})$ and $\widetilde{\Psi}(\lambda,\mu;\mathbf{s},\mathfrak{X}_{\mathfrak{n}})$ are the meromorphic functions defined in \textsection\ref{sec3.3.1}. 

\begin{lemma}\label{lem5.1}
 Let $\pi=\otimes_{v}\pi_{v}$ be a unitary generic automorphic representation of $G(\mathbb{A}_{F})$ with central character $\omega$, such that $\pi_{v}$ is unramified at $v\mid\mathfrak{n}$. Let $\phi\in \pi$ be a right-$\prod_{v\mid\mathfrak{n}}K_{v}$-invariant vector with Whittaker function $W_{\phi}$. Then 
\begin{equation}\label{4.12}
\widetilde{\Psi}(\overline{W_{\phi}},W_{1}(\cdot,0),h_{2}(\cdot,0))=\overline{\lambda_{\pi}^{*}(\mathfrak{n})}\cdot \widetilde{\Psi}(\overline{W_{\phi}},W_{3}(\cdot,0),h_{4}(\cdot,0)).
\end{equation} 
\end{lemma}
\begin{proof}
By linearity, take $\phi$ to be a pure tensor. Rankin--Selberg theory gives
\begin{align*}
\widetilde{\Psi}(\overline{W_{\phi}},W_{i}(\cdot,0),h_{j}(\cdot,0))=\Lambda(1/2,\widetilde{\pi})^{2}\prod_{v}\frac{\Psi_{v}(\overline{W_{\phi,v}},W_{i,v}(\cdot,0),h_{j,v}(\cdot,0))}{L_{v}(1/2,\widetilde{\pi}_{v})^{2}},
\end{align*}
where $(i,j)\in\{(1,2),(3,4)\}$, and 
\begin{align*}
\Psi_{v}(\overline{W_{\phi,v}},W_{i,v}(\cdot,0),h_{j,v}(\cdot,0))=\int_{N(F_{v})\backslash \overline{G}(F_{v})}\overline{W_{\phi,v}(g_{v})}W_{i,v}(g_{v},0)h_{j,v}(g_{v},0)dg_{v},
\end{align*}
which converges absolutely. By \eqref{e4.10} and \eqref{e4.11}, and a change of variable, 
\begin{equation}\label{eq5.2}
\Psi_{v}(\overline{W_{\phi,v}},W_{1,v}(\cdot,0),h_{2,v}(\cdot,0))=	\Psi_{v}(\overline{W_{\phi,v}},W_{3,v}(\cdot,0),h_{4,v}(\cdot,0))	
\end{equation}
when $v\nmid\mathfrak{n}$, and for $v\mid\mathfrak{n}$, we have
\begin{equation}\label{e5.2}
\Psi_{v}(\cdots)=\overline{\omega}_{v}(\varpi_{v}^{l_{v}})q_{v}^{l_{v}/2}(1+q_{v}^{-1})\int\ \overline{W_{\phi,v}(g_{v}\mathbf{d}_{v}^{-1})}W_{3,v}(g_{v},0)h_{4,v}(g_{v},0)dg_{v},
\end{equation}
where $\Psi_{v}(\cdots)=\Psi_{v}(\overline{W_{\phi,v}},W_{1,v}(\cdot,0),h_{2,v}(\cdot,0))$, and $g_{v}$ ranges over $N(F_{v})\backslash \overline{G}(F_{v})$. 

Let $v\mid\mathfrak{n}$. The coset decomposition gives
\begin{align*}
K_{v}=\bigsqcup_{\alpha\in \mathcal{O}_{v}/\mathfrak{p}_{v}^{l_{v}}}\begin{pmatrix}
1&\alpha\\
&1
\end{pmatrix}wK_{0,v}[{l_{v}}]w\bigsqcup\bigsqcup_{\beta\in \mathfrak{p}_{v}/\mathfrak{p}_{v}^{l_{v}}}\begin{pmatrix}
&1\\
1 & \beta
\end{pmatrix}wK_{0,v}[{l_{v}}]w.
\end{align*}

Since $wK_{0,v}[{l_{v}}]w\begin{pmatrix}
\varpi_{v}^{l_{v}}\\
&1	
\end{pmatrix}=\begin{pmatrix}
\varpi_{v}^{l_{v}}\\
&1	
\end{pmatrix}K_{0,v}[{l_{v}}]$, we have 
\begin{multline}\label{5.1}
K_{v}\begin{pmatrix}
\varpi_{v}^{l_{v}}\\
&1	
\end{pmatrix}=\bigsqcup_{\alpha\in \mathcal{O}_{v}/\mathfrak{p}_{v}^{l_{v}}}\begin{pmatrix}
\varpi_{v}^{l_{v}}&\alpha\\
&1	
\end{pmatrix}K_{0,v}[{l_{v}}]\bigsqcup \begin{pmatrix}
1&\\
& \varpi_{v}^{l_{v}}
\end{pmatrix}wK_{0,v}[{l_{v}}]\\
\bigsqcup \bigsqcup_{j=1}^{l_{v}-1}\bigsqcup_{\beta\in \mathcal{O}_{v}^{\times}/(1+\mathfrak{p}_{v}^{l_{v}-j})} \begin{pmatrix}
\varpi_{v}^{{l_{v}}-j} &\beta\\
& \varpi_{v}^{j}
\end{pmatrix}\begin{pmatrix}
-\beta &\\
\varpi_{v}^{l_{v}-j}& \beta^{-1}
\end{pmatrix}K_{0,v}[{l_{v}}].
\end{multline}

Right-multiplying \eqref{5.1} by $K_{v}$ gives the corresponding double-coset decomposition, with the same representatives modulo $K_{v}$. Since $W_{\phi}$ is right-$K_{v}$-invariant and $[K_{v}:K_{0,v}[l_{v}]]=q_{v}^{l_{v}}(1+q_{v}^{-1})$, we obtain
\begin{equation}\label{5.2}
\lambda_{\pi}^{*}(\mathfrak{p}^{l_{v}})W_{\phi}(g)=T_{\mathfrak{p}^{l_{v}}}^{*}(W_{\phi})(g)=q_{v}^{l_{v}/2}(1+q_{v}^{-1})\int_{K_{v}}W_{\phi}\left(gk_{v}\begin{pmatrix}
\varpi_{v}^{l_{v}}\\
&1
\end{pmatrix}\right)dk_{v},
\end{equation}

Since $W_{3,v}(\cdot,0)$ and $h_{4,v}(\cdot,0)$ are right-$K_{v}$-invariant at $v\mid\mathfrak{n}$, formula \eqref{e5.2} becomes
\begin{equation}\label{e5.5}
q_{v}^{l_{v}/2}(1+q_{v}^{-1})\int_{N(F_{v})\backslash \overline{G}(F_{v})}\int_{K_{v}}\overline{W_{\phi,v}(g_{v}k_{v}\diag(\varpi_{v}^{l_{v}},1))}dk_{v}W_{3,v}(g_{v},0)h_{4,v}(g_{v},0)dg_{v}.
\end{equation}

Combining \eqref{5.2} with \eqref{e5.5} we derive 
\begin{equation}\label{e5.6}
\Psi_{v}(\overline{W_{\phi,v}},W_{1,v}(\cdot,0),h_{2,v}(\cdot,0))=	\overline{\lambda_{\pi}^{*}(\mathfrak{p}^{l_{v}})}\Psi_{v}(\overline{W_{\phi,v}},W_{3,v}(\cdot,0),h_{4,v}(\cdot,0)).
\end{equation}

Therefore, \eqref{4.12} follows from \eqref{eq5.2} and \eqref{e5.6}.
\end{proof}

\begin{cor}\label{cor5.2}
 Suppose $\mathfrak{n}+\mathfrak{q}=\mathcal{O}_{F}$. Then 
\begin{align*}
\mathcal{J}_{\mathrm{Cusp}}^{\heartsuit}(\mathbf{0},\mathfrak{X}_{\mathfrak{n}})=\sum_{\pi\in \mathcal{F}_{0}(\mathfrak{q},\omega)}\overline{\lambda_{\pi}^{*}(\mathfrak{n})}\cdot \sum_{\phi\in\mathfrak{B}(\pi)^{K_{0}[\mathfrak{q}]}}\big|\widetilde{\Psi}(W_{\phi},\overline{W_{3}(\cdot,0)},\overline{h_{4}(\cdot,0)})\big|^{2},
\end{align*}
and $\mathcal{J}_{\mathrm{Eis}}^{\heartsuit}(\mathbf{0},\mathfrak{X}_{\mathfrak{n}})$ is equal to 
\begin{align*}
\sum_{\mu\in \mathfrak{X}(\mathfrak{q},\omega)}\frac{1}{4\pi i}\int_{i\mathbb{R}}\overline{\lambda_{\pi_{\mu,\overline{\mu}\omega,\lambda}}^{*}(\mathfrak{n})}\sum_{h\in \mathfrak{B}(\mu,\overline{\mu}\omega)^{K_{0}[\mathfrak{q}]}}\big|\widetilde{\Psi}(W_{E(\cdot,h,\lambda)},\overline{W_{3}(\cdot,0)},\overline{h_{4}(\cdot,0)})\big|^{2}d\lambda.
\end{align*}
\end{cor}
\begin{proof}
Choose orthonormal bases containing bases of the $K_{0}[\mathfrak{q}]$-isotypic subspaces. Apply \eqref{reg1} and Lemma \ref{lem5.1}, noting that $W_{3}(\cdot,s)$ is right-$K_{0}[\mathfrak{q}]$-invariant, while $h_{4}(\cdot,0)$ is right-$\eta_{\omega}$-equivariant in the convention of \textsection\ref{sec4.4}.
\end{proof}

\begin{lemma}\label{lem5.arch}
Let $0<\varepsilon<10^{-2}$, let $v\mid\infty$, and let $\tau$ be a unitary generic representation of $G(F_{v})$ with central character $\omega_{v}$. Suppose that $C_{v}(\tau)\leq \mathbf{C}_{v}^{1-10\varepsilon}$. There is a unit vector $W_{v}$ in the Whittaker model of $\tau$ such that
\begin{equation}\label{5.19}
\big|\Psi_{v}(\overline{W_{v}},W_{3,v}(\cdot,0),h_{4,v}(\cdot,0))\big|\gg_{F_{v},\varepsilon,\alpha_{v}}\mathbf{C}_{v}^{-\frac{1}{2}-\varepsilon}.
\end{equation}
The same estimate holds for the complex-conjugate local period.
\end{lemma}
\begin{proof}
Use the Archimedean test-vector argument in \cite[\textsection 3.6.3--\textsection 3.6.5]{MV10} with $\pi_{1}=\widetilde{\tau}$, $\pi_{2}=\mathbf{1}\boxplus\mathbf{1}$, and $\pi_{3}=\mathbf{1}\boxplus\omega_{v}$. Our fixed $W_{3,v}(\cdot,0)$ is spherical and tempered, with bounded Sobolev norms. Choose $W_{v}$ in the Kirillov model using a fixed nonnegative smooth function supported near $1$, normalized to have norm $1$. Here $W_{3,v}(a(y),0)$ is positive near $y=1$. On $K_{v}$, the section $h_{4,v}(\cdot,0)$ is supported where $|k_{21}|_{v}\ll\mathbf{C}_{v}^{-1}$, with bounded $L^{2}$-norm and $L^{1}$-norm $\asymp\mathbf{C}_{v}^{-1/2}$. The factor $\omega_{v}$ in \eqref{eq6.28} cancels the central-character factor from $\overline{W_{v}}$. Since $W_{3,v}(\cdot,0)$ is spherical, only $W_{v}$ is translated in the $K_{v}$-average. The condition $C_{v}(\tau)\leq\mathbf{C}_{v}^{1-10\varepsilon}$ makes the error smaller than the main term, whose size is $\asymp\mathbf{C}_{v}^{-1/2}$. This proves \eqref{5.19} over both $\mathbb{R}$ and $\mathbb{C}$. Complex conjugation gives the final assertion.
\end{proof}

\begin{lemma}\label{lemma5.3}
Let $\pi\in\mathcal{A}_{0}([G],\omega)$ and $0<\varepsilon<10^{-2}$. Suppose that $C_{v}(\pi)\leq \mathbf{C}_{v}^{1-10\varepsilon}$ for every $v\mid\infty$ and that $\mathfrak{B}(\pi)^{K_{0}[\mathfrak{q}]}$ is nonempty. Then 
\begin{equation}\label{5.8}
\sum_{\phi\in\mathfrak{B}(\pi)^{K_{0}[\mathfrak{q}]}}\big|\widetilde{\Psi}(\overline{W_{\phi}},W_{3}(\cdot,0),h_{4}(\cdot,0))\big|^{2}\gg_{F,\varepsilon} \mathbf{C}_{\infty}^{\ -1-\varepsilon}N_{F}(\mathfrak{q})^{-\varepsilon}|L(1/2,\pi)|^{4}.
\end{equation}
\end{lemma}
\begin{proof}
For a unit pure tensor $\phi\in\pi$, write $W_{\phi}=c_{\pi}\otimes_{v}W_{v}$, with the local normalizations chosen below. Rankin--Selberg theory gives
\begin{multline}\label{e5.9}
\frac{\widetilde{\Psi}(\overline{W_{\phi}},W_{3}(\cdot,0),h_{4}(\cdot,0))}{L(1/2,\widetilde{\pi})^{2}}=\overline{c_{\pi}}\prod_{v\mid\infty}\Psi_{v}(\overline{W_{v}},W_{3,v}(\cdot,0),h_{4,v}(\cdot,0))\\
\prod_{v<\infty}\frac{\Psi_{v}(\overline{W_{v}},W_{3,v}(\cdot,0),h_{4,v}(\cdot,0))}{L_{v}(1/2,\widetilde{\pi}_{v})^{2}}.
\end{multline}

For each $v<\infty$, take the local newvector $W_{v}$, normalized by $W_{v}(I_{2})=1$ when both $\pi_{v}$ and $\psi_{v}$ are unramified, and by $\|W_{v}\|_{v}=1$ otherwise. The standard local newvector calculation (compare \cite[\textsection 3.6.2]{MV10}), with the volume normalization in \eqref{equa4.5}, gives
\begin{equation}\label{e5.10}
|c_{\pi}|\prod_{v<\infty}\bigg|\frac{\Psi_{v}(\overline{W_{v}},W_{3,v}(\cdot,0),h_{4,v}(\cdot,0))}{L_{v}(1/2,\widetilde{\pi}_{v})^{2}}\bigg|\gg_{F,\varepsilon} C(\pi)^{-\varepsilon}N_{F}(\mathfrak{q})^{-\varepsilon}.
\end{equation}
Here we have also used the Rankin--Selberg norm formula for $|c_{\pi}|$ and the bound $L(1,\pi,\Ad)\ll_{F,\varepsilon}C(\pi)^{\varepsilon}$.

At every Archimedean place choose $W_{v}$ as in Lemma \ref{lem5.arch}, and let $\phi\in\pi^{K_{0}[\mathfrak{q}]}$ be the resulting unit pure tensor. Combining \eqref{e5.9}, \eqref{e5.10}, and \eqref{5.19}, and using $|L(1/2,\widetilde{\pi})|=|L(1/2,\pi)|$, gives
\begin{align*}
|\widetilde{\Psi}(\overline{W_{\phi}},W_{3}(\cdot,0),h_{4}(\cdot,0))|\gg_{F,\varepsilon} \mathbf{C}_{\infty}^{-\frac{1}{2}-\varepsilon}C(\pi)^{-\varepsilon}N_{F}(\mathfrak{q})^{-\varepsilon}|L(1/2,\pi)|^{2}.
\end{align*}

The square of this period is a lower bound for the sum in \eqref{5.8}, by finite $K$-type approximation and Cauchy--Schwarz. Since $\pi^{K_{0}[\mathfrak{q}]}\neq 0$ and $C_{v}(\pi)\leq \mathbf{C}_{v}^{1-10\varepsilon}$ at infinity, we have $C(\pi)\ll_{F} \mathbf{C}_{\infty} N_{F}(\mathfrak{q})$. Absorbing the resulting powers of $\varepsilon$ proves \eqref{5.8}.
\end{proof}

\begin{lemma}\label{lemma5.4}
Let $\mu\in\widehat{C}_{F}$, let $t\in\mathbb{R}$, and set
\begin{align*}
\pi_{\mu,it}:=\mu|\cdot|^{it}\boxplus\overline{\mu}\omega|\cdot|^{-it}.
\end{align*}
Let $0<\varepsilon<10^{-2}$. Suppose that $C_{v}(\pi_{\mu,it})\leq \mathbf{C}_{v}^{1-10\varepsilon}$ for every $v\mid\infty$ and that $\mathfrak{B}(\mu,\overline{\mu}\omega)^{K_{0}[\mathfrak{q}]}$ is nonempty. Then 
\begin{align*}
\sum_{h\in \mathfrak{B}(\mu,\overline{\mu}\omega)^{K_{0}[\mathfrak{q}]}}\big|\widetilde{\Psi}(W_{E(\cdot,h,it)},\overline{W_{3}(\cdot,0)},\overline{h_{4}(\cdot,0)})\big|^{2}\gg_{F,\varepsilon} \frac{|L(1/2+it,\mu)|^{4}|L(1/2-it,\overline{\mu}\omega)|^{4}}{\mathbf{C}_{\infty}^{1+\varepsilon}N_{F}(\mathfrak{q})^{\varepsilon}|L(1+2it,\mu^{2}\overline{\omega})|^{2}}.
\end{align*}
\end{lemma}
\begin{proof}
Apply Lemma \ref{lem5.arch} at every Archimedean place to the unitary induced representation $\pi_{\mu,it}$. At the finite places choose the normalized newvectors in the induced models, and put $h=\rho_{F}^{-1/2}\bigotimes_{v}h_{v}$, a global unit vector. The local newvector calculation and the normalization of Eisenstein Whittaker functions give
\begin{align*}
\big|\widetilde{\Psi}(W_{E(\cdot,h,it)},\overline{W_{3}(\cdot,0)},\overline{h_{4}(\cdot,0)})\big|
\gg_{F,\varepsilon}\frac{|L(1/2+it,\mu)L(1/2-it,\overline{\mu}\omega)|^{2}}{\mathbf{C}_{\infty}^{\frac{1}{2}+\varepsilon}C(\pi_{\mu,it})^{\varepsilon}N_{F}(\mathfrak{q})^{\varepsilon}|L(1+2it,\mu^{2}\overline{\omega})|}.
\end{align*}
The square of this period is a lower bound for the sum in the statement, by finite $K$-type approximation and Cauchy--Schwarz. Since $\pi_{\mu,it}$ has level dividing $\mathfrak{q}$, we have $C(\pi_{\mu,it})\ll_{F} \mathbf{C}_{\infty} N_{F}(\mathfrak{q})$. Rescaling $\varepsilon$ completes the proof.
\end{proof}

\section{Bounds for the Residues $\widetilde{\Psi}_{\mathrm{RS}}^{(i)}(\mathbf{0}\mid\mathfrak{X}_{\mathfrak{n}})$}\label{sec6}
We now estimate the eight Rankin--Selberg residues.

\begin{prop}\label{prop6.1}
 Let $1\leq i\leq 8$. Then 
\begin{align*}
\big|\widetilde{\Psi}_{\mathrm{RS}}^{(i)}(\mathbf{0}\mid\mathfrak{X}_{\mathfrak{n}})\big|\ll_{F,\varepsilon} C_{\infty}(\omega)^{\frac{1}{2}+\varepsilon}\mathbf{C}_{\infty}^{-1+\varepsilon}N_{F}(\mathfrak{q})^{\varepsilon}N_{F}(\mathfrak{q}')^{\frac{1}{2}+\varepsilon}N_{F}(\mathfrak{n})^{\frac{1}{2}+\varepsilon}.
\end{align*}
\end{prop}

\subsection{Euler Products for $\widetilde{\Psi}_{\mathrm{RS}}^{(i)}(\mathbf{s},\mathfrak{X}_{\mathfrak{n}})$}
Proposition~\ref{prop3.residual-integrals} expresses these residues as integrals in the initial data. We use the Euler product in Lemma~\ref{lem3.3} to estimate them locally.

\subsubsection{Interpretation via $L$-functions}\label{sec6.1.1}
Let 
\begin{align*}
\boldsymbol{e}(\lambda,\mu;\mathbf{s}):=\widetilde{\Psi}(\lambda,\mu;\mathbf{s},\mathfrak{X}_{\mathfrak{n}})/\mathbf{L}(\lambda,\mu;\mathbf{s}),
\end{align*}
where $\mathbf{L}(\lambda,\mu;\mathbf{s})$ is the product of the complete $L$-functions 
\begin{multline*}
 \Lambda(1+2\lambda,\mu^{2}\overline{\omega})^{-1}\Lambda(1-2\lambda,\overline{\mu}^{2}\omega)^{-1}
 \Lambda(1/2+s_{2}+s_{1}-\lambda,\overline{\mu})\Lambda(1/2+s_{2}+s_{1}+\lambda,\mu\overline{\omega})\\
\Lambda(1/2+s_{2}-s_{1}-\lambda,\overline{\mu})\Lambda(1/2+s_{2}-s_{1}+\lambda,\mu\overline{\omega})
\Lambda(1/2+s_{4}+s_{3}+\lambda,\mu)\\\Lambda(1/2+s_{4}+s_{3}-\lambda,\overline{\mu}\omega)\Lambda(1/2+s_{4}-s_{3}+\lambda,\mu)\Lambda(1/2+s_{4}-s_{3}-\lambda,\overline{\mu}\omega).
\end{multline*}

By Rankin--Selberg factorization, taking the residues in Definition~\ref{defn3.5} gives
\begin{equation}\label{c7.1}
\begin{aligned}
\widetilde{\Psi}_{\mathrm{RS}}^{(1)}(\mathbf{s},\mathfrak{X}_{\mathfrak{n}})&=\boldsymbol{e}(s_{2}+s_{1}-1/2,\mathbf{1};\mathbf{s})\underset{\lambda=s_{2}+s_{1}-\frac{1}{2}}{\Res}\ \mathbf{L}(\lambda,\mathbf{1};\mathbf{s}),\\
\widetilde{\Psi}_{\mathrm{RS}}^{(2)}(\mathbf{s},\mathfrak{X}_{\mathfrak{n}})&=-\boldsymbol{e}(1/2-s_{2}-s_{1},\omega;\mathbf{s})\underset{\lambda=\frac{1}{2}-s_{2}-s_{1}}{\Res}\ \mathbf{L}(\lambda,\omega;\mathbf{s}).
\end{aligned}
\end{equation}
The identities are first obtained when the indicated poles are simple and distinct and then continued meromorphically. The other residual terms have the same form.

Using the local sections and Whittaker functions of \textsection\ref{sec2.1.1} and \textsection\ref{sec2.1.whittaker}, let
\begin{equation}\label{eq7.2}
W_{h_{v},\lambda}(g_{v}):=\int_{N(F_{v})}S(h_{v})(wu_{v}g_{v},\lambda)\overline{\theta_{v}(u_{v})}du_{v},
\end{equation}
where $S(h_{v})(\cdot,\lambda)$ is the $v$-th component of $S(h)(\cdot,\lambda)$ defined by \eqref{eq2.2}. 
\begin{defn}
 	Define the function 
\begin{align*}
\widetilde{\Psi}_{v}(\lambda,h_{v};\mathbf{s}):=\widetilde{\Psi}_{v}(\overline{W_{h_{v},-\overline{\lambda}}},W_{1,v}(\cdot,s_{1}),h_{2,v}(\cdot,s_{2}))\widetilde{\Psi}_{v}(W_{h_{v},\lambda},\overline{W_{3,v}(\cdot,\overline{s_{3}})},\overline{h_{4,v}(\cdot,\overline{s_{4}})}),
\end{align*}
where the two factors on the right are the meromorphic continuations of the local period integrals defined as in \eqref{c2.6}.
\end{defn}

By definition and Rankin--Selberg theory,
\begin{equation}\label{eq6.euler}
\boldsymbol{e}(\lambda,\mu;\mathbf{s})=\frac{1}{2\rho_{F}}\prod_{v}\boldsymbol{e}_{v}(\lambda,\mu_{v};\mathbf{s}),
\end{equation}
where $1/2$ comes from \eqref{c3.5}, $\rho_{F}^{-1}$ comes from the global compact-model norm, and
\begin{equation}\label{eq6.1}
\boldsymbol{e}_{v}(\lambda,\mu_{v};\mathbf{s})=\sum_{h_{v}\in \mathfrak{B}(\mu_{v},\overline{\mu}_{v}\omega_{v})}\frac{\widetilde{\Psi}_{v}(\lambda,h_{v};\mathbf{s})}{\mathbf{L}_{v}(\lambda,\mu_{v};\mathbf{s})}.
\end{equation}
Here $\mathfrak{B}(\mu_{v},\overline{\mu}_{v}\omega_{v})$ is an orthonormal basis of the induced representation $\Ind (\mu_{v}\otimes\overline{\mu}_{v}\omega_{v})$, and $\mathbf{L}_{v}(\lambda,\mu_{v};\mathbf{s})$ is the $v$-th factor of $\mathbf{L}(\lambda,\mu;\mathbf{s})$. 

Unramified calculations give $\boldsymbol{e}_{v}(\lambda,\mu_{v};\mathbf{s})\equiv 1$ at all but finitely many places. At the residual parameters, we continue one period meromorphically and apply Parseval's identity on $K_{v}$.

\subsection{The Archimedean Integrals}
\subsubsection{Local Bessel functions}
Let $v$ be a place of $F$, and let $\pi_{v}$ be a generic representation of $G(F_{v})$ with central character $\omega_{\pi_{v}}$. Let $j_{\pi_{v}}$ be the Bessel function associated with $\pi_{v}$; see, for example, \cite{Cog14}. If $W_{\pi_{v}}$ belongs to the Whittaker model of $\pi_{v}$, then for all $g_{v}\in G(F_{v})$,
\begin{equation}\label{eq5.7}
W_{\pi_{v}}\left(\begin{pmatrix}
a_{v}\\
&1	
\end{pmatrix}wg_{v}\right)=\omega_{\pi_{v}}(a_{v})\int_{F_{v}^{\times}}j_{\pi_{v}}(a_{v}y_{v})W_{\pi_{v}}\left(\begin{pmatrix}
y_{v}\\
&1
\end{pmatrix}g_{v}\right)d^{\times}y_{v}.
\end{equation}

Let $\chi_{v}$ be a multiplicative character of $F_{v}^{\times}$. Upon replacing $W_{\pi_{v}}(\cdot)$ with the function $\chi_{v}(\det(\cdot))W_{\pi_{v}}(\cdot)$ in \eqref{eq5.7} we obtain 
\begin{equation}\label{eq5.8bessel}
j_{\pi_{v}\otimes\chi_{v}}(y_{v})=\chi_{v}^{-1}(-y_{v})j_{\pi_{v}}(y_{v}).
\end{equation}

The local functional equation gives the Mellin identity
\begin{equation}\label{f5.8}
\gamma_{v}(s,\pi_{v},\psi_{v})
=\int_{F_{v}^{\times}}j_{\pi_{v}}(y_{v})|y_{v}|_{v}^{1/2-s}d^{\times}y_{v},
\end{equation}
where $\gamma_{v}(s,\pi_{v},\psi_{v})$ is the local $\gamma$-factor. Outside the convergence range, use \eqref{f5.8} to continue the Bessel integrals meromorphically. Combining \eqref{eq5.8bessel} with \eqref{f5.8} yields
\begin{equation}\label{twist}
\gamma_{v}(1/2-s,\pi_{v}\otimes\chi_{v}^{-1},\psi_{v})=\chi_{v}(-1)\int_{F_{v}^{\times}}\chi_{v}(y_{v})j_{\pi_{v}}(y_{v})|y_{v}|_{v}^{s}d^{\times}y_{v},
\end{equation}
by meromorphic continuation in $s$.

\subsubsection{Switching sections in the Rankin--Selberg periods}\label{sec6.2.2}
\begin{lemma}\label{lem6.1}
We have
\begin{multline}\label{c6.6}
\widetilde{\Psi}_{v}(\overline{W_{h_{v},-\overline{\lambda}}},W_{1,v}(\cdot,s_{1}),h_{2,v}(\cdot,s_{2}))=\mu_{v}(-1)\gamma(1/2-\lambda-s_{2}+s_{1},\overline{\mu}_{v}\omega_{v},\psi_{v})\\
\gamma(1/2-\lambda-s_{2}-s_{1},\overline{\mu}_{v}\omega_{v},\psi_{v})
\widetilde{\Psi}_{v}(W_{1,v}(\cdot,s_{1}),W_{2,v}^{*}(\cdot,s_{2}),\overline{S(h_{v})(\cdot,-\overline{\lambda})}).
\end{multline}
\end{lemma}
\begin{proof}
First take the original period in its region of absolute convergence:
\begin{equation}\label{eq6.6}
\widetilde{\Psi}_{v}(\overline{W_{h_{v},-\overline{\lambda}}},W_{1,v}(\cdot,s_{1}),h_{2,v}(\cdot,s_{2}))=\Psi_{v}(\overline{W_{h_{v},-\overline{\lambda}}},W_{1,v}(\cdot,s_{1}),h_{2,v}(\cdot,s_{2}))
\end{equation}
In the Bessel calculations below, insert compact cutoffs before interchanging integrals and continue meromorphically using \eqref{f5.8}. Substituting the definition
\begin{align*}
W_{h_{v},-\overline{\lambda}}(g_{v})=\int_{N(F_{v})}S(h_{v})(wu_{v}g_{v},-\overline{\lambda})\overline{\theta_{v}(u_{v})}du_{v}
\end{align*} 
into \eqref{eq6.6} gives
\begin{equation}\label{eq6.7}
\int_{\overline{G}(F_{v})}\overline{S(h_{v})(g_{v},-\overline{\lambda})}W_{1,v}(wg_{v},s_{1})h_{2,v}(wg_{v},s_{2})dg_{v}.
\end{equation}

By \eqref{eq5.7}, we have
\begin{equation}\label{eq6.8}
W_{1,v}(wg_{v},s_{1})=\int_{F_{v}^{\times}}j(y_{v})W_{1,v}\left(\begin{pmatrix}
y_{v}\\
&1
\end{pmatrix}g_{v},s_{1}\right)d^{\times}y_{v},
\end{equation}
where $j(\cdot)=j_{|\cdot|^{s_{1}}\boxplus |\cdot|^{-s_{1}}}(\cdot)$ is the Bessel function. 

Substituting \eqref{eq6.8} into \eqref{eq6.7} and changing variables $g_{v}\mapsto a(y_{v}^{-1})g_{v}$ gives
\begin{multline}\label{eq6.9}
\int_{F_{v}^{\times}}j(y_{v})\int_{\overline{G}(F_{v})}\overline{S(h_{v})(\diag(y_{v}^{-1},1)g_{v},-\overline{\lambda})}\\
h_{2,v}(w\diag(y_{v}^{-1},1)g_{v},s_{2})W_{1,v}\left(g_{v},s_{1}\right)dg_{v}d^{\times}y_{v}.
\end{multline}

Since 
\begin{align*}
\int_{N(F_{v})}h_{2,v}(wu g_{v},s_{2})\theta_{v}(u)du=\omega_{v}(-1)W_{2,v}^{*}(g_{v},s_{2}),
\end{align*}
the transformation laws of the sections turn \eqref{eq6.9} into
\begin{align*}
\omega_{v}(-1)\int_{F_{v}^{\times}}j(y_{v})|y_{v}|_{v}^{\lambda+s_{2}}\mu_{v}\overline{\omega}_{v}(y_{v})d^{\times}y_{v}\Psi_{v}(W_{1,v}(\cdot,s_{1}),W_{2,v}^{*}(\cdot,s_{2}),\overline{S(h_{v})(\cdot,-\overline{\lambda})}).
\end{align*}

Applying \eqref{twist} with $\chi_{v}=\mu_{v}\overline{\omega}_{v}$ gives \eqref{c6.6}, and meromorphic continuation proves it for all parameters.
\end{proof}

\begin{lemma}
We have
\begin{multline}\label{c6.12}
\widetilde{\Psi}_{v}(W_{h_{v},\lambda},\overline{W_{3,v}(\cdot,\overline{s_{3}})},\overline{h_{4,v}(\cdot,\overline{s_{4}})})=\mu_{v}(-1)\gamma(1/2+\lambda-s_{4}+s_{3},\mu_{v}\overline{\omega}_{v},\psi_{v})\\
\gamma(1/2+\lambda-s_{4}-s_{3},\mu_{v}\overline{\omega}_{v},\psi_{v})
\widetilde{\Psi}_{v}(\overline{W_{3,v}(\cdot,\overline{s_{3}})},\overline{W_{4,v}^{*}(\cdot,\overline{s_{4}})},S(h_{v})(\cdot,\lambda)).
\end{multline}
\end{lemma}
\begin{proof}
Apply the proof of Lemma \ref{lem6.1} to the second period and continue meromorphically.
\end{proof}

\subsubsection{Local periods at the residual parameters}
Recall \eqref{eq6.1} and write
\begin{align*}
\mathcal{Q}_{v}(\lambda,\mu_{v};\mathbf{s}):=\mathbf{L}_{v}(\lambda,\mu_{v};\mathbf{s})\boldsymbol{e}_{v}(\lambda,\mu_{v};\mathbf{s}).
\end{align*}
As in \textsection\ref{sec2.1.fourier}, $h_{2,v}^{\Diamond}$ denotes the unnormalized local intertwining section, defined initially for $\Re(s_{2})>0$ and then meromorphically continued.

\begin{lemma}\label{lem6.3}
Let $v\mid\infty$ or $v\mid\mathfrak{q}$, $v\nmid\mathfrak{D}_{F}$, with the datum of \textsection\ref{sec4.4}.
\begin{enumerate}
\item[(A).] As meromorphic identities,
\begin{align*}
\mathcal{Q}_{v}(\lambda,\mu_{v};\mathbf{s})=\mathcal{Q}_{v}(-\lambda,\overline{\mu}_{v}\omega_{v};\mathbf{s}),\qquad \overline{\mathcal{Q}_{v}(\lambda,\mu_{v};\mathbf{s})}=\mathcal{Q}_{v}(-\overline{\lambda},\mu_{v};\overline{s}_{3},\overline{s}_{4},\overline{s}_{1},\overline{s}_{2}).
\end{align*}

\item[(B).] Suppose $\mu_{v}=\omega_{v}$ and $\lambda=1/2+\epsilon_{1} s_{1}-s_{2}$, with $\epsilon_{1}\in\{\pm1\}$. Then
\begin{equation}\label{eq6.residual-projection}
\widetilde{\Psi}_{v}(\overline{W_{h_{v},-\overline{\lambda}}},W_{1,v}(\cdot,s_{1}),h_{2,v}(\cdot,s_{2}))
=c_{v} L_{v}(1+2\epsilon_{1} s_{1},\mathbf{1})\int_{K_{v}}\overline{h_{v}(k)}h_{2,v}^{\Diamond}(k,s_{2})\,dk,
\end{equation}
Here the continued period on the left is evaluated at the stated $\lambda$, and $c_{v}$ depends only on the local measures. At a finite place with unramified $\psi_{v}$, $c_{v}=\zeta_{F_{v}}(1)$.
\end{enumerate}
\end{lemma}
\begin{proof}
On $\lambda\in i\mathbb{R}$, the intertwining operator normalized to preserve Whittaker functions is unitary \cite[Proposition 3.1.4]{Sha81}. It therefore replaces the orthonormal basis by one in the Weyl-conjugate induced model without changing $\mathcal{Q}_{v}$. This gives the first identity, and meromorphic continuation gives it for all $\lambda$. The second follows by conjugating the two periods and interchanging them, since $\Phi_{1,v}=\Phi_{3,v}$ and $\Phi_{2,v}=\Phi_{4,v}$.

For (B), substitute \eqref{c2.5} and the Jacquet integral for $W_{2,v}^{*}$ into \eqref{c6.6} and apply Tate's local functional equation and residue formula at zero, as in \cite[\textsection 1.2, Lemma~2]{JZ87}. Initially work in the domain of absolute convergence with $\Re(s_{2})>0$ and $\Re(\epsilon_{1}s_{1})<0$. The symmetry $W_{1,v}(\cdot,s_{1})=W_{1,v}(\cdot,-s_{1})$ of the standard spherical datum covers both signs. Multiplication by the vanishing gamma factor evaluates the Tate test function at zero and contributes $c_{v}$. The remaining integrals give $L_{v}(-2\epsilon_{1}s_{1},\mathbf{1})$ and $h_{2,v}^{\Diamond}(k,s_{2})$ by the Godement and intertwining definitions. The other gamma factor converts the former into $L_{v}(1+2\epsilon_{1}s_{1},\mathbf{1})$, while the two factors $\omega_{v}(-1)$ cancel. Integrating against $\overline{h_{v}(k)}$ over $K_{v}$ proves \eqref{eq6.residual-projection}; meromorphic continuation gives the general case.
\end{proof}

\subsubsection{An upper bound for $\boldsymbol{e}_{v}(\lambda,\mu_{v};\mathbf{s})$}
\begin{lemma}\label{lem6.4}
Let $v\mid\infty$, let $\mathbf{s}\in\mathbf{S}_{\varepsilon}^{4}$, and let $(\lambda,\mu)$ be one of the eight pairs in Definition \ref{defn3.5} for the datum \eqref{4.5}. On the contour chosen in \textsection\ref{sec4.regularization},
\begin{equation}\label{f7.15}
|\boldsymbol{e}_{v}(\lambda,\mu_{v};\mathbf{s})|\ll_{F_{v},\varepsilon} |\mathbf{L}_{v}(\lambda,\mu_{v};\mathbf{s})|^{-1}C_{v}(\omega)^{1000\varepsilon}\mathbf{C}_{v}^{-1+1000\varepsilon}.
\end{equation}
\end{lemma}
\begin{proof}
By part (A) of Lemma \ref{lem6.3}, it suffices to consider $\mu_{v}=\omega_{v}$ and $\lambda=1/2+\epsilon_{1} s_{1}-s_{2}$, $\epsilon_{1}=\pm1$; conjugation also allows the pairs $(s_{1},s_{2})$ and $(s_{3},s_{4})$ to be interchanged. By Lemma~\ref{lem6.3} (B), the fixed spherical section $h_{1,v}$ contributes the factor $L_{v}(1+2\epsilon_{1}s_{1},\mathbf{1})$ to the first period. This factor and the gamma factors in \eqref{c6.12} are $O_{F_{v},\varepsilon}(1)$ on the contour, since the relevant signed sums of the $s_{i}$ are bounded away from zero. Combining \eqref{eq6.residual-projection} and \eqref{c6.12} and applying Parseval's identity on $K_{v}$ gives
\begin{equation}\label{eq6.parseval}
|\mathcal{Q}_{v}(\lambda,\omega_{v};\mathbf{s})|\ll_{F_{v},\varepsilon}\left|\int_{K_{v}}h_{2,v}^{\Diamond}(k,s_{2})B_{v}(k)\,dk\right|,
\end{equation}
where
\begin{align*}
B_{v}(k):=\int_{F_{v}^{\times}}\overline{W_{3,v}(a(y),\overline{s}_{3})}\,\overline{W_{4,v}^{*}(a(y)k,\overline{s}_{4})}\,\omega_{v}(y)|y|_{v}^{\lambda-1/2}\,d^{\times} y.
\end{align*}
This integral is absolutely convergent.

We claim that
\begin{equation}\label{eq6.compact-pair}
\|h_{2,v}^{\Diamond}(\cdot,s_{2})\|_{L^{1}(K_{v})}\ll_{F_{v},\varepsilon} \frac{C_{v}(\omega)^{100\varepsilon}}{\mathbf{C}_{v}^{1/2-100\varepsilon}},\qquad
\|B_{v}\|_{L^{\infty}(K_{v})}\ll_{F_{v},\varepsilon}\mathbf{C}_{v}^{-1/2+100\varepsilon}.
\end{equation}

For the first bound, apply Tate's local functional equation to the Godement integral and use the Fourier estimates in the proof of \cite[Lemma~3.7.2]{MV10}, averaging over $K_{v}$. Allowing $|s_{2}|\leq11\varepsilon$ costs at most $(\mathbf{C}_{v}C_{v}(\omega))^{O(\varepsilon)}$.

For the second bound, insert \eqref{c2.5} and integrate over $y$ before taking absolute values. The factor $\omega_{v}(y)$ cancels, and all remaining character factors have modulus one. The expansion at zero and rapid decay of the fixed spherical $W_{3,v}$ bound the remaining integral over the support of \eqref{eq6.28} by $\mathbf{C}_{v}^{-1+O(\varepsilon)}$, uniformly in $k\in K_{v}$. Multiplying by $\mathbf{C}_{v}^{1/2}$ proves the bound. The resulting integral is absolutely convergent throughout the stated range and agrees with $B_{v}$ by meromorphic continuation from $\Re(s_{4})>0$.

Combining \eqref{eq6.parseval} and \eqref{eq6.compact-pair} proves \eqref{f7.15}.
\end{proof}

\subsection{The non-Archimedean Integrals}

\subsubsection{Calculation of $\boldsymbol{e}_{v}(\lambda,\mu_{v};\mathbf{s})$ at  $v\nmid \mathfrak{q}$} 
Suppose $v\nmid\mathfrak{q}$. Since $W_{3,v}$ and $h_{4,v}$ are spherical, the second period vanishes unless $\mu_{v}|\cdot|_{v}^{\lambda}\boxplus \overline{\mu}_{v}\omega_{v}|\cdot|_{v}^{-\lambda}$ is unramified. Therefore, by \eqref{eq6.1},
\begin{align*}
\boldsymbol{e}_{v}(\lambda,\mu_{v};\mathbf{s})=\sum_{h_{v}\in \mathfrak{B}(\mu_{v},\overline{\mu}_{v}\omega_{v})}\frac{\widetilde{\Psi}_{v}(\lambda,h_{v};\mathbf{s})}{\mathbf{L}_{v}(\lambda,\mu_{v};\mathbf{s})}=\frac{\widetilde{\Psi}_{v}(\lambda,h_{v}^{\circ};\mathbf{s})}{\mathbf{L}_{v}(\lambda,\mu_{v};\mathbf{s})},
\end{align*}
where $h_{v}^{\circ}$ is the normalized spherical vector in $\mu_{v}\boxplus \overline{\mu}_{v}\omega_{v}$. An unramified calculation, combined with the analysis at $v\mid\mathfrak{n}$ as in Lemma \ref{lem5.1}, gives  
\begin{equation}\label{eq7.15}
\boldsymbol{e}_{v}(\lambda,\mu_{v};\mathbf{s})=\lambda_{\widetilde{\pi}_{\lambda}}^{*}(\mathfrak{p}_{v}^{e_{v}(\mathfrak{n})})|\det\mathbf{d}_{v}|_{v}^{-s_{1}-s_{2}}q_{v}^{(s_{2}+s_{4}-s_{1}-s_{3}-4)d_{v}},
\end{equation}
where $\pi_{\lambda}=\mu|\cdot|^{\lambda}\boxplus \overline{\mu}\omega|\cdot|^{-\lambda}$ and $d_{v}=e_{v}(\mathfrak{D}_{F})$ is the different exponent. Changing to an unramified additive character in the two periods gives the displayed power of $q_{v}^{d_{v}}$; the character values introduced in the two periods cancel. Here the contragredient occurs because the first local period contains $\overline{W_{h_{v},-\overline{\lambda}}}$; equivalently, $\lambda_{\widetilde{\pi}_{\lambda}}^{*}(\mathfrak{p}_{v}^{l})=\omega_{v}(\varpi_{v})^{-l}\lambda_{\pi_{\lambda}}^{*}(\mathfrak{p}_{v}^{l})$.

\subsubsection{Calculation of $\boldsymbol{e}_{v}(\lambda,\mu_{v};\mathbf{s})$ at  $v\mid \mathfrak{q}$} 
Suppose $v\mid\mathfrak{q}$, put $m=e_{v}(\mathfrak{q})$, and let $(\lambda,\mu)$ be one of the eight pairs in Definition \ref{defn3.5}. By Lemma \ref{lem6.3}, we may take $\mu_{v}=\omega_{v}$ and $\lambda=1/2+\epsilon_{1}s_{1}-s_{2}$, $\epsilon_{1}=\pm1$, interchanging $(s_{1},s_{2})$ and $(s_{3},s_{4})$ if necessary. First suppose $\psi_{v}$ is unramified, and define $B_{v}$ by the same integral as in the proof of Lemma \ref{lem6.4}. Combining \eqref{eq6.residual-projection} and \eqref{c6.12} with Parseval and canceling matching Tate factors gives
\begin{align*}
|\boldsymbol{e}_{v}(\lambda,\mu_{v};\mathbf{s})|\ll\left\|\frac{h_{2,v}^{\Diamond}(\cdot,s_{2})}{L_{v}(2s_{2},\overline{\omega}_{v})}\right\|_{L^{1}(K_{v})}\|B_{v}\|_{L^{\infty}(K_{v})}.
\end{align*}
Division by $L_{v}(2s_{2},\overline{\omega}_{v})$ removes the Tate pole. Since $h_{2,v}|_{K_{v}}$ has $L^{1}$-norm one and is invariant under the principal congruence subgroup of level $m$, the remaining intertwining integral is bounded by a geometric sum of at most $m+1$ terms. Thus the first norm is $\ll(m+1)q_{v}^{O(\varepsilon m)}$.

For $B_{v}$, insert \eqref{c2.5} and integrate over $y$ first, using the finite-place calculation in the proof of \cite[Lemma~3.7.2]{MV10}. The factor $\omega_{v}(y)$ cancels as in Lemma \ref{lem6.4}, and the support volume is canceled by the normalization in \eqref{equa4.5}. The parameters $s_{j}$ contribute only $q_{v}^{O(\varepsilon m)}$, so $\|B_{v}\|_{L^{\infty}(K_{v})}\ll q_{v}^{O(\varepsilon m)}$.

At $v\mid\mathfrak{D}_{F}$, changing to an unramified additive character contributes the power of $q_{v}^{d_{v}}$ in \eqref{eq7.15}, whose modulus is at most $1$ on $\mathbf{S}_{\varepsilon}^{4}$. Absorbing $m+1$ gives
\begin{equation}\label{7.16}
|\boldsymbol{e}_{v}(\lambda,\mu_{v};\mathbf{s})|\ll_{F_{v},\varepsilon} q_{v}^{100\varepsilon e_{v}(\mathfrak{q})},\qquad (\lambda,\mu)\ \text{as above},\quad\mathbf{s}\in\mathbf{S}_{\varepsilon}^{4}.
\end{equation}

\subsection{Proof of Proposition \ref{prop6.1}}
Let $\mathbf{s}\in\mathbf{S}_{\varepsilon}^{4}$ be on the contour chosen in \textsection\ref{sec4.regularization}. By Lemma \ref{lem6.4} and the non-Archimedean calculations \eqref{eq7.15} and \eqref{7.16}, we deduce
\begin{equation}\label{7.25}
|\boldsymbol{e}(\lambda,\mu;\mathbf{s})|\ll_{F,\varepsilon} \frac{|\lambda_{\widetilde{\pi}_{\lambda}}^{*}(\mathfrak{n})|\cdot N_{F}(\mathfrak{n})^{O(\varepsilon)}C_{\infty}(\omega)^{1000\varepsilon}\mathbf{C}_{\infty}^{-1+1000\varepsilon}N_{F}(\mathfrak{q})^{100\varepsilon}}{\prod_{v\mid\infty}|\mathbf{L}_{v}(\lambda,\mu_{v};\mathbf{s})|}.
\end{equation} 

For each pair $(\lambda_{0},\mu)$ in Definition~\ref{defn3.5}, the only finite $L$-factor in $\underset{\lambda=\lambda_{0}}{\Res}\mathbf{L}(\lambda,\mu;\mathbf{s})$ involving $\omega$ with argument near zero is $L(2s_{2},\overline{\omega})$ for $i=1,\ldots,4$, or $L(2s_{4},\omega)$ for $i=5,\ldots,8$. Its functional equation and the bounds for the other factors on the contour give
\begin{equation}\label{7.26}
\frac{\big|\underset{\lambda=\lambda_{0}}{\Res}\ \mathbf{L}(\lambda,\mu;\mathbf{s})\big|}{\prod_{v\mid\infty}|\mathbf{L}_{v}(\lambda_{0},\mu_{v};\mathbf{s})|}\ll_{F,\varepsilon} C_{\infty}(\omega)^{1/2+O(\varepsilon)}N_{F}(\mathfrak{q}')^{1/2+O(\varepsilon)}.
\end{equation}

Since each $\widetilde{\Psi}_{\mathrm{RS}}^{(i)}(\mathbf{s},\mathfrak{X}_{\mathfrak{n}})$ is of the form \eqref{c7.1}, multiplying \eqref{7.25} and \eqref{7.26}, using $|\lambda_{\widetilde{\pi}_{\lambda_{0}}}^{*}(\mathfrak{n})|\ll N_{F}(\mathfrak{n})^{1/2+O(\varepsilon)}$, and taking the contour integral of \textsection\ref{sec4.regularization} proves Proposition \ref{prop6.1}. Rescaling $\varepsilon$ absorbs all displayed $O(\varepsilon)$-powers.

\section{An Upper Bound for \texorpdfstring{$\mathcal{I}_{\Spec}^{\heartsuit}(\mathbf{0},\mathfrak{X}_{\mathfrak{n}})$}{}}
We bound
\begin{align*}
\mathcal{I}_{\mathrm{Spec}}^{\heartsuit}(\mathbf{0},\mathfrak{X}_{\mathfrak{n}})=\mathcal{I}_{\mathrm{Cusp}}^{\heartsuit}(\mathbf{0},\mathfrak{X}_{\mathfrak{n}})+\mathcal{I}_{\mathrm{Eis}}^{\heartsuit}(\mathbf{0},\mathfrak{X}_{\mathfrak{n}}),
\end{align*}
as defined in Definition \ref{defn3.7}, in terms of moments of automorphic $L$-functions on $[\mathrm{PGL}_{2}]$ of level $\mathfrak{n}$. The cuspidal and Eisenstein contributions are treated in Propositions \ref{prop8.2} and \ref{prop8.9}, respectively.

Assume there is $\pi\in\mathcal{F}_{0}(\mathfrak{q},\omega)$ with $C_{v}(\pi)\leq\mathbf{C}_{v}$ for every $v\mid\infty$. The local conductor inequality then gives $C_{v}(\omega)\ll_{F_{v}}\mathbf{C}_{v}$ at these places.

\subsection{A Bound for $\mathcal{I}_{\mathrm{Cusp}}^{\heartsuit}(\mathbf{0},\mathfrak{X}_{\mathfrak{n}})$}
Let $\mathbf{s}\in \mathbf{B}_{\varepsilon}^{4}$. Recall the definition in \textsection\ref{sec3.dual-continuation}:
\begin{align*}
\mathcal{I}_{\mathrm{Cusp}}^{\heartsuit}(\mathbf{s},\mathfrak{X}_{\mathfrak{n}}):=\sum_{\substack{\sigma\in \mathcal{A}_{0}([G],\mathbf{1})}}\widetilde{\Psi}^{*}(\sigma;\mathbf{s},\mathfrak{X}_{\mathfrak{n}}),
\end{align*}
where 
\begin{align*}
\widetilde{\Psi}^{*}(\sigma;\mathbf{s},\mathfrak{X}_{\mathfrak{n}}):=\sum_{\substack{\phi\in\mathfrak{B}(\sigma)}}\widetilde{\Psi}(\overline{W_{\phi}},W_{1}(\cdot,s_{1}),\overline{h_{3}(\cdot,\overline{s_{3}})})\widetilde{\Psi}(W_{\phi}^{*},W_{2}(\cdot,s_{2}),\overline{h_{4}(\cdot,\overline{s_{4}})}).
\end{align*}

\begin{prop}\label{prop8.2}
Let $\vartheta\leq 7/64$ be a parameter towards the Ramanujan conjecture (see \textsection\ref{sec4.ramanujan}). Then
\begin{multline*}
\big|\mathcal{I}_{\mathrm{Cusp}}^{\heartsuit}(\mathbf{0},\mathfrak{X}_{\mathfrak{n}})\big|\ll_{F,\varepsilon} N_{F}(\mathfrak{q})^{1/2+\varepsilon}N_{F}(\mathfrak{n})^{-1/2+\varepsilon}\big[\mathbf{C}_{\infty}C_{\infty}(\omega)^{-1}\big]^{\vartheta}\mathbf{C}_{\infty}^{-1/2+\varepsilon}\\
\sum_{\substack{\sigma\in \mathcal{F}_{0}(\mathfrak{n},\mathbf{1})\\ \sigma_{\infty}\ \mathrm{unramified}}}e^{-\frac{\pi}{2}\sum_{v\mid\infty}|\nu_{\sigma_{v}}|}\sum_{\mathfrak{a}\mid\mathfrak{q}\mathfrak{q}'^{-1}}|\lambda_{\sigma}(\mathfrak{a})|\cdot \frac{|L(1/2,\sigma)|^{3}|L(1/2,\sigma\times\overline{\omega})|}{|L(1,\sigma,\Ad)|},
\end{multline*}  
where $\nu_{\sigma_{v}}$, $v\mid\infty$, is the spectral parameter of $\sigma_{v}$.  	
\end{prop}

\subsubsection{Euler Products for $\widetilde{\Psi}^{*}(\sigma;\mathbf{0},\mathfrak{X}_{\mathfrak{n}})$}\label{sec8.1.1}
By Rankin--Selberg theory, we may write 
\begin{equation}\label{eq8.2}
\widetilde{\Psi}^{*}(\sigma;\mathbf{0},\mathfrak{X}_{\mathfrak{n}})=\frac{\rho_{F}}{2D_{F}^{1/2}\Lambda(2,\mathbf{1})}\frac{\Lambda(1/2,\sigma)^{3}\Lambda(1/2,\sigma\times\omega)}{\Lambda(1,\sigma,\Ad)}\prod_{v\leq\infty}\boldsymbol{e}_{v}(\sigma_{v}),
\end{equation}
where 
\begin{equation}\label{c8.3}
\boldsymbol{e}_{v}(\sigma_{v}):=\frac{\zeta_{F_{v}}(1)}{\zeta_{F_{v}}(2)}\sum_{W_{v}}\frac{\widetilde{\Psi}_{v}(\overline{W_{v}},W_{1,v}(\cdot,0),\overline{h_{3,v}(\cdot,0)})\widetilde{\Psi}_{v}(W_{v}^{*},W_{2,v}(\cdot,0),\overline{h_{4,v}(\cdot,0)})}{L_{v}(1/2,\sigma_{v})^{3}L_{v}(1/2,\sigma_{v}\times\omega_{v})L_{v}(1,\sigma_{v},\Ad)^{-1}}.
\end{equation}
Here the sum ranges over an orthonormal basis of $\mathcal{W}(\sigma_{v})$. The global factor and the ratio $\zeta_{F_{v}}(1)/\zeta_{F_{v}}(2)$ come from the Rankin--Selberg norm formula \cite[(2.2)--(2.5)]{MV10}. Since $\sigma$ is unitary and self-dual, $|L(1/2,\sigma\times\omega)|=|L(1/2,\sigma\times\overline{\omega})|$; the same equality holds locally.

At infinity, $W_{1,v}$ and $h_{3,v}$ are right-$K_{v}$-invariant, so only unramified $\sigma_{v}$ contribute. At finite places:
\begin{itemize}
\item If $v<\infty$ and $v\nmid\mathfrak{n}$, then $W_{1,v}$ and $h_{3,v}$ are spherical, so only the spherical vector in $\mathcal{W}(\sigma_{v})$ can contribute. If also $v\nmid\mathfrak{q}\mathfrak{D}_{F}$, the unramified calculation gives $\boldsymbol{e}_{v}(\sigma_{v})=1$. At $v\mid\mathfrak{D}_{F}$ with $v\nmid\mathfrak{q}$, changing $\psi_{v}$ to an unramified additive character gives $|\boldsymbol{e}_{v}(\sigma_{v})|\ll_{F} q_{v}^{O(d_{v})}$; these factors are absorbed into the $F$-dependent constants in the bounds.

\item If $v\mid\mathfrak{n}$, the function $W_{1,v}(\cdot,0)\overline{h_{3,v}(\cdot,0)}$ is right-$K_{0,v}[e_{v}(\mathfrak{n})]$-invariant. Thus the sum in \eqref{c8.3} may be restricted to an orthonormal basis of $\mathcal{W}(\sigma_{v})^{K_{0,v}[e_{v}(\mathfrak{n})]}$, whose dimension is $\ll 1+e_{v}(\mathfrak{n})$.
\end{itemize}  

Consequently, $\widetilde{\Psi}^{*}(\sigma;\mathbf{0},\mathfrak{X}_{\mathfrak{n}})$ vanishes unless $\sigma\in\mathcal{F}_{0}(\mathfrak{n},\mathbf{1})$, and the product in \eqref{eq8.2} may be restricted to $v\mid\mathfrak{n}\mathfrak{q}\mathfrak{D}_{F}\infty$.

\subsubsection{Bounds for Local Periods}
\begin{lemma}\label{lem8.2}
 Let $\sigma\in \mathcal{F}_{0}(\mathfrak{n},\mathbf{1})$. Let $v\mid\mathfrak{n}$ and $W_{v}$ be a vector in the Whittaker model of $\sigma_{v}$. Let  
\begin{align*}
\boldsymbol{\Psi}(W_{v},\mathfrak{n}):=\int_{N(F_{v})\backslash G(F_{v})}\overline{W_{v}\left(
g_{v}\right)}
W_{3,v}(g_{v}\mathbf{d}_{v},0)\overline{\Phi_{3,v}(\mathbf{e}_{2}g_{v})}|\det g_{v}|^{\frac{1}{2}}dg_{v}.
\end{align*}
Then $\boldsymbol{\Psi}(W_{v},\mathfrak{n})$ equals  
\begin{align*}
\varepsilon_{v}(1/2,\sigma_{v},\psi_{v})\int_{N(F_{v})\backslash\overline{G}(F_{v})}\overline{W_{v}\left(g_{v}\right)
W_{3,v}^{*}(g_{v},0)}h_{3,v}(g_{v}\mathbf{d}_{v},0)dg_{v},
\end{align*}
where $\varepsilon_{v}(1/2,\sigma_{v},\psi_{v})$ is the $\varepsilon$-factor of $\sigma_{v}$ relative to the additive character $\psi_{v}$.
\end{lemma}
\begin{proof}
The integral $\boldsymbol{\Psi}(W_{v},\mathfrak{n})$ converges absolutely. By a change of variable,
\begin{equation}\label{eq8.5}
\boldsymbol{\Psi}(W_{v},\mathfrak{n})=\int_{\overline{G}(F_{v})}\overline{W_{v}\left(wg_{v}\right)}
h_{3,v}(g_{v}\mathbf{d}_{v},0)\overline{h_{3,v}(wg_{v},0)}dg_{v}.
\end{equation}

By \eqref{eq5.7}, we have 
\begin{align*}
W_{v}\left(wg_{v}\right)=\int_{F_{v}^{\times}}j_{\sigma_{v}}(y_{v})W_{v}\left(\begin{pmatrix}
y_{v}\\
&1
\end{pmatrix}g_{v}\right)d^{\times}y_{v}.
\end{align*}

Insert compact cutoffs before interchanging integrals. Substitute into the conjugate of \eqref{eq8.5}, first in the convergence region of \eqref{f5.8} and then by meromorphic continuation. This expresses $\overline{\boldsymbol{\Psi}(W_{v},\mathfrak{n})}$ as
\begin{equation}\label{eq8.6}
\int_{F_{v}^{\times}}j_{\sigma_{v}}(y_{v})d^{\times}y_{v}\int_{\overline{G}(F_{v})}W_{v}\left(g_{v}\right)
\overline{h_{3,v}(g_{v}\mathbf{d}_{v},0)}h_{3,v}(wg_{v},0)dg_{v}.
\end{equation}

The $y_{v}$-integral is computed by \eqref{f5.8} and meromorphic continuation. Since $\sigma_{v}$ has trivial central character and is self-dual,
\begin{equation}\label{eq8.7}
\int_{F_{v}^{\times}}j_{\sigma_{v}}(y_{v})d^{\times}y_{v}=\gamma_{v}(1/2,\sigma_{v},\psi_{v})=\varepsilon_{v}(1/2,\sigma_{v},\psi_{v}).
\end{equation}

Substituting \eqref{eq8.7} into \eqref{eq8.6} and changing variables proves Lemma \ref{lem8.2}. 
\end{proof}

\begin{lemma}\label{lem8.3}
 Let $v\mid\mathfrak{n}$, so that $l_{v}=e_{v}(\mathfrak{n})\geq 1$, and let $\Re(s_{3})>-1/2$. Then, for $k_{v}\in K_{v}$, we have
\begin{multline}\label{4.10}
h_{3,v}(k_{v}\mathbf{d}_{v},s_{3})=\zeta_{F_{v}}(1+2s_{3})\Big[q_{v}^{-l_{v}(1/2+s_{3})}\mathbf{1}_{k_{v}\in K_{v}-K_{0,v}[1]}+q_{v}^{l_{v}(1/2+s_{3})}\mathbf{1}_{k_{v}\in K_{0,v}[l_{v}]}\Big]\\
+\zeta_{F_{v}}(1+2s_{3})\sum_{j=1}^{l_{v}-1}q_{v}^{(2j-l_{v})(1/2+s_{3})}\big[\mathbf{1}_{k_{v}\in K_{0,v}[j]}-\mathbf{1}_{k_{v}\in K_{0,v}[j+1]}\big].
\end{multline}
\end{lemma}
\begin{proof}
We use the standard decomposition
\begin{equation}\label{8.9}
K_{v}=\bigsqcup_{\alpha\in \mathcal{O}_{v}/\mathfrak{p}_{v}^{l_{v}}}\begin{pmatrix}
1& \alpha\\
& 1
\end{pmatrix}wK_{0,v}[l_{v}]\bigsqcup\bigsqcup_{\beta\in \mathfrak{p}_{v}/\mathfrak{p}_{v}^{l_{v}}}\begin{pmatrix}
1\\
\beta & 1
\end{pmatrix}K_{0,v}[l_{v}].
\end{equation}

By \eqref{8.9}, there are three cases. 
\begin{itemize}
\item Suppose $k_{v}\in K_{v}-K_{0,v}[1]$. Then $k_{v}\in \begin{pmatrix}
1& \alpha\\
& 1
\end{pmatrix}wK_{0,v}[l_{v}]$ for some $\alpha\in \mathcal{O}_{v}/\mathfrak{p}_{v}^{l_{v}}$. Note that $K_{0,v}[l_{v}]\mathbf{d}_{v}\subset \mathbf{d}_{v}K_{v}$ and the function $h_{3,v}(\cdot,s_{3})$ is right-$K_{v}$-invariant. Then $h_{3,v}(k_{v}\mathbf{d}_{v},s_{3})$ is equal to 
\begin{equation}\label{8.10}
h_{3,v}(w\mathbf{d}_{v},s_{3})=q_{v}^{-l_{v}(1/2+s_{3})}h_{3,v}(I_{2},s_{3})=q_{v}^{-l_{v}(1/2+s_{3})}\zeta_{F_{v}}(1+2s_{3}).
\end{equation}

\item Suppose $k_{v}\in K_{0,v}[l_{v}]$. Then $k_{v}\mathbf{d}_{v}\in \mathbf{d}_{v}K_{v}$. Thus
\begin{equation}\label{8.11}
h_{3,v}(k_{v}\mathbf{d}_{v},s_{3})=h_{3,v}(\mathbf{d}_{v},s_{3})=q_{v}^{l_{v}(1/2+s_{3})}\zeta_{F_{v}}(1+2s_{3}).
\end{equation}

\item Suppose $k_{v}\in K_{0,v}[j]-K_{0,v}[j+1]$ for some $1\leq j<l_{v}$. By \eqref{8.9}, $k_{v}\in \begin{pmatrix}
1\\
\beta & 1
\end{pmatrix}K_{0,v}[l_{v}]$ for some $\beta\in \mathcal{O}_{v}$ with $e_{v}(\beta)=j$. Using 
\begin{align*}
k_{v}\mathbf{d}_{v}\in \begin{pmatrix}
1\\
\beta & \varpi_{v}^{l_{v}}
\end{pmatrix}K_{v}=\begin{pmatrix}
\varpi_{v}^{l_{v}}\beta^{-1} &1\\
& \beta
\end{pmatrix}K_{v}\subset N(F_{v})\begin{pmatrix}
\varpi_{v}^{l_{v}-j} &\\
& \varpi_{v}^{j}
\end{pmatrix}K_{v}.
\end{align*}

As a consequence, we obtain 
\begin{equation}\label{8.12}
h_{3,v}(k_{v}\mathbf{d}_{v},s_{3})=h_{3,v}(\diag(\varpi_{v}^{l_{v}-j},\varpi_{v}^{j}),s_{3})=q_{v}^{(2j-l_{v})(1/2+s_{3})}\zeta_{F_{v}}(1+2s_{3}).
\end{equation}
\end{itemize}

Therefore, \eqref{4.10} follows from \eqref{8.10}, \eqref{8.11} and \eqref{8.12}. 
\end{proof}

\subsubsection{Growth of $\boldsymbol{e}_{v}(\sigma_{v})$ at $v\mid\mathfrak{n}\mathfrak{q}\infty$}
\begin{lemma}\label{lem8.5}
 Let $v\mid\mathfrak{n}$. Let $\sigma\in \mathcal{F}_{0}(\mathfrak{n},\mathbf{1})$. Then 
\begin{equation}\label{8.21}
|\boldsymbol{e}_{v}(\sigma_{v})|\ll_{F} l_{v}^{8}\cdot q_{v}^{-l_{v}/2}.
\end{equation}
\end{lemma}
\begin{proof}
Let $r_{\sigma_{v}}$ be the exponent of the arithmetic conductor of $\sigma_{v}$. Let $W_{v}^{\circ}$ be the normalized local new vector in the Whittaker model of $\sigma_{v}$. The local orthogonalization in the proof of \cite[Lemma 9]{BM15} depends only on the residue cardinality and the local Hecke eigenvalues, and applies over $F_{v}$. It gives an orthonormal basis $\{W_{v,n_{v}}:\ 0\leq n_{v} \leq e_{v}(\mathfrak{n})-r_{\sigma_{v}}\}$ of $\mathcal{W}(\sigma_{v})^{K_{0,v}[e_{v}(\mathfrak{n})]}$, where
\begin{equation}\label{e7.5}
W_{v,n_{v}}=\sum_{i=0}^{n_{v}}\xi_{\sigma_{v}}(\mathfrak{p}_{v}^{i},\mathfrak{p}_{v}^{n_{v}})\sigma_{v}\left(\begin{pmatrix}
1\\
&\varpi_{v}^{i}
\end{pmatrix}\right)W_{v}^{\circ}
\end{equation}
for some explicit coefficients $\xi_{\sigma_{v}}(\mathfrak{p}_{v}^{i},\mathfrak{p}_{v}^{n_{v}})\ll 1$. The local $L$-factors in \eqref{c8.3} are bounded above and below uniformly by the Ramanujan bound of \textsection\ref{sec4.ramanujan}, including when the inducing characters are ramified. Using this basis in \eqref{c8.3}, we obtain
\begin{align*}
|\boldsymbol{e}_{v}(\sigma_{v})|\ll \sum_{n_{v}}\big|\Psi_{v}(\overline{W_{v,n_{v}}},W_{1,v}(\cdot,0),\overline{h_{3,v}(\cdot,0)})\Psi_{v}(W_{v,n_{v}}^{*},W_{2,v}(\cdot,0),\overline{h_{4,v}(\cdot,0)})\big|,
\end{align*}
where both periods converge absolutely.

Let $\alpha_{\mathfrak{n},v}:=\overline{\omega}_{v}(\varpi_{v}^{l_{v}})q_{v}^{l_{v}/2}(1+q_{v}^{-1})$. By definitions in \textsection\ref{sec4.4}, we have
\begin{align*}
W_{1,v}(g_{v},0)=\alpha_{\mathfrak{n},v}W_{3,v}(g_{v}\mathbf{d}_{v},0),\qquad  \mathbf{d}_{v}=\diag(1,\varpi_{v}^{l_{v}}).
\end{align*}
As a consequence, $\Psi_{v}(\overline{W_{v,n_{v}}},W_{1,v}(\cdot,0),\overline{h_{3,v}(\cdot,0)})$ becomes  
\begin{align*}
\alpha_{\mathfrak{n},v}\int_{N(F_{v})\backslash G(F_{v})}\overline{W_{v,n_{v}}(g_{v})}W_{3,v}(g_{v}\mathbf{d}_{v},0)\overline{\Phi_{3,v}(\mathbf{e}_{2}g_{v})}|\det g_{v}|^{\frac{1}{2}}dg_{v}.
\end{align*}

Thus, with all sections evaluated at $0$, we have
\begin{equation}\label{e8.16}
|\boldsymbol{e}_{v}(\sigma_{v})|\ll l_{v}q_{v}^{l_{v}/2}\sum_{0\leq n_{v}\leq l_{v}-r_{\sigma_{v}}}\big|\boldsymbol{\Psi}(W_{v,n_{v}},\mathfrak{n})\Psi_{v}(W_{v,n_{v}}^{*},W_{2,v},\overline{h_{4,v}})\big|,
\end{equation}
where $\boldsymbol{\Psi}(W_{v,n_{v}},\mathfrak{n})$ is given by Lemma \ref{lem8.2}: 
\begin{align*}
\boldsymbol{\Psi}(W_{v,n_{v}},\mathfrak{n})=\varepsilon_{v}(1/2,\sigma_{v},\psi_{v})\int_{N(F_{v})\backslash\overline{G}(F_{v})}\overline{W_{v,n_{v}}\left(g_{v}\right)
W_{3,v}^{*}(g_{v},0)}h_{3,v}(g_{v}\mathbf{d}_{v},0)dg_{v}.
\end{align*}

By Iwasawa decomposition, $\big|\boldsymbol{\Psi}(W_{v,n_{v}},\mathfrak{n})\big|$ is equal to 
\begin{equation}\label{8.16}
\bigg|\int_{F_{v}^{\times}}W_{3,v}^{*}(a(y_{v}),0)|y_{v}|_{v}^{-\frac{1}{2}}\int_{K_{v}}W_{v,n_{v}}\left(a(y_{v})k_{v}\right)\overline{h_{3,v}(k_{v}\mathbf{d}_{v},0)
}dk_{v}d^{\times}y_{v}\bigg|.
\end{equation}

Let $j\geq0$, and let $\sigma_{v}(\mathbf{1}_{K_{0,v}[j]})$ denote the unnormalized averaging operator
\begin{align*}
\sigma_{v}(\mathbf{1}_{K_{0,v}[j]})W:=\int_{K_{0,v}[j]}\sigma_{v}(k_{v})Wdk_{v}.
\end{align*}
Dividing by $\Vol(K_{0,v}[j])$ gives the orthogonal projection onto the right-$K_{0,v}[j]$-invariant subspace, which is zero for $j<r_{\sigma_{v}}$ and otherwise is spanned by $W_{v,i}$, $0\leq i\leq j-r_{\sigma_{v}}$. Orthogonality therefore gives
\begin{equation}\label{8.18}
\sigma_{v}(\mathbf{1}_{K_{0,v}[j]})W_{v,n_{v}}=\Vol(K_{0,v}[j])\mathbf{1}_{j\geq n_{v}+r_{\sigma_{v}}}W_{v,n_{v}}.
\end{equation}

It follows from \eqref{8.18} and  Lemma \ref{lem8.3} that 
\begin{equation}\label{8.19}
\int_{K_{v}}W_{v,n_{v}}\left(a(y_{v})k_{v}\right)\overline{h_{3,v}(k_{v}\mathbf{d}_{v},0)
}dk_{v}\ll l_{v}q_{v}^{-l_{v}/2}\big|W_{v,n_{v}}(a(y_{v}))\big|. 
\end{equation}

Inserting the oldform expansion \eqref{e7.5} and using the diagonal newvector bounds for $W_{v}^{\circ}$ and the Casselman--Shalika formula for $W_{3,v}^{*}$ gives
\begin{align*}
\int_{F_{v}^{\times}}\big|W_{v,n_{v}}(a(y_{v}))W_{3,v}^{*}(a(y_{v}),0)\big||y_{v}|_{v}^{-1/2}d^{\times}y_{v}\ll (1+n_{v})^{2}.
\end{align*}
The Ramanujan bound gives uniform convergence of the geometric series, and summing the translates in \eqref{e7.5} contributes at most $(1+n_{v})^{2}$. Thus \eqref{8.16} and \eqref{8.19} give $|\boldsymbol{\Psi}(W_{v,n_{v}},\mathfrak{n})|\ll l_{v}^{3}q_{v}^{-l_{v}/2}$.

For the second period, apply the calculation of Lemma \ref{lem8.2}, including meromorphic continuation, with $h_{3,v}$ replaced by $h_{4,v}$. Both the initial and final periods converge absolutely. By \eqref{twist} and unitarity, the resulting $\gamma$-factor has modulus one. On the sets indexed by $0\leq j\leq l_{v}$ in Lemma \ref{lem8.3}, the section $h_{4,v}(k_{v}\mathbf{d}_{v},0)$ equals $h_{4,v}(I_{2},0)\omega_{v}(\varpi_{v})^{j}q_{v}^{j-l_{v}/2}$. Since $|h_{4,v}(I_{2},0)|=|L_{v}(1,\overline{\omega}_{v})|\ll1$ and $|\omega_{v}(\varpi_{v})|=1$, \eqref{8.18} gives the same bound for the $K_{v}$-average. The remaining Whittaker integral has the same bound as above, since $|W_{4,v}^{*}(a(\varpi_{v}^{k}),0)|\leq(k+1)q_{v}^{-k/2}$ for $k\geq0$. Hence
\begin{align*}
\big|\Psi_{v}(W_{v,n_{v}}^{*},W_{2,v},\overline{h_{4,v}})\big|\ll l_{v}^{3}q_{v}^{-l_{v}/2}.
\end{align*}
There are at most $l_{v}+1\ll l_{v}$ relevant oldvectors, and substitution into \eqref{e8.16} yields
\begin{align*}
|\boldsymbol{e}_{v}(\sigma_{v})|\ll l_{v}q_{v}^{l_{v}/2}\cdot l_{v}\cdot l_{v}^{6}q_{v}^{-l_{v}}\ll l_{v}^{8}q_{v}^{-l_{v}/2}.
\end{align*}
Hence, \eqref{8.21} follows. 
\end{proof}

\begin{lemma}\label{lem8.6}
 Let $v\mid\mathfrak{q}$. Let $\sigma\in \mathcal{F}_{0}(\mathfrak{n},\mathbf{1})$. Then 
\begin{equation}\label{8.24}
|\boldsymbol{e}_{v}(\sigma_{v})|\ll_{F} q_{v}^{m/2}\big[|\lambda_{\sigma}(\mathfrak{p}_{v}^{m-r})|+|\lambda_{\sigma}(\mathfrak{p}_{v}^{m-r-1})|\big]\ll_{F}(m+1)q_{v}^{m/2+(m-r)\vartheta},
\end{equation}
where $m=e_{v}(\mathfrak{q})$, $r=e_{v}(\mathfrak{q}')$, and $\lambda_{\sigma}(\mathfrak{p}_{v}^{-1}):=0$. 
\end{lemma}
\begin{proof}
Since $\mathfrak{n}+\mathfrak{q}=\mathcal{O}_{F}$, the representation $\sigma_{v}$ is unramified. By \eqref{c8.3} and the fact that $W_{1,v}(\cdot,0)\overline{h_{3,v}(\cdot,0)}$ is right-$K_{v}$-invariant, only the normalized spherical vector contributes. The normalized local $L$-factors in \eqref{c8.3} are bounded above and below uniformly by the Ramanujan bound, and hence
\begin{align*}
|\boldsymbol{e}_{v}(\sigma_{v})|\ll \big|\widetilde{\Psi}_{v}(\overline{W_{v}^{\circ}},W_{1,v}(\cdot,0),\overline{h_{3,v}(\cdot,0)})\widetilde{\Psi}_{v}(W_{v}^{\circ,*},W_{2,v}(\cdot,0),\overline{h_{4,v}(\cdot,0)})\big|,
\end{align*}
where $\overline{W_{v}^{\circ}}$ is the normalized spherical vector. 

By Casselman-Shalika formula, 
\begin{equation}\label{8.22}
\widetilde{\Psi}_{v}(\overline{W_{v}^{\circ}},W_{1,v}(\cdot,0),\overline{h_{3,v}(\cdot,0)})\ll |L_{v}(1/2,\sigma_{v})|^{2}\ll 1.
\end{equation}

First suppose that $\psi_{v}$ is unramified. The $K_{v}$-integration in the second period cancels $h_{4,v}(I_{2},0)=\Vol(K_{0,v}[m])^{-1}$. If $r\geq1$, the Gauss integral in \eqref{c2.5} gives
\begin{align*}
|W_{2,v}(a(\varpi_{v}^{k}),0)|=\Vol(K_{0,v}[m])^{-1}q_{v}^{-(k+r)/2}\mathbf{1}_{k\geq m-r}.
\end{align*}
Its phase is independent of $k$ on this support. Factoring out this phase, the Hecke relation gives
\begin{align*}
\sum_{k\geq m-r}\lambda_{\sigma}(\mathfrak{p}_{v}^{k})q_{v}^{-k/2}=q_{v}^{-(m-r)/2}L_{v}(1/2,\sigma_{v})(\lambda_{\sigma}(\mathfrak{p}_{v}^{m-r})-q_{v}^{-1/2}\lambda_{\sigma}(\mathfrak{p}_{v}^{m-r-1})).
\end{align*}
Thus the second period is bounded by $q_{v}^{m/2}(|\lambda_{\sigma}(\mathfrak{p}_{v}^{m-r})|+|\lambda_{\sigma}(\mathfrak{p}_{v}^{m-r-1})|)$.

If $r=0$, the same calculation also includes the term $k=m-1$:
\begin{align*}
W_{2,v}(a(\varpi_{v}^{k}),0)=(1+q_{v}^{-1})q_{v}^{m-1-k/2}\bigg[-\omega_{v}(\varpi_{v})^{-1}\mathbf{1}_{k\geq m-1}+(q_{v}-1)\mathbf{1}_{k\geq m}\sum_{a=0}^{k-m}\omega_{v}(\varpi_{v})^{a}\bigg].
\end{align*}
Separate the term $k=m-1$ and apply the same Hecke relation to the sum over $k\geq m$. Since $\omega_{v}$ is unitary, the coefficients of $\lambda_{\sigma}(\mathfrak{p}_{v}^{m})$ and $\lambda_{\sigma}(\mathfrak{p}_{v}^{m-1})$ are bounded by $\sum_{k\geq0}(k+1)^{2}q_{v}^{-k(1/2-\vartheta)}\ll 1$. This gives the same bound. At places dividing $\mathfrak{D}_{F}$, changing $\psi_{v}$ to an unramified additive character changes only the $F$-dependent constant. Combining with \eqref{8.22} and the Ramanujan bound proves \eqref{8.24}.
\end{proof}

\begin{lemma}\label{lem8.7}
 Let $v\mid\infty$ and $\sigma\in \mathcal{F}_{0}(\mathfrak{n},\mathbf{1})$ with $\sigma_{v}$ unramified. Put $\nu=\nu_{\sigma_{v}}$, $d=[F_{v}:\mathbb{R}]$, and $a=\varepsilon+d|\Re(\nu)|$. Define 
\begin{align*}
\boldsymbol{e}_{v}^{*}(\sigma_{v}):=L_{v}(1/2,\sigma_{v})^{3}L_{v}(1/2,\sigma_{v}\times\omega_{v})L_{v}(1,\sigma_{v},\Ad)^{-1}\boldsymbol{e}_{v}(\sigma_{v}),
\end{align*}
where $\boldsymbol{e}_{v}(\sigma_{v})$ is defined by \eqref{c8.3}. Then 
\begin{equation}\label{eq8.24}
|\boldsymbol{e}_{v}^{*}(\sigma_{v})|\ll_{F,\varepsilon} \frac{|L_{v}(1/2,\sigma_{v})|^{2}\big[\mathbf{C}_{v}C_{v}(\omega)^{-1}\big]^{|\Re(\nu)|}\sqrt{|\Gamma(a+d\nu)\Gamma(a-d\nu)|}}{\mathbf{C}_{v}^{1/2-10\varepsilon}|\Gamma((1/2+\nu)d)\Gamma((1/2-\nu)d)|}.
\end{equation}
\end{lemma}
\begin{proof}
Since $\sigma_{v}$ is unramified with trivial central character, write $\sigma_{v}=|\cdot|^{\nu}\boxplus|\cdot|^{-\nu}$ with $|\Re(\nu)|\leq 7/64$ (see \textsection\ref{sec4.ramanujan}). Since $W_{1,v}(\cdot,0)\overline{h_{3,v}(\cdot,0)}$ is right-$K_{v}$-invariant, 
\begin{align*}
\boldsymbol{e}_{v}^{*}(\sigma_{v})\ll \big|\widetilde{\Psi}_{v}(\overline{W_{v}^{\circ}},W_{1,v}(\cdot,0),\overline{h_{3,v}(\cdot,0)})\widetilde{\Psi}_{v}(W_{v}^{\circ,*},W_{2,v}(\cdot,0),\overline{h_{4,v}(\cdot,0)})\big|,
\end{align*}
where $W_{v}^{\circ}$ is the normalized spherical vector relative to the generic character $\theta_{v}$, and $W_{v}^{\circ,*}$ is the vector relative to $\overline{\theta}_{v}$ (see \textsection\ref{sec1.3.3}). There is a scalar $c_{v}(\nu)\in\mathbb{C}^{\times}$ such that
\begin{align*}
&W_{v}^{\circ}(a(y_{v}))=c_{v}(\nu)|y_{v}|_{v}^{1/2}K_{\nu}(2\pi |y_{v}|_{v}),\quad \ \ \text{if $F_{v}\simeq \mathbb{R}$},\\
&W_{v}^{\circ}(a(y_{v}))=c_{v}(\nu)|y_{v}|_{v}^{1/2}K_{2\nu}(4\pi |y_{v}|_{v}^{1/2}),\quad \ \ \text{if $F_{v}\simeq \mathbb{C}$}.
\end{align*}
The absolute value $|c_{v}(\nu)|$ is determined by $\langle W_{v}^{\circ},W_{v}^{\circ}\rangle=1$. Since the factor $|y_{v}|_{v}^{1/2}$ is part of $W_{v}^{\circ}$, \cite[6.576.4]{GR15} gives
\begin{align*}
|c_{v}(\nu)|^{-2}&\asymp_{F_{v}}\Gamma(1/2+\nu)\Gamma(1/2-\nu), &&F_{v}\simeq\mathbb{R},\\
|c_{v}(\nu)|^{-2}&\asymp_{F_{v}}\Gamma(1+2\nu)\Gamma(1-2\nu), &&F_{v}\simeq\mathbb{C}.
\end{align*}
Here $|\cdot|$ means the modulus of a complex number (so that $|\cdot|_{v}=|\cdot|^{2}$ if $F_{v}\simeq\mathbb{C}$). Consequently, 
\begin{align*}
|\Psi_{v}(\overline{W_{v}^{\circ}},W_{1,v}(\cdot,0),\overline{h_{3,v}(\cdot,0)})|\ll_{F_{v}} |c_{v}(\nu)|\cdot |L_{v}(1/2,\sigma_{v})|^{2}.
\end{align*}
Analogously,
\begin{align*}
\widetilde{\Psi}_{v}(W_{v}^{\circ,*},W_{2,v}(\cdot,0),\overline{h_{4,v}(\cdot,0)})=\Psi_{v}(W_{v}^{\circ,*},W_{2,v}(\cdot,0),\overline{h_{4,v}(\cdot,0)}).
\end{align*}

Using these formulas for $W_{v}^{\circ}$, we obtain 
\begin{multline*}
|\widetilde{\Psi}_{v}(W_{v}^{\circ,*},W_{2,v}(\cdot,0),\overline{h_{4,v}(\cdot,0)})|\ll |c_{v}(\nu)|\int_{K_{v}}|h_{4,v}(k_{v},0)|\\
\int_{F_{v}^{\times}}|K_{d\nu}(2^{d}\pi |y_{v}|_{v}^{d^{-1}})|\cdot |W_{2,v}(a(y_{v})k_{v},0)|d^{\times}y_{v}dk_{v}.
\end{multline*}

Let $m\geq 0$. By \cite[Lemma 3.7.2]{MV10}, for $k_{v}$ with $h_{4,v}(k_{v},0)\neq 0$, 
\begin{align*}
W_{2,v}(a(y_{v})k_{v},0)\ll_{F_{v},\varepsilon,m} \mathbf{C}_{v}^{1/2+\varepsilon}C_{v}(\omega)^{-1/2}|y_{v}|_{v}^{1/2-\varepsilon}\left(1+\frac{|y_{v}|_{v}}{C_{v}(\omega)\mathbf{C}_{v}^{-1}}\right)^{-m}.
\end{align*}

Combining these estimates gives  
\begin{equation}\label{8.28}
|\boldsymbol{e}_{v}^{*}(\sigma_{v})|\ll \frac{|L_{v}(1/2,\sigma_{v})|^{2}\mathbf{C}_{v}^{1/2}}{|\Gamma((1/2+\nu)d)\Gamma((1/2-\nu)d)|}\cdot \mathbf{C}_{v}^{-1/2+\varepsilon}C_{v}(\omega)^{-1/2}\cdot \mathcal{I}_{m},
\end{equation}
where $\mathbf{C}_{v}^{1/2}$ comes from the sup-norm of $h_{4,v}(\cdot,0)$, and 
\begin{align*}
\mathcal{I}_{m}:=\int_{F_{v}^{\times}}|K_{d\nu}(2^{d}\pi |y_{v}|_{v}^{d^{-1}})|\left(1+\frac{|y_{v}|_{v}}{C_{v}(\omega)\mathbf{C}_{v}^{-1}}\right)^{-m}|y_{v}|_{v}^{1/2-\varepsilon}d^{\times}y_{v}.
\end{align*}

Recall that $|\Re(\nu)|\leq \vartheta\leq 7/64$ (see \textsection\ref{sec4.ramanujan}). By the Cauchy--Schwarz inequality, we obtain 
\begin{equation}\label{8.29}
\mathcal{I}_{m}\ll \sqrt{\mathcal{I}_{m}^{(1)}\cdot \mathcal{I}_{m}^{(2)}},
\end{equation}
where 
\begin{align*}
\mathcal{I}_{m}^{(1)}&:=\int_{F_{v}^{\times}}|K_{d\nu}(2^{d}\pi |y_{v}|_{v}^{d^{-1}})|^{2}\cdot |y_{v}|_{v}^{2a/d}d^{\times}y_{v},\\
\mathcal{I}_{m}^{(2)}&:=\int_{F_{v}^{\times}}\left(1+\frac{|y_{v}|_{v}}{C_{v}(\omega)\mathbf{C}_{v}^{-1}}\right)^{-2m}|y_{v}|_{v}^{1-2\varepsilon-2a/d}d^{\times}y_{v}.
\end{align*}

Since $a>|\Re(d\nu)|$, the integral $\mathcal{I}_{m}^{(1)}$ converges at zero. By \cite[6.576.4]{GR15} and Stirling's formula, 
\begin{equation}\label{8.30}
\mathcal{I}_{m}^{(1)}\ll_{\varepsilon} \left|\Gamma(a+d\nu)\Gamma(a-d\nu)\right|.
\end{equation}

Since $1-2\varepsilon-2a/d\geq 1-4\varepsilon-7/32>0$, the integral $\mathcal{I}_{m}^{(2)}$ converges absolutely. Take $m=\floor{1000\varepsilon^{-1}}$ and split the integral at $|y_{v}|_{v}=C_{v}(\omega)\mathbf{C}_{v}^{-1+\varepsilon}$. We obtain
\begin{equation}\label{8.31}
\mathcal{I}_{m}^{(2)}\ll_{\varepsilon} \big[C_{v}(\omega)\mathbf{C}_{v}^{-1+\varepsilon}\big]^{1-2\varepsilon-2a/d}.
\end{equation}

Substituting \eqref{8.29}, \eqref{8.30} and \eqref{8.31} into \eqref{8.28} yields \eqref{eq8.24}.
\end{proof}

\subsubsection{An upper bound for $\widetilde{\Psi}^{*}(\sigma;\mathbf{0},\mathfrak{X}_{\mathfrak{n}})$}
\begin{lemma}\label{lem8.8}
 Let $\sigma\in \mathcal{F}_{0}(\mathfrak{n},\mathbf{1})$ with $\sigma_{\infty}$ unramified. Let $\vartheta\leq 7/64$ be a parameter towards the Ramanujan conjecture (see \textsection\ref{sec4.ramanujan}), and let $\nu_{\sigma_{v}}$ be the spectral parameter of $\sigma_{v}$ for $v\mid\infty$. Then 
\begin{multline}\label{8.33.}
|\widetilde{\Psi}^{*}(\sigma;\mathbf{0},\mathfrak{X}_{\mathfrak{n}})|\ll_{F,\varepsilon} \frac{|L(1/2,\sigma)|^{3}|L(1/2,\sigma\times\overline{\omega})|}{|L(1,\sigma,\Ad)|}\sum_{\mathfrak{a}\mid\mathfrak{q}\mathfrak{q}'^{-1}}|\lambda_{\sigma}(\mathfrak{a})|\\
N_{F}(\mathfrak{n})^{-\frac{1}{2}+\varepsilon}N_{F}(\mathfrak{q})^{1/2+\varepsilon}\mathbf{C}_{\infty}^{-1/2+\varepsilon}\cdot \big[\mathbf{C}_{\infty}C_{\infty}(\omega)^{-1}\big]^{\vartheta}\cdot e^{-\frac{\pi}{2}\sum_{v\mid\infty}|\nu_{\sigma_{v}}|}.
\end{multline}  
\end{lemma}
\begin{proof}
Multiplying the two-term sums in \eqref{8.24} over $v\mid\mathfrak{q}$ gives a sum over divisors of $\mathfrak{q}\mathfrak{q}'^{-1}$; the divisor bound absorbs the local constants. Substitute Lemmas \ref{lem8.5}--\ref{lem8.7} into \eqref{eq8.2}. By Stirling's formula, the Gamma quotient multiplied by $|L_{v}(1/2,\sigma_{v})|^{2}$ in \eqref{eq8.24} is $O_{\varepsilon}(e^{-\pi|\nu_{\sigma_{v}}|/2})$ at real places and $O_{\varepsilon}(e^{-\pi|\nu_{\sigma_{v}}|})$ at complex places. Since $|\Re(\nu_{\sigma_{v}})|\leq\vartheta$ and $C_{v}(\omega)\ll_{F_{v}}\mathbf{C}_{v}$, these bounds give \eqref{8.33.} after rescaling $\varepsilon$.
\end{proof}

\subsubsection{Proof of Proposition \ref{prop8.2}}
Proposition \ref{prop8.2} follows from Lemma \ref{lem8.8}.

\subsection{A Bound for $\mathcal{I}_{\mathrm{Eis}}^{\heartsuit}(\mathbf{0},\mathfrak{X}_{\mathfrak{n}})$}
Let $\mathbf{s}\in \mathbf{B}_{\varepsilon}^{4}$. Recall the definition in \textsection\ref{sec3.dual-continuation}:
\begin{align*}
\mathcal{I}_{\mathrm{Eis}}^{\heartsuit}(\mathbf{s},\mathfrak{X}_{\mathfrak{n}}):=\sum_{\substack{\mu\in \widehat{C}_{F}}}\frac{1}{2\pi i}\int_{i\mathbb{R}}\widetilde{\Psi}^{*}(\lambda,\mu;\mathbf{s},\mathfrak{X}_{\mathfrak{n}})d\lambda,
\end{align*}
where $\widetilde{\Psi}^{*}(\lambda,\mu;\mathbf{s},\mathfrak{X}_{\mathfrak{n}})$ is defined as  
\begin{align*}
\frac{1}{2}\sum_{h\in \mathfrak{B}(\mu,\overline{\mu})}\widetilde{\Psi}(\overline{W_{E(\cdot,h,-\overline{\lambda})}},W_{1}(\cdot,s_{1}),\overline{h_{3}(\cdot,\overline{s_{3}})})\widetilde{\Psi}(W_{E(\cdot,h,\lambda)}^{*},W_{2}(\cdot,s_{2}),\overline{h_{4}(\cdot,\overline{s_{4}})}).
\end{align*}

\begin{prop}\label{prop8.9}
 Let $\mathfrak{X}(\mathfrak{n},\mathbf{1})$ be as in Theorem \ref{thmF}. We have
\begin{multline}\label{8.37}
\big|\mathcal{I}_{\mathrm{Eis}}^{\heartsuit}(\mathbf{0},\mathfrak{X}_{\mathfrak{n}})\big|\ll_{F,\varepsilon} N_{F}(\mathfrak{q})^{1/2+\varepsilon}N_{F}(\mathfrak{n})^{-1/2+\varepsilon}\mathbf{C}_{\infty}^{-1/2+\varepsilon}\sum_{\substack{\mu\in\mathfrak{X}(\mathfrak{n},\mathbf{1})\\ \mu_{\infty}\ \mathrm{unramified}}}\\
\int_{\mathbb{R}}e^{-\frac{\pi}{2}\sum_{v\mid\infty}|\nu_{\sigma_{\mu_{v},it}}|}\cdot \frac{|L(1/2+it,\mu)|^{6}|L(1/2+it,\mu\overline{\omega})L(1/2+it,\mu\omega)|}{|L(1+2it,\mu^{2})|^{2}}dt,
\end{multline}  
where $\nu_{\sigma_{\mu_{v},it}}$, $v\mid\infty$, is the spectral parameter of $\sigma_{\mu_{v},it}:=\mu_{v}|\cdot|^{it}\boxplus \overline{\mu}_{v}|\cdot|^{-it}$. 
\end{prop}
\begin{proof}
Write $\sigma_{\mu,it}=\mu|\cdot|^{it}\boxplus \overline{\mu}|\cdot|^{-it}$. By Rankin--Selberg theory, we may write 
\begin{align*}
\widetilde{\Psi}^{*}(it,\mu;\mathbf{0},\mathfrak{X}_{\mathfrak{n}})=\frac{\Lambda(1/2,\sigma_{\mu,it})^{3}\Lambda(1/2,\sigma_{\mu,it}\times\omega)}{2\rho_{F}\Lambda(1+2it,\mu^{2})\Lambda(1-2it,\overline{\mu}^{2})}\prod_{v\leq\infty}\boldsymbol{e}_{v}(\sigma_{\mu_{v},it}),
\end{align*}
where 
\begin{align*}
\boldsymbol{e}_{v}(\sigma_{\mu_{v},it}):=\sum_{W_{v}}\frac{\widetilde{\Psi}_{v}(\overline{W_{v}},W_{1,v}(\cdot,0),\overline{h_{3,v}(\cdot,0)})\widetilde{\Psi}_{v}(W_{v}^{*},W_{2,v}(\cdot,0),\overline{h_{4,v}(\cdot,0)})}{L_{v}(1/2,\sigma_{\mu_{v},it})^{3}L_{v}(1/2,\sigma_{\mu_{v},it}\times\omega_{v})|L_{v}(1+2it,\mu_{v}^{2})|^{-2}},
\end{align*}
where $W_{v}$ ranges over the Jacquet images of an orthonormal compact-model basis of $\sigma_{\mu_{v},it}$, with the induced-model norm.

As in \textsection\ref{sec8.1.1}, we have $\boldsymbol{e}_{v}(\sigma_{\mu_{v},it})=1$ for $v\nmid\mathfrak{n}\mathfrak{q}\mathfrak{D}_{F}\infty$, and the factors at $v\mid\mathfrak{D}_{F}$, $v\nmid\mathfrak{q}$, are absorbed into the $F$-dependent constants. For $v\mid\mathfrak{n}\mathfrak{q}\infty$, both local periods converge absolutely, so $\widetilde{\Psi}_{v}=\Psi_{v}$ here. Apply the local arguments of Lemmas \ref{lem8.5}--\ref{lem8.7}: the representation is generic and tempered, and is unramified at $v\mid\mathfrak{q}$. At infinity, $\Re(\nu_{\sigma_{\mu_{v},it}})=0$, so $[\mathbf{C}_{v}C_{v}(\omega)^{-1}]^{|\Re(\nu_{\sigma_{\mu_{v},it}})|}=1$ in \eqref{eq8.24}. The proof of Lemma \ref{lem8.8} then gives
\begin{align*}
|\widetilde{\Psi}^{*}(it,\mu;\mathbf{0},\mathfrak{X}_{\mathfrak{n}})|\ll_{F,\varepsilon} \frac{|L(1/2,\sigma_{\mu,it})|^{3}|L(1/2,\sigma_{\mu,it}\times\overline{\omega})|}{|L(1+2it,\mu^{2})|^{2}}\cdot \mathbf{1}_{\mu\in \mathfrak{X}(\mathfrak{n},\mathbf{1})}\sum_{\mathfrak{a}\mid\mathfrak{q}\mathfrak{q}'^{-1}}|\lambda_{\sigma_{\mu,it}}(\mathfrak{a})|\\
N_{F}(\mathfrak{n})^{-\frac{1}{2}+\varepsilon}N_{F}(\mathfrak{q})^{1/2+\varepsilon}\mathbf{C}_{\infty}^{-1/2+\varepsilon}\cdot e^{-\frac{\pi}{2}\sum_{v\mid\infty}|\nu_{\sigma_{\mu_{v},it}}|}.
\end{align*}

Since $\sigma_{\mu,it}$ is tempered, the divisor sum of its Hecke eigenvalues is $\ll_{F,\varepsilon}N_{F}(\mathfrak{q})^{\varepsilon}$. This proves \eqref{8.37}.
\end{proof}

\section{Bounds for the Residues $\widetilde{\Psi}_{\mathrm{Dual}}^{(i)}(\mathbf{0}\mid\mathfrak{X}_{\mathfrak{n}})$}\label{sec8}

We prove the following bound for the eight residual terms in the dual spectral expansion. Here and in \textsection\ref{sec9}, we assume $C_{v}(\omega)\ll_{F_{v}}\mathbf{C}_{v}$ for $v\mid\infty$.

\begin{prop}\label{prop9.1}
 Let $1\leq i\leq 8$. Then 
\begin{align*}
\big|\widetilde{\Psi}_{\mathrm{Dual}}^{(i)}(\mathbf{0}\mid\mathfrak{X}_{\mathfrak{n}})\big|\ll_{F,\varepsilon} C(\omega)^{\varepsilon}\mathbf{C}_{\infty}^{\varepsilon}N_{F}(\mathfrak{q})^{1+\varepsilon}N_{F}(\mathfrak{n})^{-1/2+\varepsilon}.
\end{align*}
\end{prop}

\subsection{Euler Products for $\widetilde{\Psi}_{\mathrm{Dual}}^{(i)}(\mathbf{s},\mathfrak{X}_{\mathfrak{n}})$}
We use Proposition~\ref{prop3.residual-integrals} and factor out the global $L$-functions as in Lemma~\ref{lem4.1}.

Let $\boldsymbol{e}^{*}(\lambda,\mu;\mathbf{s}):=\widetilde{\Psi}^{*}(\lambda,\mu;\mathbf{s},\mathfrak{X}_{\mathfrak{n}})/\mathbf{L}^{*}(\lambda,\mu;\mathbf{s})$, where $\mathbf{L}^{*}(\lambda,\mu;\mathbf{s})$ is the product of the complete $L$-functions 
\begin{multline*}
\Lambda(1+2\lambda,\mu^{2})^{-1}\Lambda(1-2\lambda,\overline{\mu}^{2})^{-1}\Lambda(1/2+s_{3}+s_{1}-\lambda,\overline{\mu})\Lambda(1/2+s_{3}+s_{1}+\lambda,\mu)\\
\Lambda(1/2+s_{3}-s_{1}-\lambda,\overline{\mu})\Lambda(1/2+s_{3}-s_{1}+\lambda,\mu)
\Lambda(1/2+s_{4}+s_{2}+\lambda,\mu)\\
\Lambda(1/2+s_{4}+s_{2}-\lambda,\overline{\mu})\Lambda(1/2+s_{4}-s_{2}+\lambda,\mu\omega)\Lambda(1/2+s_{4}-s_{2}-\lambda,\overline{\mu}\omega).
\end{multline*}

As in \textsection\ref{sec6.1.1}, Rankin--Selberg factorization gives, for $\mathbf{s}\in\mathbf{S}_{\varepsilon}^{4}$,
\begin{equation}\label{c9.1}
\begin{aligned}
\widetilde{\Psi}_{\mathrm{Dual}}^{(1)}(\mathbf{s},\mathfrak{X}_{\mathfrak{n}})&=-\boldsymbol{e}^{*}(s_{3}+s_{1}-1/2,\mathbf{1};\mathbf{s})\underset{\lambda=s_{3}+s_{1}-\frac{1}{2}}{\Res}\ \mathbf{L}^{*}(\lambda,\mathbf{1};\mathbf{s}),\\
\widetilde{\Psi}_{\mathrm{Dual}}^{(6)}(\mathbf{s},\mathfrak{X}_{\mathfrak{n}})&=\boldsymbol{e}^{*}(1/2+s_{2}-s_{4},\overline{\omega};\mathbf{s})\underset{\lambda=\frac{1}{2}+s_{2}-s_{4}}{\Res}\ \mathbf{L}^{*}(\lambda,\overline{\omega};\mathbf{s}).
\end{aligned}
\end{equation}
The identities are first obtained when the indicated poles are simple and distinct and then continued meromorphically as in Definition~\ref{defn3.9}. The other residual terms have the same form.

By definition and Rankin--Selberg theory,
\begin{align*}
\boldsymbol{e}^{*}(\lambda,\mu;\mathbf{s})=\frac{1}{2\rho_{F}}\prod_{v}\boldsymbol{e}_{v}^{*}(\lambda,\mu_{v};\mathbf{s}),
\end{align*}
where
\begin{equation}\label{c9.2}
\boldsymbol{e}_{v}^{*}(\lambda,\mu_{v};\mathbf{s})=\sum_{h_{v}\in \mathfrak{B}(\mu_{v},\overline{\mu}_{v})}\frac{\widetilde{\Psi}_{v}^{*}(\lambda,h_{v};\mathbf{s})}{\mathbf{L}_{v}^{*}(\lambda,\mu_{v};\mathbf{s})}.
\end{equation}
Here $\mathbf{L}_{v}^{*}(\lambda,\mu_{v};\mathbf{s})$ is the $v$-th factor of $\mathbf{L}^{*}(\lambda,\mu;\mathbf{s})$, and $\widetilde{\Psi}_{v}^{*}(\lambda,h_{v};\mathbf{s})$ is defined by 
\begin{align*}
\widetilde{\Psi}_{v}(\overline{W_{h_{v},-\overline{\lambda}}},W_{1,v}(\cdot,s_{1}),\overline{h_{3,v}(\cdot,\overline{s_{3}})})\widetilde{\Psi}_{v}(W_{h_{v},\lambda}^{*},W_{2,v}(\cdot,s_{2}),\overline{h_{4,v}(\cdot,\overline{s_{4}})}),
\end{align*}
with $W_{h_{v},\lambda}$ as in \eqref{eq7.2}.

If $v\nmid\mathfrak{n}\mathfrak{q}\mathfrak{D}_{F}\infty$, all the local data and the induced representation $\mu_{v}|\cdot|^{\lambda}\boxplus \overline{\mu}_{v}|\cdot|^{-\lambda}$ are unramified, and the self-dual measures have the standard unramified volumes. The local Rankin--Selberg calculation, first in the region of absolute convergence and then by meromorphic continuation, gives $\boldsymbol{e}_{v}^{*}(\lambda,\mu_{v};\mathbf{s})\equiv 1$. At a place $v\mid\mathfrak{D}_{F}$ with $v\nmid\mathfrak{q}$, the same calculation leaves only powers of $q_{v}^{d_{v}}$ arising from the self-dual measures; uniformly on $\mathbf{S}_{\varepsilon}^{4}$,
\begin{align*}
|\boldsymbol{e}_{v}^{*}(\lambda,\mu_{v};\mathbf{s})|\ll_{F,\varepsilon}q_{v}^{O(d_{v})}.
\end{align*}
Their product is absorbed into the implied constant.

\subsection{Estimates for \texorpdfstring{$\boldsymbol{e}_{v}^{*}(\lambda,\mu_{v};\mathbf{s})$}{} at $v\mid\mathfrak{n}\mathfrak{q}\mathfrak{D}_{F}\infty$}
Throughout this subsection, $\mathbf{s}\in\mathbf{S}_{\varepsilon}^{4}$ and $(\lambda,\mu)$ is one of the eight pairs in Definition~\ref{defn3.9} for the datum~\eqref{4.5}.
\begin{lemma}\label{lem9.2}
 Let $v\mid\mathfrak{n}$. Then
\begin{equation}\label{9.5}
|\boldsymbol{e}_{v}^{*}(\lambda,\mu_{v};\mathbf{s})|\ll_{F_{v},\varepsilon} q_{v}^{-l_{v}(1/2-\varepsilon)}.
\end{equation}
\end{lemma}
\begin{proof}
All inducing characters are unramified at $v$. The intertwining argument in Lemma~\ref{lem6.3} (A) applies to \eqref{c9.2}, so it suffices to treat $\Re(\lambda)=1/2+O(\varepsilon)$. By Proposition~\ref{prop3.residual-integrals}, the local factor is the ratio of the residual integral with and without $\mathbf{d}_{v}$. Intertwining commutes with right translation, so the Tate factors and intertwining scalars cancel, leaving $\alpha_{\mathfrak{n},v}q_{v}^{l_{v}(s_{1}+s_{2})}$.

Consider $(\lambda,\mu_{v})=(1/2-s_{1}-s_{3},\mathbf{1})$; the other cases follow by the replacements in Proposition~\ref{prop3.residual-integrals}. Choose a fixed orthonormal basis $h_{n}$, $0\leq n\leq l_{v}$, of the $K_{0,v}[l_{v}]$-fixed subspace in the compact model, adapted to the nested $K_{0,v}[n]$-fixed subspaces, with $h_{0}=1$. Put
\begin{align*}
d_{n}:=\frac{\int_{K_{v}}h_{3,v}(k\mathbf{d}_{v},-s_{1})\overline{h_{n}(k)}dk}{h_{3,v}(I_{2},-s_{1})},\qquad
c_{n}:=\frac{\int_{K_{v}}h_{n}(k)h_{4,v}(k\mathbf{d}_{v},s_{2})dk}{h_{n}(I_{2})h_{4,v}(I_{2},s_{2})}.
\end{align*}
The formula \eqref{8.18} for the integral over $K_{v}$ is independent of $\lambda$. Thus
\begin{equation}\label{9.4}
\boldsymbol{e}_{v}^{*}(\lambda,\mu_{v};\mathbf{s})=\alpha_{\mathfrak{n},v}q_{v}^{l_{v}(s_{1}+s_{2})}\sum_{n=0}^{l_{v}}d_{n}c_{n}J_{n},
\end{equation}
where $J_{n}$ is the following integral divided by the same integral with $h_{n}$ replaced by $h_{0}$:
\begin{align*}
\int_{F_{v}^{\times}}W_{h_{n},\lambda}(a(y))\overline{W_{4,v}^{*}(a(y),\overline{s_{4}})}|y|_{v}^{s_{2}-1/2}d^{\times}y.
\end{align*}
The spherical scalar factors and the gamma factor from the local functional equation cancel against the corresponding factors in the calculation without $\mathbf{d}_{v}$. All identities are continued meromorphically before taking absolute values. Lemma~\ref{lem8.3}, applied also after an unramified twist, and $|h_{n}(I_{2})|\ll q_{v}^{n/2}$ give
\begin{align*}
|d_{n}|\ll(l_{v}+1)^{O(1)}q_{v}^{(n-l_{v})/2+O(\varepsilon l_{v})},\qquad |c_{n}|\ll(l_{v}+1)^{O(1)}q_{v}^{-l_{v}/2+O(\varepsilon l_{v})}.
\end{align*}

Put $A=\mu_{v}(\varpi_{v})q_{v}^{-\lambda}$. For $n\geq2$, take $h_{n}=C_{n}\cdot (\mathbf{1}_{K_{0,v}[n]}-q_{v}^{-1}\mathbf{1}_{K_{0,v}[n-1]})$, where $C_{n}^{2}=(q_{v}+1)q_{v}^{n}/(q_{v}-1)$. Direct evaluation of the Jacquet integral gives
\begin{align*}
\frac{\sum_{k\geq0}W_{h_{n},\lambda}(a(\varpi_{v}^{k}))T^{k}}{\sum_{k\geq0}W_{h_{0},\lambda}(a(\varpi_{v}^{k}))T^{k}}=\frac{C_{n}(Aq_{v}^{-1/2}T)^{n-2}(Aq_{v}^{1/2}T-1)(Aq_{v}^{-1/2}T-A^{2}/q_{v})}{q_{v}(1-A^{2}/q_{v})}.
\end{align*}
Since $|A|=q_{v}^{-1/2+O(\varepsilon)}$, this polynomial has at most three coefficients, each $\ll q_{v}^{-n/2+O(\varepsilon n)}$; $n=0,1$ satisfy the same bound. After division by the spherical Euler factors, shifting the spherical Rankin--Selberg series by $j$ costs at most $(j+1)q_{v}^{O(\varepsilon j)}$, since the other spherical Satake parameters have modulus $q_{v}^{O(\varepsilon)}$. Hence $|J_{n}|\ll(n+1)^{O(1)}q_{v}^{-n/2+O(\varepsilon n)}$, uniformly in the unramified unitary character $\omega_{v}$.

Substituting these bounds and $|\alpha_{\mathfrak{n},v}q_{v}^{l_{v}(s_{1}+s_{2})}|\ll q_{v}^{l_{v}/2+O(\varepsilon l_{v})}$ into \eqref{9.4} gives $|\boldsymbol{e}_{v}^{*}|\ll(l_{v}+1)^{O(1)}q_{v}^{-l_{v}/2+O(\varepsilon l_{v})}$. Absorbing the polynomial in $l_{v}$ and rescaling $\varepsilon$ proves \eqref{9.5}.
\end{proof}

\begin{lemma}\label{lem9.3}
 Let $v\mid\mathfrak{q}$. Then
\begin{equation}\label{9.6}
|\boldsymbol{e}_{v}^{*}(\lambda,\mu_{v};\mathbf{s})|\ll_{F_{v},\varepsilon} q_{v}^{(1/2+100\varepsilon)e_{v}(\mathfrak{q})}\cdot q_{v}^{(e_{v}(\mathfrak{q})-e_{v}(\mathfrak{q}'))|\Re(\lambda)|}\mathbf{1}_{\mu_{v}\ \mathrm{unramified}},\ \ \ \mathbf{s}\in \mathbf{S}_{\varepsilon}^{4}.
\end{equation}
\end{lemma}
\begin{proof}
Since $W_{1,v}\overline{h_{3,v}}$ is right-$K_{v}$-invariant, only the $K_{v}$-fixed vector contributes to \eqref{c9.2}, and the local factor vanishes unless $\mu_{v}$ is unramified. The Euler factors of the first period cancel, and we apply the diagonal calculation in the proof of Lemma~\ref{lem8.6} to the second period. Put $m=e_{v}(\mathfrak{q})$ and $r=e_{v}(\mathfrak{q}')$. Sum the geometric series in their convergence region before meromorphic continuation. The lower limit is $k=m-r$ for $r>0$; for $r=0$, include the term $k=m-1$. With Satake parameters $\mu_{v}(\varpi_{v})q_{v}^{-\lambda}$ and $\overline{\mu}_{v}(\varpi_{v})q_{v}^{\lambda}$, division by the Euler factors gives
\begin{align*}
|\boldsymbol{e}_{v}^{*}(\lambda,\mu_{v};\mathbf{s})|\ll_{F_{v},\varepsilon}(m+1)q_{v}^{m/2+(m-r)|\Re(\lambda)|+O(\varepsilon m)}.
\end{align*}
At places dividing $\mathfrak{D}_{F}$, changing $\psi_{v}$ to an unramified additive character changes only the $F_{v}$-dependent constant. Absorbing $m+1$ and rescaling $\varepsilon$ proves \eqref{9.6}.
\end{proof}

\begin{lemma}\label{lem9.4}
 Let $v\mid\infty$ and $m\geq0$. For the pairs $(\lambda,\mu)$ specified above, and also for $(\lambda,\mu_{v})=(1/2+s_{1}+s_{3},\mathbf{1})$, we have
\begin{equation}\label{eq9.11}
|\mathbf{L}_{v}\boldsymbol{e}_{v}^{*}(\lambda,\mu_{v};\mathbf{s})|\ll_{F_{v},\varepsilon,m} \mathbf{C}_{v}^{1000\varepsilon}C_{v}(\omega)^{-1/2+500\varepsilon}\big[\mathbf{C}_{v}^{\varepsilon}C_{v}(\mu|\cdot|^{\lambda})^{-1}\big]^{m},\ \  \mathbf{s}\in \mathbf{S}_{\varepsilon}^{4},
\end{equation}
where $\mathbf{L}_{v}:=L_{v}(1+2\lambda,\mu_{v}^{2})^{-1}L_{v}(1/2+s_{4},(\mu_{v}|\cdot|^{\lambda}\boxplus \overline{\mu}_{v}|\cdot|^{-\lambda})\otimes (|\cdot|^{s_{2}}\boxplus \omega_{v}|\cdot|^{-s_{2}}))$. 
\end{lemma}
\begin{proof}
Because $W_{1,v}(\cdot,s_{1})\overline{h_{3,v}(\cdot,\overline{s_{3}})}$ is spherical, the local factor vanishes unless $\mu_{v}\boxplus\overline{\mu}_{v}$ contains a spherical vector. In the nonzero case only its normalized spherical vector $h_{v}^{\circ}$ occurs, and hence
\begin{align*}
\mathbf{L}_{v}\boldsymbol{e}_{v}^{*}(\lambda,\mu_{v};\mathbf{s})=\widetilde{\Psi}_{v}(W_{h_{v},\lambda}^{\circ,*},W_{2,v}(\cdot,s_{2}),\overline{h_{4,v}(\cdot,\overline{s_{4}})}).
\end{align*}

As in the proof of Lemma~\ref{lem8.7}, the support of $h_{4,v}(\cdot,\overline{s_{4}})$ restricts the $K_{v}$-integration to $k_{v}=(k_{ij})$ with $|k_{21}|_{v}\leq\mathbf{C}_{v}^{-1}$ and gives
\begin{align*}
\int_{K_{v}}|h_{4,v}(k_{v},\overline{s_{4}})|dk_{v}\ll_{F_{v},\varepsilon} \mathbf{C}_{v}^{-1/2+10\varepsilon}.
\end{align*}

Define
\begin{equation}\label{fc9.8}
\mathcal{M}W_{2,v}(s;k_{v}):=2^{-1}\int_{F_{v}^{\times}}W_{2,v}(a(y_{v}')k_{v},s_{2})|y_{v}'|_{v}^{s}d^{\times}y_{v}'.
\end{equation}
Integration over signs or angles gives $2\,dr/r$ with $r=|y_{v}'|_{v}$, which accounts for $2^{-1}$ at both real and complex places. By \cite[Lemma 3.7.2]{MV10}, the integral converges absolutely in $\Re(s)>-1/2+|\Re(s_{2})|$.

Insert \eqref{c2.5} and \eqref{eq6.28} into \eqref{fc9.8} and make the substitutions $b_{v}\mapsto b_{v}t_{v}^{-1}$ and $y_{v}'\mapsto y_{v}'t_{v}^{-1}$. This gives
\begin{equation}\label{equa9.9}
\mathcal{M}W_{2,v}(s;k_{v})=2^{-1}\mathbf{C}_{v}^{\frac{1}{2}}\int |y_{v}'|_{v}^{s+s_{2}+\frac{1}{2}}\int G(y_{v}',t_{v};k_{v})\overline{\omega}_{v}(t_{v})|t_{v}|_{v}^{-s+s_{2}-\frac{1}{2}}d^{\times}t_{v}d^{\times}y_{v}',
\end{equation}
where $t_{v},y_{v}'\in F_{v}^{\times}$, $\Re(s)>-1/2+|\Re(s_{2})|$, and $G(y_{v}',t_{v};k_{v})$ is defined by 
\begin{align*}
\int_{F_{v}}\alpha_{v}(\mathbf{C}_{v}|k_{11}y_{v}'+k_{21}b_{v}|_{v})\alpha_{v}(|k_{12}y_{v}'+k_{22}b_{v}|_{v}-1)
\overline{\psi}_{v}(b_{v}t_{v}^{-1})\omega_{v}(k_{12}y_{v}'+k_{22}b_{v})db_{v}.
\end{align*}

The support of $\alpha_{v}$ gives $G(y_{v}',t_{v};k_{v})=0$ unless $|y_{v}'|_{v}\leq10\mathbf{C}_{v}^{-1}$. Put $\delta_{v}=[F_{v}:\mathbb{R}]$. Writing $y_{v}'=ru$ with $r>0$ and $|u|=1$, define
\begin{align*}
H_{v}(r,t_{v};k_{v}):=\int_{|u|=1}G(ru,t_{v};k_{v})du,
\end{align*}
where $du$ is the angular measure in the polar decomposition of $d^{\times}y_{v}'$. Smoothness and angular symmetry give $H_{v}(r,t_{v};k_{v})=H_{v}(0,t_{v};k_{v})+O(r^{2})$ as $r\rightarrow0$. Initially in the above region of absolute convergence, integration by parts in $r$ gives
\begin{multline}\label{c9.9}
\mathcal{M}W_{2,v}(s;k_{v})=-\frac{2^{-1}\mathbf{C}_{v}^{1/2}}{\delta_{v}(s+s_{2}+1/2)}\int_{0}^{\infty} r^{\delta_{v}(s+s_{2}+1/2)+1}\\
\int_{F_{v}^{\times}}\overline{\omega}_{v}(t_{v})|t_{v}|_{v}^{-s+s_{2}-1/2}\frac{\partial}{\partial r}H_{v}(r,t_{v};k_{v})d^{\times} t_{v}\frac{dr}{r}.
\end{multline}
After meromorphically continuing the inner integral as described below, the right-hand side continues $\mathcal{M}W_{2,v}$ to $\Re(s)>-3/2-\Re(s_{2})$; the factor in front shows the possible pole at $s=-1/2-s_{2}$. Other poles may come from the inner integral over $t_{v}$.

Let $m_{1}, m_{2}\geq 1$. Following the proof of \cite[Lemma 3.1.14]{MV10} we have
\begin{multline}\label{cf9.10}
G(y_{v}',t_{v};k_{v})\ll_{F_{v},\varepsilon,m_{1},m_{2}} C_{v}(\omega)^{-1/2+\varepsilon}\cdot \mathbf{1}_{
|y_{v}'|_{v}\leq 10\mathbf{C}_{v}^{-1}}\\
\cdot \min\big\{C_{v}(\omega)^{m_{1}}|t_{v}|_{v}^{m_{1}},C_{v}(\omega)^{-m_{2}}(1+|t_{v}|_{v}^{-1})^{m_{2}}\big\}.
\end{multline}

After adjusting $m_{1}$ and $m_{2}$, $\partial_{r}H_{v}$ satisfies the same estimate with the additional factor $\min\{\mathbf{C}_{v}^{1/\delta_{v}},\mathbf{C}_{v}^{2/\delta_{v}}r\}$. The second term in this minimum follows from $\partial_{r}H_{v}(0,t_{v};k_{v})=0$.

To continue the inner $t_{v}$-integral, split at $|t_{v}|_{v}=1$ and subtract a sufficiently long Taylor expansion of $\overline{\psi}_{v}(b_{v}t_{v}^{-1})$ in $t_{v}^{-1}$ in the range $|t_{v}|_{v}\geq1$. The remainder is integrated absolutely; the subtracted terms are evaluated as meromorphic Tate integrals. In any fixed vertical strip they have only finitely many poles and satisfy the bounds below away from their $\varepsilon$-neighborhoods.

Substituting \eqref{cf9.10} into \eqref{equa9.9} and \eqref{c9.9} gives, in any fixed vertical strip,
\begin{equation}\label{equ9.13}
\mathcal{M}W_{2,v}(s;k_{v})\ll_{F_{v},\varepsilon} \mathbf{C}_{v}^{-\Re(s)-\Re(s_{2})}C_{v}(\omega)^{\Re(s)-\Re(s_{2})+\varepsilon}
\end{equation}
if $\Re(s)\geq -3/2-\Re(s_{2})+\varepsilon$ and $s$ remains a distance $\gg\varepsilon$ from the poles of $\mathcal{M}W_{2,v}$. The contour chosen in \textsection\ref{sec4.regularization} keeps the arguments used below away from these poles.

Take $\beta=-1/2+40\varepsilon$. For large $\Re(s_{4})$, Mellin inversion applies using the spherical Whittaker formula
\begin{align*}
\int_{F_{v}^{\times}}W_{h_{v},\lambda}^{\circ,*}(a(y_{v}))|y_{v}|_{v}^{-1/2+s_{4}-s}d^{\times}y_{v}=\frac{L_{v}(s_{4}-s+\lambda,\mu_{v})L_{v}(s_{4}-s-\lambda,\overline{\mu}_{v})}{L_{v}(1+2\lambda,\mu_{v}^{2})}. 
\end{align*}

Write $\mu_{v}=|\cdot|_{v}^{\gamma_{v}}$, with $\gamma_{v}\in i\mathbb{R}$. Writing $W_{\mu_{v},\lambda}^{\circ}$ for the Jacquet integral of the unit spherical section, its functional equation is
\begin{align*}
W_{\mu_{v},\lambda}^{\circ}=\frac{L_{v}(1-2\lambda,\overline{\mu}_{v}^{2})}{L_{v}(1+2\lambda,\mu_{v}^{2})}W_{\overline{\mu}_{v},-\lambda}^{\circ}.
\end{align*}
For $|\lambda+1/2|\leq20\varepsilon$, Stirling's formula on the chosen contour bounds this multiplier by a fixed polynomial in $C_{v}(\mu)$, absorbed by increasing $m$. It therefore suffices to treat $|\lambda-1/2|\leq20\varepsilon$. The pole $s=s_{4}-\lambda-\gamma_{v}$ lies to the left of $(\beta)$, while every other pole of the two $L_{v}$-factors lies to its right. Continuing from large $\Re(s_{4})$ therefore expresses $\mathbf{L}_{v}\boldsymbol{e}_{v}^{*}$ as
\begin{multline*}
\frac{1}{2\pi i}\int_{K_{v}}\frac{\overline{h_{4,v}(k_{v},\overline{s_{4}})}}{L_{v}(1+2\lambda,\mu_{v}^{2})}\int_{(\beta)}\mathcal{M}W_{2,v}(s;k_{v})L_{v}(s_{4}-s+\lambda,\mu_{v})L_{v}(s_{4}-s-\lambda,\overline{\mu}_{v})dsdk_{v}\\
+\frac{L_{v}(2\lambda,\mu_{v}^{2})}{L_{v}(1+2\lambda,\mu_{v}^{2})}\underset{s'=\gamma_{v}}{\Res}\ L_{v}(s',\overline{\mu}_{v})\int_{K_{v}}\overline{h_{4,v}(k_{v},\overline{s_{4}})}\mathcal{M}W_{2,v}(s_{4}-\lambda-\gamma_{v};k_{v})dk_{v}.
\end{multline*}

The positive sign follows because the residue of $L_{v}(s_{4}-s-\lambda,\overline{\mu}_{v})$ in $s$ is the negative of its residue in its own argument. The choice in \textsection\ref{sec4.regularization} keeps $s_{4}-\lambda-\gamma_{v}$ away from the remaining poles of $\mathcal{M}W_{2,v}$. After rescaling $y_{v}'$ by $\mathbf{C}_{v}^{-1/\delta_{v}}$, every fixed radial derivative satisfies the same bounds as \eqref{cf9.10}: differentiating the character costs at most $C_{v}(\omega)^{1/\delta_{v}}\mathbf{C}_{v}^{-2/\delta_{v}}\ll1$. Repeated integration by parts in the outer integral gives arbitrary decay in $\Im(s)$ on the continued strip: $r\partial_{r}$ removes $H_{v}(0,t_{v};k_{v})$ and preserves the $O(r^{2})$ remainder. Differentiate the Taylor terms and the convergent remainder separately; the same decay holds after continuation. Shift the remaining $s$-contour sufficiently far to the right, to a fixed integer real part. The residues of the Gamma factors involve values and derivatives of $\mathcal{M}W_{2,v}$ at imaginary parts $\pm\Im(\gamma_{v})+O(\varepsilon)$ and are bounded by this decay. On the new line use Stirling's formula for $|\Im(s)|\leq|\gamma_{v}|/2$ and the decay in $\Im(s)$ otherwise; the factor $(C_{v}(\omega)/\mathbf{C}_{v})^{\Re(s)}$ is bounded by a fixed local constant. In the residue already displayed, the factor $\mathbf{C}_{v}^{1/2}$ cancels the $L^{1}(K_{v})$-bound for $h_{4,v}$. The number of derivatives depends only on $m$ and $\varepsilon$. Together with \eqref{equ9.13}, this bounds the expression by
\begin{align*}
\frac{\big[\mathbf{C}_{v}^{\varepsilon}C_{v}(\mu|\cdot|^{\lambda})^{-1}\big]^{m}\mathbf{C}_{v}^{1/2+500\varepsilon}}{C_{v}(\omega)^{1/2-500\varepsilon}}\int_{K_{v}}\big|\overline{h_{4,v}(k_{v},\overline{s_{4}})}\big|dk_{v}\ll \frac{\mathbf{C}_{v}^{1000\varepsilon}\big[\mathbf{C}_{v}^{\varepsilon}C_{v}(\mu|\cdot|^{\lambda})^{-1}\big]^{m}}{C_{v}(\omega)^{1/2-500\varepsilon}}.
\end{align*} 

Therefore, \eqref{eq9.11} follows. 
\end{proof}

\subsection{Proof of Proposition \ref{prop9.1}}
Let $m\geq 10$, to be chosen below. Combining Lemmas \ref{lem9.2}, \ref{lem9.3} and \ref{lem9.4}, the preceding estimate at $v\mid\mathfrak{D}_{F}$, and \eqref{c9.2}, we obtain, uniformly for the eight pairs above and $\mathbf{s}\in\mathbf{S}_{\varepsilon}^{4}$,
\begin{equation}\label{9.16}
|\boldsymbol{e}^{*}(\lambda,\mu;\mathbf{s})|\ll_{F,\varepsilon,m} \frac{N_{F}(\mathfrak{n})^{-1/2+\varepsilon}N_{F}(\mathfrak{q})^{1+\varepsilon}N_{F}(\mathfrak{q}')^{-1/2+\varepsilon}\mathbf{C}_{\infty}^{(m+10)\varepsilon}}{\prod_{v\mid\infty}|\mathbf{L}_{v}|\cdot C_{\infty}(\omega)^{1/2}C_{\infty}(\mu|\cdot|^{\lambda})^{m}}\cdot\mathbf{1}_{C_{\fin}(\mu)=1},
\end{equation}
with $\mathbf{L}_{v}$ as in Lemma \ref{lem9.4}.

Let $\mathbf{s}\in\mathbf{S}_{\varepsilon}^{4}$ and let $(\lambda_{0},\mu)$ be one of the eight pairs above. In view of the indicator in \eqref{9.16}, it remains only to consider $C_{\fin}(\mu)=1$. When $\mu=\mathbf{1}$, the residue of $\mathbf{L}^{*}$, after division by the Archimedean factors in \eqref{9.16}, contains at most one finite $L$-factor involving $\omega$ or $\overline{\omega}$ with argument of real part $O(\varepsilon)$; the accompanying factor at $1+O(\varepsilon)$ is $\ll C(\omega)^{O(\varepsilon)}$. When $\mu\in\{\omega,\overline{\omega}\}$, the remaining dependence is polynomial in the Archimedean conductor of $\mu$. Only a fixed number of degree-one factors occurs, with fixed Gamma shifts at each place. Thus Stirling's formula, the functional equations, and convexity give an absolute exponent $A\geq0$ such that, after rescaling $\varepsilon$ and evaluating the $\mathbf{L}_{v}$ at $\lambda_{0}$,
\begin{equation}\label{9.17}
\prod_{v\mid\infty}|\mathbf{L}_{v}|^{-1}\big|\underset{\lambda=\lambda_{0}}{\Res}\ \mathbf{L}^{*}(\lambda,\mu;\mathbf{s})\big|\ll_{F,\varepsilon} C(\omega)^{\varepsilon}\mathbf{C}_{\infty}^{\varepsilon}C_{\infty}(\mu|\cdot|^{\lambda_{0}})^{A}\Big(1+\max_{|s|\leq 20\varepsilon}|L(s,\overline{\omega})|\Big).
\end{equation}

Substituting \eqref{9.16} and \eqref{9.17} into \eqref{c9.1}, and using the 
convexity bound for $L(s,\overline{\omega})$ (taking $m\geq A+10$, independently of $[F:\mathbb{Q}]$), we obtain, uniformly for $\mathbf{s}\in\mathbf{S}_{\varepsilon}^{4}$,
\begin{align*}
\big|\widetilde{\Psi}_{\mathrm{Dual}}^{(i)}(\mathbf{s},\mathfrak{X}_{\mathfrak{n}})\big|\ll_{F,\varepsilon} \frac{C(\omega)^{3\varepsilon}N_{F}(\mathfrak{q})^{1+\varepsilon}\mathbf{C}_{\infty}^{\varepsilon}}{N_{F}(\mathfrak{n})^{1/2-\varepsilon}}.
\end{align*}

Taking the contour integral of \textsection\ref{sec4.regularization} and rescaling $\varepsilon$ proves Proposition \ref{prop9.1}.

\section{An Upper Bound for $\widetilde{\Psi}_{\mathrm{Geo}}^{(i)}(\mathbf{0}\mid\mathfrak{X}_{\mathfrak{n}})$}\label{sec9}
We estimate the eight geometric terms $\widetilde{\Psi}_{\mathrm{Geo}}^{(i)}(\mathbf{0}\mid\mathfrak{X}_{\mathfrak{n}})$ in Theorem~\ref{thmD}, treating them in four pairs.

\subsection{Bounds of $\widetilde{\Psi}_{\mathrm{Geo}}^{(1)}(\mathbf{0}\mid\mathfrak{X}_{\mathfrak{n}})$ and $\widetilde{\Psi}_{\mathrm{Geo}}^{(2)}(\mathbf{0}\mid\mathfrak{X}_{\mathfrak{n}})$}
By Definition \ref{defn2.5},
\begin{align*}
&\widetilde{\Psi}_{\mathrm{Geo}}^{(1)}(\mathbf{s},\mathfrak{X}_{\mathfrak{n}}):=\widetilde{\Psi}(W_{1}(\cdot,s_{1}),\overline{W_{3}(\cdot,\overline{s_{3}})},h_{2}(\cdot,s_{2})\overline{h_{4}(\cdot,\overline{s_{4}})}),\\
&\widetilde{\Psi}_{\mathrm{Geo}}^{(2)}(\mathbf{s},\mathfrak{X}_{\mathfrak{n}}):=\widetilde{\Psi}(W_{1}(\cdot,s_{1}),\overline{W_{3}(\cdot,\overline{s_{3}})},h_{2}^{\Diamond}(\cdot,s_{2})\overline{h_{4}(\cdot,\overline{s_{4}})}).
\end{align*}

\begin{lemma}\label{lem9.1}
 Let $i\in\{1,2\}$. Then 
\begin{equation}\label{e9.1}
\big|\widetilde{\Psi}_{\mathrm{Geo}}^{(i)}(\mathbf{0}\mid\mathfrak{X}_{\mathfrak{n}})\big|\ll_{F,\varepsilon} N_{F}(\mathfrak{n})^{-1/2+\varepsilon}N_{F}(\mathfrak{q})^{1+\varepsilon}C(\omega)^{\varepsilon}\mathbf{C}_{\infty}^{\varepsilon}.
\end{equation}
\end{lemma}
\begin{proof}
Suppose $i=1$ first. Notice that $h_{2}(\cdot,s_{2})\overline{h_{4}(\cdot,\overline{s_{4}})}$ is a section in $|\cdot|^{1/2+s_{2}+s_{4}}\boxplus |\cdot|^{-1/2-s_{2}-s_{4}}$. The unramified Rankin--Selberg calculation gives
\begin{align*}
\widetilde{\Psi}_{\mathrm{Geo}}^{(1)}(\mathbf{s},\mathfrak{X}_{\mathfrak{n}})=\frac{\Lambda(1+s_{2}+s_{4},(|\cdot|^{s_{1}}\boxplus |\cdot|^{-s_{1}})\times(|\cdot|^{s_{3}}\boxplus |\cdot|^{-s_{3}}))}{\Lambda(2+2s_{2}+2s_{4},\mathbf{1})\Lambda(1+2s_{2},\overline{\omega})^{-1}\Lambda(1+2s_{4},\omega)^{-1}}\prod_{v\leq\infty}\mathbf{e}_{1,v}(\mathbf{s},\mathfrak{X}_{\mathfrak{n}}),
\end{align*}
where $\mathbf{e}_{1,v}(\mathbf{s},\mathfrak{X}_{\mathfrak{n}})$, for $\mathbf{s}\in \mathbf{B}_{\varepsilon}^{4}$, is defined by 
\begin{align*}
\frac{L_{v}(2+2s_{2}+2s_{4},\mathbf{1})\Psi_{v}(W_{1,v}(\cdot,s_{1}),\overline{W_{3,v}(\cdot,\overline{s_{3}})},h_{2,v}(\cdot,s_{2})\overline{h_{4,v}(\cdot,\overline{s_{4}})})}{L_{v}(1+s_{2}+s_{4},(|\cdot|_{v}^{s_{1}}\boxplus |\cdot|_{v}^{-s_{1}})\times(|\cdot|_{v}^{s_{3}}\boxplus |\cdot|_{v}^{-s_{3}}))L_{v}(1+2s_{2},\overline{\omega}_{v})L_{v}(1+2s_{4},\omega_{v})}.
\end{align*}

Direct calculation gives $\mathbf{e}_{1,v}(\mathbf{s},\mathfrak{X}_{\mathfrak{n}})\equiv 1$ for all $v\nmid \mathfrak{n}\mathfrak{q}\mathfrak{D}_{F}\infty$. We estimate the remaining local factors as follows.
\begin{itemize}
\item Suppose $v\mid \mathfrak{D}_{F}$ and $v\nmid\mathfrak{q}$. Similarly to \eqref{eq7.15} we have 
\begin{equation}\label{e9.2}
|\mathbf{e}_{1,v}(\mathbf{s},\mathfrak{X}_{\mathfrak{n}})|\ll q_{v}^{O(d_{v})},\ \ \mathbf{s}\in \mathbf{S}_{\varepsilon}^{4},
\end{equation}
where $d_{v}=e_{v}(\mathfrak{D}_{F})$ is the different exponent. Places dividing both $\mathfrak{D}_{F}$ and $\mathfrak{q}$ are covered by the next case, with their fixed measure factors absorbed into the $F$-dependent implied constant.

\item Suppose $v\mid\mathfrak{q}$. Let $\mathbf{s}\in \mathbf{S}_{\varepsilon}^{4}$. By the construction of $W_{j,v}(\cdot,s_{j})$ and $h_{j,v}(\cdot,s_{j})$ (see \textsection\ref{sec4.4}), 
\begin{multline}\label{ee9.3}
\Psi_{v}(W_{1,v}(\cdot,s_{1}),\overline{W_{3,v}(\cdot,\overline{s_{3}})},h_{2,v}(\cdot,s_{2})\overline{h_{4,v}(\cdot,\overline{s_{4}})})\ll \Big|\int_{F_{v}^{\times}}W_{1,v}(a(y_{v}),s_{1})\\
\overline{W_{3,v}(a(y_{v}),\overline{s_{3}})}|y_{v}|_{v}^{s_{2}+s_{4}}d^{\times}y_{v}\int_{K_{v}}h_{2,v}(k_{v},s_{2})\overline{h_{4,v}(k_{v},\overline{s_{4}})}dk_{v}\Big|,
\end{multline}
which is $\ll q_{v}^{e_{v}(\mathfrak{q})}$. Here we have used the fact that 
\begin{equation}\label{fc9.4}
\overline{h_{4,v}(k_{v},\overline{s_{4}})}=\overline{\eta_{\omega,v}(k_{v})}\,\overline{h_{4,v}(I_{2},\overline{s_{4}})}\mathbf{1}_{k_{v}\in K_{0,v}[e_{v}(\mathfrak{q})]}.	
\end{equation}
Hence
\begin{equation}\label{e9.3}
|\mathbf{e}_{1,v}(\mathbf{s},\mathfrak{X}_{\mathfrak{n}})|\ll_{F_{v},\varepsilon} q_{v}^{e_{v}(\mathfrak{q})},\ \ \mathbf{s}\in \mathbf{S}_{\varepsilon}^{4}.
\end{equation}

\item Suppose $v\mid\infty$. The same argument as in \eqref{ee9.3} gives
\begin{equation}\label{e9.5}
|\Psi_{v}(W_{1,v}(\cdot,s_{1}),\overline{W_{3,v}(\cdot,\overline{s_{3}})},h_{2,v}(\cdot,s_{2})\overline{h_{4,v}(\cdot,\overline{s_{4}})})|\ll_{F_{v},\varepsilon} 1,\ \ \mathbf{s}\in \mathbf{S}_{\varepsilon}^{4}.	
\end{equation} 

\item Suppose $v\mid\mathfrak{n}$. Let $\mathbf{s}\in\mathbf{S}_{\varepsilon}^{4}$. By the construction of $W_{j,v}(\cdot,s_{j})$ and $h_{j,v}(\cdot,s_{j})$, 
\begin{multline}\label{equa9.6}
\Psi_{v}(W_{1,v}(\cdot,s_{1}),\overline{W_{3,v}(\cdot,\overline{s_{3}})},h_{2,v}(\cdot,s_{2})\overline{h_{4,v}(\cdot,\overline{s_{4}})})\ll q_{v}^{\frac{l_{v}}{2}+3\varepsilon l_{v}}\Big|\int_{F_{v}^{\times}}\overline{W_{3,v}(a(y_{v}),\overline{s_{3}})}\\
|y_{v}|_{v}^{s_{2}+s_{4}}d^{\times}y_{v}\int_{K_{v}}W_{3,v}(a(y_{v})k_{v}\mathbf{d}_{v},s_{1})h_{4,v}(k_{v}\mathbf{d}_{v},s_{2})dk_{v}\overline{h_{4,v}(I_{2},\overline{s_{4}})}\Big|.
\end{multline}

There are three cases:
\begin{itemize}
\item Suppose $k_{v}\in K_{v}-K_{0,v}[1]$. Then $k_{v}\in \begin{pmatrix}
1& \alpha\\
& 1
\end{pmatrix}wK_{0,v}[l_{v}]$ for some $\alpha\in \mathcal{O}_{v}/\mathfrak{p}_{v}^{l_{v}}$. Similar to \eqref{8.10}, we have 
\begin{equation}\label{ee9.6}
W_{3,v}(a(y_{v})k_{v}\mathbf{d}_{v},s_{1})h_{4,v}(k_{v}\mathbf{d}_{v},s_{2})\ll \big|W_{3,v}(a(y_{v}\varpi_{v}^{l_{v}}),s_{1})h_{4,v}(a(\varpi_{v}^{l_{v}}),s_{2})\big|.
\end{equation}

\item Suppose $k_{v}\in K_{0,v}[l_{v}]$. Similar to \eqref{8.11}, we have 
\begin{equation}\label{ee9.7}
W_{3,v}(a(y_{v})k_{v}\mathbf{d}_{v},s_{1})h_{4,v}(k_{v}\mathbf{d}_{v},s_{2})\ll \big|W_{3,v}(a(y_{v})\mathbf{d}_{v},s_{1})h_{4,v}(\mathbf{d}_{v},s_{2})\big|.
\end{equation}

\item Suppose $k_{v}\in K_{0,v}[j]-K_{0,v}[j+1]$ for some $1\leq j<l_{v}$. Write $k_{v}=\begin{psmallmatrix}1&0\\\beta&1\end{psmallmatrix}k_{0}$, where $k_{0}\in K_{0,v}[l_{v}]$ and $e_{v}(\beta)=j$. The Iwasawa decomposition in \eqref{8.12} gives
\begin{equation}\label{ee9.8}
W_{3,v}(a(y_{v})k_{v}\mathbf{d}_{v},s_{1})h_{4,v}(k_{v}\mathbf{d}_{v},s_{2})=\psi_{v}(y_{v}/\beta)W_{3,v}(a(y_{v})\mathbf{d}_{v}^{(j)},s_{1})h_{4,v}(\mathbf{d}_{v}^{(j)},s_{2}),
\end{equation}
where $\mathbf{d}_{v}^{(j)}:=\diag(\varpi_{v}^{l_{v}-j},\varpi_{v}^{j})$.
\end{itemize}

For $1\leq j<l_{v}$, retain the character in \eqref{ee9.8} and put $m=e_{v}(y_{v})\geq0$. On the support $m+l_{v}-2j\geq0$, inversion of units and additive orthogonality give
\begin{align*}
\frac{1}{\#\{\beta\bmod\mathfrak{p}_{v}^{l_{v}}:e_{v}(\beta)=j\}}\sum_{\substack{\beta\bmod\mathfrak{p}_{v}^{l_{v}}\\e_{v}(\beta)=j}}\psi_{v}(y_{v}/\beta)=\begin{cases}0,&m<j-1,\\-1/(q_{v}-1),&m=j-1,\\1,&m\geq j.\end{cases}
\end{align*}
The set $K_{0,v}[j]-K_{0,v}[j+1]$ has measure $(q_{v}-1)q_{v}^{-j}/(q_{v}+1)$. Apart from the character, the integrand in \eqref{equa9.6} is $\ll(1+m+l_{v})^{2}q_{v}^{-(m+l_{v}-2j)+O(\varepsilon(m+l_{v}))}$. The sum over $m\geq j$ is therefore $\ll q_{v}^{-l_{v}+O(\varepsilon l_{v})}$, and the term $m=j-1$ has the same bound because of the factor $(q_{v}-1)^{-1}$. The cases \eqref{ee9.6} and \eqref{ee9.7} give this bound directly. Summing over $j$ and multiplying by $q_{v}^{l_{v}/2+O(\varepsilon l_{v})}$ yields
\begin{equation}\label{e9.6}
|\mathbf{e}_{1,v}(\mathbf{s},\mathfrak{X}_{\mathfrak{n}})|\ll_{F_{v},\varepsilon} q_{v}^{-l_{v}(1/2-\varepsilon)},\ \ \mathbf{s}\in \mathbf{S}_{\varepsilon}^{4}.	
\end{equation} 
\end{itemize}

On the chosen contour $\mathbf{S}_{\varepsilon}^{4}$, the global $L$-factors satisfy
\begin{align*}
\Big|\frac{\Lambda(1+s_{2}+s_{4},(|\cdot|^{s_{1}}\boxplus |\cdot|^{-s_{1}})\times(|\cdot|^{s_{3}}\boxplus |\cdot|^{-s_{3}}))}{L(2+2s_{2}+2s_{4},\mathbf{1})L(1+2s_{2},\overline{\omega})^{-1}L(1+2s_{4},\omega)^{-1}}\Big|\ll_{F,\varepsilon}N_{F}(\mathfrak{q})^{\varepsilon}C(\omega)^{\varepsilon}.
\end{align*}
Combining this with \eqref{e9.2}, \eqref{e9.3}, \eqref{e9.5}, and \eqref{e9.6}, and taking the contour integral of \textsection\ref{sec4.regularization}, proves \eqref{e9.1} for $i=1$.

For $i=2$, the product section belongs to $\omega|\cdot|^{1/2-s_{2}+s_{4}}\boxplus\overline{\omega}|\cdot|^{-1/2+s_{2}-s_{4}}$. Use the local intertwining operators normalized to preserve Whittaker functions \cite[Proposition 3.1.4]{Sha81}. Their normalizing scalars are the Tate factors $\gamma_{v}(2s_{2},\overline{\omega}_{v},\psi_{v})$, up to unit-modulus Weyl factors; their global product is one by the functional equation. At an unramified place the spherical multiplier is $L_{v}(1-2s_{2},\omega_{v})/L_{v}(1+2s_{2},\overline{\omega}_{v})$, so the global Euler factors involve $L(1-2s_{2},\omega)$ and the four twists $L(1-s_{2}+s_{4}\pm s_{1}\pm s_{3},\omega)$. Their finite parts are $\ll C(\omega)^{O(\varepsilon)}$ on the chosen contour. The rank-one multiplier formulas give norm $\ll q_{v}^{O(\varepsilon m_{v})}$ on the level-$m_{v}$ subspace at $v\mid\mathfrak{q}$; Cauchy--Schwarz bounds the integral over the compact subgroup by $q_{v}^{m_{v}+O(\varepsilon m_{v})}$. At $v\mid\mathfrak{n}$, replace $s_{2}$ by $-s_{2}$ in the preceding calculation. The additional unramified character factors depend only on valuations, so the additive averages are unchanged. At $v\mid\infty$, the same multiplier formulas have order $O(\varepsilon)$; the bump in \eqref{eq6.28} therefore costs at most $(\mathbf{C}_{v} C_{v}(\omega))^{O(\varepsilon)}$. These estimates prove \eqref{e9.1} for $i=2$.
\end{proof}

\subsection{Bounds of $\widetilde{\Psi}_{\mathrm{Geo}}^{(3)}(\mathbf{0}\mid\mathfrak{X}_{\mathfrak{n}})$ and $\widetilde{\Psi}_{\mathrm{Geo}}^{(4)}(\mathbf{0}\mid\mathfrak{X}_{\mathfrak{n}})$}
By Definition \ref{defn2.5},
\begin{align*}
&\widetilde{\Psi}_{\mathrm{Geo}}^{(3)}(\mathbf{s},\mathfrak{X}_{\mathfrak{n}}):=-\widetilde{\Psi}(W_{1}^{*}(\cdot,s_{1}),W_{2}(\cdot,s_{2}),\overline{h_{3}(\cdot,\overline{s_{3}})}\overline{h_{4}(\cdot,\overline{s_{4}})}),\\
&\widetilde{\Psi}_{\mathrm{Geo}}^{(4)}(\mathbf{s},\mathfrak{X}_{\mathfrak{n}}):=-\widetilde{\Psi}(W_{1}^{*}(\cdot,s_{1}),W_{2}(\cdot,s_{2}),\overline{h_{3}^{\Diamond}(\cdot,\overline{s_{3}})}\overline{h_{4}(\cdot,\overline{s_{4}})}).
\end{align*}

\begin{lemma}\label{lemma9.2}
 Let $i\in\{3,4\}$. Then 
\begin{equation}\label{e9.11}
\big|\widetilde{\Psi}_{\mathrm{Geo}}^{(i)}(\mathbf{0}\mid\mathfrak{X}_{\mathfrak{n}})\big|\ll_{F,\varepsilon} N_{F}(\mathfrak{n})^{1/2+\varepsilon}C(\omega)^{1/2+\varepsilon}\mathbf{C}_{\infty}^{-1+\varepsilon}N_{F}(\mathfrak{q})^{\varepsilon}.
\end{equation}
\end{lemma}
\begin{proof}
Suppose $i=3$ first. Notice that $\overline{h_{3}(\cdot,\overline{s_{3}})}\overline{h_{4}(\cdot,\overline{s_{4}})}$ is a section in $|\cdot|^{1/2+s_{3}+s_{4}}\boxplus \overline{\omega}|\cdot|^{-1/2-s_{3}-s_{4}}$. The unramified Rankin--Selberg calculation expresses $\widetilde{\Psi}_{\mathrm{Geo}}^{(3)}(\mathbf{s},\mathfrak{X}_{\mathfrak{n}})$ as
\begin{equation}\label{c9.13}
-\frac{\Lambda(1+s_{3}+s_{4},(|\cdot|^{s_{1}}\boxplus |\cdot|^{-s_{1}})\times(|\cdot|^{s_{2}}\boxplus \omega|\cdot|^{-s_{2}}))}{\Lambda(2+2s_{3}+2s_{4},\omega)\Lambda(1+2s_{3},\mathbf{1})^{-1}\Lambda(1+2s_{4},\omega)^{-1}}\prod_{v\leq\infty}\mathbf{e}_{3,v}(\mathbf{s},\mathfrak{X}_{\mathfrak{n}}),
\end{equation}
where $\mathbf{e}_{3,v}(\mathbf{s},\mathfrak{X}_{\mathfrak{n}})$, for $\mathbf{s}\in \mathbf{B}_{\varepsilon}^{4}$, is defined by 
\begin{align*}
\frac{L_{v}(2+2s_{3}+2s_{4},\omega_{v})\Psi_{v}(W_{1,v}^{*}(\cdot,s_{1}),W_{2,v}(\cdot,s_{2}),\overline{h_{3,v}(\cdot,\overline{s_{3}})}\overline{h_{4,v}(\cdot,\overline{s_{4}})})}{L_{v}(1+s_{3}+s_{4},(|\cdot|_{v}^{s_{1}}\boxplus |\cdot|_{v}^{-s_{1}})\times(|\cdot|_{v}^{s_{2}}\boxplus \omega_{v}|\cdot|_{v}^{-s_{2}}))L_{v}(1+2s_{3},\mathbf{1})L_{v}(1+2s_{4},\omega_{v})}.
\end{align*}

Direct calculation gives $\mathbf{e}_{3,v}(\mathbf{s},\mathfrak{X}_{\mathfrak{n}})\equiv 1$ for all $v\nmid \mathfrak{n}\mathfrak{q}\mathfrak{D}_{F}\infty$. We estimate the remaining local factors as follows.
\begin{itemize}
\item Suppose $v\mid \mathfrak{D}_{F}$ and $v\nmid\mathfrak{q}$. Similarly to \eqref{eq7.15} we have 
\begin{equation}\label{e9.12}
|\mathbf{e}_{3,v}(\mathbf{s},\mathfrak{X}_{\mathfrak{n}})|\ll q_{v}^{O(d_{v})},\ \ \mathbf{s}\in \mathbf{S}_{\varepsilon}^{4}.
\end{equation}

\item Suppose $v\mid\mathfrak{q}$. The integral $\Psi_{v}(W_{1,v}^{*}(\cdot,s_{1}),W_{2,v}(\cdot,s_{2}),\overline{h_{3,v}(\cdot,\overline{s_{3}})}\overline{h_{4,v}(\cdot,\overline{s_{4}})})$ unfolds to 
\begin{align*}
\overline{h_{3,v}(I_{2},\overline{s_{3}})}\int_{K_{v}}\int_{F_{v}^{\times}}W_{1,v}^{*}(a(y_{v}),s_{1})W_{2,v}(a(y_{v})k_{v},s_{2})|y_{v}|_{v}^{s_{3}+s_{4}}d^{\times}y_{v}\overline{h_{4,v}(k_{v},\overline{s_{4}})}dk_{v}.
\end{align*}

Put $m_{v}=e_{v}(\mathfrak{q})$ and $r_{v}=e_{v}(\mathfrak{q}')$. Use \eqref{fc9.4} and the Gauss calculation in the proof of Lemma~\ref{lem8.6}, including the term at valuation $m_{v}-1$ when $r_{v}=0$. Since the spherical Whittaker coefficients here have parameters of real part $O(\varepsilon)$, summing the resulting geometric series gives
\begin{equation}\label{9.14.}
|\Psi_{v}(W_{1,v}^{*}(\cdot,s_{1}),W_{2,v}(\cdot,s_{2}),\overline{h_{3,v}(\cdot,\overline{s_{3}})}\overline{h_{4,v}(\cdot,\overline{s_{4}})})|\ll_{F_{v},\varepsilon} q_{v}^{r_{v}/2+20\varepsilon m_{v}}.
\end{equation}

As a result, we obtain 
\begin{equation}\label{9.14}
|\mathbf{e}_{3,v}(\mathbf{s},\mathfrak{X}_{\mathfrak{n}})|\ll_{F_{v},\varepsilon} q_{v}^{r_{v}/2+20\varepsilon m_{v}},\ \ \mathbf{s}\in \mathbf{S}_{\varepsilon}^{4}.
\end{equation}

\item Suppose $v\mid\mathfrak{n}$. Put $\sigma_{v}:=|\cdot|_{v}^{1/2+s_{3}+s_{4}}\boxplus\overline{\omega}_{v}|\cdot|_{v}^{-1/2-s_{3}-s_{4}}$. The change of variables $g_{v}\mapsto g_{v}\mathbf{d}_{v}^{-1}$ removes the common right translation from $W_{1,v}^{*},W_{2,v}$, leaving $W_{3,v}^{*}(\cdot,s_{1}),W_{4,v}(\cdot,s_{2})$ and the factor $\alpha_{\mathfrak{n},v}|\det\mathbf{d}_{v}|_{v}^{-s_{1}-s_{2}}$. The $K_{v}$-average of the remaining spherical section at $\mathbf{d}_{v}^{-1}$ contributes
\begin{align*}
\omega_{v}(\varpi_{v})^{l_{v}}q_{v}^{-l_{v}/2}(1+q_{v}^{-1})^{-1}\lambda_{\sigma_{v}}^{*}(\mathfrak{p}_{v}^{l_{v}}).
\end{align*}
Its phase and volume factor cancel those in $\alpha_{\mathfrak{n},v}$. The radial integral is the unramified Rankin--Selberg integral. Dividing by the local $L$-factors gives $\mathbf{e}_{3,v}=q_{v}^{l_{v}(s_{1}+s_{2})}\lambda_{\sigma_{v}}^{*}(\mathfrak{p}_{v}^{l_{v}})$. The Satake expansion gives
\begin{equation}\label{9.15}
|\mathbf{e}_{3,v}(\mathbf{s},\mathfrak{X}_{\mathfrak{n}})|\ll_{F_{v},\varepsilon} q_{v}^{3\varepsilon l_{v}}\big|\lambda_{\sigma_{v}}^{*}(\mathfrak{p}_{v}^{l_{v}})\big|\ll_{F_{v},\varepsilon} (1+l_{v})q_{v}^{(1/2+18\varepsilon)l_{v}}\ll_{\varepsilon} q_{v}^{(1/2+20\varepsilon)l_{v}}.
\end{equation}

\item Suppose $v\mid\infty$. The support and rapid-decay estimates for $W_{2,v}$ from \cite[Lemma 3.7.2]{MV10}, together with the $K_{v}$-support and $L^{1}$ estimate used in the proof of Lemma \ref{lem9.4}, give, for $\mathbf{s}\in \mathbf{S}_{\varepsilon}^{4}$,
\begin{equation}\label{e9.16}
|\Psi_{v}(W_{1,v}^{*}(\cdot,s_{1}),W_{2,v}(\cdot,s_{2}),\overline{h_{3,v}(\cdot,\overline{s_{3}})}\overline{h_{4,v}(\cdot,\overline{s_{4}})})|\ll_{F_{v},\varepsilon}  \frac{C_{v}(\omega)^{1/2+100\varepsilon}}{\mathbf{C}_{v}^{1-100\varepsilon}}.
\end{equation} 
\end{itemize}

Combining \eqref{c9.13}, \eqref{e9.12}, \eqref{9.14}, \eqref{9.15}, and \eqref{e9.16}, and bounding the remaining degree-one factors on the chosen contour by their functional equations and convexity, gives
\begin{align*}
\big|\widetilde{\Psi}_{\mathrm{Geo}}^{(3)}(\mathbf{s},\mathfrak{X}_{\mathfrak{n}})\big|\ll_{F,\varepsilon}N_{F}(\mathfrak{n})^{1/2+\varepsilon}C(\omega)^{1/2+\varepsilon}\mathbf{C}_{\infty}^{-1+\varepsilon}N_{F}(\mathfrak{q})^{\varepsilon},\qquad \mathbf{s}\in\mathbf{S}_{\varepsilon}^{4}.
\end{align*}
The contour integral proves \eqref{e9.11} for $i=3$. Since $h_{3}$ is spherical, its intertwining section is a scalar multiple of $h_{3}(\cdot,-s_{3})$. Multiplying these scalars over all places and using the global Tate functional equation bounds their product by $O_{F,\varepsilon}(1)$ on the contour. Thus $i=4$ follows by $s_{3}\mapsto-s_{3}$.
\end{proof}

\subsection{Bounds of $\widetilde{\Psi}_{\mathrm{Geo}}^{(5)}(\mathbf{0}\mid\mathfrak{X}_{\mathfrak{n}})$ and $\widetilde{\Psi}_{\mathrm{Geo}}^{(6)}(\mathbf{0}\mid\mathfrak{X}_{\mathfrak{n}})$}
Recall the definition:
\begin{align*}
&\widetilde{\Psi}_{\mathrm{Geo}}^{(5)}(\mathbf{s},\mathfrak{X}_{\mathfrak{n}}):=\widetilde{\Psi}(W_{h_{1},\overline{h_{3}}}^{*}(\cdot,s_{1},\overline{s_{3}}),W_{2}(\cdot,s_{2}),\overline{h_{4}(\cdot,\overline{s_{4}})}),\\
&\widetilde{\Psi}_{\mathrm{Geo}}^{(6)}(\mathbf{s},\mathfrak{X}_{\mathfrak{n}}):=\widetilde{\Psi}(W_{h_{1}^{\Diamond},\overline{h_{3}}}^{*}(\cdot,s_{1},\overline{s_{3}}),W_{2}(\cdot,s_{2}),\overline{h_{4}(\cdot,\overline{s_{4}})}).
\end{align*}

\begin{lemma}\label{lemma9.3}
 Let $i\in\{5,6\}$. Then 
\begin{equation}\label{e9.17}
\big|\widetilde{\Psi}_{\mathrm{Geo}}^{(i)}(\mathbf{0}\mid\mathfrak{X}_{\mathfrak{n}})\big|\ll_{F,\varepsilon} C_{\infty}(\omega)^{\varepsilon}\mathbf{C}_{\infty}^{\varepsilon}N_{F}(\mathfrak{q})^{1+\varepsilon}N_{F}(\mathfrak{n})^{-1/2+\varepsilon}.
\end{equation}
\end{lemma}
\begin{proof}
By the convention in \textsection\ref{sec1.3.3} and a change of variables,
\begin{align*}
W_{h_{1},\overline{h_{3}}}^{*}(g,s_{1},\overline{s_{3}})=\int_{N(\mathbb{A}_{F})}h_{1}(wug,s_{1})\overline{h_{3}(wug,\overline{s_{3}})}\theta(u)du.
\end{align*}

Notice that $h_{1}(\cdot,s_{1})\overline{h_{3}(\cdot,\overline{s_{3}})}$ is a section in $|\cdot|^{1/2+s_{1}+s_{3}}\boxplus |\cdot|^{-1/2-s_{1}-s_{3}}$. Moreover, at $v\nmid \mathfrak{n}\mathfrak{D}_{F}$ the function 
\begin{align*}
\frac{h_{1,v}(\cdot,s_{1})\overline{h_{3,v}(\cdot,\overline{s_{3}})}}{h_{1,v}(I_{2},s_{1})\overline{h_{3,v}(I_{2},\overline{s_{3}})}}=\frac{h_{1,v}(\cdot,s_{1})\overline{h_{3,v}(\cdot,\overline{s_{3}})}}{L_{v}(1+2s_{1},\mathbf{1})L_{v}(1+2s_{3},\mathbf{1})}
\end{align*}
is the normalized spherical section. Therefore, 
\begin{multline}\label{c9.20}
\widetilde{\Psi}_{\mathrm{Geo}}^{(5)}(\mathbf{s},\mathfrak{X}_{\mathfrak{n}})=\Lambda(1+2s_{1},\mathbf{1})\Lambda(1+2s_{3},\mathbf{1})\Lambda(1/2+s_{4},(|\cdot|^{1/2+s_{1}+s_{3}}\boxplus \\
|\cdot|^{-1/2-s_{1}-s_{3}})
\times (|\cdot|^{s_{2}}\boxplus \omega|\cdot|^{-s_{2}}))\prod_{v\leq\infty}\mathbf{e}_{5,v}(\mathbf{s},\mathfrak{X}_{\mathfrak{n}}),
\end{multline}
where $\mathbf{e}_{5,v}(\mathbf{s},\mathfrak{X}_{\mathfrak{n}})$ is defined by 
\begin{align*}
\frac{L_{v}(1+2s_{1},\mathbf{1})^{-1}L_{v}(1+2s_{3},\mathbf{1})^{-1}\widetilde{\Psi}_{v}(W_{h_{1},\overline{h_{3}},v}^{*}(\cdot,s_{1},\overline{s_{3}}),W_{2,v}(\cdot,s_{2}),\overline{h_{4,v}(\cdot,\overline{s_{4}})})}{L_{v}(1/2+s_{4},(|\cdot|_{v}^{1/2+s_{1}+s_{3}}\boxplus 
|\cdot|_{v}^{-1/2-s_{1}-s_{3}})
\times (|\cdot|_{v}^{s_{2}}\boxplus \omega_{v}|\cdot|_{v}^{-s_{2}}))}.
\end{align*}

At $v\nmid \mathfrak{n}\mathfrak{q}\mathfrak{D}_{F}\infty$, we have $\mathbf{e}_{5,v}=L_{v}(2+2s_{1}+2s_{3},\mathbf{1})^{-1}$. The product of these factors converges absolutely and is bounded above and below uniformly near $\mathbf{0}$. It remains to handle $v\mid \mathfrak{n}\mathfrak{q}\mathfrak{D}_{F}\infty$.
\begin{itemize}
\item Suppose $v\mid\mathfrak{D}_{F}$ and $v\nmid\mathfrak{q}$. Then 
\begin{equation}\label{9.18}
|\mathbf{e}_{5,v}(\mathbf{s},\mathfrak{X}_{\mathfrak{n}})|\ll q_{v}^{O(d_{v})},\ \ \mathbf{s}\in \mathbf{S}_{\varepsilon}^{4}.
\end{equation}

\item Suppose $v\mid\mathfrak{q}$. In the half-plane $\Re(s_{4})\gg1$, unfolding the local period in the numerator defining $\mathbf{e}_{5,v}$ gives
\begin{equation}\label{9.19}
q_{v}^{-d_{v}}L_{v}(1+2s_{1},\mathbf{1})L_{v}(1+2s_{3},\mathbf{1})\int_{N(F_{v})\backslash\overline{G}(F_{v})}W_{v}^{\circ}(g_{v})W_{2,v}(g_{v},s_{2})\overline{h_{4,v}(g_{v},\overline{s_{4}})}dg_{v},
\end{equation}
where $W_{v}^{\circ}$ is the Jacquet integral relative to $\overline{\psi}_{v}$ of the spherical section with value $1$ at $I_{2}$. First suppose that $\psi_{v}$ is unramified, so $W_{v}^{\circ}(I_{2})=L_{v}(2+2s_{1}+2s_{3},\mathbf{1})^{-1}$. Put $m_{v}=e_{v}(\mathfrak{q})$ and $r_{v}=e_{v}(\mathfrak{q}')$. Apply the diagonal calculation in the proof of Lemma~\ref{lem9.3} with spherical parameter $1/2+s_{1}+s_{3}$. The geometric series starts at $m_{v}-r_{v}$ for $r_{v}>0$; for $r_{v}=0$, include the term at $m_{v}-1$. Sum before meromorphic continuation and divide by the local $L$-factors. This gives $q_{v}^{m_{v}-r_{v}/2+O(\varepsilon m_{v})}$, while the additional factor $L_{v}(2+2s_{1}+2s_{3},\mathbf{1})^{-1}$ is bounded. Thus
\begin{equation}\label{9.20}
|\mathbf{e}_{5,v}(\mathbf{s}, \mathfrak{X}_{\mathfrak{n}})| \ll_{F_{v},\varepsilon} q_{v}^{m_{v}-r_{v}/2+100\varepsilon m_{v}}, \quad \mathbf{s} \in \mathbf{S}_{\varepsilon}^{4}.
\end{equation}
If also $v\mid\mathfrak{D}_{F}$, the full local period is $q_{v}^{-(7/2+s_{1}+s_{2}+s_{3}-s_{4})d_{v}}$ times the period for an unramified additive character, up to a unit-modulus factor. This total scalar already includes $q_{v}^{-d_{v}}$ in \eqref{9.19} and has modulus at most $1$ on the contour, so \eqref{9.20} still holds.

\item Suppose $v\mid\mathfrak{n}$. Similar to Lemma \ref{lem6.1} in \textsection\ref{sec6.2.2}, we have
\begin{multline}\label{9.21}
\widetilde{\Psi}_{v}(W_{h_{1},\overline{h_{3}},v}^{*}(\cdot,s_{1},\overline{s_{3}}),W_{2,v}(\cdot,s_{2}),\overline{h_{4,v}(\cdot,\overline{s_{4}})})=\omega_{v}(-1)\\
\gamma_{v}(1+s_{1}+s_{3}-s_{4},|\cdot|_{v}^{-s_{2}}\boxplus \overline{\omega}_{v}|\cdot|_{v}^{s_{2}},\psi_{v})\widetilde{\Psi}_{v}(W_{2,v}(\cdot,s_{2}),\overline{W_{4,v}(\cdot,\overline{s_{4}})},h_{1,v}(\cdot,s_{1})\overline{h_{3,v}(\cdot,\overline{s_{3}})})
\end{multline}
as an equality of meromorphic functions. Since $\omega_{v}$ and $\psi_{v}$ are unramified here, the numerator of the gamma factor consists of $L_{v}(-s_{1}-s_{3}+s_{4}+s_{2},\mathbf{1})$ and $L_{v}(-s_{1}-s_{3}+s_{4}-s_{2},\omega_{v})$. These cancel the identical factors in the normalization of $\mathbf{e}_{5,v}$ before evaluation. The remaining Euler factors are those of the transformed period, with the bounded factor $L_{v}(2+2s_{1}+2s_{3},\mathbf{1})^{-1}$.

After division by these Euler factors, the calculation proving \eqref{e9.6} applies to the period on the right of \eqref{9.21}. The unramified character factors depend only on valuations and do not change its additive-character averages. Hence
\begin{equation}\label{9.22}
|\mathbf{e}_{5,v}(\mathbf{s}, \mathfrak{X}_{\mathfrak{n}})| \ll_{F_{v},\varepsilon} q_{v}^{-l_{v}(1/2-\varepsilon)},\ \ \mathbf{s}\in \mathbf{S}_{\varepsilon}^{4}.
\end{equation}

\item Suppose $v\mid\infty$. Apply the second-period estimate in Lemma~\ref{lem9.4} with $\mu_{v}=\mathbf{1}$ and $\lambda=1/2+s_{1}+s_{3}$. The spherical scalar $L_{v}(1+2s_{1},\mathbf{1})L_{v}(1+2s_{3},\mathbf{1})$ is bounded on the contour. Mellin inversion and meromorphic continuation, as in that proof, give
\begin{equation}\label{9.23}
|\widetilde{\Psi}_{v}(W_{h_{1},\overline{h_{3}},v}^{*}(\cdot,s_{1},\overline{s_{3}}),W_{2,v}(\cdot,s_{2}),\overline{h_{4,v}(\cdot,\overline{s_{4}})})|\ll_{F_{v},\varepsilon} \mathbf{C}_{v}^{1000\varepsilon}C_{v}(\omega)^{-1/2+500\varepsilon}.
\end{equation}
\end{itemize}

On the chosen contour, the untwisted finite $L$-factors in \eqref{c9.20} are bounded, and the $\omega$-factor near $1$ is $\ll_{F,\varepsilon}C(\omega)^{O(\varepsilon)}$. For the $\omega$-factor near $0$, the functional equation and convexity give
\begin{align*}
|L(s,\omega)|\ll_{F,\varepsilon}C(\omega)^{1/2+100\varepsilon}.
\end{align*}
Combining this with \eqref{9.18}, \eqref{9.20}, \eqref{9.22}, and \eqref{9.23}, and taking the contour integral of \textsection\ref{sec4.regularization}, proves \eqref{e9.17} for $i=5$. For $i=6$, intertwining commutes with right translation, so the replacement used for $i=4$ changes $s_{1}$ to $-s_{1}$ and contributes $|\det\mathbf{d}_{\mathfrak{n}}|^{-2s_{1}}$. Its size $N_{F}(\mathfrak{n})^{O(\varepsilon)}$ is absorbed by rescaling $\varepsilon$.
\end{proof}

\subsection{Bounds of $\widetilde{\Psi}_{\mathrm{Geo}}^{(7)}(\mathbf{0}\mid\mathfrak{X}_{\mathfrak{n}})$ and $\widetilde{\Psi}_{\mathrm{Geo}}^{(8)}(\mathbf{0}\mid\mathfrak{X}_{\mathfrak{n}})$}

\begin{align*}
&\widetilde{\Psi}_{\mathrm{Geo}}^{(7)}(\mathbf{s},\mathfrak{X}_{\mathfrak{n}}):=-\widetilde{\Psi}(W_{h_{1},h_{2}}(\cdot,s_{1},s_{2}),\overline{W_{3}(\cdot,\overline{s_{3}})},\overline{h_{4}(\cdot,\overline{s_{4}})}),\\
&\widetilde{\Psi}_{\mathrm{Geo}}^{(8)}(\mathbf{s},\mathfrak{X}_{\mathfrak{n}}):=-\widetilde{\Psi}(W_{h_{1}^{\Diamond},h_{2}}(\cdot,s_{1},s_{2}),\overline{W_{3}(\cdot,\overline{s_{3}})},\overline{h_{4}(\cdot,\overline{s_{4}})}).
\end{align*}

\begin{lemma}\label{lemma9.4}
 Let $i\in\{7,8\}$. Then 
\begin{equation}\label{e9.18}
\big|\widetilde{\Psi}_{\mathrm{Geo}}^{(i)}(\mathbf{0}\mid\mathfrak{X}_{\mathfrak{n}})\big|\ll_{F,\varepsilon} N_{F}(\mathfrak{n})^{1/2+\varepsilon}C_{\infty}(\omega)^{1/2+\varepsilon}\mathbf{C}_{\infty}^{-1+\varepsilon}N_{F}(\mathfrak{q}')^{1/2+\varepsilon}N_{F}(\mathfrak{q})^{\varepsilon}.
\end{equation}
\end{lemma}
\begin{proof}  
By the construction of $W_{j,v}(\cdot,s_{j})$ and $h_{j,v}(\cdot,s_{j})$, we have
\begin{multline}\label{e9.28}
\widetilde{\Psi}_{\mathrm{Geo}}^{(7)}(\mathbf{s},\mathfrak{X}_{\mathfrak{n}})=-\Lambda(1+2s_{1},\mathbf{1})\Lambda(1+2s_{2},\overline{\omega})\Lambda(1/2+s_{4},(|\cdot|^{1/2+s_{1}+s_{2}}\boxplus \\
\omega|\cdot|^{-1/2-s_{1}-s_{2}})
\times (|\cdot|^{s_{3}}\boxplus |\cdot|^{-s_{3}}))\prod_{v\leq\infty}\mathbf{e}_{7,v}(\mathbf{s},\mathfrak{X}_{\mathfrak{n}}),
\end{multline}
where $\mathbf{e}_{7,v}(\mathbf{s},\mathfrak{X}_{\mathfrak{n}})$ is defined by 
\begin{align*}
\frac{L_{v}(1+2s_{1},\mathbf{1})^{-1}L_{v}(1+2s_{2},\overline{\omega}_{v})^{-1}\widetilde{\Psi}_{v}(W_{h_{1},h_{2},v}(\cdot,s_{1},s_{2}),\overline{W_{3,v}(\cdot,\overline{s_{3}})},\overline{h_{4,v}(\cdot,\overline{s_{4}})})}{L_{v}(1/2+s_{4},(|\cdot|_{v}^{1/2+s_{1}+s_{2}}\boxplus 
\omega_{v}|\cdot|_{v}^{-1/2-s_{1}-s_{2}})
\times (|\cdot|_{v}^{s_{3}}\boxplus |\cdot|_{v}^{-s_{3}}))}.
\end{align*}

At $v\nmid \mathfrak{n}\mathfrak{q}\mathfrak{D}_{F}\infty$, we have $\mathbf{e}_{7,v}=L_{v}(2+2s_{1}+2s_{2},\overline{\omega}_{v})^{-1}$. Again, the product of these factors converges absolutely and is bounded above and below uniformly near $\mathbf{0}$. Consider the remaining places:
\begin{itemize}
\item Suppose $v\mid\mathfrak{D}_{F}$ and $v\nmid\mathfrak{q}$. Then 
\begin{equation}\label{9.25}
|\mathbf{e}_{7,v}(\mathbf{s},\mathfrak{X}_{\mathfrak{n}})|\ll q_{v}^{O(d_{v})},\ \ \mathbf{s}\in \mathbf{S}_{\varepsilon}^{4}.
\end{equation} 

\item Suppose $v\mid\mathfrak{n}$. Notice that $\overline{W_{3,v}(\cdot,\overline{s_{3}})}\overline{h_{4,v}(\cdot,\overline{s_{4}})}$ is right-$K_{v}$-invariant. Initially for $\Re(s_{4})\gg1$, the local period is represented by
\begin{align*}
\int_{F_{v}^{\times}}\int_{K_{v}}W_{h_{1},h_{2},v}(a(y_{v})k_{v},s_{1},s_{2})dk_{v}\overline{W_{3,v}(a(y_{v}),\overline{s_{3}})}\overline{h_{4,v}(a(y_{v}),\overline{s_{4}})}|y_{v}|_{v}^{-1}d^{\times}y_{v}.
\end{align*}

The product section is a scalar multiple of a right translate of the spherical section in $\sigma_{v}:=|\cdot|_{v}^{1/2+s_{1}+s_{2}}\boxplus\omega_{v}|\cdot|_{v}^{-1/2-s_{1}-s_{2}}$. Jacquet integration commutes with this translation. The double-coset identity for $T_{\mathfrak{p}_{v}^{l_{v}}}^{*}$ cancels $q_{v}^{l_{v}/2}(1+q_{v}^{-1})$ in $\alpha_{\mathfrak{n},v}$, and the remaining radial integral is the unramified Rankin--Selberg integral. Dividing by the local $L$-factors gives
\begin{align*}
\mathbf{e}_{7,v}=\frac{\overline{\omega}_{v}(\varpi_{v})^{l_{v}}q_{v}^{l_{v}(s_{1}+s_{2})}\lambda_{\sigma_{v}}^{*}(\mathfrak{p}_{v}^{l_{v}})}{L_{v}(2+2s_{1}+2s_{2},\overline{\omega}_{v})}.
\end{align*}
The denominator has argument near $2$ and is uniformly bounded above and below. Thus, for $\mathbf{s}\in \mathbf{S}_{\varepsilon}^{4}$, we have 
\begin{equation}\label{9.27}
|\mathbf{e}_{7,v}(\mathbf{s},\mathfrak{X}_{\mathfrak{n}})|\ll_{F_{v},\varepsilon} q_{v}^{3\varepsilon l_{v}}|\lambda_{\sigma_{v}}^{*}(\mathfrak{p}_{v}^{l_{v}})|\ll_{F_{v},\varepsilon} (1+l_{v})q_{v}^{(1/2+6\varepsilon)l_{v}}\ll_{\varepsilon} q_{v}^{(1/2+10\varepsilon)l_{v}}.
\end{equation}

\item Suppose $v\mid\mathfrak{q}\infty$. By Lemma \ref{lem6.1}, the local period equals
\begin{align*}
\gamma_{v}(1+s_{1}+s_{2}-s_{4},\overline{\omega}_{v}(|\cdot|_{v}^{s_{3}}\boxplus|\cdot|_{v}^{-s_{3}}),\psi_{v})\widetilde{\Psi}_{v}(\overline{W_{3,v}},\overline{W_{4,v}^{*}},h_{1,v}h_{2,v}).
\end{align*}

The final local period above is absolutely convergent for $\mathbf{s}\in \mathbf{B}_{\varepsilon}^{4}$. We treat the finite and Archimedean places separately. 
\begin{itemize}
\item Suppose $v\mid\mathfrak{q}$. Put $m_{v}=e_{v}(\mathfrak{q})$ and $r_{v}=e_{v}(\mathfrak{q}')$. By the identities among the local data, this is the period estimated in \eqref{9.14.}. Thus
\begin{align*}
|\Psi_{v}(\overline{W_{3,v}(\cdot,\overline{s_{3}})},\overline{W_{4,v}^{*}(\cdot,\overline{s_{4}})},h_{1,v}(\cdot,s_{1})h_{2,v}(\cdot,s_{2}))|\ll_{F_{v},\varepsilon} q_{v}^{r_{v}/2+20\varepsilon m_{v}}.
\end{align*}
If $r_{v}\geq1$, the ramified gamma factor has size $q_{v}^{-r_{v}+O(\varepsilon r_{v})}$, up to fixed factors at $v\mid\mathfrak{D}_{F}$. If $r_{v}=0$, its numerator factors $L_{v}(-s_{1}-s_{2}+s_{4}\pm s_{3},\omega_{v})$ cancel the identical factors in the normalization of $\mathbf{e}_{7,v}$; its denominators cancel the corresponding Euler factors of the transformed period. Only the bounded scalar $L_{v}(2+2s_{1}+2s_{2},\overline{\omega}_{v})^{-1}$ remains. Taking absolute values after these cancellations gives
\begin{equation}\label{9.29}
|\mathbf{e}_{7,v}(\mathbf{s},\mathfrak{X}_{\mathfrak{n}})|\ll_{F_{v},\varepsilon} q_{v}^{-r_{v}/2+50\varepsilon m_{v}}.
\end{equation} 

\item Suppose $v\mid\infty$. The argument for \eqref{e9.16} applies to the transformed period. On the chosen contour, Stirling's formula bounds the product of the gamma factors by $C_{v}(\omega)^{-1+O(\varepsilon)}$. Thus
\begin{equation}\label{9.32}
|\widetilde{\Psi}_{v}(W_{h_{1},h_{2},v}(\cdot,s_{1},s_{2}),\overline{W_{3,v}(\cdot,\overline{s_{3}})},\overline{h_{4,v}(\cdot,\overline{s_{4}})})|\ll_{F_{v},\varepsilon} C_{v}(\omega)^{-1/2+100\varepsilon}\mathbf{C}_{v}^{-1+50\varepsilon}.
\end{equation}
\end{itemize}
\end{itemize}

Substituting \eqref{9.25}, \eqref{9.27}, \eqref{9.29} and \eqref{9.32} into \eqref{e9.28} gives, on the chosen contour,
\begin{align*}
\big|\widetilde{\Psi}_{\mathrm{Geo}}^{(7)}(\mathbf{s},\mathfrak{X}_{\mathfrak{n}})\big|\ll_{F,\varepsilon} N_{F}(\mathfrak{n})^{1/2+\varepsilon}C_{\infty}(\omega)^{-1/2+\varepsilon}\mathbf{C}_{\infty}^{-1+\varepsilon}N_{F}(\mathfrak{q}')^{-1/2}N_{F}(\mathfrak{q})^{\varepsilon}C(\omega)^{1+200\varepsilon}.
\end{align*}
Here the functional equation and convexity bound each of the two remaining $L$-factors near $0$ by $C(\omega)^{1/2+100\varepsilon}$; when $\omega=\mathbf{1}$, the contour stays a distance $\gg\varepsilon$ from the poles. The contour integral now gives \eqref{e9.18}. The case $i=8$ follows by the same $s_{1}\mapsto-s_{1}$ replacement used for $i=6$.
\end{proof}

\section{Proof of the Main Theorems}\label{sec10}
We combine the preceding estimates with amplification to prove the main results.

For the two dual moments, put
\begin{align*}
&\mathcal{I}_{1}:=\sum_{\substack{\sigma\in \mathcal{F}_{0}(\mathfrak{n},\mathbf{1})\\ \sigma_{\infty}\ \mathrm{unramified}}}e^{-\frac{\pi}{2}\sum_{v\mid\infty}|\nu_{\sigma_{v}}|}\sum_{\mathfrak{a}\mid\mathfrak{q}\mathfrak{q}'^{-1}}|\lambda_{\sigma}(\mathfrak{a})|\cdot \frac{|L(1/2,\sigma)|^{3}|L(1/2,\sigma\times\overline{\omega})|}{|L(1,\sigma,\Ad)|},\\
&\mathcal{I}_{2}:=\sum_{\substack{\mu\in\mathfrak{X}(\mathfrak{n},\mathbf{1})\\ \mu_{\infty}\ \mathrm{unramified}}}
\int_{\mathbb{R}}e^{-\frac{\pi}{2}\sum_{v\mid\infty}|\nu_{\sigma_{\mu_{v},it}}|}\frac{|L(1/2+it,\mu)|^{6}|L(1/2+it,\mu\overline{\omega})L(1/2+it,\mu\omega)|}{|L(1+2it,\mu^{2})|^{2}}dt.
\end{align*}

\subsection{Proof of Theorems \ref{thmE} and \ref{thmF}}
We first deduce the two fourth-moment estimates from Theorem \ref{thmD}.

\begin{proof}[Proof of Theorem \ref{thmE}]
Take $\mathfrak{n}=\mathcal{O}_{F}$ in Theorem \ref{thmD}. To cover the range $C_{v}(\pi)\leq\mathbf{C}_{v}$ in Theorem \ref{thmE}, use $\mathbf{C}_{v}^{1/(1-10\varepsilon)}$ as the test-function parameter in \textsection\ref{sec4.4}; the resulting factors $\mathbf{C}_{\infty}^{O(\varepsilon)}$ are absorbed by rescaling $\varepsilon$. By Corollary \ref{cor5.2} and Lemma \ref{lemma5.3}, we then have
\begin{equation}\label{10.2}
\mathcal{J}_{\mathrm{Cusp}}^{\heartsuit}(\mathbf{0},\mathfrak{X}_{\mathfrak{n}})\gg_{F,\varepsilon} \mathbf{C}_{\infty}^{\ -1-\varepsilon}N_{F}(\mathfrak{q})^{-\varepsilon}\sum_{\substack{\pi\in \mathcal{F}_{0}(\mathfrak{q},\omega)\\
C_{v}(\pi)\leq \mathbf{C}_{v},\ v\mid\infty}}|L(1/2,\pi)|^{4}.
\end{equation}

On the other hand, substituting the bounds of Propositions \ref{prop6.1}, \ref{prop8.2}, \ref{prop8.9}, \ref{prop9.1}, and Lemmas \ref{lem9.1}, \ref{lemma9.2}, \ref{lemma9.3} and \ref{lemma9.4} into Theorem \ref{thmD}, we obtain 
\begin{multline}\label{10.3}
\mathcal{J}_{\mathrm{Cusp}}^{\heartsuit}(\mathbf{0},\mathfrak{X}_{\mathfrak{n}}) \ll_{F,\varepsilon} N_{F}(\mathfrak{q})^{1+\varepsilon}C(\omega)^{\varepsilon}\mathbf{C}_{\infty}^{\varepsilon}\\
+N_{F}(\mathfrak{q})^{1/2+\varepsilon}\mathbf{C}_{\infty}^{-1/2+\varepsilon}\Big[\big[\mathbf{C}_{\infty}C_{\infty}(\omega)^{-1}\big]^{\vartheta}\mathcal{I}_{1}+\mathcal{I}_{2}\Big].
\end{multline}

Applying convexity to the twisted $L$-factors in $\mathcal{I}_{1}$ and $\mathcal{I}_{2}$, we deduce from \eqref{10.3} that
\begin{equation}\label{10.4}
\mathcal{J}_{\mathrm{Cusp}}^{\heartsuit}(\mathbf{0},\mathfrak{X}_{\mathfrak{n}}) \ll_{F,\varepsilon} N_{F}(\mathfrak{q})^{1+\varepsilon}C(\omega)^{\varepsilon}\mathbf{C}_{\infty}^{\varepsilon}+N_{F}(\mathfrak{q})^{1/2+\vartheta+\varepsilon}C(\omega)^{1/2-\vartheta+\varepsilon}\mathbf{C}_{\infty}^{-1/2+\vartheta+\varepsilon}.
\end{equation} 

Since $N_{F}(\mathfrak{q}')\leq N_{F}(\mathfrak{q})$, $C_{\infty}(\omega)\ll_{F}\mathbf{C}_{\infty}$, and $C(\omega)=D_{F}N_{F}(\mathfrak{q}')C_{\infty}(\omega)$, the second term in \eqref{10.4} is $\ll_{F} N_{F}(\mathfrak{q})^{1+O(\varepsilon)}\mathbf{C}_{\infty}^{O(\varepsilon)}$. Therefore, \eqref{10.1} follows from \eqref{10.2} and \eqref{10.4}. 
\end{proof}

Theorem \ref{thmF} follows using Lemma \ref{lemma5.4} in place of Lemma \ref{lemma5.3}; $C_{v}(\pi_{\mu,it})=C_{v}(\mu|\cdot|^{it})C_{v}(\mu\overline{\omega}|\cdot|^{it})$, so the conductor condition in Lemma \ref{lemma5.4} gives the conductor restriction in Theorem \ref{thmF} after the same enlargement of the test-function parameters.

\begin{proof}[Proof of Corollary \ref{cor1.eighth}]
Apply Theorem \ref{thmF} with $\omega=\mathbf{1}$, level $\mathfrak{q}^{2}$ and local cutoffs $\ll_{F}\mathbf{C}_{v}^{2}$. The standard short-interval comparison, using Taylor's formula and convexity bounds for derivatives, gives the asserted central-value estimate with an additional $\varepsilon$-power. This also covers $\mu^{2}=\mathbf{1}$, for which $|\zeta_{F}(1+2it)|^{-2}$ vanishes quadratically at $t=0$. Rescaling $\varepsilon$ proves the assertion.
\end{proof}

\subsection{Hybrid Moment Estimates}
We now bound $\mathcal{I}_{1}$ and $\mathcal{I}_{2}$ for general $\mathfrak{n}$.

\begin{lemma}\label{lem10.1}
We have
\begin{align*}
\mathcal{I}_{1}\ll_{F,\varepsilon} N_{F}(\mathfrak{q}\mathfrak{q}'^{-1})^{\vartheta+\varepsilon}N_{F}(\mathfrak{n})^{1+\varepsilon}
\min\big\{
C(\omega)^{\frac{1}{2}+\varepsilon}, N_{F}(\mathfrak{n})^{\frac{1}{4}} C(\omega)^{\frac{3}{8}+\varepsilon}+N_{F}(\mathfrak{n})^{\frac{1}{2}} C(\omega)^{\frac{1}{4}+\varepsilon}
\big\}.
\end{align*}
\end{lemma}
\begin{proof}
Put
\begin{align*}
\mathcal{I}_{1}^{\mathrm{untw}}&:=\sum_{\substack{\sigma\in \mathcal{F}_{0}(\mathfrak{n},\mathbf{1})\\ \sigma_{\infty}\ \mathrm{unramified}}}e^{-\frac{\pi}{2}\sum_{v\mid\infty}|\nu_{\sigma_{v}}|}C(\sigma)^{2\varepsilon}|L(1/2,\sigma)|^{4},\\
\mathcal{I}_{1}^{\mathrm{tw}}&:=\sum_{\substack{\sigma\in \mathcal{F}_{0}(\mathfrak{n},\mathbf{1})\\ \sigma_{\infty}\ \mathrm{unramified}}}e^{-\frac{\pi}{2}\sum_{v\mid\infty}|\nu_{\sigma_{v}}|}|L(1/2,\sigma\times\overline{\omega})|^{4}.
\end{align*}
By the Ramanujan and divisor bounds, H\"{o}lder's inequality, and $L(1,\sigma,\Ad)\gg_{F,\varepsilon} C(\sigma)^{-\varepsilon}$, we have 
\begin{equation}\label{eq10.4}
\mathcal{I}_{1}\ll_{F,\varepsilon} N_{F}(\mathfrak{q}\mathfrak{q}'^{-1})^{\vartheta+\varepsilon}\big(\mathcal{I}_{1}^{\mathrm{untw}}\big)^{\frac{3}{4}}\big(\mathcal{I}_{1}^{\mathrm{tw}}\big)^{\frac{1}{4}}.
\end{equation}

By Theorem \ref{thmE},
\begin{equation}\label{10.5}
\mathcal{I}_{1}^{\mathrm{untw}}\ll_{F,\varepsilon} N_{F}(\mathfrak{n})^{1+\varepsilon}.
\end{equation}

We may regard $\sigma\otimes\overline{\omega}$ as a representation in $\mathcal{F}_{0}(\mathfrak{n}\mathfrak{q}'^{2},\overline{\omega}^{2})$. Applying Theorem \ref{thmE} after a dyadic decomposition of the Archimedean parameters, and summing against the exponential weight, gives
\begin{equation}\label{10.6}
\mathcal{I}_{1}^{\mathrm{tw}}\ll_{F,\varepsilon} C_{\infty}(\omega)^{2+\varepsilon}N_{F}(\mathfrak{n}\mathfrak{q}'^{2})^{1+\varepsilon}.
\end{equation}

On the other hand, \cite[Corollary 1.9]{Yan23c} gives
\begin{equation}\label{10.7}
\mathcal{I}_{1}^{\mathrm{tw}}\ll_{F,\varepsilon} N_{F}(\mathfrak{n})^{1+\varepsilon} \big[N_{F}(\mathfrak{n})^{\frac{1}{4}} C(\omega)^{\frac{3}{8}+\varepsilon}+N_{F}(\mathfrak{n})^{\frac{1}{2}} C(\omega)^{\frac{1}{4}+\varepsilon}\big]^{4}.
\end{equation}

Substituting \eqref{10.5}, \eqref{10.6} and \eqref{10.7} into \eqref{eq10.4} proves Lemma \ref{lem10.1}.
\end{proof}

\begin{lemma}\label{lem10.2}
We have
\begin{align*}
\mathcal{I}_{2}\ll_{F,\varepsilon} N_{F}(\mathfrak{q}\mathfrak{q}'^{-1})^{\vartheta}N_{F}(\mathfrak{n})^{1+\varepsilon}
\min\big\{
C(\omega)^{\frac{1}{2}+\varepsilon}, N_{F}(\mathfrak{n})^{\frac{3}{16}} C(\omega)^{\frac{3}{8}+\varepsilon}
\big\}.
\end{align*}	
\end{lemma}
\begin{proof}
Use throughout the Eisenstein measure in $\mathcal{I}_{2}$, including $|L(1+2it,\mu^{2})|^{-2}$. H\"{o}lder's inequality with exponents $4/3$ and $4$, followed by Theorem~\ref{thmF} at levels $\mathfrak{n}$ and $\mathfrak{n}\mathfrak{q}'^{2}$, gives $N_{F}(\mathfrak{n})^{1+\varepsilon}C(\omega)^{1/2+\varepsilon}$; the twist has central character $\overline{\omega}^{2}$ and the same denominator. For the second bound, $C_{\fin}(\mu)\leq N_{F}(\mathfrak{n})^{1/2}$, so \cite[Corollary 1.2]{Yan23c} bounds the product of the two twisted $L$-values by $N_{F}(\mathfrak{n})^{3/16+\varepsilon}C(\omega)^{3/8+\varepsilon}$ up to polynomial Archimedean factors. For the remaining sixth moment, Theorem~\ref{thmF} bounds the eighth moment by $N_{F}(\mathfrak{n})^{1+\varepsilon}$, while the weighted Weyl law bounds the integral of the weight against this measure by $N_{F}(\mathfrak{n})^{1+\varepsilon}$; H\"{o}lder with exponents $4/3$ and $4$ gives the same bound for the sixth moment. The exponential weight absorbs all polynomial Archimedean factors. Insert $N_{F}(\mathfrak{q}\mathfrak{q}'^{-1})^{\vartheta}\geq1$ to obtain the stated form.
\end{proof}

\subsection{The $\lambda_{\pi}^{*}(\mathfrak{n})$-weighted Fourth Moment}
\begin{thm}\label{thm10.3}
 Let $\mathfrak{n}$ be an integral ideal as in \textsection\ref{sec4.4}. Let $\mathcal{J}_{\mathrm{\Spec}}^{\heartsuit}(\mathbf{0},\mathfrak{X}_{\mathfrak{n}})$ be the left hand side of the formula in Theorem \ref{thmD}. Then 
\begin{multline}\label{10.8}
\mathcal{J}_{\mathrm{\Spec}}^{\heartsuit}(\mathbf{0},\mathfrak{X}_{\mathfrak{n}})\ll_{F,\varepsilon} N_{F}(\mathfrak{n})^{-1/2+\varepsilon}N_{F}(\mathfrak{q})^{1+\varepsilon}\mathbf{C}_{\infty}^{\varepsilon}+N_{F}(\mathfrak{q})^{1/2+\vartheta+\varepsilon}\mathbf{C}_{\infty}^{-1/2+\vartheta+\varepsilon}\\
C(\omega)^{-\vartheta}\min\big\{
N_{F}(\mathfrak{n})^{\frac{1}{2}+\varepsilon}C(\omega)^{\frac{1}{2}+\varepsilon}, N_{F}(\mathfrak{n})^{\frac{3}{4}+\varepsilon} C(\omega)^{\frac{3}{8}+\varepsilon}+N_{F}(\mathfrak{n})^{1+\varepsilon}C(\omega)^{\frac{1}{4}+\varepsilon}
\big\}.
\end{multline}
\end{thm}
\begin{proof}
Substituting the bounds of Propositions \ref{prop6.1}, \ref{prop8.2}, \ref{prop8.9}, \ref{prop9.1}, and Lemmas \ref{lem9.1}, \ref{lemma9.2}, \ref{lemma9.3} and \ref{lemma9.4} into Theorem \ref{thmD}, we obtain 
\begin{multline*}
\mathcal{J}_{\mathrm{\Spec}}^{\heartsuit}(\mathbf{0},\mathfrak{X}_{\mathfrak{n}})\ll_{F,\varepsilon} N_{F}(\mathfrak{n})^{1/2+\varepsilon}C(\omega)^{1/2+\varepsilon}\mathbf{C}_{\infty}^{-1+\varepsilon}N_{F}(\mathfrak{q})^{\varepsilon}
+N_{F}(\mathfrak{n})^{-1/2+\varepsilon}N_{F}(\mathfrak{q})^{1+\varepsilon}\mathbf{C}_{\infty}^{\varepsilon}\\
+N_{F}(\mathfrak{q})^{1/2+\varepsilon}N_{F}(\mathfrak{n})^{-1/2+\varepsilon}\big[\mathbf{C}_{\infty}C_{\infty}(\omega)^{-1}\big]^{\vartheta}\mathbf{C}_{\infty}^{-1/2+\varepsilon}\cdot (\mathcal{I}_{1}+\mathcal{I}_{2}).
\end{multline*}
Proposition \ref{prop6.1} and Lemma \ref{lemma9.4} both contribute to the first term.

Substitution of Lemmas \ref{lem10.1} and \ref{lem10.2} gives the second term in \eqref{10.8}. Both entries of its minimum are at least $N_{F}(\mathfrak{n})^{1/2}C(\omega)^{1/4}$. Since $C(\omega)\ll_{F} N_{F}(\mathfrak{q})\mathbf{C}_{\infty}$, the factor preceding the minimum, multiplied by $\mathbf{C}_{\infty}$, is $\gg_{F} C(\omega)^{1/2}$. Thus this term absorbs the first term in the preceding display. This proves \eqref{10.8}.
\end{proof}

\begin{remark}\label{rmk10.degenerate}
In Theorem~\ref{thmD}, the terms $\widetilde{\Psi}_{\mathrm{Geo}}^{(i)}$, $i\in\{1,2,5,6\}$, and all $\widetilde{\Psi}_{\mathrm{Dual}}^{(i)}$ contribute to the main-term bound in \eqref{10.8}. By Lemmas \ref{lem9.1} and \ref{lemma9.3} and Proposition \ref{prop9.1}, they are bounded, up to $\varepsilon$-powers of the parameters, by
\begin{align*}
N_{F}(\mathfrak{q})N_{F}(\mathfrak{n})^{-1/2}.
\end{align*}
The remaining geometric terms $\widetilde{\Psi}_{\mathrm{Geo}}^{(i)}$, $i\in\{3,4,7,8\}$, and all $\widetilde{\Psi}_{\mathrm{RS}}^{(i)}$ are error terms. By Lemmas \ref{lemma9.2} and \ref{lemma9.4} and Proposition \ref{prop6.1}, they are bounded, up to the same $\varepsilon$-powers, by
\begin{align*}
N_{F}(\mathfrak{n})^{1/2}C(\omega)^{1/2}\mathbf{C}_{\infty}^{-1}.
\end{align*}
These errors are absorbed by the dual-moment term in \eqref{10.8}. With the amplifier of \textsection\ref{sec10.4}, after multiplication by $\mathbf{C}_{\infty}$ and division by the squared amplifier size, the main-term and error contributions are bounded by $C(\pi)^{1+\varepsilon}L^{-1+\varepsilon}$ and $L^{2+\varepsilon}C(\omega)^{1/2+\varepsilon}$, respectively. Only the former enters the parameter optimization.
\end{remark}

\subsection{Amplification}\label{sec10.4}
Let $\pi$ be a unitary cuspidal automorphic representation of central character $\omega$ and arithmetic conductor $\mathfrak{q}$. In the test functions take $\mathbf{C}_{v}=2C_{v}(\pi)^{1/(1-10\varepsilon)}$ for $v\mid\infty$. Thus $\mathbf{C}_{\infty}\asymp_{F} C_{\infty}(\pi)^{1+O(\varepsilon)}$, and throughout this subsection the resulting factor $C(\pi)^{O(\varepsilon)}$ is absorbed by rescaling $\varepsilon$.

\begin{thm}\label{thm10.4}
Let $L>100D_{F}$ and assume that $L\gg_{F}\log(2+N_{F}(\mathfrak{q}))$. Then
\begin{multline}\label{10.14}
|L(1/2,\pi)|^{4}\ll_{F,\varepsilon} C(\pi)^{1+\varepsilon}L^{-1+\varepsilon}+C(\pi)^{1/2+\vartheta+\varepsilon}C(\omega)^{-\vartheta+\varepsilon}L^{\varepsilon}\\
\min\big\{
L^{2}C(\omega)^{\frac{1}{2}}, L^{3} C(\omega)^{\frac{3}{8}}+L^{4}C(\omega)^{\frac{1}{4}}
\big\}.
\end{multline}
\end{thm}
\begin{proof}
For a prime ideal $\mathfrak{p}\nmid\mathfrak{q}$, we let $\ell_{\mathfrak{p}}$ be the smallest positive integer such that $|\lambda_{\pi}^{*}(\mathfrak{p}^{\ell_{\mathfrak{p}}})|\geq 10^{-1}$. Due to the Hecke relation $|\lambda_{\pi}^{*}(\mathfrak{p})|^{2}=\omega_{v}^{-1}(\varpi_{v})\lambda_{\pi}^{*}(\mathfrak{p}^{2})+q_{v}^{-1}+1$ (see \eqref{c4.5} in Lemma \ref{lemma4.1}), we have $\ell_{\mathfrak{p}}\in\{1,2\}$. 

Let $\mathfrak{L}$ be the set of prime ideals $\mathfrak{p}$ such that $\mathfrak{p}\nmid\mathfrak{q}\mathfrak{D}_{F}$ and $L<N_{F}(\mathfrak{p})\leq 2L$. The prime ideal theorem gives $\gg_{F} L/\log L$ primes in this interval, while at most $O(\log N_{F}(\mathfrak{q})/\log L)$ of them divide $\mathfrak{q}$. Thus the condition on $L$ gives $\#\mathfrak{L}\asymp_{F} L/\log L$. For each $\mathfrak{p}\in \mathfrak{L}$, we define $\alpha_{\mathfrak{p}}:=\overline{\lambda_{\pi}^{*}(\mathfrak{p}^{\ell_{\mathfrak{p}}})}|\lambda_{\pi}^{*}(\mathfrak{p}^{\ell_{\mathfrak{p}}})|^{-1}$. So $|\alpha_{\mathfrak{p}}|=1$ and $\alpha_{\mathfrak{p}}\lambda_{\pi}^{*}(\mathfrak{p}^{\ell_{\mathfrak{p}}})=|\lambda_{\pi}^{*}(\mathfrak{p}^{\ell_{\mathfrak{p}}})|\geq 10^{-1}$. Let $\boldsymbol{\alpha}_{\pi}=(\alpha_{\mathfrak{p}})_{\mathfrak{p}\in \mathfrak{L}}$.  

According to the Hecke relations \eqref{c4.5}, there are constants $\kappa_{\mathfrak{p},j}\ll 1$ and $\kappa_{\mathfrak{p}_{1},\mathfrak{p}_{2}}\ll 1$ such that 
\begin{align*}
\Big|\sum_{\mathfrak{p}\in \mathfrak{L}}\alpha_{\mathfrak{p}}\lambda_{\pi}^{*}(\mathfrak{p}^{\ell_{\mathfrak{p}}})\Big|^{2}=\sum_{\substack{\mathfrak{p}\in \mathfrak{L}\\ 0\leq j\leq \ell_{\mathfrak{p}}}}\kappa_{\mathfrak{p},j}\lambda_{\pi}^{*}(\mathfrak{p}^{2\ell_{\mathfrak{p}}-2j})
+\sum_{\substack{\mathfrak{p}_{1},\mathfrak{p}_{2}\in \mathfrak{L}\\ \mathfrak{p}_{1}\neq\mathfrak{p}_{2}}}\kappa_{\mathfrak{p}_{1},\mathfrak{p}_{2}}\alpha_{\mathfrak{p}_{1}}\overline{\alpha_{\mathfrak{p}_{2}}}\lambda_{\pi}^{*}(\mathfrak{p}_{1}^{\ell_{\mathfrak{p}_{1}}}\mathfrak{p}_{2}^{\ell_{\mathfrak{p}_{2}}}).
\end{align*}

We thus define
\begin{multline}\label{10.10}
\mathcal{J}_{\mathrm{\Spec}}^{\heartsuit}(\boldsymbol{\alpha}_{\pi},\mathfrak{L}):=\sum_{\mathfrak{p}\in \mathfrak{L}}\sum_{j=0}^{\ell_{\mathfrak{p}}}\overline{\kappa_{\mathfrak{p},j}}\mathcal{J}_{\mathrm{\Spec}}^{\heartsuit}(\mathbf{0},\mathfrak{X}_{\mathfrak{p}^{2\ell_{\mathfrak{p}}-2j}})\\
+\sum_{\substack{\mathfrak{p}_{1},\mathfrak{p}_{2}\in \mathfrak{L}\\ \mathfrak{p}_{1}\neq\mathfrak{p}_{2}}}\overline{\kappa_{\mathfrak{p}_{1},\mathfrak{p}_{2}}}\overline{\alpha_{\mathfrak{p}_{1}}}\alpha_{\mathfrak{p}_{2}}\mathcal{J}_{\mathrm{\Spec}}^{\heartsuit}(\mathbf{0},\mathfrak{X}_{\mathfrak{p}_{1}^{\ell_{\mathfrak{p}_{1}}}\mathfrak{p}_{2}^{\ell_{\mathfrak{p}_{2}}}}).
\end{multline}
The coefficients are conjugated because Corollary \ref{cor5.2} weights each spectral term by $\overline{\lambda_{\sigma}^{*}(\mathfrak{n})}$. Thus the amplified spectral weight is $\big|\sum_{\mathfrak{p}\in\mathfrak{L}}\alpha_{\mathfrak{p}}\lambda_{\sigma}^{*}(\mathfrak{p}^{\ell_{\mathfrak{p}}})\big|^{2}$.

By the choice of $\mathfrak{L}$, $N_{F}(\mathfrak{p}^{\ell_{\mathfrak{p}}})\ll L^{2}$. Substituting the bound in Theorem \ref{thm10.3} into \eqref{10.10}, we derive 
\begin{multline}\label{10.11}
\mathbf{C}_{\infty}\cdot \mathcal{J}_{\mathrm{\Spec}}^{\heartsuit}(\boldsymbol{\alpha}_{\pi},\mathfrak{L})\ll_{F,\varepsilon} C(\pi)^{1+\varepsilon}L+C(\pi)^{1/2+\vartheta+\varepsilon}C(\omega)^{-\vartheta}L^{2}\\
\min\big\{
L^{2+\varepsilon}C(\omega)^{\frac{1}{2}+\varepsilon}, L^{3+\varepsilon} C(\omega)^{\frac{3}{8}+\varepsilon}+L^{4+\varepsilon}C(\omega)^{\frac{1}{4}+\varepsilon}
\big\}.
\end{multline}

Positivity, Corollary \ref{cor5.2}, and Lemma \ref{lemma5.3} give 
\begin{equation}\label{10.12}
\mathbf{C}_{\infty}\cdot \mathcal{J}_{\mathrm{\Spec}}^{\heartsuit}(\boldsymbol{\alpha}_{\pi},\mathfrak{L})\gg_{F,\varepsilon} \mathbf{C}_{\infty}^{-\varepsilon}N_{F}(\mathfrak{q})^{-\varepsilon}|L(1/2,\pi)|^{4}\cdot \Big|\sum_{\mathfrak{p}\in \mathfrak{L}}\alpha_{\mathfrak{p}}\lambda_{\pi}^{*}(\mathfrak{p}^{\ell_{\mathfrak{p}}})\Big|^{2}.
\end{equation}

The amplifier at $\pi$ is at least $10^{-1}\#\mathfrak{L}$, so \eqref{10.12} and $\#\mathfrak{L}\asymp_{F} L/\log L$ give
\begin{equation}\label{10.13}
\mathbf{C}_{\infty}\cdot \mathcal{J}_{\mathrm{\Spec}}^{\heartsuit}(\boldsymbol{\alpha}_{\pi},\mathfrak{L})\gg_{F,\varepsilon} \mathbf{C}_{\infty}^{-\varepsilon}N_{F}(\mathfrak{q})^{-\varepsilon}|L(1/2,\pi)|^{4}\cdot L^{2-\varepsilon}.
\end{equation}

Combining \eqref{10.13} with \eqref{10.11} leads to \eqref{10.14}. 
\end{proof}

\subsection{A Uniform Subconvexity} 
The optimizing lengths $L$ below may be enlarged by an $F$-dependent constant and an arbitrarily small power of $C(\pi)$. This ensures the lower bound on $L$ in Theorem \ref{thm10.4}; the resulting factors are absorbed into $C(\pi)^{\varepsilon}$.
\begin{thmx}\label{thmBweak}
 Let $F$ be a number field. Let $\pi$ be a unitary cuspidal automorphic representation of $\mathrm{GL}_{2}/F$ with central character $\omega$. Then 
\begin{multline*}
L(1/2,\pi)\ll_{F,\varepsilon} \min\Big\{C(\pi)^{\frac{7}{32}+\frac{\vartheta}{16}+\varepsilon}C(\omega)^{\frac{3}{128}-\frac{\vartheta}{16}+\varepsilon}\mathbf{1}_{C(\pi)^{1/2-\vartheta}<C(\omega)^{7/8-\vartheta}}\\
+C(\pi)^{\frac{9}{40}+\frac{\vartheta}{20}+\varepsilon}C(\omega)^{\frac{1}{80}-\frac{\vartheta}{20}+\varepsilon}\mathbf{1}_{C(\pi)^{1/2-\vartheta}\geq C(\omega)^{7/8-\vartheta}},\ 
C(\pi)^{\frac{5}{24}+\frac{\vartheta}{12}+\varepsilon}C(\omega)^{\frac{1}{24}-\frac{\vartheta}{12}+\varepsilon}\Big\}.
\end{multline*}	
\end{thmx}
\begin{proof}
Choose $L$ in Theorem \ref{thm10.4} as follows.
\begin{itemize}
\item Suppose $L\ll C(\omega)^{1/8}$. Then \eqref{10.14} implies
\begin{equation}\label{e10.15}
|L(1/2,\pi)|^{4}\ll_{F,\varepsilon} C(\pi)^{1+\varepsilon}L^{-1+\varepsilon}+C(\pi)^{1/2+\vartheta+\varepsilon}C(\omega)^{-\vartheta+\varepsilon}L^{3+\varepsilon}C(\omega)^{\frac{3}{8}}.
\end{equation}

Balancing the two main terms, with the $\varepsilon$-powers absorbed as above, gives $L^{4}=C(\pi)^{1/2-\vartheta}C(\omega)^{-3/8+\vartheta}$. In the range $C(\pi)^{1/2-\vartheta}<C(\omega)^{7/8-\vartheta}$ we have $L^{4}\ll C(\omega)^{1/2}$, i.e., $L\ll C(\omega)^{1/8}$. 

Thus, in this range, substituting $L^{4}=C(\pi)^{1/2-\vartheta}C(\omega)^{-3/8+\vartheta}$ into \eqref{e10.15} gives
\begin{equation}\label{10.15}
L(1/2,\pi)\ll_{F,\varepsilon} C(\pi)^{\frac{7}{32}+\frac{\vartheta}{16}+\varepsilon}C(\omega)^{\frac{3}{128}-\frac{\vartheta}{16}+\varepsilon}.
\end{equation}

\item Suppose $L\gg  C(\omega)^{1/8}$. Then \eqref{10.14} gives
\begin{equation}\label{e10.17}
|L(1/2,\pi)|^{4}\ll_{F,\varepsilon} C(\pi)^{1+\varepsilon}L^{-1+\varepsilon}+C(\pi)^{1/2+\vartheta+\varepsilon}C(\omega)^{-\vartheta+\varepsilon}L^{4}C(\omega)^{\frac{1}{4}}.
\end{equation}

Balancing the two main terms gives $L^{5}=C(\pi)^{1/2-\vartheta}C(\omega)^{\vartheta-1/4}$. In the range $C(\pi)^{1/2-\vartheta}\geq C(\omega)^{7/8-\vartheta}$ we have $L^{5}\geq C(\omega)^{5/8}$, i.e., $L\geq C(\omega)^{1/8}$. 

Thus, in this range, substituting $L^{5}=C(\pi)^{1/2-\vartheta}C(\omega)^{\vartheta-1/4}$ into \eqref{e10.17} gives
\begin{equation}\label{10.16}
L(1/2,\pi)\ll_{F,\varepsilon} C(\pi)^{\frac{9}{40}+\frac{\vartheta}{20}+\varepsilon}C(\omega)^{\frac{1}{80}-\frac{\vartheta}{20}+\varepsilon}.
\end{equation}

\item The first entry of the minimum in \eqref{10.14} gives
\begin{equation}\label{10.19}
|L(1/2,\pi)|^{4}\ll_{F,\varepsilon} C(\pi)^{1+\varepsilon}L^{-1+\varepsilon}+C(\pi)^{1/2+\vartheta+\varepsilon}C(\omega)^{-\vartheta+\varepsilon}L^{2+\varepsilon}C(\omega)^{\frac{1}{2}}.
\end{equation}

Balancing the two terms gives $L^{3}=C(\pi)^{1/2-\vartheta}C(\omega)^{\vartheta-1/2}$. Substitution into \eqref{10.19} yields
\begin{equation}\label{10.20}
L(1/2,\pi)\ll_{F,\varepsilon} C(\pi)^{\frac{5}{24}+\frac{\vartheta}{12}+\varepsilon}C(\omega)^{\frac{1}{24}-\frac{\vartheta}{12}+\varepsilon}.
\end{equation}
\end{itemize}

Therefore, Theorem \ref{thmBweak} follows from \eqref{10.15}, \eqref{10.16}, and \eqref{10.20}. 
The bound \eqref{10.15} also holds globally: outside its stated range, \eqref{10.20} is smaller.
\end{proof}

We obtain the following corollary.
\begin{cor}
 Let $F$ be a number field. Let $\pi$ be a unitary cuspidal automorphic representation of $\mathrm{GL}_{2}/F$. Then 
\begin{equation}\label{e10.20}
L(1/2,\pi)\ll_{F,\varepsilon} C(\pi)^{\frac{1}{4}-\frac{1}{128}+\varepsilon}.	
\end{equation}
\end{cor}
\begin{proof}
In the first range of Theorem \ref{thmBweak}, the exponents of $C(\pi)$ and $C(\omega)$ add to $31/128$, and the exponent of $C(\omega)$ is nonnegative for $0\leq\vartheta\leq7/64$. In the second range they add to $19/80<31/128$, again with a nonnegative exponent of $C(\omega)$. The result follows from $C(\omega)\leq C(\pi)$.
\end{proof}
\begin{remark}
The exponent $1/128$ agrees with Blomer-Khan \cite[Theorem 1]{BK19a}.
\end{remark}

\subsection{Hybrid Subconvexity via the Relative Trace Formula}\label{sec10.6}
Write $\mathfrak{c}(\cdot)$ for the finite conductor ideal. To improve \eqref{e10.20}, we refine the hybrid bound \cite[Corollary 1.9]{Yan23c} used above:
\begin{multline}\label{10.21}
L(1/2,\sigma\times\overline{\omega})\ll_{F,\varepsilon}  C(\sigma)^{\frac{1}{2}+\varepsilon}C_{\fin}(\overline{\omega})^{\frac{1}{4}+\varepsilon}C_{\infty}(\sigma\times\overline{\omega})^{\frac{1}{8}+\varepsilon}\\
+N_{F}(\mathfrak{c}(\sigma)\cap\mathfrak{c}(\omega_{\sigma})\mathfrak{c}(\overline{\omega}))^{\frac{1}{4}+\varepsilon}C_{\fin}(\overline{\omega})^{\frac{1}{8}+\varepsilon}C_{\infty}(\sigma)^{\frac{3}{8}+\varepsilon}C_{\infty}(\sigma\times\overline{\omega})^{\frac{3}{16}+\varepsilon},
\end{multline}
where $\omega_{\sigma}$ is the central character of $\sigma$.

We refine this estimate when $\omega_{\sigma}=\mathbf{1}$ and $\mathfrak{c}(\sigma)=\mathfrak{p}_{1}^{2}\mathfrak{p}_{2}^{2}$ with $\mathfrak{p}_{1},\mathfrak{p}_{2}\in\mathfrak{L}$; see \textsection\ref{sec10.4}. These representations give the dominant contribution after amplification. For the other conductors, we use \eqref{10.21}.

We refine \cite[Theorem 7.6]{Yan23c} as follows.
\begin{thm}\label{thm10.7}
Suppose $\sigma=\otimes_{v}\sigma_{v}$ is a unitary generic automorphic representation of $\mathrm{PGL}_{2}/F$, and let $\omega$ be a unitary Hecke character of $F^{\times}\backslash\mathbb{A}_{F}^{\times}$. Assume that
\begin{itemize}
\item at $v\mid 2\mathfrak{D}_{F}$, $\sigma_{v}$ is unramified;
\item at each $v<\infty$ and $v\nmid 2\mathfrak{D}_{F}$, $\sigma_{v}$ is either unramified or of conductor $\mathfrak{p}_{v}^{2}$; 
\item $\mathfrak{c}(\sigma)+\mathfrak{c}(\overline{\omega})=\mathcal{O}_{F}$. 
\end{itemize}
Let $\varepsilon>0$ and $L\gg_{F,\varepsilon} C_{\infty}(\sigma)^{\frac{1}{2}+\varepsilon}C_{\fin}(\sigma)^{1+\varepsilon}C_{\fin}(\overline{\omega})^{\varepsilon}$. Then 
\begin{multline}\label{10.22}
|L(1/2,\sigma\times\overline{\omega})|^{2}\ll_{F,\varepsilon}  L^{-1+\varepsilon}C_{\fin}(\sigma)^{\frac{1}{2}+\varepsilon}C_{\fin}(\overline{\omega})^{1+\varepsilon}C_{\infty}(\sigma)^{1+\varepsilon}C_{\infty}(\sigma\times\overline{\omega})^{\frac{1}{2}+\varepsilon}\\
+L^{1+\varepsilon}C_{\fin}(\sigma)^{\frac{1}{4}+\varepsilon}C_{\fin}(\overline{\omega})^{\frac{1}{2}+\varepsilon}C_{\infty}(\sigma)^{\frac{1}{2}+\varepsilon}C_{\infty}(\sigma\times\overline{\omega})^{\frac{1}{4}+\varepsilon}.
\end{multline}
\end{thm}
\begin{proof}
For each place $v$, let $f_{v}\in C_{c}^{\infty}(\overline{G}(F_{v}))$. For $-1<\Re(s)<1$, define 
\begin{align*}
&J_{\mathrm{sm},v}(f_{v},s):=\int_{F_{v}^{\times}}\int_{F_{v}^{\times}}\int_{F_{v}}f_{v}\left(\begin{pmatrix}
	y_{v}&b_{v}\\
	&1 
\end{pmatrix}\right)\psi_{v}(x_{v}b_{v})\omega_{v}(y_{v})db_{v}d^{\times}y_{v}|x_{v}|_{v}^{1+s}d^{\times}x_{v},\\
&J_{\mathrm{du},v}(f_{v},s):=\int_{F_{v}^{\times}}\int_{F_{v}^{\times}}\int_{F_{v}}f_{v}\left(\begin{pmatrix}
	y_{v}&\\
	y_{v}b_{v}&1
\end{pmatrix}\right)\psi_{v}(x_{v}b_{v})\omega_{v}(y_{v})db_{v}d^{\times}y_{v}|x_{v}|_{v}^{1-s}d^{\times}x_{v}.
\end{align*}

Let $f=\otimes_{v}f_{v}$. Define the global integrals
\begin{align*}
J^{\mathrm{Reg}}_{\mathrm{Geo},\mathrm{Small}}(f,s,\overline{\omega}):= \prod_{v}J_{\mathrm{sm},v}(f_{v},s),\qquad J^{\mathrm{Reg}}_{\mathrm{Geo},\mathrm{Dual}}(f,s,\overline{\omega}):=\prod_{v}J_{\mathrm{du},v}(f_{v},s).
\end{align*} 

Then $J^{\mathrm{Reg}}_{\mathrm{Geo},\mathrm{Small}}(f,s,\overline{\omega})$ converges absolutely in $\Re(s)>0$ and $J^{\mathrm{Reg}}_{\mathrm{Geo},\mathrm{Dual}}(f,s,\overline{\omega})$ converges absolutely in $\Re(s)<0$. After unfolding, both are finite linear combinations of Tate zeta integrals, with local factors determined by $f$. In particular, they admit meromorphic continuation to $s\in\mathbb{C}$. 

Define
\begin{align*}
J^{\mathrm{Reg},\RNum{2}}_{\mathrm{Geo},\mathrm{Big}}(f,\mathbf{0},\overline{\omega}):=\sum_{t \in F - \{0,1\}}\prod_{v} \mathcal{E}_{v}(t),
\end{align*}
where
\begin{align*}
\mathcal{E}_{v}(t) := \int_{F_{v}^{\times}} \int_{F_{v}^{\times}} 
f_{v}\left(\begin{pmatrix} y_{v} & x_{v}^{-1}t \\ x_{v} y_{v} & 1 \end{pmatrix} \right) 
\omega_{v}(y_{v}) \, d^{\times}y_{v} \, d^{\times}x_{v}.
\end{align*} 
Then $J^{\mathrm{Reg},\RNum{2}}_{\mathrm{Geo},\mathrm{Big}}(f,\mathbf{0},\overline{\omega})$ converges absolutely; see \cite[Theorem 6.1]{Yan23c}.

Recall the regularized relative trace formula in  \cite[Theorem 2.6]{Yan23c}: 
\begin{multline}\label{eq10.23}
J_{\mathrm{Spec}}^{\mathrm{Reg},\heartsuit}(f,\mathbf{0},\overline{\omega})
=\frac{1}{2\pi i}\int_{|s|=\varepsilon}\frac{J^{\mathrm{Reg}}_{\mathrm{Geo},\mathrm{Small}}(f,s,\overline{\omega})}{s}ds\\
+\frac{1}{2\pi i}\int_{|s|=\varepsilon}\frac{J^{\mathrm{Reg}}_{\mathrm{Geo},\mathrm{Dual}}(f,s,\overline{\omega})}{s}ds+J^{\mathrm{Reg},\RNum{2}}_{\mathrm{Geo},\mathrm{Big}}(f,\mathbf{0},\overline{\omega}),
\end{multline}
where $J^{\mathrm{Reg}}_{\mathrm{Geo},\mathrm{Small}}(f,s,\overline{\omega})$ and $J^{\mathrm{Reg}}_{\mathrm{Geo},\mathrm{Dual}}(f,s,\overline{\omega})$ denote their meromorphic continuations. 

Let $S_{\sigma}$ be the set of finite places where $\sigma_{v}$ is ramified. Let $\mathcal{L}$ be the set of integral ideals for the inner amplifier, as in \cite[\textsection 7.1.1]{Yan23c}; the outer amplifier uses $\mathfrak{L}$ from Theorem \ref{thm10.4}. Normalize the Archimedean tests by their spectral period lower bounds. By \cite[(7.19) and (7.21)]{Yan23c}, this contributes at most $C_{\infty}(\sigma)^{1+\varepsilon}C_{\infty}(\sigma\times\overline{\omega})^{1/2+\varepsilon}$ to the small and dual terms, and $C_{\infty}(\sigma)^{1/2+\varepsilon}C_{\infty}(\sigma\times\overline{\omega})^{1/4+\varepsilon}$ to the regular term. Define
\begin{align*}
f=\sum_{\mathfrak{m}_{1}\in\mathcal{L}}\sum_{\mathfrak{m}_{2}\in\mathcal{L}}\overline{\lambda_{\sigma}(\mathfrak{m}_{1})\lambda_{\sigma}(\mathfrak{m}_{2})}\sum_{\mathfrak{a}\mid(\mathfrak{m}_{1}+\mathfrak{m}_{2})}\bigotimes_{u\in S_{\sigma}}f_{u}\bigotimes\bigotimes_{v\not\in S_{\sigma}}f_{\mathfrak{m}_{1}\mathfrak{m}_{2}\mathfrak{a}^{-2},v},
\end{align*}
where for $v\not\in S_{\sigma}$, $f_{\mathfrak{m}_{1}\mathfrak{m}_{2}\mathfrak{a}^{-2},v}$ is the local test function constructed in \cite[\textsection 2.2]{Yan23c}. As a consequence, we have  
\begin{align*}
J_{\mathrm{Spec}}^{\mathrm{Reg},\heartsuit}(f,\mathbf{0},\overline{\omega})\gg_{F,\varepsilon} \frac{|L(1/2,\sigma\times\overline{\omega})|^{2}L^{2-\varepsilon}}{C(\sigma)^{\varepsilon}C(\omega)^{\varepsilon}},
\end{align*}
since the local tests below give a spectral contribution $\gg1$ at each $u\in S_{\sigma}$.

Let $u<\infty$ be a place where $\sigma_{u}$ is ramified, and put $k_{\mathbf{F}}=\mathcal{O}_{u}/\mathfrak{p}_{u}$. By assumption, $\sigma_{u}$ is supercuspidal of depth $0$, a twisted Steinberg representation $\mathrm{St}\otimes\xi_{u}$, or a unitary ramified principal series $\xi_{u}|\cdot|^{r_{u}}\boxplus\xi_{u}^{-1}|\cdot|^{-r_{u}}$, where $\xi_{u}$ is a unitary character of $F_{u}^{\times}$ with conductor exponent $1$ and $r_{u}\in[-\vartheta,\vartheta]$. In the principal-series case, $r_{u}\neq0$ implies $\xi_{u}^{2}=\mathbf{1}$.

\begin{itemize}
\item In the Steinberg and principal-series cases, use the character idempotent of a finite $K_{u}$-type. In the finite-group formulas below, $\xi_{u}$ denotes its induced character of $k_{\mathbf{F}}^{\times}$. If this character is nonquadratic, take $\tilde{\rho}_{u}=\operatorname{Ind}_{B(k_{\mathbf{F}})}^{G(k_{\mathbf{F}})}(\xi_{u}\otimes\xi_{u}^{-1})$, of dimension $q_{u}+1$; otherwise take $\tilde{\rho}_{u}=\mathrm{St}_{k_{\mathbf{F}}}\otimes(\xi_{u}\circ\det)$, of dimension $q_{u}$. Inflate to $K_{u}$ and extend trivially on $Z(F_{u})$, and set
\begin{align*}
f_{u}(g_{u})=(\dim\tilde{\rho}_{u})\int_{F_{u}^{\times}}\mathbf{1}_{K_{u}}(z_{u}g_{u})\chi_{\tilde{\rho}_{u}}(z_{u}g_{u})d^{\times}z_{u}.
\end{align*}
The characters are real and invariant under inversion, so $f_{u}$ is a positive projector. These types occur in $\sigma_{u}$, including the complementary series. Each contains every nontrivial character of $N(k_{\mathbf{F}})$ once, with Kirillov functions $\mathbf{1}_{\varpi_{u}^{-1}\alpha(1+\mathfrak{p}_{u})}$, $\alpha\in k_{\mathbf{F}}^{\times}$. Their normalized squared periods sum to $1$, so the local spectral contribution is at least $1$.

The finite-group character formulas give
\begin{align*}
J_{\mathrm{sm},u}(f_{u},s)=\frac{\dim\tilde{\rho}_{u}}{q_{u}-1}q_{u}^{1+s},\qquad J_{\mathrm{du},u}(f_{u},s)=\frac{\dim\tilde{\rho}_{u}}{q_{u}-1}q_{u}^{1-s}.
\end{align*}
In the range $e_{u}(t)=e_{u}(1-t)=0$, reduction modulo $\mathfrak{p}_{u}$ gives, up to bounded factors and $O(1)$ terms, the sums
\begin{align*}
\sum_{a\in k_{\mathbf{F}}-\{0,1,1-t\}}\xi_{u}\left(\frac{a(a+t-1)}{(a-1)(1-t)}\right),\qquad
\sum_{y\in k_{\mathbf{F}}^{\times}}\xi_{u}\big(y((y-1)^{2}+4yt)\big),
\end{align*}
respectively. The first has distinct zeros and poles at $0,1-t,1,\infty$, and in the second, $\xi_{u}$ is quadratic and the cubic polynomial has no repeated roots. The character-sum bound used below gives $O(q_{u}^{1/2})$ for both. In the other valuation ranges, use the decomposition below: the endpoint sums are $O(1)$ and each interior sum is $O(q_{u})$. These give
\begin{equation}\label{equ10.25}
\begin{cases}
J_{\mathrm{sm},u}(f_{u},s)\ll_{\varepsilon} q_{u}^{1+10\varepsilon},\ \ |s|=\varepsilon,\\
J_{\mathrm{du},u}(f_{u},s)\ll_{\varepsilon} q_{u}^{1+10\varepsilon},\ \ |s|=\varepsilon,\\
\mathcal{E}_{u}(t)\ll (1-e_{u}(t))^{2} q_{u}^{\frac{1}{2}-\frac{e_{u}(t)}{4}}\mathbf{1}_{e_{u}(t)<0}+ (e_{u}(t)+1)q_{u}^{\frac{1+e_{u}(t)}{2}}\mathbf{1}_{\substack{e_{u}(t)\geq 0\\
e_{u}(t-1)=0}}.
\end{cases}
\end{equation}

\item Suppose $\sigma_{u}$ is supercuspidal of depth $0$. Then $\sigma_{u}$ is induced from a representation $\rho_{u}$ of $Z(F_{u})K_{u}$ that is inflated from a cuspidal representation $\tilde{\rho}_{u}$ of $G(\mathcal{O}_{u}/\mathfrak{p}_{u})$ of dimension $q_{u}-1$. The representation $\tilde{\rho}_{u}$ is parameterized by Galois conjugacy classes of regular characters $\eta$ of $\mathbb{F}_{q_{u}^{2}}^{\times}$.

Let $k_{\mathbf{E}}\simeq\mathbb{F}_{q_{u}^{2}}$ be the quadratic extension of $k_{\mathbf{F}}$. We can regard $\eta$ as a character on $k_{\mathbf{E}}^{\times}/k_{\mathbf{F}}^{\times}$. The character $\chi_{\tilde{\rho}_{u}}$ of $\tilde{\rho}_{u}$ is given by \cite[(6.4.1)]{BH06}: 
\begin{equation}\label{e10.23}
\chi_{\tilde{\rho}_{u}}(g_{u})=\begin{cases}
q_{u}-1,\ & \text{if $g_{u}\sim Z(k_{\mathbf{F}})$,}\\
-1,\ & \text{if $g_{u}\sim Z(k_{\mathbf{F}})N(k_{\mathbf{F}})-Z(k_{\mathbf{F}})$,}\\
-\eta(\gamma_{u})-\eta(\gamma_{u}^{q_{u}}),\ & \text{if $g_{u}\sim k_{\mathbf{E}}-k_{\mathbf{F}}$,}\\
0,\ & \text{otherwise}. 
\end{cases}
\end{equation}
Here, $\gamma_{u}$ and $\gamma_{u}^{q_{u}}$ are the eigenvalues of $g_{u}$, and for a set $S$, $g_{u}\sim S$ means that $g_{u}$ is conjugate to an element of $S$. 

We can extend $\chi_{\tilde{\rho}_{u}}$ to a function on $K_{u}$ by congruence modulo $\mathfrak{p}_{u}$, with extension to $Z(F_{u})K_{u}$ by triviality on $Z(F_{u})$. Define the local test function by 
\begin{align*}
f_{u}(g_{u})=(q_{u}-1)\int_{F_{u}^{\times}}\mathbf{1}_{K_{u}}(z_{u}g_{u})\chi_{\tilde{\rho}_{u}}(z_{u}g_{u})d^{\times}z_{u},
\end{align*}
which is the character idempotent projecting onto the $\tilde{\rho}_{u}$-isotypic subspace of $\sigma_{u}$. Indeed, $\eta^{q_{u}}=\eta^{-1}$ on $k_{\mathbf{E}}^{\times}/k_{\mathbf{F}}^{\times}$, so $\chi_{\tilde{\rho}_{u}}$ is real-valued and $\chi_{\tilde{\rho}_{u}}(g^{-1})=\chi_{\tilde{\rho}_{u}}(g)$; thus the displayed formula agrees with the usual idempotent $(\dim\tilde{\rho}_{u})\chi_{\tilde{\rho}_{u}}(g^{-1})$ for the Haar measure giving $K_{u}$ volume one. This subspace contains a minimal vector whose corresponding Kirillov vector is 
\begin{equation}\label{eq10.25}
W_{u}\left(\begin{pmatrix}
y_{u}\\
& 1
\end{pmatrix}\right)=\mathbf{1}_{\varpi_{u}^{-1}\alpha_{u}(1+\mathfrak{p}_{u})}(y_{u}),\ \text{for some $\alpha_{u}\in k_{\mathbf{F}}^{\times}$}. 
\end{equation}

The translates by $\diag(a,1)$, $a\in\mathcal{O}_{u}^{\times}$, give vectors $W_{u,\alpha}$ of the form \eqref{eq10.25} for every $\alpha\in k_{\mathbf{F}}^{\times}$. Their Kirillov supports are disjoint, so they form an orthogonal basis of the $(q_{u}-1)$-dimensional type. Since $f_{u}$ projects onto this entire type and $\omega_{u}$ is unramified, its contribution to the local spectral side is the sum of the squared normalized periods of these vectors. Thus
\begin{equation}\label{10.29}
\sum_{\alpha\in k_{\mathbf{F}}^{\times}}\frac{1}{\langle W_{u,\alpha},W_{u,\alpha}\rangle_{u}}\bigg|\int_{F_{u}^{\times}}W_{u,\alpha}(a(y_{u}))d^{\times} y_{u}\bigg|^{2}=1.
\end{equation}

Notice that $f_{u}(g_{u})=f_{u}(wg_{u}w)$, where $w=\begin{pmatrix}
	& 1\\
1	& 
\end{pmatrix}\in K_{u}$. Hence, it follows from  \eqref{e10.23} that 
\begin{multline*}
f_{u}\left(\begin{pmatrix}
	y_{u}&\\
b_{u}&1
\end{pmatrix}\right)=f_{u}\left(\begin{pmatrix}
	y_{u}&b_{u}\\
&1
\end{pmatrix}\right)=-(q_{u}-1)\mathbf{1}_{1+\mathfrak{p}_{u}}(y_{u})\mathbf{1}_{\mathcal{O}_{u}-\mathfrak{p}_{u}}(b_{u})\\
+(q_{u}-1)^{2}\mathbf{1}_{1+\mathfrak{p}_{u}}(y_{u})\mathbf{1}_{\mathfrak{p}_{u}}(b_{u}).
\end{multline*}

The $b_{u}$-Fourier transform is supported on $e_{u}(x_{u})=-1$ and equals $(q_{u}-1)\mathbf{1}_{1+\mathfrak{p}_{u}}(y_{u})$ there. Since $\Vol(1+\mathfrak{p}_{u},d^{\times} y_{u})=(q_{u}-1)^{-1}$, the two integrals are $q_{u}^{1+s}$ and $q_{u}^{1-s}$, respectively. Thus
\begin{equation}\label{10.25}
\begin{cases}
J_{\mathrm{sm},u}(f_{u},s)\ll_{\varepsilon} q_{u}^{1+10\varepsilon},\ \ |s|=\varepsilon,\\
J_{\mathrm{du},u}(f_{u},s)\ll_{\varepsilon} q_{u}^{1+10\varepsilon},\ \ |s|=\varepsilon.
\end{cases}
\end{equation}

We treat the valuation ranges as in \cite[\textsection 6.2]{Yan23c}, with different finite-field character sums. Write $k=e_{u}(z_{u})$, $r_{1}=e_{u}(x_{u})$, and $r_{2}=e_{u}(y_{u})$. The support condition $z_{u}\begin{pmatrix}y_{u}&x_{u}^{-1}t\\x_{u}y_{u}&1\end{pmatrix}\in K_{u}$ gives
\begin{align*}
\begin{cases}
2k+r_{2}+e_{u}(1-t)=0,\\ 
0\leq k\leq-e_{u}(1-t),\\ 
k+e_{u}(1-t)\leq r_{1}\leq k+e_{u}(t).
\end{cases}
\end{align*}

Thus $e_{u}(1-t)\leq0$ and $e_{u}(t)\geq e_{u}(1-t)$. Suppose $e_{u}(1-t)<0$. Then $e_{u}(t)=e_{u}(1-t)$ and $r_{1}=k+e_{u}(t)$. After reduction modulo $\mathfrak{p}_{u}$ and conjugation by the unit part of $x_{u}$, the contribution for each $k$ is, up to a factor of absolute value one,
\begin{align*}
\sum_{\alpha\in k_{\mathbf{F}}^{\times}}\chi_{\tilde{\rho}_{u}}\left(\begin{pmatrix}\mathbf{1}_{k=-e_{u}(t)}\alpha&\overline{\varpi_{u}^{-e_{u}(t)}t}\\\alpha&\mathbf{1}_{k=0}\end{pmatrix}\right).
\end{align*}
Here the bar denotes reduction modulo $\mathfrak{p}_{u}$, and the factor $q_{u}-1$ in $f_{u}$ cancels the measure $(q_{u}-1)^{-1}$ of each residue class in $\mathcal{O}_{u}^{\times}$. 
\begin{itemize}
\item For $k=0$ and $k=-e_{u}(t)$, the two character sums agree after multiplication by the central scalar $\alpha^{-1}$ and the substitution $\alpha\mapsto\alpha^{-1}$. By \eqref{e10.23}, their common value is
\begin{align*}
-\sum_{\gamma\in k_{\mathbf{E}}:\  \operatorname{Tr}_{k_{\mathbf{E}}/k_{\mathbf{F}}}\gamma=1}\eta(\gamma),
\end{align*}
which has absolute value $1$ by character orthogonality, since $\eta$ is nontrivial and trivial on $k_{\mathbf{F}}^{\times}$. 
\item For $0<k<-e_{u}(t)$, the trivial bound is $O(q_{u})$. These terms occur only when $e_{u}(t)\leq-2$.
\end{itemize}

Combining these estimates and summing over $k$, we obtain
\begin{align*}
\mathcal{E}_{u}(t)\ll (1-e_{u}(t))^{2}q_{u}^{\frac{1}{2}-\frac{e_{u}(t)}{4}},\qquad e_{u}(t)=e_{u}(1-t)<0.
\end{align*}

It remains to consider $e_{u}(1-t)=0$. Choose a nonsquare $\epsilon_{u}\in k_{\mathbf{F}}^{\times}$ so that $k_{\mathbf{E}}=k_{\mathbf{F}}(\sqrt{\epsilon_{u}})$. Suppose first that $e_{u}(t)=0$. Write $\alpha=z_{u}(y_{u}+1)$ and $\beta^{2}\epsilon_{u}=z_{u}^{2}(y_{u}-1)^{2}+4z_{u}^{2}y_{u}t$. Then the eigenvalues of $z_{u}\begin{pmatrix} y_{u} & x_{u}^{-1}t \\ x_{u} y_{u} & 1 \end{pmatrix}$ are $(\alpha\pm \beta\sqrt{\epsilon_{u}})/2$. Since $\eta$ is trivial on $k_{\mathbf{F}}^{\times}$, the scalar factor $2$ does not affect the character values. Therefore,
\begin{align*}
|\mathcal{E}_{u}(t)|\ll q_{u}^{-1}\bigg|\sum_{\substack{z_{u}, y_{u}\in k_{\mathbf{F}}^{\times}\\ 
\alpha=z_{u}(y_{u}+1),\ 
\beta^{2}\epsilon_{u}=(\alpha-2z_{u})^{2}+4(z_{u}\alpha-z_{u}^{2})t}}\big[\eta(\alpha+\beta\sqrt{\epsilon_{u}})+\eta(\alpha-\beta\sqrt{\epsilon_{u}})\big]\bigg|.
\end{align*}
Here and below, the omitted terms with repeated eigenvalues contribute $O(1)$.

Let $\eta'$ be the quadratic character of $k_{\mathbf{F}}^{\times}$, extended by zero at $0$. Since $\beta^{2}\epsilon_{u}=(\alpha-2z_{u})^{2}+4(z_{u}\alpha-z_{u}^{2})t$ amounts to $(z_{u}-\alpha/2)^{2}=\frac{\beta^{2}\epsilon_{u}-\alpha^{2}t}{4(1-t)}$, the number of admissible $z_{u}$ is the sum of a constant term and the quadratic character $\eta'(\beta^{2}\epsilon_{u}-\alpha^{2}t)\eta'(1-t)$. The constant term factorizes after $\alpha\mapsto\beta\alpha$; the resulting one-variable sums $\sum_{\alpha\in k_{\mathbf{F}}^{\times}}\eta(\alpha\pm\sqrt{\epsilon_{u}})$ are nontrivial Kummer sums and are $O(q_{u}^{1/2})$. Their total contribution is therefore $O(q_{u}^{1/2})$. Hence, we obtain 
\begin{equation}\label{10.30}
|\mathcal{E}_{u}(t)|\ll q_{u}^{-1}\bigg|\sum_{\alpha, \beta\in k_{\mathbf{F}}^{\times}}\eta(\alpha+\beta\sqrt{\epsilon_{u}})\eta'(\beta^{2}\epsilon_{u}-\alpha^{2}t)\eta'(1-t)\bigg|+q_{u}^{1/2},
\end{equation}
Making the change of variable $\alpha\mapsto \beta\alpha$ in \eqref{10.30} leads to 
\begin{equation}\label{e10.30a}
|\mathcal{E}_{u}(t)|\ll \bigg|\sum_{\alpha\in k_{\mathbf{F}}^{\times}}\eta(\alpha+\sqrt{\epsilon_{u}})\eta'(\epsilon_{u}-\alpha^{2}t)\bigg|+q_{u}^{1/2}.
\end{equation}

The summand in \eqref{e10.30a} is a mixed multiplicative-character trace function of bounded conductor. The regularity condition $\eta\neq\eta^{q_{u}}$ rules out a geometrically trivial constituent; this can be checked from the divisors at the two conjugate points $\pm\sqrt{\epsilon_{u}}$, while the finitely many singular points contribute $O(1)$. Hence the Riemann hypothesis over finite fields for such character sums (see, for example, \cite{Kat89}, or the quasi-orthogonality formulation \cite[Theorem 4.1]{FKM14}) gives square-root cancellation in \eqref{e10.30a}, and therefore $\mathcal{E}_{u}(t)\ll q_{u}^{1/2}$ when $e_{u}(t)=e_{u}(1-t)=0$. If $e_{u}(t)=n\geq1$ and $e_{u}(1-t)=0$, the support has $k=r_{2}=0$ and $0\leq r_{1}\leq n$. The terms $r_{1}=0$ and $r_{1}=n$ contribute $-1$ each, and each term with $0<r_{1}<n$ contributes $q_{u}-1$: its reduction is diagonal, and only the scalar residue class contributes. Hence $\mathcal{E}_{u}(t)=(n-1)(q_{u}-1)-2\ll(n+1)q_{u}^{(n+1)/2}$. Combining all valuation ranges, we obtain
\begin{equation}\label{10.31}
\mathcal{E}_{u}(t)\ll (1-e_{u}(t))^{2}q_{u}^{\frac{1}{2}-\frac{e_{u}(t)}{4}}\mathbf{1}_{e_{u}(t)<0}+(e_{u}(t)+1)q_{u}^{\frac{e_{u}(t)+1}{2}}\mathbf{1}_{\substack{e_{u}(t)\geq0\\e_{u}(t-1)=0}}.
\end{equation}
By \eqref{10.29}, the normalized estimates \eqref{10.25} and \eqref{10.31} are bounded by the corresponding estimates in \eqref{equ10.25}.
\end{itemize}

Insert the ramified estimates \eqref{equ10.25}, \eqref{10.29}, \eqref{10.25} and \eqref{10.31} into \eqref{eq10.23} and use the remaining local estimates and global counting argument of \cite[\textsection 6--\textsection 7]{Yan23c} to obtain
\begin{multline}\label{10.27}
\frac{|L(1/2,\sigma\times\overline{\omega})|^{2}L^{2-\varepsilon}}{C(\sigma)^{\varepsilon}C(\omega)^{\varepsilon}}\ll_{F,\varepsilon} L^{1+\varepsilon}C_{\fin}(\sigma)^{\frac{1}{2}+\varepsilon}C_{\fin}(\overline{\omega})^{1+\varepsilon}C_{\infty}(\sigma)^{1+\varepsilon}C_{\infty}(\sigma\times\overline{\omega})^{\frac{1}{2}+\varepsilon}\\
+L^{3+\varepsilon}C_{\fin}(\sigma)^{\frac{1}{4}+\varepsilon}C_{\fin}(\overline{\omega})^{\frac{1}{2}+\varepsilon}C_{\infty}(\sigma)^{\frac{1}{2}+\varepsilon}C_{\infty}(\sigma\times\overline{\omega})^{\frac{1}{4}+\varepsilon}.
\end{multline}

Therefore, \eqref{10.22} follows from \eqref{10.27}.
\end{proof}

We also have the following variant using different test functions at the two primes. 
\begin{thm}\label{thm10.8}
Suppose $\sigma=\otimes_{v}\sigma_{v}$ is a unitary generic automorphic representation of $\mathrm{PGL}_{2}/F$, and let $\omega$ be a unitary Hecke character of $F^{\times}\backslash\mathbb{A}_{F}^{\times}$. Assume that
\begin{itemize}
\item the set $S_{\sigma}:=\{v<\infty:\ \text{$\sigma_{v}$ is ramified}\ \}$ has exactly two elements, say $S_{\sigma}=\{u_{1},u_{2}\}$;
\item $u_{1},u_{2}\nmid 2\mathfrak{D}_{F}$;
\item $\mathfrak{c}(\sigma)=\mathfrak{p}_{u_{1}}^{2}\mathfrak{p}_{u_{2}}^{2}$ and $\mathfrak{c}(\sigma)+\mathfrak{c}(\overline{\omega})=\mathcal{O}_{F}$. 
\end{itemize}
Let $\varepsilon>0$ and $L\gg_{F,\varepsilon} C_{\infty}(\sigma)^{\frac{1}{2}+\varepsilon}q_{u_{1}}^{2+\varepsilon}q_{u_{2}}^{2+\varepsilon}C_{\fin}(\overline{\omega})^{\varepsilon}$. Then 
\begin{multline}\label{f10.22}
|L(1/2,\sigma\times\overline{\omega})|^{2}\ll_{F,\varepsilon}  L^{-1+\varepsilon}q_{u_{1}}^{2+\varepsilon}q_{u_{2}}^{1+\varepsilon}C_{\fin}(\overline{\omega})^{1+\varepsilon}C_{\infty}(\sigma)^{1+\varepsilon}C_{\infty}(\sigma\times\overline{\omega})^{\frac{1}{2}+\varepsilon}\\
+L^{1+\varepsilon}q_{u_{1}}^{\varepsilon}q_{u_{2}}^{\frac{1}{2}+\varepsilon}C_{\fin}(\overline{\omega})^{\frac{1}{2}+\varepsilon}C_{\infty}(\sigma)^{\frac{1}{2}+\varepsilon}C_{\infty}(\sigma\times\overline{\omega})^{\frac{1}{4}+\varepsilon}.
\end{multline}
\end{thm}
\begin{proof}
In the relative trace formula, take the Gross--Prasad vector at $u_{1}$ and the test function constructed in the proof of Theorem \ref{thm10.7} at $u_{2}$. The period normalization and orbital-integral estimates in \cite[\textsection 3.6.1 and \textsection 4--\textsection 6]{Yan23c} give the factors $q_{u_{1}}^{2+\varepsilon}$ in the small and dual terms and $q_{u_{1}}^{\varepsilon}$ in the regular term. At $u_{2}$, \eqref{equ10.25}--\eqref{10.31} give, with the same normalization, $q_{u_{2}}^{1+\varepsilon}$ and $q_{u_{2}}^{1/2+\varepsilon}$, respectively. Thus the products of the normalized local factors are $q_{u_{1}}^{2+\varepsilon}q_{u_{2}}^{1+\varepsilon}$ and $q_{u_{1}}^{\varepsilon}q_{u_{2}}^{1/2+\varepsilon}$. Substitution into \eqref{eq10.23}, followed by the same amplification as in the proof of Theorem \ref{thm10.7}, gives \eqref{f10.22}.
\end{proof}

\begin{remark}
Theorem~\ref{thm10.7} alone suffices for the uniform saving $1/120$. Theorem~\ref{thm10.8} gives $1/124$ in the balanced range $C(\pi)\asymp C(\omega)$ but yields the sharper hybrid bounds in the intermediate conductor ranges of Theorem~\ref{thmB}.
\end{remark}

\subsection{Proofs of Theorems \ref{thmB} and \ref{thmC}, Corollary \ref{cor1.1}, and Theorem \ref{cor1.7}}

\subsubsection{Proof of  Theorem \ref{thmB}}
The optimizing amplifier lengths below are polynomial in $C(\pi)$. By convexity and the weighted Weyl law, the exponential weight in $\mathcal{I}_{1}$ and $\mathcal{I}_{2}$ makes the contribution of $C_{\infty}(\sigma)>C(\pi)^{10^{-2}\varepsilon}$ negligible, even after summing over the amplifier. We may therefore assume $C_{\infty}(\sigma)\leq C(\pi)^{10^{-2}\varepsilon}$, so $C_{\infty}(\sigma\times\overline{\omega})\ll C_{\infty}(\omega)^{2}C(\pi)^{\varepsilon}$. Let $L_{\mathrm{RTF}}$ denote the amplifier length in Theorem~\ref{thm10.7}. Balancing the two terms in \eqref{10.22} gives
\begin{align*}
L_{\mathrm{RTF}}=C_{\fin}(\sigma)^{1/8}C_{\fin}(\overline{\omega})^{1/4}C_{\infty}(\sigma)^{1/4}C_{\infty}(\sigma\times\overline{\omega})^{1/8}.
\end{align*}
Using $C_{\fin}(\omega)C_{\infty}(\omega)\leq C(\omega)$ and absorbing $C(\pi)^{O(\varepsilon)}$-factors, we optimize the resulting upper bound subject to the lower bound in Theorem \ref{thm10.7}. The two choices of length give
\begin{multline}\label{10.23}
L(1/2,\sigma\times\overline{\omega})\ll_{F,\varepsilon} C_{\fin}(\sigma)^{\frac{3}{16}+\varepsilon}C(\omega)^{\frac{3}{8}+\varepsilon}C(\pi)^{\varepsilon}\mathbf{1}_{C_{\fin}(\sigma)\leq C(\omega)^{2/7}}\\
+C_{\fin}(\sigma)^{\frac{5}{8}+\varepsilon}C(\omega)^{\frac{1}{4}+\varepsilon}C(\pi)^{\varepsilon}\mathbf{1}_{C_{\fin}(\sigma)\geq C(\omega)^{2/7}}.
\end{multline}

For Theorem \ref{thm10.8}, the same conductor comparison gives the balancing length for the resulting upper bound, up to $C(\pi)^{O(\varepsilon)}$,
\begin{align*}
L_{\mathrm{RTF}}=q_{u_{1}}q_{u_{2}}^{1/4}C(\omega)^{1/4}.
\end{align*}
It satisfies the lower bound $L_{\mathrm{RTF}}\gg_{F,\varepsilon} q_{u_{1}}^{2}q_{u_{2}}^{2}C(\pi)^{O(\varepsilon)}$ in the range $q_{u_{1}}^{4}q_{u_{2}}^{7}\leq C(\omega)$; in the complementary range we again use the smallest admissible value. Thus
\begin{multline}\label{f10.23}
L(1/2,\sigma\times\overline{\omega})\ll_{F,\varepsilon} q_{u_{1}}^{\frac{1}{2}+\varepsilon}q_{u_{2}}^{\frac{3}{8}+\varepsilon}C(\omega)^{\frac{3}{8}+\varepsilon}C(\pi)^{\varepsilon}\mathbf{1}_{q_{u_{1}}^{4}q_{u_{2}}^{7}\leq C(\omega)}\\
+q_{u_{1}}^{1+\varepsilon}q_{u_{2}}^{\frac{5}{4}+\varepsilon}C(\omega)^{\frac{1}{4}+\varepsilon}C(\pi)^{\varepsilon}\mathbf{1}_{q_{u_{1}}^{4}q_{u_{2}}^{7}\geq C(\omega)}.
\end{multline}

We now separate the amplified sums according to conductor. In an off-diagonal term, write $\mathfrak{n}=\mathfrak{p}_{1}^{\ell_{1}}\mathfrak{p}_{2}^{\ell_{2}}$ with $\ell_{i}\in\{1,2\}$ and $\mathfrak{p}_{1}\neq\mathfrak{p}_{2}$. If $\mathfrak{c}(\sigma)=\mathfrak{p}_{1}^{a_{1}}\mathfrak{p}_{2}^{a_{2}}$, then $0\leq a_{i}\leq\ell_{i}$. Since $N_{F}(\mathfrak{p}_{i})\asymp L$, the case $(a_{1},a_{2})=(2,2)$ has $C_{\fin}(\sigma)\asymp L^{4}$, whereas all other off-diagonal terms have $C_{\fin}(\sigma)\ll L^{3}$. In these other cases, \eqref{10.21} gives
\begin{equation}\label{f10.23a}
L(1/2,\sigma\times\overline{\omega})\ll_{F,\varepsilon} C(\pi)^{\varepsilon}\big[L^{\frac{3}{4}+\varepsilon}C(\omega)^{\frac{3}{8}+\varepsilon}+L^{\frac{3}{2}+\varepsilon}C(\omega)^{\frac{1}{4}+\varepsilon}\big].
\end{equation}
If $L\leq C(\omega)^{1/14}$, the right-hand side of \eqref{f10.23a} is bounded by the first term of \eqref{10.23}, evaluated at $C_{\fin}(\sigma)\asymp L^{4}$; if $L\geq C(\omega)^{1/14}$, it is bounded by the second term. Indeed, the only nontrivial comparisons are $L^{3/2}C(\omega)^{1/4}\leq L^{3/4}C(\omega)^{3/8}$ in the first range and $L^{3/4}C(\omega)^{3/8}\leq L^{5/2}C(\omega)^{1/4}$ in the second. For \eqref{f10.23}, the comparisons require only $L\leq C(\omega)^{1/5}$ in the first range and $L\geq C(\omega)^{1/12}$ in the second; these follow from $L\leq C(\omega)^{1/11}$ and $L\geq C(\omega)^{1/11}$, respectively. To pass from these pointwise estimates to the cuspidal dual sum, apply H\"older's inequality, Theorem \ref{thmE}, and the weighted Weyl bound separately for each conductor. These estimates multiply the twisted bound by $N_{F}(\mathfrak{q}\mathfrak{q}'^{-1})^{\vartheta+\varepsilon}N_{F}(\mathfrak{n})^{1+\varepsilon}$. The coefficient of the dual moment in Theorem \ref{thm10.3} reduces the power of $N_{F}(\mathfrak{n})$ to $1/2+\varepsilon$. Summing over the $O(L^{2})$ pairs gives the final term below, with the other off-diagonal contributions bounded by the same estimate. For the continuous spectrum, we use the sharper estimate of Lemma \ref{lem10.2}. For $N_{F}(\mathfrak{n})\asymp L^{4}$, its total off-diagonal contribution is
\begin{align*}
C(\pi)^{1/2+\vartheta+\varepsilon}C(\omega)^{-\vartheta}L^{19/4+\varepsilon}C(\omega)^{3/8+\varepsilon}.
\end{align*}
This equals the last term in \eqref{10.35} in its first range and is absorbed by that term in the second range because $L\geq C(\omega)^{1/14}$. It is also absorbed by the last term in \eqref{f10.35}: in the first range there is an additional factor $L^{1/8}$, and in the second one uses $L\geq C(\omega)^{1/11}$. The diagonal terms, for which the exponent of a single prime may be as large as $4$, occur only $O(L)$ times and are bounded by \eqref{10.8}; their contribution is the middle term with leading factor $L$ in \eqref{10.35} and \eqref{f10.35}. For each off-diagonal term, the number of possible conductors and the oldform dimensions at the two primes are bounded in terms of the differences $\ell_{i}-a_{i}\leq2$; hence their total multiplicity is $O(1)$.

Combining \eqref{10.8}, \eqref{10.21}, and \eqref{10.23} in \eqref{10.10}, we derive 
\begin{multline}\label{10.35}
\mathbf{C}_{\infty}\cdot \mathcal{J}_{\mathrm{\Spec}}^{\heartsuit}(\boldsymbol{\alpha}_{\pi},\mathfrak{L})\ll_{F,\varepsilon} C(\pi)^{1+\varepsilon}L+C(\pi)^{1/2+\vartheta+\varepsilon}C(\omega)^{-\vartheta}L\\
\min\big\{
L^{2+\varepsilon}C(\omega)^{\frac{1}{2}+\varepsilon}, L^{3+\varepsilon} C(\omega)^{\frac{3}{8}+\varepsilon}+L^{4+\varepsilon}C(\omega)^{\frac{1}{4}+\varepsilon}
\big\}+C(\pi)^{1/2+\vartheta+\varepsilon}\\
C(\omega)^{-\vartheta}L^{2}\Big[L^{\frac{11}{4}+\varepsilon} C(\omega)^{\frac{3}{8}+\varepsilon}\mathbf{1}_{L\leq C(\omega)^{1/14}}+L^{\frac{9}{2}+\varepsilon}C(\omega)^{\frac{1}{4}+\varepsilon}\mathbf{1}_{L\geq C(\omega)^{1/14}}
\Big].
\end{multline}

Similarly, using \eqref{f10.23} in place of \eqref{10.23} gives
\begin{multline}\label{f10.35}
\mathbf{C}_{\infty}\cdot \mathcal{J}_{\mathrm{\Spec}}^{\heartsuit}(\boldsymbol{\alpha}_{\pi},\mathfrak{L})\ll_{F,\varepsilon} C(\pi)^{1+\varepsilon}L+C(\pi)^{1/2+\vartheta+\varepsilon}C(\omega)^{-\vartheta}L\\
\min\big\{
L^{2+\varepsilon}C(\omega)^{\frac{1}{2}+\varepsilon}, L^{3+\varepsilon} C(\omega)^{\frac{3}{8}+\varepsilon}+L^{4+\varepsilon}C(\omega)^{\frac{1}{4}+\varepsilon}
\big\}+C(\pi)^{1/2+\vartheta+\varepsilon}\\
C(\omega)^{-\vartheta}L^{2}\Big[L^{\frac{23}{8}+\varepsilon} C(\omega)^{\frac{3}{8}+\varepsilon}\mathbf{1}_{L\leq C(\omega)^{1/11}}+L^{\frac{17}{4}+\varepsilon}C(\omega)^{\frac{1}{4}+\varepsilon}\mathbf{1}_{L\geq C(\omega)^{1/11}}
\Big].
\end{multline}

We now optimize \eqref{10.35} and \eqref{f10.35}. The middle term in \eqref{10.35} is absorbed by its final term: in the first range, estimate the minimum by its second entry and use $L\leq C(\omega)^{1/14}\leq C(\omega)^{1/2}$; in the second, estimate it by $L^{2}C(\omega)^{1/2}$ and use $L\geq C(\omega)^{1/14}$. The same comparison in \eqref{f10.35} uses $L\leq C(\omega)^{1/11}\leq C(\omega)$ in the first range and $L\geq C(\omega)^{1/11}$ in the second. From the first and second ranges in \eqref{10.35}, respectively, we obtain
\begin{align}
L(1/2,\pi)&\ll_{F,\varepsilon} C(\pi)^{\frac{13}{60}+\frac{\vartheta}{15}+\varepsilon}C(\omega)^{\frac{1}{40}-\frac{\vartheta}{15}},\qquad C(\pi)\leq C(\omega)^{\frac{9-14\vartheta}{7(1-2\vartheta)}},\label{10.38}\\
L(1/2,\pi)&\ll_{F,\varepsilon} C(\pi)^{\frac{5}{22}+\frac{\vartheta}{22}+\varepsilon}C(\omega)^{\frac{1}{88}-\frac{\vartheta}{22}},\qquad C(\pi)\geq C(\omega)^{\frac{9-14\vartheta}{7(1-2\vartheta)}}.\label{f10.41}
\end{align}
Indeed, the optimizing parameter in the first range is $L=C(\pi)^{\frac{2}{15}-\frac{4\vartheta}{15}}C(\omega)^{-\frac{1}{10}+\frac{4\vartheta}{15}}$. In the second range it is $L=C(\pi)^{\frac{1-2\vartheta}{11}}C(\omega)^{-\frac{1/2-2\vartheta}{11}}$. Their admissibility conditions $L\leq C(\omega)^{1/14}$ and $L\geq C(\omega)^{1/14}$ are both equivalent to the displayed transition point. Likewise, the two ranges in \eqref{f10.35} give
\begin{align*}
L(1/2,\pi)&\ll_{F,\varepsilon} C(\pi)^{\frac{1}{4}-\frac{1-2\vartheta}{31}+\varepsilon}C(\omega)^{\frac{3}{124}-\frac{2\vartheta}{31}},\qquad C(\pi)\leq C(\omega)^{\frac{2(8-11\vartheta)}{11(1-2\vartheta)}},\\
L(1/2,\pi)&\ll_{F,\varepsilon} C(\pi)^{\frac{19}{84}+\frac{\vartheta}{21}+\varepsilon}C(\omega)^{\frac{1}{84}-\frac{\vartheta}{21}},\qquad C(\pi)\geq C(\omega)^{\frac{2(8-11\vartheta)}{11(1-2\vartheta)}}.
\end{align*}
Here the optimizing parameter in the first range is $L=C(\pi)^{\frac{4(1-2\vartheta)}{31}}C(\omega)^{-\frac{3}{31}+\frac{8\vartheta}{31}}$. In the second range it is $L=C(\pi)^{\frac{2-4\vartheta}{21}}C(\omega)^{-\frac{1-4\vartheta}{21}}$. Their transition point is $L=C(\omega)^{1/11}$.

The first two bounds intersect when $C(\pi)=C(\omega)^{\frac{9-14\vartheta}{7(1-2\vartheta)}}$, the middle two intersect when $C(\pi)=C(\omega)^{\frac{35-52\vartheta}{26(1-2\vartheta)}}$, and the last two intersect when $C(\pi)=C(\omega)^{\frac{2(8-11\vartheta)}{11(1-2\vartheta)}}$. Moreover,
\begin{align*}
1\leq\frac{9-14\vartheta}{7(1-2\vartheta)}\leq\frac{35-52\vartheta}{26(1-2\vartheta)}\leq\frac{2(8-11\vartheta)}{11(1-2\vartheta)},\qquad 0\leq\vartheta\leq7/64.
\end{align*}
Since $C(\omega)\leq C(\pi)$, comparison on the admissible ranges gives the lower envelope of these two piecewise bounds in the following order:
\begin{multline}\label{10.46}
L(1/2,\pi)\ll_{F,\varepsilon} C(\pi)^{\frac{13}{60}+\frac{\vartheta}{15}+\varepsilon}C(\omega)^{\frac{1}{40}-\frac{\vartheta}{15}}\mathbf{1}_{C(\pi)\leq C(\omega)^{\frac{9-14\vartheta}{7(1-2\vartheta)}}}\\
+C(\pi)^{\frac{5}{22}+\frac{\vartheta}{22}+\varepsilon}C(\omega)^{\frac{1}{88}-\frac{\vartheta}{22}}\mathbf{1}_{C(\omega)^{\frac{9-14\vartheta}{7(1-2\vartheta)}}<C(\pi)\leq C(\omega)^{\frac{35-52\vartheta}{26(1-2\vartheta)}}}\\
+C(\pi)^{\frac{1}{4}-\frac{1-2\vartheta}{31}+\varepsilon}C(\omega)^{\frac{3}{124}-\frac{2\vartheta}{31}}\mathbf{1}_{C(\omega)^{\frac{35-52\vartheta}{26(1-2\vartheta)}}<C(\pi)\leq C(\omega)^{\frac{2(8-11\vartheta)}{11(1-2\vartheta)}}}\\
+C(\pi)^{\frac{19}{84}+\frac{\vartheta}{21}+\varepsilon}C(\omega)^{\frac{1}{84}-\frac{\vartheta}{21}}\mathbf{1}_{C(\pi)>C(\omega)^{\frac{2(8-11\vartheta)}{11(1-2\vartheta)}}}.
\end{multline}

Taking the minimum of \eqref{10.46} and the global bounds \eqref{10.15} and \eqref{10.20} gives \eqref{1.1}. This proves Theorem \ref{thmB}.
\begin{remark}\label{rmk10.9}
Let $\varepsilon_{1} = 10^{-5}\varepsilon$.
Applying Theorem \ref{thmB} to the unitary twists $\pi\otimes|\cdot|^{it}$ gives its bounds on the critical line. The functional equation and the Phragm\'en--Lindel\"of principle extend these bounds to $|s-1/2|\leq\varepsilon_{1}$, with an additional factor $C(\pi)^{O(\varepsilon_{1})}$. This factor is absorbed into $C(\pi)^{\varepsilon}$ after rescaling $\varepsilon$. Cauchy's integral formula then gives the corresponding derivative bounds.
\end{remark}

\subsubsection{Proof of Theorem \ref{cor1.7}}
In this case, we use a different argument to bound $\mathcal{I}_{1}$ and $\mathcal{I}_{2}$. Put
\begin{align*}
A_{0}(\mathfrak{n})&:=\sum_{\substack{\sigma\in \mathcal{F}_{0}(\mathfrak{n},\mathbf{1})\\ \sigma_{\infty}\ \mathrm{unramified}}}e^{-\frac{\pi}{100}\sum_{v\mid\infty}|\nu_{\sigma_{v}}|}|L(1/2,\sigma\otimes\omega)|^{3},\\
A_{\mathrm{Eis}}(\mathfrak{n})&:=\sum_{\substack{\mu\in\mathfrak{X}(\mathfrak{n},\mathbf{1})\\ \mu_{\infty}\ \mathrm{unramified}}}\int_{\mathbb{R}}e^{-\frac{\pi}{100}\sum_{v\mid\infty}|\nu_{\sigma_{\mu_{v},it}}|}\frac{|L(1/2+it,\mu\omega)L(1/2-it,\overline{\mu}\omega)|^{3}}{|L(1+2it,\mu^{2})|^{2}}dt,\\
B_{\mathrm{Eis}}(\mathfrak{n})&:=\sum_{\substack{\mu\in\mathfrak{X}(\mathfrak{n},\mathbf{1})\\ \mu_{\infty}\ \mathrm{unramified}}}\int_{\mathbb{R}}e^{-\frac{\pi}{100}\sum_{v\mid\infty}|\nu_{\sigma_{\mu_{v},it}}|}\frac{|L(1/2+it,\mu)L(1/2-it,\overline{\mu})|^{9/2}}{|L(1+2it,\mu^{2})|^{2}}dt.
\end{align*}
\begin{lemma}
Suppose that $\omega^{2}=\mathbf{1}$. Uniformly for the integral ideals $\mathfrak{n}$ occurring in the amplification, we have
\begin{equation}\label{10.47}
A_{0}(\mathfrak{n})+A_{\mathrm{Eis}}(\mathfrak{n})\ll_{F,\varepsilon} N_{F}(\mathfrak{n})^{1+\varepsilon}C(\omega)^{1+\varepsilon},\quad B_{\mathrm{Eis}}(\mathfrak{n})\ll_{F,\varepsilon} N_{F}(\mathfrak{n})^{\frac{9}{8}-\frac{1-2\vartheta}{48}+\varepsilon}.
\end{equation}
\end{lemma}
\begin{proof}
The first estimate follows from Nelson's cubic-moment argument \cite[\textsection 3--\textsection 4 and \textsection 13]{Nel19}, with the level dependence explicit in the local weights. Since $\mathfrak{n}+\mathfrak{c}(\omega)=\mathcal{O}_{F}$, the level weights at primes dividing $\mathfrak{n}$ and the quadratic-twist weights at primes dividing $\mathfrak{c}(\omega)$ give the factor $N_{F}(\mathfrak{n})^{1+\varepsilon}C(\omega)^{1+\varepsilon}$. The divisor bound absorbs the sum over conductors and the oldform dimensions, and the exponential weight absorbs the polynomial Archimedean dependence. The identity includes the Eisenstein spectrum; on the cuspidal spectrum, $L(1,\sigma,\Ad)\ll C(\sigma)^{\varepsilon}$ removes the factor $L(1,\sigma,\Ad)^{-1}$.

For the second estimate, $C_{\fin}(\mu)\leq N_{F}(\mathfrak{n})^{1/2}$, so the Hecke-character Burgess bound \cite[Corollary 1.2]{Yan23c} gives $|L(1/2+it,\mu)|\ll N_{F}(\mathfrak{n})^{3/32+\varepsilon}$ up to polynomial Archimedean factors. Thus $|L(1/2,\sigma_{\mu,it})|^{1/2}\ll N_{F}(\mathfrak{n})^{3/32+\varepsilon}$, again up to polynomial Archimedean factors. Apply Theorem \ref{thmF} to the remaining fourth power; the exponential weight absorbs the Archimedean factors. The resulting exponent $35/32$ proves the second estimate in \eqref{10.47}.
\end{proof}

By Corollary \ref{cor1.6} and the estimate \eqref{10.5}, we obtain
\begin{equation}\label{e10.48}
\sum_{\substack{\sigma\in \mathcal{F}_{0}(\mathfrak{n},\mathbf{1})\\ \sigma_{\infty}\ \mathrm{unramified}}}e^{-\frac{\pi}{100}\sum_{v\mid\infty}|\nu_{\sigma_{v}}|}|L(1/2,\sigma)|^{\frac{9}{2}}\ll_{F,\varepsilon} N_{F}(\mathfrak{n})^{\frac{9}{8}-\frac{1-2\vartheta}{48}+\varepsilon}.
\end{equation}

Combining \eqref{10.47} and \eqref{e10.48} with H\"{o}lder's inequality, we obtain
\begin{equation}\label{10.48}
\mathcal{I}_{1}\ll_{F,\varepsilon} N_{F}(\mathfrak{q}\mathfrak{q}'^{-1})^{\vartheta+\varepsilon}N_{F}(\mathfrak{n})^{\frac{13}{12}-\frac{1-2\vartheta}{72}+\varepsilon}C(\omega)^{\frac{1}{3}+\varepsilon}.
\end{equation}

The two Eisenstein estimates in \eqref{10.47} similarly give
\begin{equation}\label{e10.49}
\mathcal{I}_{2}\ll_{F,\varepsilon} N_{F}(\mathfrak{q}\mathfrak{q}'^{-1})^{\vartheta}N_{F}(\mathfrak{n})^{\frac{13}{12}-\frac{1-2\vartheta}{72}+\varepsilon}C(\omega)^{\frac{1}{3}+\varepsilon}.
\end{equation}

Recall the estimate from the proof of Theorem \ref{thm10.3}:
\begin{multline*}
\mathcal{J}_{\mathrm{\Spec}}^{\heartsuit}(\mathbf{0},\mathfrak{X}_{\mathfrak{n}})\ll_{F,\varepsilon} N_{F}(\mathfrak{n})^{\frac{1}{2}+\varepsilon}C(\omega)^{\frac{1}{2}+\varepsilon}\mathbf{C}_{\infty}^{-1+\varepsilon}N_{F}(\mathfrak{q})^{\varepsilon}
+N_{F}(\mathfrak{n})^{-\frac{1}{2}+\varepsilon}N_{F}(\mathfrak{q})^{1+\varepsilon}\mathbf{C}_{\infty}^{\varepsilon}\\
+N_{F}(\mathfrak{q})^{\frac{1}{2}+\varepsilon}N_{F}(\mathfrak{n})^{-\frac{1}{2}+\varepsilon}\big[\mathbf{C}_{\infty}C_{\infty}(\omega)^{-1}\big]^{\vartheta}\mathbf{C}_{\infty}^{-\frac{1}{2}+\varepsilon}\cdot (\mathcal{I}_{1}+\mathcal{I}_{2}).
\end{multline*}
Substituting \eqref{10.48} and \eqref{e10.49} into this estimate yields
\begin{multline*}
\mathcal{J}_{\mathrm{\Spec}}^{\heartsuit}(\mathbf{0},\mathfrak{X}_{\mathfrak{n}})\ll_{F,\varepsilon} N_{F}(\mathfrak{n})^{\frac{1}{2}+\varepsilon}C(\omega)^{\frac{1}{2}+\varepsilon}\mathbf{C}_{\infty}^{-1+\varepsilon}N_{F}(\mathfrak{q})^{\varepsilon}
+N_{F}(\mathfrak{n})^{-\frac{1}{2}+\varepsilon}N_{F}(\mathfrak{q})^{1+\varepsilon}\mathbf{C}_{\infty}^{\varepsilon}\\
+N_{F}(\mathfrak{q})^{\frac{1}{2}+\varepsilon}N_{F}(\mathfrak{n})^{\frac{7}{12}-\frac{1-2\vartheta}{72}+\varepsilon}\big[\mathbf{C}_{\infty}C_{\infty}(\omega)^{-1}\big]^{\vartheta}\mathbf{C}_{\infty}^{-\frac{1}{2}+\varepsilon}\cdot N_{F}(\mathfrak{q}\mathfrak{q}'^{-1})^{\vartheta+\varepsilon}C(\omega)^{\frac{1}{3}+\varepsilon}.
\end{multline*}

After summing over the amplifier, the contribution of the first term in this display is $\ll L^{4}C(\omega)^{1/2+\varepsilon}$. It is absorbed by the contribution of the last term, since $C(\omega)\leq C(\pi)$ and $13/3-(1-2\vartheta)/18>4$. Substituting the resulting bound into \eqref{10.10}, we get
\begin{equation}
\label{e10.50}
\mathbf{C}_{\infty}\cdot \mathcal{J}_{\mathrm{\Spec}}^{\heartsuit}(\boldsymbol{\alpha}_{\pi},\mathfrak{L})\ll_{F,\varepsilon} C(\pi)^{1+\varepsilon}L
+C(\pi)^{\frac{1}{2}+\vartheta+\varepsilon}C(\omega)^{-\vartheta}C(\omega)^{\frac{1}{3}+\varepsilon}L^{\frac{13}{3}-\frac{1-2\vartheta}{18}}.
\end{equation}

Combining \eqref{10.13} with \eqref{e10.50} gives
\begin{equation}\label{e10.51}
|L(1/2,\pi)|^{4}\ll_{F,\varepsilon} C(\pi)^{1+\varepsilon}L^{-1+\varepsilon}
+C(\pi)^{\frac{1}{2}+\vartheta+\varepsilon}C(\omega)^{-\vartheta}C(\omega)^{\frac{1}{3}+\varepsilon}L^{\frac{7}{3}-\frac{1-2\vartheta}{18}+\varepsilon}. 
\end{equation}

Optimizing by choosing
\begin{align*}
C(\pi)
=C(\pi)^{\frac{1}{2}+\vartheta}C(\omega)^{-\vartheta}C(\omega)^{\frac{1}{3}}L^{\frac{10}{3}-\frac{1-2\vartheta}{18}}
\end{align*}
in \eqref{e10.51} yields the desired bound \eqref{1.8}. This proves Theorem \ref{cor1.7}.

\begin{remark}
If the fifth moment estimate
\begin{align*}
\sum_{\substack{\sigma\in \mathcal{F}_{0}(\mathfrak{n},\mathbf{1})\\ \sigma_{\infty}\ \mathrm{unramified}}}e^{-\frac{\pi}{100}\sum_{v\mid\infty}|\nu_{\sigma_{v}}|}|L(1/2,\sigma)|^{5}\ll_{F,\varepsilon} N_{F}(\mathfrak{n})^{1+\varepsilon},
\end{align*}
together with its Eisenstein analogue were available, then \eqref{e10.51} would improve to
\begin{align*}
|L(1/2,\pi)|^{4}\ll_{F,\varepsilon} C(\pi)^{1+\varepsilon}L^{-1+\varepsilon}
+C(\pi)^{\frac{1}{2}+\vartheta+\varepsilon}C(\omega)^{-\vartheta}C(\omega)^{\frac{1}{3}+\varepsilon}L^{2+\varepsilon},
\end{align*}
which would yield the subconvexity bound
\begin{align*}
L(1/2,\pi)\ll_{F,\varepsilon} C(\pi)^{\frac{1}{4}-\frac{1}{24}+\frac{\vartheta}{12}+\varepsilon}C(\omega)^{-\frac{\vartheta}{12}+\frac{1}{36}}.
\end{align*}

Corollary \ref{cor1.fifth} treats representations of exact prime conductor, whereas this argument requires all representations whose conductor divides the amplifier level $\mathfrak{n}$. For related prime and prime-power results, see \cite{BK19,KY23,Yan25}.
\end{remark}

\subsubsection{Proof of Corollary \ref{cor1.1}}
The two bounds \eqref{10.38} and \eqref{f10.41} cover all conductor ranges. Their $C(\omega)$-exponents are nonnegative for $0\leq\vartheta\leq7/64$, and their total exponents are $29/120$ and $21/88<29/120$, respectively. Since $C(\omega)\leq C(\pi)$, they give
\begin{align*}
L(1/2,\pi)\ll_{F,\varepsilon} C(\pi)^{\frac{1}{4}-\frac{1}{120}+\varepsilon}.
\end{align*}

By Remark \ref{rmk10.9}, after rescaling $\varepsilon$, one obtains
\begin{align*}
L\bigl(1/2+\lambda,\pi\bigr)
\ll_{F,\varepsilon}
C(\pi)^{\frac{1}{4}-\frac{1}{120}+\varepsilon},
\qquad |\lambda|=\varepsilon_{1}.
\end{align*}
Cauchy's integral formula proves Corollary \ref{cor1.1}.

\subsubsection{Proof of Theorem \ref{thmC}}
Follow the amplifier argument of \cite[\textsection 14--\textsection 15]{HY26}. Choose the two amplifier primes away from $2\mathfrak{D}_{F}\mathfrak{c}(\pi')$ and use the local tests from the proof of Theorem \ref{thm10.7} when the dual conductor exponent is $2$. Retain the ordinary local tests at primes dividing $\mathfrak{c}(\pi')$. There the target conductor exponent is at most $e_{v}(\mathfrak{c}(\pi'))$, and the ordinary test level is the maximum of the target and twisting levels. The extra level and valuation shifts therefore cost $q_{v}^{O(e_{v}(\mathfrak{c}(\pi')))}$, without changing the twisting-conductor exponents. Their product, the Archimedean weights, and the amplifier conditions cost only a polynomial in $C(\pi')$. The smaller conductors, diagonal terms, and corrections are handled as in the proof of Theorem \ref{thmB}. Using \eqref{10.23} and $C(\omega\omega')\ll C(\pi)C(\pi')$, where $\omega$ and $\omega'$ are the central characters, gives, for some $A>0$,
\begin{equation}\label{eq10.C.amplifier}
|L(1/2,\pi\times\pi')|^{2}\ll_{F,\varepsilon}C(\pi')^{A}\big[C(\pi)^{1+\varepsilon}L^{-1}+C(\pi)^{7/8+\varepsilon}L^{11/4}+C(\pi)^{3/4+\varepsilon}L^{9/2}\big].
\end{equation}
Take $L=C(\pi)^{1/30}$, enlarged by a polynomial in $C(\pi')$ and a sufficiently small $C(\pi)^{\varepsilon}$-factor to meet the amplifier conditions. The three terms in \eqref{eq10.C.amplifier} are bounded by a polynomial in $C(\pi')$ times $C(\pi)^{29/30+\varepsilon}$, $C(\pi)^{29/30+\varepsilon}$, and $C(\pi)^{9/10+\varepsilon}$, respectively. Taking square roots gives
\begin{align*}
L(1/2,\pi\times\pi')
\ll_{F,\varepsilon}
C(\pi')^{A}C(\pi)^{\frac{1}{2}-\frac{1}{60}+\varepsilon}.
\end{align*}
Unitary twists and the argument of Remark \ref{rmk10.9} give the derivative bounds, with polynomial dependence on $C(\pi')$.

\subsection{Dependence on the Number Field}\label{sec10.field}
Michel and Venkatesh \cite[Theorem~1.2]{MV10} obtain polynomial dependence on $D_{F}$ for fields of fixed degree. The following bound gives the saving $1/120$ of Corollary \ref{cor1.1} without restricting the degree. The analytic conductor includes $D_{F}^{2}$, as in \textsection\ref{sec4.1.1}.

\begin{thm}
For every $\varepsilon>0$, there is $A_{\varepsilon}>0$, depending only on $\varepsilon$, such that, for every number field $F$ and every unitary cuspidal representation $\pi$ of $\mathrm{GL}_{2}(\mathbb{A}_{F})$,
\begin{equation}\label{eq10.field}
L(1/2,\pi)\ll_{\varepsilon}D_{F}^{A_{\varepsilon}}C(\pi)^{\frac{1}{4}-\frac{1}{120}+\varepsilon}.
\end{equation}
\end{thm}
\begin{proof}
At the central parameters, the factor $D_{F}^{-4}$ from the finite measures cancels between the two spectral sides; the shifts $s_{j}$ contribute $D_{F}^{O(\varepsilon)}$. Taking residues cancels $\rho_{F}^{-1}$ in \eqref{eq6.euler}; three remaining Tate factors near zero give $D_{F}^{3/2+O(\varepsilon)}$ by their global functional equations. The normalized geometric terms in \textsection\ref{sec9} give the same power of $D_{F}$. Expressing the coefficient of the dual moment in terms of full conductors gives a factor $D_{F}^{-1-\vartheta}\leq1$.

Put $d=[F:\mathbb{Q}]$ and take $\delta>0$ sufficiently small in terms of $\varepsilon$. Use the same fixed cutoff function $\alpha_{v}$ from \textsection\ref{sec4.4} at every Archimedean place, independently of $F$. The differentiation orders are independent of $d$, so the local constants multiply to at most $B_{\delta}^{d}$, absorbed into $D_{F}^{A_{\delta}}$ by Minkowski's bound; the different contributes $D_{F}^{O(1)}$. In \textsection\ref{sec6}, \textsection\ref{sec8}--\textsection\ref{sec9}, only a fixed number of global poles remain, so the contour radii can be bounded below in terms of $\delta$ alone. In the proof of \cite{Yan23c}, use fixed smooth cutoffs in the Archimedean multiplicative integrals. Counting cosets using the product formula and summing geometric series preserve the conductor exponents, apart from a factor $B_{\delta}^{d}$ and arbitrarily small powers of the conductors. For the inner amplifier, enlarging the admissible length by a power of $D_{F}$ gives the sum of the two bounds in \eqref{10.23}, for all conductor ranges, up to $D_{F}^{A_{\delta}}$ and a fixed power of $C_{\infty}(\sigma)$.

The class number formula and \cite{Fri89} bound the constants relating the global and local norms by powers of $D_{F}$; the argument of \cite[Lemmas~3--4 and the proof of Theorem~2]{Li10} gives $L(1,\pi,\Ad)\ll_{\delta}B_{\delta}^{d}C(\pi)^{\delta}$. The dual representations $\sigma$ are unramified at infinity, have trivial central character, and are ramified only at amplifier primes. For their adjoint lower bounds, combine the zero-free argument of \cite{RW03} with Page's theorem \cite{Kad12}. If the exceptional character is ramified, omit an amplifier prime at which it is ramified; if it is unramified at every finite place, use residues of Dedekind zeta functions. The logarithmic derivatives of the gamma factors are $O(\log C(\sigma)+d)$, and the lower bounds for the Euler products use fixed powers of $\zeta_{F}$.

For the fourth and eighth moments in \eqref{10.5} and Lemma~\ref{lem10.2}, apply reciprocity with $\mathfrak{n}=\mathcal{O}_{F}$. The dual spectrum has level one and is bounded by convexity and the full Kuznetsov formula \cite[\textsection 3.4]{Mag13}, keeping the adjoint and Eisenstein denominators. Count ideal cosets by the product formula; the factor involving the class number cancels the sum over ideal classes \cite[(3.19) and (3.21)]{Qi24}. This introduces only $D_{F}^{A_{\delta}}B_{\delta}^{d}$ and fixed polynomial powers of the Archimedean parameters. Since $1/2+\vartheta<1$, the level exponent in the proof of Theorems~\ref{thmE} and~\ref{thmF} remains $1+O(\delta)$. Summing against the exponential weight at each place costs at most $B_{\delta}^{d}$.

For the twisted factors in Lemma~\ref{lem10.2}, follow the proof of \cite[Corollary~1.2]{Yan23c} for the Eisenstein representation $\mathbf{1}\boxplus\mathbf{1}$. The local tests can be chosen uniformly near $t=0$. Comparing the central value with the integral over a short interval, keeping $|\zeta_{F}(1+2it)|^{-2}$, uses derivatives of order depending only on $\delta$, not on $d$. Thus this step also costs only $D_{F}^{A_{\delta}}$ and $O(\delta)$ conductor powers.

Finally, \cite[Theorem~1.3.1]{Zam17} gives sufficiently many amplifier primes in a suitable interval $[L,2L]$. The two auxiliary primes used in the relative trace formula of \cite{Yan23c} contribute only a power of $D_{F}$ and arbitrarily small conductor powers. Choose the interval before omitting a prime as above. We may therefore take
\begin{align*}
D_{F}^{B_{\delta}}C(\pi)^{\frac{1}{30}+\delta}\ll_{\delta}L\ll_{\delta}D_{F}^{B'_{\delta}}C(\pi)^{\frac{1}{30}+\delta}.
\end{align*}
Up to small conductor powers, the error terms in Remark~\ref{rmk10.degenerate} contribute at most $D_{F}^{A_{\delta}}C(\pi)^{1/2}L^{2}$ after amplification and are absorbed by the last term below. Using both bounds in \eqref{10.23}, the proof of \eqref{10.35} gives
\begin{align*}
|L(1/2,\pi)|^{4}\ll_{\delta}D_{F}^{A_{\delta}}(C(\pi)L)^{O(\delta)}\big[C(\pi)L^{-1}+C(\pi)^{7/8}L^{11/4}+C(\pi)^{3/4}L^{9/2}\big].
\end{align*}
Substituting the lower bound for $L$ in the first term and the upper bound in the other two gives conductor exponents $29/30$, $29/30$, and $9/10$, up to $O(\delta)$. Combine the discriminant powers into $D_{F}^{A_{\delta}}$, rescale $\delta$, and take fourth roots to obtain \eqref{eq10.field}.
\end{proof}

\subsection{A Prime-Level Fifth Moment}\label{sec10.fifth}
\begin{proof}[Proof of Corollary \ref{cor1.fifth}]
It suffices to take $0<\varepsilon<1/4$. Let $\beta_{\infty}$ be the positive Archimedean fourth-moment weight for the spherical data \eqref{f4.6}, including the fixed global normalization. Write the spherical parameters as $\boldsymbol{\nu}=(\nu_{v})_{v\mid\infty}$, so that $L_{v}(s,\pi_{v})=\Gamma_{\mathbb{R}}(s+\nu_{v})\Gamma_{\mathbb{R}}(s-\nu_{v})$ at real places, with $\Gamma_{\mathbb{R}}$ replaced by $\Gamma_{\mathbb{C}}$ at complex places. Set $\iota_{v}=2$ at real places and $\iota_{v}=1$ at complex places. Choose a sufficiently wide fixed strip in each $\nu_{v}$ containing the unitary spherical dual. For a sufficiently large fixed integer $m$, define the polynomial 
\begin{align*}
P_{\boldsymbol{\nu}}(z):=\prod_{v\mid\infty}\prod_{j=0}^{m}\prod_{\epsilon_{1},\epsilon_{2}\in\{\pm1\}}(1/2+\epsilon_{1}\nu_{v}+\epsilon_{2} z+\iota_{v}j).
\end{align*}

Take $h_{\infty}=c\,\beta_{\infty} e^{\sum_{v}\nu_{v}^{2}}P_{\boldsymbol{\nu}}(0)$. This weight is positive on the spherical cuspidal spectrum; choose $c>0$ so that $h_{\infty}\geq1$ on the cuspidal spectrum in $\Omega_{\infty}$. All weights in this proof are extended by zero off the spherical spectrum. Since $\pi$ has trivial central character, $L(1/2,\pi)\geq0$ by \cite[Theorem~1]{LR03}. Together with $L(1,\pi,\Ad)\ll_{F,\varepsilon}C(\pi)^{\varepsilon}$, this reduces the proof to bounding
\begin{align}\label{10.fifth.1}
\sum_{\pi\in\mathcal{F}_{\mathrm{new}}(\mathfrak{p},\mathbf{1})}
\frac{L(1/2,\pi)^{5}}{L(1,\pi,\Ad)}h_{\infty}(\pi_{\infty}).
\end{align}

Apply the self-dual approximate functional equation as in \cite[\textsection 12.1]{BK19}, using the even function $e^{z^{2}}P_{\boldsymbol{\nu}}(z)/P_{\boldsymbol{\nu}}(0)$, whose value at $0$ is $1$; representations with root number $-1$ contribute zero. Choose $m$ so that $P_{\boldsymbol{\nu}}(z)$ cancels all numerator Gamma poles in the chosen strip for $\Re(z)\geq\varepsilon$, with multiplicity. For $\tau\in\mathbb{R}$, the resulting weights are
\begin{align*}
h_{\infty,\tau}(\pi_{\infty})=c\,\beta_{\infty}e^{\sum_{v}\nu_{v}^{2}+(\varepsilon+i\tau)^{2}}P_{\boldsymbol{\nu}}(\varepsilon+i\tau)
\frac{L_{\infty}(1/2+\varepsilon+i\tau,\pi)}{L_{\infty}(1/2,\pi)}.
\end{align*}
Stirling's formula and spherical inversion \cite[Theorem~1 and p.~348]{Ank91} give $h_{\infty,\tau}=\beta_{\infty}\prod_{v\mid\infty}\widehat{\kappa}_{v}$, where $\kappa_{v}$ is smooth and bi-$K_{v}$-invariant and the hat denotes its spherical transform. The products of the weighted $L^{1}$ norms of $\kappa_{v}$ and their derivatives used below are $\ll_{F,\Omega_{\infty},\varepsilon}e^{-\tau^{2}/2}$.

For $\mathfrak{n}+\mathfrak{p}=\mathcal{O}_{F}$, we have
\begin{multline}\label{10.fifth}
\bigg|\sum_{\pi\in\mathcal{F}_{\mathrm{new}}(\mathfrak{p},\mathbf{1})}
\frac{\lambda_{\pi}(\mathfrak{n})L(1/2,\pi)^{4}}{L(1,\pi,\Ad)}h_{\infty,\tau}(\pi_{\infty})\bigg|\\
\ll_{F,\Omega_{\infty},\varepsilon}e^{-\tau^{2}/2}
\left[N_{F}(\mathfrak{p})N_{F}(\mathfrak{n})^{-\frac{1}{2}}
+N_{F}(\mathfrak{p})^{\frac{1}{2}+\vartheta}N_{F}(\mathfrak{n})^{\frac{1}{2}}\right]
\big[N_{F}(\mathfrak{p})N_{F}(\mathfrak{n})\big]^{\varepsilon}.
\end{multline}

Primes dividing $\mathfrak{D}_{F}$ are covered by Theorem \ref{thmE} and $|\lambda_{\pi}(\mathfrak{n})|\ll_{F,\varepsilon}N_{F}(\mathfrak{n})^{\vartheta+\varepsilon}$, after summing the rapidly decreasing Archimedean weight. Thus assume $\mathfrak{p}\nmid\mathfrak{D}_{F}$. In the datum of \textsection\ref{sec4.4}, take $\omega=\mathbf{1}$ and $\mathfrak{q}=\mathfrak{p}$. At every $v\mid\infty$, start with spherical data, simultaneously translate the two sections in the first Rankin--Selberg period by $x_{v}$, and integrate against $\kappa_{v}(x_{v})$. Moving the translation to the automorphic representation multiplies the original spectral weight by $\widehat{\kappa}_{v}$, giving $h_{\infty,\tau}$. At $v\mid\mathfrak{D}_{F}$, use the unramified character $\psi_{v}(\varpi_{v}^{-d_{v}}\,\cdot)$ as in the calculation of \eqref{eq7.15}; the Hecke calculation is unchanged, up to factors depending only on $F$.

On the dual side, use unit Kirillov vectors and divide the finite local periods by their Euler factors as in \eqref{e5.9}. The fixed smooth Archimedean data give the decay needed for Theorems \ref{thmE} and \ref{thmF}. For the residual and degenerate terms, use Proposition~\ref{prop3.residual-integrals} and \eqref{9.21} on the contour of \textsection\ref{sec4.regularization}. These estimates grow polynomially under translation, so integration against $\kappa_{v}$ is justified. Combining this with the finite-place estimates of \textsection\ref{sec6}--\textsection\ref{sec9} gives the right side of \eqref{10.fifth} for the cuspidal and continuous spectra together, weighted by $\lambda_{\pi}^{*}(\mathfrak{n})$.

For $\pi$ of conductor $\mathfrak{p}$ and trivial central character, $\pi_{\mathfrak{p}}$ is an unramified twist of $\mathrm{St}$. For its newvector $W$ normalized by $W(I_{2})=1$, \eqref{equa4.5} gives
\begin{align*}
\Psi_{\mathfrak{p}}(\overline{W},W_{3,\mathfrak{p}},h_{4,\mathfrak{p}})=L_{\mathfrak{p}}(1/2,\pi)^{2},\qquad \|W\|^{2}=\zeta_{F_{\mathfrak{p}}}(2).
\end{align*}
The Rankin--Selberg norm formula \cite[\textsection 2.2.2]{MV10} therefore gives the coefficient $1+N_{F}(\mathfrak{p})^{-1}$ for every representation of exact conductor $\mathfrak{p}$. Divide the identity by this factor. For representations unramified at $\mathfrak{p}$, the local period calculation and oldform orthogonalization (compare \cite[\textsection 3.2]{BK19}) give a uniformly bounded coefficient, since $\vartheta\leq7/64$. Theorem \ref{thmE}, with rapid Archimedean decay, bounds their total contribution by $O_{F,\Omega_{\infty},\varepsilon}(e^{-\tau^{2}/2}N_{F}(\mathfrak{n})^{\vartheta+\varepsilon})$. The continuous spectrum here is unramified, since its conductor exponent at $\mathfrak{p}$ is even. Theorem \ref{thmF} gives the same bound for its contribution. Both are absorbed by the second term in \eqref{10.fifth}. Finally,
\begin{align*}
\lambda_{\pi}(\mathfrak{n})=\sum_{\mathfrak{d}^{2}\mid\mathfrak{n}}N_{F}(\mathfrak{d})^{-1}\lambda_{\pi}^{*}(\mathfrak{n}\mathfrak{d}^{-2}),
\end{align*}
and the divisor bound converts the estimate for $\lambda_{\pi}^{*}$ into \eqref{10.fifth}.

The terms with $N_{F}(\mathfrak{n})>N_{F}(\mathfrak{p})^{1/2+\varepsilon}$ contribute $O_{F,\Omega_{\infty},\varepsilon}(N_{F}(\mathfrak{p})^{-100})$, as in \cite[\textsection 12.1]{BK19}. The remaining ideals are coprime to $\mathfrak{p}$. Applying \eqref{10.fifth} and integrating in $\tau$ gives
\begin{align*}
\eqref{10.fifth.1}\ll_{F,\Omega_{\infty},\varepsilon}N_{F}(\mathfrak{p})^{1+O(\varepsilon)}\left[\log(2N_{F}(\mathfrak{p}))+N_{F}(\mathfrak{p})^{\vartheta}\right]\ll_{F,\Omega_{\infty},\varepsilon}N_{F}(\mathfrak{p})^{1+\vartheta+O(\varepsilon)}.
\end{align*}
Rescaling $\varepsilon$ proves \eqref{1.fifth}.
\end{proof}

\subsection{Proof of Theorem \ref{thmeC}}
Let $X>1$ and $V$ be a fixed nonzero smooth non-negative function compactly supported on $[1/2,9/10]$. Let $\eta$ be a character of $\mathrm{Pic}(\mathcal{O})$, and let $\mathfrak{f}\subseteq\mathcal{O}_{E}$ be the conductor of $\mathcal{O}$. All ideal sums below are restricted to ideals coprime to $\mathfrak{f}$; extension to $\mathcal{O}_{E}$ preserves their norms and identifies them with $\mathcal{O}_{E}$-ideals coprime to $\mathfrak{f}$. Write $L^{(\mathfrak{f})}(s,\eta)$ for the Hecke $L$-function with the Euler factors at primes dividing $\mathfrak{f}$ omitted, and put $\rho=\Res_{s=1}L^{(\mathfrak{f})}(s,\mathbf{1})$. Mellin inversion gives
\begin{align*}
\sum_{\substack{\mathfrak{a}\subseteq \mathcal{O}\\ \mathfrak{a}\ {\rm invertible}}}\eta([\mathfrak{a}])V(X^{-1}N(\mathfrak{a}))=\frac{1}{2\pi i}\int_{(10)}L^{(\mathfrak{f})}(s,\eta)\widehat{V}(s)X^{s}ds,
\end{align*}
where $\widehat{V}$ is the Mellin transform of $V$. Shifting the contour to $\Re(s)=1/2$ gives
\begin{multline}\label{10.50}
\sum_{\substack{\mathfrak{a}\subseteq \mathcal{O}\\ \mathfrak{a}\ {\rm invertible}}}\eta([\mathfrak{a}])V(X^{-1}N(\mathfrak{a}))=\mathbf{1}_{\eta=\mathbf{1}}X\widehat{V}(1)\rho\\
+\frac{1}{2\pi i}\int_{(1/2)}L^{(\mathfrak{f})}(s,\eta)\widehat{V}(s)X^{s}ds.
\end{multline}

Take $X=N(\mathcal{O};\chi)$ and $\eta=\chi\neq\mathbf{1}$. Every ideal contributing to the sum has norm less than $X$, so $\chi([\mathfrak{a}])=1$. Thus \eqref{10.50} gives
\begin{equation}\label{10.51}
\sum_{\substack{\mathfrak{a}\subseteq \mathcal{O}\\ \mathfrak{a}\ {\rm invertible}}}V(X^{-1}N(\mathfrak{a}))=
\frac{1}{2\pi i}\int_{(1/2)}L^{(\mathfrak{f})}(s,\chi)\widehat{V}(s)X^{s}ds.
\end{equation}

Comparing \eqref{10.51} with \eqref{10.50} in the case $\eta=\mathbf{1}$ gives 
\begin{equation}\label{10.52}
N(\mathcal{O};\chi)\cdot \widehat{V}(1)\rho=\frac{1}{2\pi i}\int_{(1/2)}\big[L^{(\mathfrak{f})}(s,\chi)-L^{(\mathfrak{f})}(s,\mathbf{1})\big]\widehat{V}(s)X^{s}ds.
\end{equation}

Since $N_{E}(\mathfrak{f})\leq D$, the omitted Euler factors contribute $\ll_{F,\varepsilon}D^{\varepsilon}$ on $\Re(s)=1/2$. Siegel's theorem gives
\begin{align*}
\rho=\underset{s=1}{\Res}\,\zeta_{E}(s)\prod_{\mathfrak{P}\mid\mathfrak{f}}\big(1-N_{E}(\mathfrak{P})^{-1}\big)\gg_{F,\varepsilon}D^{-\varepsilon}.
\end{align*}
Corollary \ref{cor1.3}, the bound for the omitted Euler factors, and the rapid decay of $\widehat{V}$ give, for every $\eta$,
\begin{equation}\label{10.54}
\frac{1}{2\pi i}\int_{(1/2)}L^{(\mathfrak{f})}(s,\eta)\widehat{V}(s)X^{s}ds\ll_{F,\varepsilon,V} X^{\frac{1}{2}}D^{\frac{1}{4}-\frac{1}{120}+2\varepsilon}.
\end{equation}

Using \eqref{10.54} in \eqref{10.52} and the lower bound for $\rho$, we obtain, after rescaling $\varepsilon$, the bound $N(\mathcal{O};\chi)\ll_{F,\varepsilon} D^{\frac{1}{2}-\frac{1}{60}+\varepsilon}$, with an ineffective constant.

For the second assertion, character orthogonality and M\"obius inversion, as in \cite[\textsection 5.2]{Mic07}, expand the indicator of the generators of $H$ in characters of $\mathrm{Pic}(\mathcal{O})$. The sum of the absolute values of the coefficients is at most the number of divisors of $|H|$. Since $|H|\leq|\mathrm{Pic}(\mathcal{O})|\ll_{F,\varepsilon}D^{1/2+\varepsilon}$, this divisor factor is $\ll_{F,\varepsilon}D^{\varepsilon}$. Applying \eqref{10.50} and \eqref{10.54}, and rescaling $\varepsilon$, gives
\begin{equation}\label{eq10.picard.generators}
\sum_{\substack{\mathfrak{a}\subseteq\mathcal{O}\ {\rm invertible}\\ \langle[\mathfrak{a}]\rangle=H}}V(X^{-1}N(\mathfrak{a}))=\frac{\varphi(|H|)}{|\mathrm{Pic}(\mathcal{O})|}X\widehat{V}(1)\rho+O_{F,\varepsilon}\big(X^{1/2}D^{1/4-1/120+\varepsilon}\big).
\end{equation}
Here $\varphi$ is Euler's totient function. Since $\varphi(|H|)/|H|\gg_{F,\varepsilon}D^{-\varepsilon}$ and $\rho\gg_{F,\varepsilon}D^{-\varepsilon}$, a rescaling of $\varepsilon$ makes \eqref{eq10.picard.generators} positive for $X$ a sufficiently large multiple of $[\mathrm{Pic}(\mathcal{O}):H]^{2}D^{1/2-1/60+\varepsilon}$. This proves Theorem \ref{thmeC}.

\bibliographystyle{alpha}

\bibliography{LY}

\end{document}